\documentclass[12pt, 
  oneside]{amsart}

\usepackage{amscd, amsmath, colonequals}

\usepackage[usenames]{xcolor}

\usepackage[hyperfootnotes=false]{hyperref}

\title[]{Real submanifolds of maximum complex tangent space at a CR singular point}
\author[]{Xianghong Gong}

\address{Department of Mathematics,
 University of Wisconsin, Madison, WI 53706, U.S.A.}
 \email{gong@math.wisc.edu}

\author{Laurent Stolovitch}
\address{CNRS and Laboratoire J.-A. Dieudonn\'e
U.M.R. 6621, Universit\'e de Nice - Sophia Antipolis, Parc Valrose
06108 Nice Cedex 02, France.}
\email{stolo@unice.fr}
\thanks{Research of L. Stolovitch was supported by ANR grant ``ANR-10-BLAN 0102'' for the project DynPDE}


\newtheorem{thm}{Theorem}[section]
\newtheorem{cor}[thm]{Corollary}
\newtheorem{prop}[thm]{Proposition}
\newtheorem{lemma}[thm]{Lemma}

\newcommand{\diag}{\operatorname{diag}}

\theoremstyle{definition}

\newtheorem{defn}[thm]{Definition}
\newtheorem{exmp}[thm]{Example}

\newtheorem{rem}[thm]{Remark}

\renewcommand{\th}[1]{\begin{thm}\label{#1}}
\newcommand{\eth}{\end{thm}}
\newcommand{\co}[1]{\begin{cor}\label{#1}}
\newcommand{\eco}{\end{cor}}
\renewcommand{\le}[1]{\begin{lemma}\label{#1}}
\newcommand{\ele}{\end{lemma}}
\newcommand{\pr}[1]{\begin{prop}\label{#1}}
\newcommand{\epr}{\end{prop}}

\newcommand{\ga}{\begin{gather}}
\newcommand{\ega}{\end{gather}}
\newcommand{\gan}{\begin{gather*}}
\newcommand{\egan}{\end{gather*}}
\newcommand{\al}{\begin{align}}
\newcommand{\eal}{\end{align}}
\newcommand{\aln}{\begin{align*}}
\newcommand{\ealn}{\end{align*}}
\newcommand{\eq}[1]{\begin{equation}\label{#1}}
\newcommand{\eeq}{\end{equation}}

\newcommand{\ci}{~\cite}

\newcommand{\f}[2]{\frac{#1}{#2}}

\newcommand{\fix}{\operatorname{Fix}}

\newcommand{\cc}{{\bf C}}
\newcommand{\nn}{{\bf N}}
\newcommand{\zz}{{\bf Z}}
\newcommand{\rr}{{\bf R}}
\newcommand{\qq}{{\bf Q}}

\newcommand{\ov}{\overline}
\newcommand{\ord}{\operatorname{ord}}

\newcommand{\id}{\operatorname{I}}

\newcommand{\RE}{\operatorname{Re}}
\newcommand{\IM}{\operatorname{Im}}

\newcommand{\cL}{\mathcal}

\newcommand{\I}{\operatorname{I}}

\newcommand{\all}{\alpha}

\newcommand{\gaa}{\gamma}

\newcommand{\del}{\delta}
\newcommand{\Del}{\Delta}
\newcommand{\var}{\varphi}
\newcommand{\e}{\epsilon}
\newcommand{\om}{\omega}
\newcommand{\Om}{\Omega}

\newcommand{\la}{\lambda}

\newcommand{\pd}{\partial}

\newcommand{\yt}{\frac{1}{2}}

\newcommand{\re}[1]{(\ref{#1})}
\newcommand{\rea}[1]{$(\ref{#1})$}
\newcommand{\rl}[1]{Lemma~\ref{#1}}
\newcommand{\nrc}[1]{Corollary~\ref{#1}}
\newcommand{\rp}[1]{Proposition~\ref{#1}}
\newcommand{\rt}[1]{Theorem~\ref{#1}}

\newcommand{\rrem}[1]{Remark~\ref{#1}}

\newcommand{\rla}[1]{Lemma~$\ref{#1}$}

\newcommand{\rpa}[1]{Proposition~$\ref{#1}$}
\newcommand{\rta}[1]{Theorem~$\ref{#1}$}

\newcounter{pp}
\newcommand{\bpp}{\begin{list}{$\hspace{-1em}\alph{pp})$}{\usecounter{pp}}}
\newcommand{\epp}{\end{list}}

\newcounter{ppp}
\newcommand{\bppp}{\begin{list}{$\hspace{-1em}(\roman{ppp})$}{\usecounter{ppp}}}
\newcommand{\eppp}{\end{list}}

\def\beq{\begin{equation}}
\def\eeq{\end{equation}}

\begin{document}

 
\begin{abstract} We study a germ of  real analytic $n$-dimensional submanifold  of ${\mathbf C}^n$ that has a complex tangent space of maximal dimension at a CR singularity. 
Under the condition that its complexification admits the maximum number of deck transformations, 
we study its transformation to a normal form under the action of local (possibly formal) biholomorphisms at the singularity. 
We first conjugate formally its associated reversible map $\sigma$ to suitable normal forms and show that all these normal forms can be divergent. If the singularity is {\it abelian}, we show, under some assumptions on the linear part of $\sigma$ at the singularity, that the real submanifold is holomorphically equivalent to an analytic normal form. 
We also show that if a real submanifold is formally equivalent 
 to a quadric,  it is actually holomorphically equivalent to it,  if a small divisors  condition is satisfied. Finally, we prove that, in general, there exists a complex submanifold of positive dimension in ${\mathbf C}^n$  that intersects 
 a real submanifold along two totally and real analytic submanifolds that intersect transversally at  a   CR singularity of the {\it complex type}. 
\end{abstract}

\date{\today}
 \maketitle

\tableofcontents

\setcounter{section}{0}
\setcounter{thm}{0}\setcounter{equation}{0}
\section{Introduction and main results}

\subsection{Introduction}
We are concerned with the local holomorphic invariants of a real analytic submanifold $M$ in $\cc^n$.
The tangent space of $M$ at a point $x$ contains a maximal complex subspace of dimension $d_{x}$. When this dimension is constant,
$M$ is called a  Cauchy-Riemann (CR) submanifold. The CR  submanifolds have been extensively studied since E.~Cartan. 
 For the analytic real hypersurfaces in $\cc^n$ with a non-degenerate Levi-form, 
 the normal form problem has a complete 
  theory   achieved through the works of E.~Cartan~\cite{Ca32}, \cite{Ca33},  Tanaka~\cite{Ta62}, and Chern-Moser~\cite{chern-moser}.  In another direction, the  relations between formal and holomorphic equivalences
   of real analytic hypersurfaces have been investigated by  Baouendi-Ebenfelt-Rothschild ~\cite{BER97}, ~\cite{BER00}, Baouendi-Mir-Rothschild~\cite{BMR02}, Juhlin-Lamel~\cite{JL13}, where  positive results were obtained. 
 In a recent preprint, Kossovskiy and  Shafikov~\cite{KS13} showed that there are real analytic real hypersurfaces
 which are formally but not holomorphically equivalent.    
    
We say that a point $x_0$ in
 a  real submanifold $M$ in $\cc^n$ is a complex tangent, or a CR singularity, 
 if the complex tangent spaces $T_xM\cap J_xT_xM$  do not 
have a constant dimension in any neighborhood of $x_0$. A real submanifold with a CR singularity must have codimension at least $2$. 
The study of real submanifolds with CR singularities was initiated by E.~Bishop in his pioneering work~\cite{Bi65},
when the complex tangent space of $M$ at a CR singularity is minimal, that is exactly one-dimensional. 
 The very elementary models of this kind of manifolds are   classified as  certain quadrics which depend on one    non-negative number, 
the Bishop invariant. 
They  are  the Bishop quadrics, given by
$$
Q\subset\cc^2\colon z_2=|z_1|^2+\gaa(z_1^2+\ov z_1^2), \quad 0\leq\gaa<\infty.
$$
The origin is a complex tangent which is said to be {\it elliptic} if $0\leq\gaa<1/2$,  {\it parabolic} if $\gaa=1/2$, or {\it hyperbolic} if $\gaa>1/2$. 

In ~\cite{MW83}, Moser and Webster studied the normal form problem of  a real analytic surface 
$M$ in $\cc^2$
which is the higher order perturbation of $Q$. They showed that
 when $0<\gaa<1/2$, $M$ is holomorphically equivalent to a normal form which
is  an algebraic surface that depends only on $\gaa$ and two discrete invariants. They also constructed a 
formal normal form of $M$ when the origin is a    non-exceptional hyperbolic complex tangent point;
 although the normal form is still convergent, they showed that the normalizing transformation is divergent in general for the hyperbolic case. 
We  mention that  the Moser-Webster normal form  theory, as in Bishop's work,
actually deals with an $n$-dimensional 
real submanifold $M$ in $\cc^n$,  of which 
the complex tangent space has   (minimum) dimension $1$
at a CR singularity. 

The main purpose of this work is to investigate an  $n$-dimensional real analytic  submanifold $M$ in $\cc^n$ of which the complex tangent space has the 
{\it largest}
possible 
dimension at a  given CR singularity. The dimension must be $p=n/2$. Therefore,  $n=2p$ is even. We are interested in 
the geometry,  the analytic classification, and  the normal form problem of such real analytic manifolds. 

In  suitable holomorphic coordinates, a $2p$-dimensional real analytic submanifold $M$ in $\cc^{2p}$
 that has a complex tangent space of maximum dimension at the origin  is given by 
\eq{mzpjintr}
M\colon z_{p+j}=E_j(z',\ov z'),
\quad 1\leq j\leq p,
\eeq
where $z'=(z_1,\ldots, z_p)$ and
$$
E_j(z',\ov{z'})=h_j(z',\ov z')+q_j(\ov z') + O(|(z',\ov z')|^3). 
$$
Moreover,  each  $h_j(z',\ov z')$ is a homogeneous quadratic polynomial in $z',\ov z'$ without holomorphic or anti-holomorphic terms,
  and each $q_j(\ov z')$ is a homogeneous quadratic polynomial in $\ov z'$. 
  One of our goals is to seek suitable normal forms 
  of perturbations of quadrics at the CR singularity (the origin). 

The study of these kind of real submanifolds, with $p>1$, was initiated in \cite{St07} by the second-named author. 
 
 \subsection{Basic invariants}
To study $M$, we consider its complexification in $\cc^{2p}\times\cc^{2p}$ defined by
\begin{equation} 
{\mathcal M}\colon
\begin{cases}
z_{p+i} = E_{i}(z',w'), & i=1,\ldots, p,
\\
w_{p+i} = \bar E_i(w',z'),& i=1,\ldots, p.\\
\end{cases}
\nonumber
\end{equation}
 It is a complex submanifold of complex
dimension $2p$ with coordinates $(z',w')\in\cc^{2p}$.  Let $\pi_1,\pi_2$ be the restrictions of the projections $(z,w)\to z$
and   $(z,w)\to w$ to $\cL M$, respectively.
Note that $\pi_2=\rho_0\pi_1\rho_0$, where  $\rho_0$ is the restriction to $\cL M$ of the anti-holomorphic involution $(z,w)\to(\ov w,\ov z)$. 

  Our basic assumption is the following condition.

\medskip
\noindent
{\bf Condition B.}  $q(z')=(q_1(z'),\ldots, q_p(z'))$ satisfies $q^{-1}(0)=\{0\}.$
\medskip

 Let us first describe the significance of condition B. 
When $p=1$ this corresponds to the case that the Bishop invariant $\gaa$ of $M$
 at the origin does not vanish.  
  When $\gaa=0$,  Moser \cite{moser-zero} obtained a  formal  normal form that is still subject to
  further  formal changes of coordinates. 
In \cite{HY09},  Huang and Yin obtained a formal normal form with a complete set of  formal holomorphic invariants of $M$ when  $\gaa=0$. 
   They used their formal normal form
 to show that  two such real analytic surfaces are holomorphically equivalent 
  if and only if they have the same formal normal form. 
    The formal normal forms for co-dimension two real submanifolds  in $\cc^n$ have been further studied  
  by Huang-Yin~\cite{HY12} and Burcea~\cite{Bu13}.
 Note that by a rapid iteration method,  Coffman~\cite{Co06} showed that  any  
 $m$ dimensional real analytic submanifold in $\cc^n$ of  one-dimensional complex tangent
 space at  a  CR singularity satisfying certain non-degeneracy conditions 
 is locally holomorphically equivalent to a unique algebraic submanifold,
 provided $2(n+1)/3\leq m<n$.

	
When $M$ is a {\it quadric}, i.e. each $E_j$ in \re{mzpjintr} is a quadratic  polynomial, our basic condition~B
is equivalent to $\pi_1$ being a $2^p$-to-1 branched covering. 
Since $\pi_2=\rho_0\pi_1\rho_0$, then $\pi_2$ is also a $2^p$-to-$1$ branched covering.
We will see that the CR singularities of the real submanifolds are closely connected with these branched coverings and their deck transformations.


We now introduce our main results. Some  of them are analogous to the Moser-Wester theory. We will underline  major differences which arise with $p>1$. 

\subsubsection{Branched covering and deck transformations} 
In section~\ref{secinv}, we study the existence of deck transformations for $\pi_1$. 
We will show that they must be involutions and they
commute pairwise. We show that they form a group of order $2^k$ for some $0\leq k\leq p$. This is a major difference between the real submanifolds with  one  dimensional complex tangent space at a 
CR singularity
and the ones with maximum complex tangent space, when $p>1$.  Indeed, we recall that 
in  the Moser-Webster theory, the branched covering $\pi_1$ is $2$-to-$1$ and consequently the group of deck transformations of $\pi_1$ has order $2$.
The group is then generated by a unique involution $\tau_1$. 

In this paper, we will focus on the case where the group of deck transformations of $\pi_1$ has the maximum order $2^p$.  Thus, we will
impose the following condition.

\medskip
\noindent
{\bf Condition D.}  {\it $M$ satisfies condition $B$ and the branched   covering   $\pi_1$ 
of $\cL M$
admits the maximum $2^p$ deck transformations.}
\medskip

Condition D gives rise to two families of commuting involutions $\{\tau_{i1},\ldots, \tau_{i2^p}\}$
intertwined by the anti-holomorphic
involution $\rho_0\colon(z',w')\to(\ov w',\ov z')$ such that $\tau_{2j}=\rho_0\tau_{1j}\rho_0$ $(1\leq j\leq 2^p)$ are deck
transformations of $\pi_2$.    We will call $\{\tau_{11},\ldots, \tau_{12^p},\rho_0\}$ the set of {\it Moser-Webster involutions}. 
We will show that there is a unique set of $p$ generators for the deck transformations  of $\pi_1$, denoted by $\tau_{11},\ldots, \tau_{1p}$, which are characterized by
the property that each $\tau_{1j}$   fixes a hypersurface in $\cL M$ pointwise.  Then
$$\tau_1=\tau_{11}\circ\cdots\circ\tau_{1p}$$
 is the unique deck transformation of which the set of fixed-points has the smallest dimension $p$. 
 Let  $\tau_2=\rho_0\tau_1\rho_0$ and 
  $$
 \sigma=\tau_1\tau_2.
 $$
 Then $\sigma$ is {\it reversible} by $\tau_j$ and $\rho_0$, i.e. 
 $\sigma^{-1}=\tau_j\sigma\tau_j^{-1}$ and $\sigma^{-1}=\rho_0\sigma\rho_0$.

 As in the Moser-Webster theory, the existence of such $2^p$ deck transformations  allows us to transfer the normal form problem for the real submanifolds into the normal form problem for the sets
  of involutions $\{\tau_{11},\ldots,\tau_{1p},
  \rho_0\}$.

In this paper we will make the following assumption. 
 
 \medskip
 \noindent
{\bf Condition E.} {\it   $M$ satisfies condition D and $M$ has distinct eigenvalues, while
the latter means that   $\sigma$ has
  $2p$ distinct eigenvalues.}
  
\medskip

  
Note that the condition excludes the higher dimensional analogous complex tangency of {\it parabolic} type, i.e. of $\gaa=1/2$. 
The normal form problem for the parabolic complex tangents has been studied by Webster~\cite{We92}, and in~\cite{Go96} where the normalization is divergent in general.  In~\cite{AG09}, Ahern and Gong constructed a moduli space for real analytic submanifolds that are formally equivalent to the Bishop quadric with $\gamma=1/2$.

\medskip

 We now introduce our main results. 
 
Our first step is to normalize   $\{\tau_1,\tau_2, 
\rho_0\}$.  When $p=1$, this normalization is   the main step in order to obtain the Moser-Webster normal form; in fact a simple further normalization allows Moser and Webster to achieve a convergent normal form under a suitable non-resonance condition even for the  non-exceptional  hyperbolic complex tangent. 

When $p>1$, we need to carry out a further normalization for $\{\tau_{11},\ldots, \tau_{1p},\rho_0\}$; this is our second step. Here the normalization has a large degree of freedom  as shown by  our  formal and convergence results.
 
In sections $2$ through $7$, we will study the formal normal forms and the relations on the
convergence of normalizations in these two steps. 
Let us first describe main results in these sections.

\subsubsection{Normal forms of quadrics with the maximum  number of deck transformations}

The basic model for quadric manifolds with such a CR singularity is a product of  3 types of quadrics defined by
\begin{gather*}
 Q_{\gamma_e}\subset\cc^2 \colon z_{2}= (z_1+2\gaa_e\ov z_1)^2;\\
 Q_{\gamma_h}\subset\cc^2\colon z_{2}= (z_{1}+2\gaa_{h}\ov z_1)^2;\\
 Q_{\gamma_s}\subset\cc^4\colon z_{3}= (z_1+2\gamma_s\ov z_{2})^2,
\quad z_{4}=( z_{2}+2
(1-\ov\gamma_{s})  \ov z_{1})^{2}.
\end{gather*}
Here
\eq{0ge1}
0<\gamma_e<1/2, \quad
1/2<\gamma_h<\infty, \quad \RE
\gaa_s>1/2, \quad  \IM\gaa_s>0.
\eeq
Note that $Q_{\gaa_e}, Q_{\gaa_h}$ are elliptic and hyperbolic Bishop quadrics, respectively.  Realizing 
a type of pairs
 of involutions  introduced in~\cite{St07}, we will say that the complex tangent of $Q_{\gaa_s}$ at the origin is {\it complex}.  We emphasize that this last type of quadric 
 is new 
as we will show that 
$Q_{\gaa_s}$ is not holomorphically equivalent to a product of two Bishop surfaces. 
A  real submanifold of dimension $n$
 in $\cc^n$ with $n=2p$ that is a product of the above quadrics
 will be called a {\em product of quadrics}, or a {\it product quadric}. 

In section~\ref{secquad}, we study all quadrics which admit the maximum number of deck transformations. For such quadrics,
 all deck transformations are linear. 
Under condition E,  we will first normalize 
$\sigma, \tau_1,\tau_2$ and $\rho_0$
into $\hat S, \hat T_1,\hat T_2$ and  $\rho$ where  
\begin{equation}
\begin{array}{rrclrcl} 
  \hat T_1\colon  &  \xi_j'   &\!\! = \!\!&  \la_j^{-1}\eta_j,  &\quad \eta_j'     &\!\!\! = \!\!\!& \la_j\xi_j, \\ 
  \hat T_2\colon  & \xi_j'    &\!\! = \!\!&  \la_j\eta_j,          &\quad  \eta_j'        &\!\!\! = \!\!\!& \la_j^{-1}\xi_j, \\ 
\hat S\colon       &\xi_j'     &\!\! = \!\!&  \mu_j\xi_j,          &\quad  \eta_j'       &\!\!\! = \!\!\!& \mu_j^{-1}\eta_j 
\end{array}
\nonumber
\end{equation}
with
\eq{muss}\nonumber
 \la_e>1,   \quad |\la_h|=1, \quad |\la_s|>1, \quad\la_{s+s_*}=\ov\la_s^{-1}, \quad \mu_j=\la_j^2.
\eeq
Here $1\leq j\leq p$.  Throughout the paper,   the indices $e,h,s$ have the ranges:  $1\leq e\leq e_*$,
$ e_*<h\leq e_*+h_*$,  $ e_*+h_*<s\leq p-s_*$. Thus $e_*+h_*+2s_*=p$.  
We will call $e_*,h_*, s_*$ 
the numbers
of {\it elliptic, hyperbolic} and {\it complex} components of a product quadric, respectively. 
As in the Moser-Webster theory, at the complex tangent (the origin)  an {\it elliptic} component of a product quadric corresponds 
a  {\it hyperbolic} component of $\hat S$, while a {\it hyperbolic} component of the quadric corresponds an {\it elliptic}
component of $\hat S$. One could identify a {\it complex} component of the quadric with a {\it hyperbolic} (instead of complex)
component of $\hat S$; however, each type of complex tangents exhibits   striking differences  in the formal normal
forms, the convergence of normalizations, and the existence of attached complex submanifolds, as illustrated 
by the  results in this section. 

For the above normal form of $\hat T_1,\hat T_2$ and $\hat S$,  we always normalize the anti-holomorphic involution
$\rho_0$ as
\begin{gather}\label{57rhoz}
\rho\colon\left\{ \begin{array}{lcllcl}
\xi_e' &=&\ov \eta_e, \qquad&\eta_e'&=&\ov\xi_e,\\
\xi_h' &=&\ov \xi_h, & \eta_h'&=&\ov\eta_h,\\
\xi_s' &=&\ov\xi_{s+s_*}, & \eta_s'&=&\ov\eta_{s+s_*},\\
\xi_{s+s_*}'   &=&\ov \xi_s, & \eta_{s+s_*}'&=&\ov\eta_s.
\end{array} \right.
\end{gather}
 With the above normal forms $\hat T_1,\hat T_2,\hat S, \rho$ with $\hat S=\hat T_1\hat T_2$, we will
then normalize the  $\tau_{11},\ldots, \tau_{1p}$ under linear transformations that
commute with $\hat T_1, \hat T_2$, and $\rho$, i.e. the
linear transformations belonging to the {\it centralizer}
of  $\hat T_1,\hat T_2$ and $\rho$. This is a subtle  step.
Instead of normalizing the involutions directly,
we will use the  pairwise commutativity of  $\tau_{11}, \ldots, \tau_{1p}$ to associate to
these $p$ involutions 
a non-singular $p\times p$ matrix $\mathbf B$. The normalization
of $\{\tau_{11}, \ldots,\tau_{1p},
\rho\}$
 is then identified with the normalization of the matrices  $\mathbf B$ under a suitable
equivalence relation.   The latter  is easy to solve. Our normal form of $\{\tau_{11}, \ldots,\tau_{1p},
\rho\}$ is then constructed from the normal forms of $T_1,T_2,\rho$, and the matrix $\mathbf B$.
Following Moser-Webster~\cite{MW83}, we will construct the normal form of the quadrics 
 from the normal form of involutions.

\begin{thm}\label{thm1intr}Let $M$ be a quadratic submanifold defined by
$$
z_{p+j}=h_j(z',\ov z')+q_j(\ov z'),
\quad 1\leq j\leq p.
$$
Suppose that $M$ satisfies condition  $E$, 
i.e. 
  the branched covering 
of  $\pi_1$ of complexification $\cL M$ has   $2^p$  deck transformations 
and  $M$ has $2p$ distinct eigenvalues. 
Then $M$ is holomorphically equivalent to
  \begin{gather}\label{qbgaintro}
  \nonumber
Q_{\mathbf B,\boldsymbol{\gamma}}\colon
z_{p+j}=L_j^2(z',\ov z'), \quad 1\leq j\leq p
\end{gather}
where $(L_1(z',\ov z'), \ldots, L_p(z',\ov z'))^t=\mathbf B(z'-2\boldsymbol{\gamma}\ov z')$,
 $\mathbf B\in GL_p(\cc)$ and 
\begin{gather}\nonumber
\boldsymbol{\gamma}:=\begin{pmatrix}
  \boldsymbol{\gaa}_{e_*}    &\mathbf{0} &\mathbf{0} &\mathbf{0}\\
    \mathbf{0}&\boldsymbol\gaa_{h_*}&\mathbf{0}& \mathbf{0}  \\
   \mathbf{0}&\mathbf{0}&\mathbf{0}&\boldsymbol{\gaa}_{s_*} \\
   \mathbf{0}&\mathbf{0}&{\mathbf I_{s_*}}-\ov{\boldsymbol{\gaa}}_{s_*}&\mathbf{0}
\end{pmatrix}.
\end{gather}
Here $p=e_*+h_*+2s_*$, $\mathbf I_{s_*}$ denotes the $s_*\times s_*$ identity matrix, and
\gan \boldsymbol{\gaa}_{e_*}=\diag(\gamma_1,\ldots, \gamma_{e_*}), \quad
\boldsymbol{\gaa}_{h_*}=\diag(\gamma_{e_*+1},\ldots, \gamma_{e_*+h_*}),\\ 
\boldsymbol{\gaa}_{s_*}=\diag(\gamma_{e_*+h_*+1},\ldots, \gamma_{p-s_*})
\end{gather*}
with  $\gaa_e,\gaa_h,$ and $\gaa_s$ satisfying \rea{0ge1}.
 Moreover, $\mathbf B$  is uniquely determined by   an equivalence
  relation 
 $
 \mathbf B\sim \mathbf C\mathbf B\mathbf R
 $ for suitable non-singular matrices $\mathbf C,\mathbf R$ which have exactly $p$ non-zero entries. 
 %
 %
\end{thm}

See~\rt{quadclass} for detail of the equivalence relation. The scheme of finding quadratic
normal forms turns out to be useful. 
It 
will be applied to the study of  normal forms of  the general real submanifolds.

\subsection{Formal normalization and divergence of normal forms}

\subsubsection{Formal submanifolds, formal involutions, and formal centralizers}

In section~\ref{fsubm}, we  show that   
the formally holomorphic classification of formal submanifolds with the maximum number of formal deck transformations 
and  
the formally holomorphic classification of suitable  families of involutions $\{\tau_{11},\ldots, \tau_{1p},\rho\}$
are equivalent. This equivalence will be used to derive the formal normal forms of the submanifolds.
As mentioned earlier, we will 
first normalize $\sigma=\tau_1\tau_2$ under general formal biholomorphic transformations.  
The normal forms of $\sigma$ turn out to be in the centralizer of $\hat S$, the  normal form of the
linear part of $\sigma$. 
 The family is subject to a second step of normalization,  
under mappings which again turn out to be in the centralizer  of $\hat S$. Thus, before we introduce normalization, we will first study various centralizers. 
We will discuss the centralizer of $\hat S$ as well as the centralizer of $\{\hat T_1,\hat T_2\}$ in section~\ref{fsubm}. The centralizer of $\{\hat T_{11},\ldots, \hat T_{1p},\rho\}$ is more complicated, which will be discussed in section~\ref{rigidquad}. 

\subsubsection{Normalization of $\sigma$}

As mentioned earlier, we will divide the normalization for the families of non-linear involutions  into two steps.
This division will serve  two purposes: first, it helps us to find the formal normal forms of the family of involutions $\{\tau_{11},\ldots, \tau_{1p},\rho\}$; second, it  helps us  understand
the convergence of normalization of the original normal form problem for the real submanifolds.  For purpose of normalization, we will assume that $M$
is {\it non-resonant}, i.e. $\sigma$ is {\it non-resonant}, if its eigenvalues $\mu_1, \ldots,\mu_p,  \mu_1^{-1}, \ldots, \mu_p^{-1}$ satisfy
$$
\mu^Q\neq1, \qquad \forall Q\in\zz^p, \quad |Q|\neq0. 
$$

In section~\ref{secnfs}, we obtain the normalization of $\sigma$ by proving the following.
\begin{thm}\label{ideal0intr} Let $\sigma$ be a holomorphic map with linear part $\hat S$. 
Assume that  $\mu_1,\ldots, \mu_p$ are non-resonant. Suppose that $\sigma=\tau_1\tau_2$ where $\tau_1$ is a holomorphic involution, $\rho$ is an anti-holomorphic involution, and $\tau_2=\rho\tau_1\rho$.
Then there exists a formal map $\Psi$
such that $\rho:=\Psi^{-1}\rho\Psi$ is given by \rea{57rhoz},  $\sigma^*=\Psi^{-1}\sigma\Psi$   and $\tau_{i}^*=\Psi^{-1}\tau_i\Psi$ have the form
 \begin{gather}\label{sigma0}
\sigma^*\colon\xi_j'=M_j(\xi\eta)\xi_j, \quad \eta_j'=M_j^{-1}(\xi\eta)\eta_j, \quad 1\leq j\leq p,\\
\tau_i^*=\Lambda_{ij}(\xi\eta)\eta_j, \quad \eta_j'=\Lambda_{ij}^{-1}(\xi\eta)\xi_j. 
\nonumber
\end{gather}
 Here, $\xi\eta=(\xi_1\eta_1,\ldots,\xi_p\eta_p)$. Assume further that  $\log M$ $($see  
 \rea{logm} for definition$)$  is tangent to the identity. 
  Under a further change of coordinates that preserves $\rho$,    $\sigma^*$  and
$\tau_i^*$ are  transformed into
\begin{gather} \label{hmjw}
\hat\sigma\colon \xi_j'=\hat M_j(\xi\eta)\xi_j, \quad \eta_j'=\hat M_j^{-1}(\xi\eta)\eta_j, \quad 1\leq j\leq p,\\
\hat\tau_i=\hat\Lambda_{ij}(\xi\eta)\eta_j, \quad \eta_j'=\hat\Lambda_{ij}^{-1}(\xi\eta)\xi_j, 
\quad\hat\Lambda_{2j}=\hat\Lambda_{1j}^{-1}.
\nonumber
\end{gather}
 Here the $j$th component of  $\log \hat M(\zeta)- I$ 
  is independent of $\zeta_ j$.
Moreover, $\hat M$ is unique. 
\end{thm}

\begin{rem} 
The condition that $\log M$ is tangent to identity at the origin has to be understood as a non-degeneracy condition of which it is the simplest instance. When there is no ambiguity, ``tangent to identity'' stands for ``tangent to identity at the origin''.
\end{rem} 

We will conclude section~\ref{secnfs} with an example showing that although $\sigma,\tau_1,\tau_2$ are both linear, $\{\tau_{11},\ldots,\tau_{1p},\rho\}$ are not necessary linear, provided $p>1$.


Section \ref{div-sect} is devoted to the proof of the following divergence result. 
\begin{thm}\label{divsig} There exists a non-resonant real analytic submanifold $M$ with
pure elliptic complex tangent in $\cc^6$ 
such that if its corresponding $\sigma$  is transformed into a map $\sigma^*$ 
which commutes with the linear part of $\sigma$ at the origin,
then $\sigma^*$ must diverge. 
\end{thm}

Note that the  theorem says that all normal forms of $\sigma$ (by definition, they belong to the centralizer of its linear part,
i.e. they are in the Poincar\/e-Dulac normal forms) are divergent.   It implies  that any transformation for $M$ that  transforms  $\sigma$ into a Poincar\'e-Dulac normal
form  must diverge. This is in contrast with the Moser-Webster theory: 
For  $p=1$, a convergent  normal form can always be achieved even if the  associated transformation is divergent (in the case of hyperbolic complex tangent), and 
furthermore in case of $p=1$ and elliptic complex tangent with a non-varnishing Bishop invariant, the normal form can   be achieved by a convergent
transformation. 
The divergent Birkhoff normal form for the classical Hamiltonian systems was  
obtained in \cite{Go12} by the first-named author. We refer to \cite{siegel-moser-book, moser-princeton} as general references concerning Hamiltonian and reversible dynamics.


\subsubsection{Normalization on the family $\{\tau_{ij},\rho\}$}
In section~\ref{nfin}, we will follow the scheme developed for the quadric
normal forms in order to normalize $\{\tau_{11},\ldots,\tau_{1p},\rho\}$. 
Let $\hat\sigma$ be given by \re{hmjw}. We define
$$
\hat\tau_{1j}\colon\xi_j'=\hat\Lambda_{1j}(\xi\eta)\eta_j, \quad \eta_j'=\hat\Lambda_{1j}^{-1}(\xi\eta)\xi_j, \quad\xi_k'=\xi_k, \quad \eta_k'=\eta_k,
$$
where $k\neq j$, 
 $\hat\Lambda_{1j}(0)=\la_j$, and
 $\hat M_j=\hat\Lambda_{1j}^2$.  We have the following formal normal form.
\begin{thm}\label{nfofM}Let $M$ be a real analytic submanifold that is a
 higher order perturbation of a non-resonant product
quadric. Suppose that 
its associated  $\sigma$ is formally equivalent to $\hat\sigma$ given by \rea{hmjw}. 
 Suppose that  the formal mapping  $\log\hat M$  in \rta{ideal0intr} is tangent to the identity. 
Then the formal normal form of $M$ is completely determined by 
\eq{Phi1}\nonumber
\hat M(\zeta), \quad  \Phi. 
\eeq
Here 
$\Phi$ is a formal invertible
mapping in $\cL C^c(\hat\tau_{11},\ldots, \hat\tau_{1p})$. Moreover, $\Phi$ is uniquely determined up to the equivalence relation $\Phi\sim R_\e\Phi R_\e^{-1}$ with  $R_\e\colon\xi_j=\e_j\xi,\eta_j'=\e_j\eta_j$ $(1\leq j\leq p)$ and $R_\e^2=I$. Furthermore, if the normal form \rea{sigma0} of $\sigma$ can be achieved by a convergent transformation, so does the normal form of $M$.
\end{thm}
The   set $\cL C^c(\hat\tau_{11},\ldots, \hat\tau_{1p})$ is defined by Definition~\ref{cnnl}.

\vspace{1cm}

The second part of the paper is devoted to geometric properties of $M$ and in particular those obtained through convergent
normalization, under additional assumptions on $M$.

We first turn to  a holomorphic normalization of a real analytic submanifold $M$ with the so-called abelian CR singularity.  This will be achieved by 
studying an integrability problem on a general family of commuting biholomorphisms described below. The holomorphic normalization will
be used to construct the local hull of holomorphy of $M$. 
 We will also study the 
 rigidity problem 
 of a quadric under higher order analytic perturbations,  
 i.e.  the problem if such a perturbation  remains holomorphically equivalent to the quadric if  it is   formally equivalent
 to the quadric.
  The rigidity problem is  reduced to a theorem of holomorphic linearization of one or several commuting diffeomorphisms along a suitable ideal that was devised in \cite{stolo-bsmf}.
 Finally,  we will study the existence of
 holomorphic  submanifolds attached to the real submanifold $M$. These are complex submanifolds of dimension $p$ intersecting $M$ along two totally real analytic submanifolds that intersect  transversally at a CR singularity.  Attaching complex submanifolds has less constraints than 
 finding a convergent normalization. Therefore, our only assumption is that $M$ is non-resonant and admits the maximum number of deck transformations.  A remarkable feature
of   attached complex submanifolds is that their existence depends only on the existence of suitable (convergent)
 invariant submanifolds of $\sigma$. When the real submanifold has a complex tangent of pure complex type, the existence 
   is ensured under
 a mild non-resonance condition but without any further restriction such as small divisors condition
  on the eigenvalues of $M$. 

\subsection{
Abelian CR singularity and analytic hull of holomorphy}
\subsubsection{Normal form of commuting biholomorphisms} 
Let $\cL F=\{F_1,\ldots, F_\ell\}$ be a finite family of germs of biholomorphisms of $\cc^n$
fixing the origin. 
 Let $D_m$ be the linear part of $F_m$ at the origin.  We say that the family $\cL F$ is (formal)  {\it completely integrable}, if there is a (formal) biholomorphic mapping $\Phi$ such that 
$\{\Phi^{-1}F_m\Phi\colon 1\leq m\leq \ell\}=\{\hat F_m\colon 1\leq m\leq \ell\}$ satisfies 
\bppp
\item $\hat F_m(z)=(\mu_{m 1}(z)z_1,\ldots, \mu_{m n}(z)z_n)$ and $ \mu_{m j}\circ D_{m'}=\mu_{m j}$ for
 $1\leq m,m'\leq \ell$  and $ 1\leq j\leq n$.   In particular,  $\hat F_m$ commutes with $D_{m'}$ for all $1\leq m,m'\leq\ell$. 
\item  For each $j$ and each $Q\in\nn^n$ with $|Q|>1$,  $\mu_m^Q(0)=\mu_{m j}(0)$ hold for all $m$ if and only if
 $\mu_{m}^Q(z)=\mu_{mj}(z)$ hold  for all $m$. 
\eppp
Note that a necessary condition for $\cL F$ to be formally completely integrable is that $F_1,\ldots, F_\ell$ commute pairwise. 
The main result of section~\ref{nfcb} is the following. 
\begin{thm} \label{lFba}
Let $\cL F$ be a family of finitely many
germs of biholomorphisms  at the origin.  Suppose that $\cL F$ is formally completely integrable.  Then it is holomorphically  completely integrable,
provided 
the family of linear parts $\cL D$ of $\cL F$ is of   the Poincar\'e type.
In particular, $\cL F$ is holomorphically equivalent to a normal form in which each element commutes with $D_1,\ldots, D_\ell$. 
\end{thm}
The definition of   Poincar\'e type is in Definition~\ref{small divisors}.
Such  results for commuting germs of vector fields were obtained in \cite{St00,  stolo-annals} under a collective Brjuno-type of small   divisors condition.  For a single germ of real analytic hyperbolic area-preserving mapping, the result  was due to 
Moser~\cite{moser-hyperbolic},
and for a single germ of reversible hyperbolic  holomorphic mapping $\sigma=\tau_1\tau_2$
of which $\tau_1$ fixes a hypersurface, 
 this result was due to Moser-Webster~\cite{MW83}. 
Our proof is inspired by these proofs.

\subsubsection{Convergence of normalization for the abelian CR singularity} In sections 7 we will obtain the convergence of normalization for an  {\it abelian} CR singularity which we now define. 
We characterize the abelian CR singularity 
as follows.   We first consider 
a  product quadric $Q$ which    satisfies condition E.  So the branched covering $\pi_1$
for  the complexification of
$Q$ are generated by $p$ involutions of which each fixes a hypersurface pointwise. We denote them by $T_{11}, \ldots, T_{1p}$.
 Let $T_{2j}=\rho T_{1j}\rho$.
It turns out that 
  each $T_{1j}$ commutes with all $T_{ik}$
   except one,  $T_{2k_j}$ for some $1\leq k_j\leq p$. When we formulate $S_j=T_{1j}T_{2k_j}$ for $1\leq j\leq p$,
  then $S_1, \ldots, S_p$ commute pairwise. 
  Consider a general  $M$ that  is a third-order perturbation of product quadric $Q$ and
  satisfies  condition E.  We   define $\sigma_j=\tau_{1j}\tau_{2k_j}$. In suitable coordinates, $T_{ij}$
  (resp. $S_j$) is the linear part of $\tau_{ij}$ (resp. $\sigma_j$) at the origin. 
  We say that the complex tangent of 
  a third order perturbation $M$ of a product quadric
  at the origin is of {\it abelian type},  if   
  $\sigma_1, \dots, \sigma_p$
  commute pairwise. 
 If each linear part  $S_j$ of $\sigma_j$ has exactly two eigenvalues $\mu_j,
\mu_j^{-1}$ that are different from $1$,  then $\cL S:=\{S_1,\ldots, S_p\}$ is of Poincar\'e type if and only if $|\mu_j|\neq1$ for all $j$.
 As mentioned previously,  Moser and Webster 
 actually dealt with $n$-dimensional real submanifolds in $\cc^n$ that
have the minimal dimension of complex tangent subspace at a CR singular point. 
In their situation, there is only one possible composition, that is $\sigma=\tau_1\tau_2$. When the complex tangent 
has an elliptic but non-vanishing Bishop invariant, $\sigma$ has exactly two positive  eigenvalues that are separated by
$1$, while the remaining eigenvalues are $1$ with multiplicity $n-2$.

  As an application of \rt{lFba}, we will prove  the following convergent normalization.
  \begin{thm}\label{iabelm}
Let $M$  be a germ of real analytic submanifold  in $\cc^{2p}$ at an abelian CR singularity. Suppose that $M$  has
distinct eigenvalues
and has no hyperbolic component of complex tangent. Then $M$ is holomorphically equivalent to
\begin{gather}\nonumber 
\widehat M\colon
z_{p+j}=\Lambda_{1j}(\zeta)\zeta_j,\quad 1\leq j\leq
p,
\end{gather}
where $\zeta=(\zeta_1,\ldots, \zeta_p)$ are the convergent solutions to
\begin{align*}\nonumber 
\zeta_e&=
 A_e(\zeta)z_e\ov z_e-
 B_e(\zeta)(z_e^2+\ov z_e^2),\\ 
\zeta_s&=
 A_s(\zeta)z_s\ov z_{s+s_*}-
 B_s(\zeta)(z_s^2+\Lambda_{1s}^2(\zeta)\ov z_{s+s_*}^2),\\
\zeta_{s+s_*}&= 
A_{s+s_*}(\zeta) \ov z_s z_{s+s_*}-
B_{s+s_*}(\zeta)(z_{s+s_*}^2+\Lambda_{1(s+s_*)}^2(\zeta)\ov z_{s}^2)
\nonumber
\end{align*}  
with 
\begin{equation}\nonumber
\begin{array}{rcl}
 A_e(\zeta) &\!\!\! :=\!\!\! &\dfrac{1+\Lambda_{1e}^{2}(\zeta)}{(1-\Lambda_{1e}^2(\zeta))^{2}}, 
  \quad
A_j(\zeta)    :=  \Lambda_{1j}(\zeta)\dfrac{1+\Lambda_{1j}^2(\zeta)}{(1-\Lambda_{1j}^2(\zeta))^2}, \  j=s,s+s_*, \\
B_j(\zeta) & \!\!\! :=\!\!\!  & \dfrac{\Lambda_{1j}(\zeta)}{(1-\Lambda_{1j}^{2}(\zeta))^2}, \quad  j=e,s,s+s_*. \end{array}
\end{equation}
Moreover,    $\Lambda_{1j}(0)=\la_j$, and 
$\Lambda_1=(\Lambda_{11}, \ldots, \Lambda_{1p})$ commutes with the anti-holomorphic involution 
$
\rho_z\colon\zeta_e\to\ov\zeta_e,  \zeta_s\to\ov\zeta_{s+s_*}, \zeta_{s+s_*}\to\ov\zeta_s.$
\end{thm} 
  We will also present a more direct proof by using a    convergence theorem of Moser and Webster~\cite{MW83}
and some formal results from 
section~\ref{nfcb}.

 In the above theorem
$M_j=\Lambda_{1j}^2$, and they are obtained by \rt{abelinv}
for the Jacobian matrix of $\log M$ to be arbitrary. 
When 
 $2\log \diag(\Lambda_{11},\dots\Lambda_{1p})$ is tangent to the identity,
 $M$ can be further  uniquely
normalized  in suitable holomorphic
coordinates to obtain a unique normal form for $M$; see Remark~\ref{presform}.  When $p>1$, the unique normal form shows
that $M$ has infinitely many holomorphic invariants and $M$ is not biholomorphic to the product
of Bishop surfaces in $\cc^2$ even if the CR singularity  has pure elliptic type.  
As an application of \rt{iabelm}, 
we will prove the following. 
\begin{cor}Under the conditions in \rta{iabelm},  the manifold $M$  can be holomorphically flattened. More precisely, in suitable holomorphic coordinates, $M$
is contained in the linear subspace $\cc^p\times\rr^p$ defined by $z_{p+e}=\ov z_{p+e}$ and $z_{p+s}=\ov z_{p+s+s_*}$ where $1\leq e\leq 
e_*$ and $ e_*< s\leq e_*+s_*$.
\end{cor}


One of significances of the Bishop quadrics is that their  higher order analytic perturbation at an elliptic    complex tangent has
 a non-trivial hull of holomorphy. 
As another application of the above normal form, we will  construct the local hull of holomorphy of $M$,  that is the intersection of domains of holomorphy in $\cc^n$ that contain $M$,
via higher dimensional non-linear analytic polydiscs. 
\begin{cor} \label{andiscver}
Let $M$ be a germ of real analytic submanifold at an abelian CR singularity. Suppose that $M$ 
  has distinct eigenvalues 
and has only elliptic component of complex tangent. Then   in suitable holomorphic coordinates, 
$\cL H_{loc}(M)$,  the local hull of holomorphy of $M$, is filled 
by a real analytic family of analytic polydiscs of dimension $p$. Moreover, $\cL H_{loc}(M)$ is the transversal
 intersection of $p$
real analytic submanifolds $\cL H_j(M)$ with boundary and in suitable holomorphic coordinates
all $\cL H_j(M)$ are contained in $\rr^p\times\cc^p$. 
\end{cor}
For a precise statement of the corollary, see \rt{hullj}.
The hulls of holomorphy for real submanifolds with a CR singularity have been studied extensively, starting with the work of Bishop. In the real analytic case with minimum complex tangent space at an elliptic
complex tangent, 
we refer to Moser-Webster~\cite{MW83} for $\gaa>0$, and Krantz-Huang \cite{HK95} for $\gaa=0$.  For the smooth case, see Kenig-Webster~\cite{KW82,KW84}, Huang~\cite{H98}.
For global
results on hull of holomorphy, we refer to \cite{bedford-gaveau, bedford-kling}.

\subsection{Rigidity of quadrics}

  In  Section 10, as an  application of   the theorem of linearization of holomorphic mappings on an ideal $\cL I$
 \cite{stolo-bsmf} (see \rt{theo-invariant}  below), 
we will prove the following theorem, which corresponds to the case $\cL I=0$~: 
\begin{thm} 
 Let $M$ be 
a germ of real analytic submanifold at the origin of $\cc^n$. Suppose that $M$ is formally equivalent
a product quadric   that has distinct eigenvalues. 
Suppose that each hyperbolic component has an eigenvalue $\mu_h$ which
is either a root of unity or satisfies   Brjuno small divisors condition. Then $M$ is holomorphically equivalent
to the product quadric.
\end{thm}
Brjuno small divisors condition is defined by \re{bruno-cond}. When $p=1$, this result is due to the first-named author under a stronger small divisor condition, namely Siegel's condition \cite{Go94}. In the case $p=1$ with a vanishing Bishop invariant, such rigidity result was obtained by 
Moser~\cite{moser-zero}  and by Huang-Yin \cite{HY09b} in a more general context. 

%
%

\subsection{Attached complex submanifolds}
We now describe  convergent results for attached complex submanifolds.  The convergent
results are for a   general 
 $M$, including the one of which the complex tangent might not be of abelian type.
 
We say that a formal complex submanifold $K$ is  {\it attached} to $M$ if $K\cap M$
 contains at least two germs of  totally real and formal submanifolds $K_1, K_2$
that intersect
 transversally at a given CR singularity. In~\cite{Kl85}, Klingenberg showed that when
$M$ is non-resonant and $p=1$, there is   a unique   formal holomorphic
curve attached to $M$ with a hyperbolic complex tangent. He also proved the convergence of the attached formal holomorphic curve under
a Siegel small divisors  
 condition.  When $p>1$, we will show that generically there is no formal complex submanifold that can be attached
to $M$ if $M$ does not admit the maximum number of  deck transformations or if the CR singularity has
an elliptic component.   When $p>1$ and $M$ is a higher order perturbation of a product quadric  of
 $Q_{\gaa_h},
Q_{\gaa_s}$,
we will encounter  various  interesting situations.

Firstly, by adapting Klingenberg's proof for $p=1$ 
and using a theorem of P\"oschel~\cite{Po86}, we will prove the following. 

\begin{prop}\label{oneconv}
Let $M$ be a third order perturbation of a product of  quadrics of which each has complex type at
the origin. 
Suppose that $M$  admits the maximum number of deck transformations
and it is non-resonant. 
 Then $M$ admits an attached complex submanifold.\end{prop}
 The proposition  does not need any small divisors condition and   a
 more detailed version   in  the presence of hyperbolic components  is in \rt{pocor}.
 Furthermore, the non resonance condition is satisfied for $\gaa_1,\ldots, \gaa_{s_*}$
 outside the union of countable algebraic hypersurfaces.

Secondly, we will show that a non-resonant product quadric has a unique attached complex
manifolds. However, under a perturbation of the quadric, 
the attached complex submanifold  of the   quadric
can split into different attached formal submanifolds which may or may not be convergent.
In fact, we will show that the coexistence of divergent and convergent attached complex submanifolds for a complex tangent of the
complex type; see \rp{coexist}.  

Finally, for the convergence of {\it all} attached  formal complex submanifolds, 
we have the following. 
\begin{thm}\label{allconv}
Let $M$ be a third order perturbation of a product quadric.  
Suppose that $M$  admits the maximum number of deck transformations and 
is non resonant. Suppose that $M$ has no elliptic component and the eigenvalues of $\sigma$ satisfy a Bruno type condition,  then
 all attached formal submanifolds are convergent. 
\end{thm}
The above  theorem  for
hyperbolic complex tangency was drafted in \cite{St07}. 
For the Bruno type of condition in the theorem, see
\rea{bruno-cond}, which was introduced in~\cite{stolo-bsmf} for linearization on ideals. 


%
%
%

\subsubsection{Notation}
We briefly introduce the notation used in the paper. The identity map is denoted by $I$.  We 
denote by $LF$ the linear part at the origin of a mapping $F\colon\cc^m\to\cc^n$
with $F(0)=0$.
 Let $F'(0)$ or $DF(0)$ denote the Jacobian matrix of the $F$ 
 at the origin. Then
$LF(z)=F'(0)z$.  
We also denote by 
$DF(z)$ or simply $DF$, the Jacobian matrix of $F$
at 
$z$, when there is no ambiguity.
 By an analytic (or holomorphic) function, we shall mean a {\it germ} of analytic function at a point (which will be defined by the context) otherwise stated. 
 We shall denote by ${\cL O}_n$ (resp. $\widehat {\cL O}_n$, $\mathfrak M_n$, $\widehat{\mathfrak M}_n$) the space of germs of holomorphic functions of $\cc^n$ at the origin (resp. of formal power series in $\cc^n$,  holomorphic
germs, and formal germs vanishing at the origin).

\medskip
\noindent 
{\bf Acknowledgment.} This joint work was completed while the first-named author was visiting at  SRC-GAIA of
POSTECH.  X.G. is grateful to Kang-Tae Kim for hospitality. 
\setcounter{thm}{0}\setcounter{equation}{0}

\section{
CR singularities and deck transformations}
\label{secinv}
 
We will consider a real submanifold $M$ of $\cc^n$. The simplest
local holomorphic invariant of $M$ is the dimension of its complex tangent
subspace $T_{x_0}^{(1,0)}M$ at a given point $x_0$. Here $T_{x_0}^{(1,0)}M$ is the  space of tangent vectors of $M$ at $x_0$   of the form $\sum_{j=1}^na_j\f{\pd}{\pd z_j}$.   Let $M$ have dimension $n$. In this paper, we assume that
$T_{x_0}^{(1,0)}M$
 has the largest possible dimension $p=n/2$ at a given   point $x_0$,   or equivalently, that
 the complexified tangent space $T_{x_0}M\otimes \cc$ is the direct sum of $T_{x_0}^{(1,0)}M$ and its complex conjugate.
We    study local invariants of   $M$ under a local holomorphic change of coordinates fixing the $x_0$.
  In suitable
holomorphic affine coordinates, we have $x_0=0$ and
\begin{equation}\label{variete-orig}
M\colon 
z_{p+j} = E_{j}(z',\bar z'), \quad 1\leq j\leq p.
\end{equation}
Here we have set $z'=(z_1,\ldots,z_p)$ and we will denote $z''=(z_{p+1},\ldots,z_{2p})$. The   $1$-jet at the origin of the complex analytic functions $E_j$ vanishes; in other words,
$E_j$ together with
their first order derivatives vanish at $0$.
 The tangent space $T_0M$ is then the $z'$-subspace. 
  For the local theory, the only interesting case is when $M$ is not a complex submanifold, that is that 
  $E(z',\ov z')$ is not holomorphic in $z'$, 
 which we assume throughout the paper.
   
  The main purpose of this section is to obtain some basic invariants and a relation between two families of involutions
 and the real analytic submanifolds which we want to normalize. 
 
\subsection{CR singular set} 
Let $M$ be given by \re{variete-orig}.  Then $$
X=\sum_{j=1}^p\left\{ a_j\frac{\partial}{\partial z_j}+b_j\frac{\partial}{\partial z_{p+j}}\right\}
$$ is tangent
to $M$ at $(z',z'')$ if and only if
$$
  b_k=\sum_{j=1}^pa_j \frac{\partial E_{k}(z',\ov z')}{\partial z_j}, \quad \sum_{j=1}^p a_j \frac{\partial\ov E_{k}(\ov z',z')}{\partial z_j}=0,
\quad 1\leq k\leq p.
$$
 To consider the   second set of equations,  we introduce
\eq{czzp}
 C(z',\ov z'):= \begin{vmatrix}
    \f{\pd E_1}{\pd\ov z_1} & \cdots & \f{\pd E_1}{\pd \ov z_p} \\
   \vdots & \vdots & \vdots \\
   \f{\pd E_p}{\pd\ov z_1} & \cdots & \f{\pd E_p}{\pd \ov z_p}
  \end{vmatrix}.
\eeq
 Note that $M$ is totally real at $(z',z'')\in M$ if and only if $C(z',\ov{z'})\neq0$. 
We will assume that   
 $C(z',\ov z')$ is not identically zero in any neighborhood of the origin. 
 Then the zero set of $C$ on $M$, denoted by $M_{CRsing}$,  is called {\it CR singular set}
 of $M$, or the set of {\it complex tangents} of $M$. 
   We assume that $M$ is real analytic. Then
 $M_{CRsing}$ is a possibly singular
  proper real analytic subset of $M$ that contains the origin. 

\subsection{Existence of deck transformations and examples}

We first derive some quadratic invariants.
Applying a quadratic change of holomorphic coordinates, we obtain
\eq{variete-orig+}
E_j(z',\ov z')=h_j(z',\ov z')+q_j(\ov z')+O(|(z',\ov z')|^3).
\eeq
Here we have used the convention that if $x=(x_1,\ldots, x_n)$,
then $O(|x|^k)$ denotes a formal power series in $x$ without
terms of order  $<k$.
A biholomorphic map $f$ that preserves the form of the above submanifolds
$M$ and fixes the origin must preserve their complex tangent spaces
at the origin, i.e.  $z''=0$. Thus if $\tilde z$
denote the old coordinates and $z$ denote the new coordinates then $f$ has the form $$
 \tilde z'=  \mathbf{A}  z'+\mathbf{B}   z''+O(|  z|^2), \quad \tilde z''= \mathbf{U}  z''+O(|  z|^2).
$$
Here $\mathbf{A}$ and $\mathbf{U}$ are non-singular  $p\times p$ complex matrices.
Now $f(M)$ is given by
$$
\mathbf{U}  z''=h(\mathbf{A}  z',\ov{\mathbf{A}}\ov{  z}')+q(\ov{\mathbf{A}}\ov z')+O(|z|^3).
$$
We multiply the both sides by $\mathbf{U}^{-1}$ and
solve  for $z''$;  
the vectors of $p$ quadratic forms  $\{h(\tilde z',\ov {\tilde z'}), q(\tilde z')\}$ are transformed into
\eq{hhzp}
\{\hat h(z',z'),\hat q(\ov z')\}=\{\mathbf{U}^{-1}h(\mathbf{A}z',\ov{\mathbf{A}}\ov z'),\mathbf{\mathbf{U}}^{-1}q(\ov {\mathbf{A}}\ov z')\}.
\eeq
This shows that if $M$ and $\hat M$ are holomorphically equivalent,  their corresponding quadratic terms  are
equivalent via \re{hhzp}.
 Therefore, we obtain a holomorphic
invariant
$$
q_*=\dim_{\cc}\{z'\colon q_1(z')=\cdots =q_p(z')=0\}.
$$
We remark that when $M,\hat M$
are quadratic (i.e. when their corresponding $E, \hat E$ are homogeneous quadratic polynomials),
the equivalence relation \re{hhzp} implies that
$M, \hat M$ are linearly equivalent,
Therefore,  the above transformation
of $h$ and $q$
via $\mathbf{A}$ and $\mathbf{U}$ determines the classifications of  the quadrics under  local biholomorphisms as wells as under linear biholomorphisms.
We have shown that the two classifications for the quadrics are identical.


Recall that $M$ is real analytic.
Let us complexify such a real submanifold $M$ by replacing $\bar z'$ by $w'$
to obtain a complex   $n$-submanifold of $\cc^{2n}$, defined by
\begin{equation}\nonumber
{\mathcal M}\colon
\begin{cases}
z_{p+i} = E_{i}(z',w'), 
\\
w_{p+i} = \bar E_i(w',z'),\quad i=1,\ldots, p.\\
\end{cases}
\end{equation}
We  use $(z',w')$ as holomorphic coordinates of $\mathcal M$ and
 define the anti-holomorphic involution $\rho$ on it by
\begin{equation}\label{antiholom-invol}
\rho(z',w')=(\bar w',\bar z').
\end{equation}
  Occasionally we will also denote the above $\rho$ by $\rho_0$ 
for clarity.  
We will  identify $M$ with a totally real and real analytic submanifold of $\cL M$ via embedding $z\to (z,\ov z)$.
We have $M={\mathcal M}\cap \text{Fix}(\rho)$ where $\text{Fix}(\rho)$ denotes the set of fixed points of $\rho$.
Let $\pi_1\colon \mathcal M\mapsto{\cc}^n$ be the restriction of the projection $(z,w)\to z$
and let $\pi_2$ be the restriction of $(z,w)\to w$. It is clear that $\pi_2=\ov{\pi_1\rho}$ on $\cL M$. Throughout the paper, $\pi_1,\pi_2,\rho$ are restricted on $\cL M$
unless stated otherwise.

Our first basic assumption on $M$ is the following condition.

\medskip
\noindent
{\bf Condition   B.}  $q_*=0$.

\medskip

Note that a necessary condition for $q_*=0$
is that functions $q_1(z')$, $ q_2(z'), \ldots, q_p(z')$
are linearly independent, since the intersection
of $k$ germs of holomorphic hypersurfaces at $0$ in $\cc^p$ has dimension at least
$p-k$. (See~\cite{Ch89}, p. 35;  \cite{Gunning2}[Corollary 8, p. 81].)

  When $\pi_1\colon\cL M\to\cc^{2p}$ is a branched covering,  we define a {\it deck transformation} on $\cL M$ for $\pi_1$ to
be   a germ of biholomorphic mapping $F$ defined at $0\in\cL M$ that satisfies  $\pi_1\circ F=\pi_1$. In other words, 
$F(z',w')=(z',f(z',w'))$  and
$$ 
E_i(z',w')=E_i(z',f(z',w')),\quad i=1,\dots, p.
$$ 

\begin{lemma}\label{2p1} Suppose that $q_*=0$.   Then $M_{CRsing}$ is a proper real analytic subset of $M$ and
 $M$ is totally real away from $M_{CRsing}$, i.e. the CR dimension of $M$
is zero.  Furthermore, $\pi_1$ is a
$2^p$-to-$1$ branched   covering.  The group of deck transformations  of $\pi_1$
consists of $2^\ell$ commuting  involutions with $0\leq \ell\leq p$.
\end{lemma}
\begin{proof}  Since $q^{-1}(0)=\{0\}$, then  $z'\to q(0,z')$ is a finite holomorphic map. Hence its Jacobian determinant
is not identically zero. In particular, $C(z',\ov z')$, defined by \re{czzp}, 
is not identically zero. This shows that $M$ has CR dimension $0$. 

Since $w'\to q(w')$ is a homogeneous quadratic
mapping of the same space  which vanishes only at the origin, then
$$
|q(w')|\geq c|w'|^2.
$$
We want to verify that
$\pi_1$ is a $2^p$--to--$1$
branched   covering.  Let $\Delta_r=\{z\in\cc\colon|z|<r\}$.
 We choose  $C>0$ such that
$\pi_1(z,w)=(z',E(z',w'))$
defines a proper and onto mapping
\eq{pi1M}
\pi_1\colon \cL M_1:=  \mathcal M\cap((\Delta_{\delta}^{p}\times \Delta_{\delta^2}^{p})
\times(\Delta_{C\delta}^{p}\times \Delta_{C\delta^2}^{p}))\mapsto \Delta_{\delta}^{p}\times \Delta_{\delta^2}^{p}.
\eeq
By Sard's theorem, the regular values of $\pi_1$ have the full measure. For each regular
value $z$,  $\pi_1^{-1}(z)$ has exactly $2^p$ distinct points (see  \cite{Ch89}, p.~105 and
p.~112). It is obvious that  $\mathcal M_1$ is smooth and connected.
We   fix a fiber $F_z$ of $2^p$ points.   Then 
 the group of deck transformations of $\pi_1$  acts on $F_z$ in such a way that if a deck transformation fixes
a point in $F_z$, then it must be the identity.  Therefore, the number of  
deck transformations divides $2^p$
and each deck transformation has period $2^\ell$ with $0\leq \ell\leq p$.

We first show that each deck transformation $f$ of $\pi_1$
is an involution.
We know that  $f$ is periodic and
has the form $$z'\to z',\quad
w'\to \mathbf{A}w'+\mathbf{B}z'+O(2),$$ where $\mathbf{A},\mathbf{B}$ are  matrices.
Assume that $f$  has period $m$. Then
$\hat f(z',w')=(z',\mathbf{A}w'+\mathbf{B}z')$ satisfies $\hat f^m= I $ and $f$  is locally equivalent to $\hat f$; indeed $\hat fgf^{-1}=g$ for
$$
g=\sum_{i=1}^m(\hat f^{i})^{-1}\circ f^i.
$$
Therefore, it suffices to show that $\hat f$ is an involution.

 We have
$$
\hat f^m(z',w')=(z',\mathbf{A}^mw'+(\mathbf{A}^{m-1}+\cdots+\mathbf{A}+\mathbf{I})\mathbf{B}z').
$$
Since $f$ is a deck
transformation, then $E(z',w')$ is invariant under $f$.
Recall from \re{variete-orig+} that $E(z',\ov z')$
starts with quadratic terms of the form $h(z',\ov z')+q(\ov z')$.
Comparing quadratic terms in
$
E(z',w')=E\circ\hat f(z',w'),
$
we see  that the linear map $\hat f$ has   invariant functions
$$
z''=h(z',w')+q(w').
$$
We know that $\mathbf{A}^m=\mathbf{I}$. By the Jordan normal form,
we choose a linear transformation $\tilde w'= \mathbf{S}w'$  such that $\mathbf{S}\mathbf{A}\mathbf{S}^{-1}$ is the diagonal matrix $\diag \mathbf{a}$
with $\mathbf{a}=(a_1,\ldots, a_p)$.
In  $(z',\tilde w')$ coordinates, the
mapping $\hat f$ has the form $( z',\tilde w')\to(z',(\diag \mathbf{ a})\tilde w'+\mathbf{S}\mathbf{B} z')$. Now  $$\tilde h_j( z',\tilde w')+\tilde q_{j}(\tilde w'):=h_j( z',\mathbf{S}^{-1}\tilde w)+q_{ j}(\mathbf{S}^{-1}\tilde w')$$  are invariant under $\hat f$. Hence
  $\tilde q_j(\tilde w')$   are invariant under    $\tilde w'\mapsto(\diag \mathbf{a})\tilde w'$. Since the common zero set of $q_1(w'),\ldots, q_p(w')$ is the origin, then
$$
V=\{\tilde w'\in\cc^p \colon \tilde q(\tilde w')=0\}=\{0\}.
 $$
We conclude that $\tilde q(\tilde w_1,0,\ldots,0)$ is not identically zero; otherwise $V$ would contain the $\tilde w_1$-axis.
 Now
$ \tilde q((\diag \mathbf{a})\tilde w')=\tilde q(\tilde w')$,  restricted to $\tilde w'=(\tilde w_1,0,\ldots, 0)$,
implies that $a_1=\pm1$.  By the same argument, we get  $a_j=\pm1$
for all $j$.
This shows that $\mathbf A^2=\mathbf I$. Let us combine it with
$$
\mathbf{A}^m=\mathbf{I}, \quad (\mathbf{A}^{m-1}+\cdots+\mathbf{A}+\mathbf{I})\mathbf{B}=\mathbf{ 0}.
$$
If $m=1$, it is obvious that $\hat f=I$.
If $m=2^\ell>1$, then $(\mathbf{A}+\mathbf{I})\mathbf{B}=\mathbf{ 0}$. Thus $\hat
f^2(z',w')=(z',\mathbf{A}^2w'+(\mathbf{A}+\mathbf{I})\mathbf{B}z')=(z',w')$.  This shows that every deck transformation of $\pi_1$ is an involution.

For any two deck transformations $f$ and $g$, $fg$ is still a deck
transformation. Hence $(fg)^2= I $ implies that $fg=gf$.
\end{proof}

Next, we want to introduce types of complex tangents. When $p=1$, the types give the classification for quadratic parts of 
the real submanifolds. For higher dimensions, the types serves a crude classification, but they are significant to characterize 
our results.

Let us first recall types of complex tangents for   surfaces. 
The Moser-Webster theory deals with the case $p=1$ for a real
analytic surface 
\eq{z2z1}
\nonumber
z_2=|z_1|^2+\gaa_1(z_1^2+\ov z_1^2)+O(|z_1|^3).
\eeq
Here $\gaa_1\geq0$ is the Bishop invariant of $M$. The complex tangent of $M$ is said to be {\it elliptic}, {\it parabolic}, or {\it hyperbolic}
according to $0\leq\gamma_1<1/2, \gamma_1=1/2$ or $\gamma_1>1/2$, respectively. 
One of most important properties of the Moser-Webster theory
is the existence of the above mentioned deck transformations.
  When $\gamma_1\neq0$, there is a pair of Moser-Webster involutions $\tau_{1}, \tau_2$ with $\tau_2=\rho\tau_1\rho$
such that
 $\tau_1$ generates the deck transformations of $\pi_1$.    In fact, $\tau_1$ is the only possible non-trivial deck 
transformation of $\pi_1$.
When $\gamma_1\neq1/2$,  in suitable coordinates
their composition $\sigma=\tau_1\tau_2$ is  of the form
$$
\tau\colon\xi'=\mu\xi+O(|(\xi,\eta)|^2), \quad \eta'=\mu^{-1}\eta+O(|(\xi,\eta)|^2). 
$$
 Here $\rho(\xi,\eta)=(\ov\eta,\ov\xi)$ when $0<\gaa<1/2$, and $\rho(\xi,\eta)
=(\ov\xi,\ov\eta)$ when $\gaa>1/2$. 
  When the complex tangent is elliptic, $\sigma$ is
{\it hyperbolic} with $\mu>1$;  when the complex tangent is  hyperbolic,  then
 $\sigma$ is {\it elliptic } with $|\mu|=1$.
When the complex tangent is parabolic,  the linear part of $\sigma$ is not diagonalizable and $1$ is the eigenvalue.
We also remark that the Moser-Webster theory deals with  a more general
case where $n$-dimensional submanifolds
$M$ in $\cc^n$ have the form
$$
z_2=|z_1|^2+\gaa_1(z_1^2+\ov z_1^2)+O(3), \quad y_j=
O(2), \quad 2\leq j\leq n
$$
with the Bishop invariant $0<\gaa_1<\infty$. Here $n>1$ is not necessarily even. The origin is then a complex
tangent of $M$ of which  the complex tangent space at the origin has the minimum dimension $1$.

Our basic model is 
 the product of the above-mentioned  Bishop
quadrics
$$
Q_{\gamma}\colon
z_{p+j}=|z_j|^2+\gaa_j(z_j^2+\ov z_j^2), \quad 1\leq j\leq p.
$$
Here $0<\gaa_j<\infty$, $\gaa_j\neq1/2$,   $\gamma=(\gaa_1,\ldots,\gaa_p)$ and $Q_{\gamma}:=Q_{\gamma_1}\times\cdots\times Q_{\gamma_p}$.
We will see later that with $p\geq2$, there is yet another simple
model that is not in the product. This is the quadric
in $\cc^4$ defined by
\eq{z3z4} Q_{\gaa_s}\colon 
  z_3=z_1\ov z_2+\gaa_s\ov z_2^2+(1- \gaa_s)z_1^2, \quad z_4=\ov z_3.
\eeq
Here $\gaa_s$ is a complex number.  
We will, however, exclude $\gaa_s=0$ or equivalently $\gaa_s=1$   by condition B. We also exclude $\gaa_s=1/2$
by condition E. Note that $\gaa_s=1/2$ does not correspond to a product  Bishop  quadrics either,   by
examining the CR singular sets. 
Under these mild non degeneracy conditions, we will show that $\gaa_s$ is an invariant when it is normalized to the range
     \eq{gsra}
\gaa_s\in (1/2, \infty)+i(0,\infty).
  \eeq
In this case, the complex  tangent is said of {\it complex} type.
Notice that    $Q_{\gaa_j}$ is contained in a real
hyperplane when  $\gaa_j\geq0$, while $Q_{\gaa_s}$ is contained in $\cc^2\times\rr^2$.

We have introduced the types of the complex tangent at the origin. Of course a product of quadrics, or a product quadric, 
can exhibit a combination of the above basic $4$ types. 
We will see soon that 
quadrics have other invariants when $p>1$. Nevertheless, in our results,  
the above invariants  
that describe the types of the complex tangent will
play a   major role in the convergence or divergence of normalizations. 

Before we proceed   to discussing
the deck transformations,
we give some examples. The first example turns out to be
a holomorphic equivalent form of a real submanifold that
admits the maximum number of deck transformations
and satisfies other mild conditions.
\begin{exmp}
Let $\mathbf{B}=(b_{jk})$ be a non-singular $p\times p$ matrix. Let $M$ be defined by
\eq{reformu}
z_{p+j}=\left(\sum_k b_{jk}\ov z_k+R_j(z',\ov z')\right)^2, \quad 1\leq j\leq p,
\eeq
  where each $R_j(0,\ov z')$ starts with terms of order at least $2$. 
Then $M$ admits $2^p$ 
deck transformations for $\pi_1$.   Indeed,  let $\mathbf E_1,\ldots,\mathbf E_{2^p}$ be the set of
 diagonal $p\times p$ matrices with $\mathbf E_j^2=
\mathbf I$, and let $\mathbf R$ is the column vector $(R_1,\dots, R_p)^t$.
Any deck transformation $(z',w')\to(z',\tilde w')$ must satisfy 
\eq{inveq}
\mathbf B\tilde w'+\mathbf R(z',\tilde w')=\mathbf E_j(\mathbf Bw'+\mathbf R(z',w')),
\eeq
for some $\mathbf E_j$. 
  Since $\mathbf B$ is invertible, it has a unique solution
$$
\tilde w=\mathbf B^{-1}\mathbf E_j\mathbf Bw'+O(|z'|)+O(|w'|^2). 
$$
Finally, $(z',w')\to (z',\tilde w')$ is an involution, as if $(z',w',\tilde w')=(z',w',f(z',w'))$ satisfy \re{inveq} if and only if $(z',f(z', w'),w')$,
substituting for $(z',w',\tilde w')$ in \re{inveq}, 
satisfy \re{inveq}.  \end{exmp}

\begin{exmp} Let $M$ be defined by
\begin{alignat*}{4}
z_{p+j}&=z_j\ov z_j+b_j\ov z_j^2+E_j(z',\ov z_j),\quad &&1\leq j\leq p-2s_*,\\
z_{s}&=\ov z_sz_{s+s_*}+b_{s+s_*}\ov z_s^2+E_s(z',\ov z_s),\\
z_{s_*+s}&=\ov z_{s+s_*}z_s+b_s\ov z_{s+s_*}^2+E_{s+s_*}(z',\ov z_{s+s_*}), \quad && p-2s_*
\leq s\leq p-s_*.
\end{alignat*}
Here $b_j\neq 0$ and $E_j=O(3)$ for $1\leq j\leq p$. 
Then $M$ admits $2^p$ 
deck transformations for $\pi_1$.
\end{exmp}
We now present two examples to show that the deck
transformations can be destroyed by perturbations
when $p>1$. This is the major difference between real submanifolds with $p>1$ and the ones with $p=1$.

The first example
shows that a small perturbation can reduce
the number of deck transformations to any number
$2^\ell$.
\begin{exmp}\label{mgae}
Let $M_{\gaa,\e}$ be defined by
$$
z_{p+j}=z_j\ov z_j+\gaa_j\ov z_j^2+\e_{j-1}\ov z_{j-1}^2, \quad 1\leq j\leq p
$$
with $z_0=z_p$. Suppose that $\e_j\neq0$ and
\begin{equation}\label{inver-cond}
\gaa_1\cdots\gaa_{p}+(-1)^{p-1}\e_0\cdots\e_{p-1}\neq0.
\end{equation}  We want to show that $M_{\gaa,\e}$
admits the identity deck transformation only.
Let $\tau(z',w')= (z',A(z',w'))$ be a deck transformation. Then
\eq{zjAj}
z_jA_j(z',w')+\gaa_jA_j^2(z',w')+
\e_{j-1}A_{j-1}^2(z',w')=z_jw_j+\gaa_jw_j^2+
\e_{j-1}w_{j-1}^2.
\eeq
Let $a_j(w')=A_j(0,w')$. Set $z'=0$. By $(\ref{inver-cond})$, we can solve
for $a_j^2$ to get unique solutions
$$
a_j^2(w')=w_j^2.
$$
This shows that $a_j(w')=\pm w_j$. Since $\e_{j-1}\neq0$,
setting $w_j=0$ and comparing the coefficients of $z_iw_{j-1}$ in
\re{zjAj} yield $A_{j-1}(z',0)= O(|z'|^2)$.
Comparing the coefficients of $z_j  w_j$ in \re{zjAj},
we conclude that $A_j(z',w')=w_j+O(|(z',w')|^2)$. This shows that $L\tau=I$.
Since $\tau$ is periodic, then $\tau= I$.
\end{exmp}
The next example shows that 
the number of deck transformations can be reduced to any number $2^\ell$ by a higher order perturbation, too.
\begin{exmp}Let $N_{\gaa,\e}$ be a perturbation
 of $Q_\gaa$ defined by
$$
z_{p+j}=z_j\ov z_j+\gaa_j\ov z_j^2+\e_{j-1} \ov z_{j-1}^3,
\quad 1\leq j\leq p.
$$
Here $\e_j\neq0$ for all $j$.
Let $\tau$ be a deck transformation of $N_{\gaa,\e}$ for $\pi_1$. We know that $\tau$ has the form
$$
z_j'=z_j, \quad w_j'=A_j(z',w')+B_j(z',w')+O(|(z',w')|^3).
$$
Here $A_j$ are linear and $B_j$ are homogeneous
quadratic polynomials.
We then have
\ga\label{zjBj-}
z_jA_j(z',w')+\gaa_jA_j^2(z',w')+ A_{j-1}^2(z',w')=z_jw_j+\gaa_jw_j^2, \\
\label{zjBj} z_jB_j(z',w')+2\gaa_j (A_jB_j)(z',w')+ A_{j-1}^3(z',w')= w_{j-1}^3.
\end{gather}
We know that $L\tau$ is a deck transformation for
$Q_\gaa$. Thus $a_j(w')=A_j(0,w')=\pm w_j$.
Set $z_j=0$ in \re{zjBj-}-\re{zjBj} to get $a_j(w')|\e_{j-1}(w_{j-1}^3-a_{j-1}^3 (w'))$. Thus $a_{j-1}(w')=w_{j-1}$. Hence, the matrix
of $L\tau$ is triangular and its diagonal entries are
$1$. Since $L\tau$
is periodic then $L\tau= I$. Since $\tau$ is periodic,
then $\tau= I$.
\end{exmp}

Based the above two examples,  we impose the second basic assumption.

\medskip
\noindent
{\bf Condition  D.} {\it  $M$ satisfies condition B and the branched   covering $\pi_1$  of $\cL M$
admits the maximum $2^p$ deck transformations.}
\medskip

Let us first derive some significant properties for real submanifolds that satisfy conditions B and D.

\subsection{
Real submanifolds and 
Moser-Webster involutions}

The main result of this subsection is to show the equivalence 
of classification of the real submanifolds with that of  families of involutions $\{\tau_{11}, \ldots, \tau_{1p},\rho\}$.  The relation between two classifications plays an important role
in the Moser-Webster theory for $p=1$. This 
 will be the base of our approach to the normal form problems.

Let $\cL F$ be a family of holomorphic maps in $\cc^n$ with coordinates $z$.  Let $L\cL F$ denote the set of linear maps $z\to f'(0)z$ with $f\in \cL F$.
Let $\cL O_n^\cL F$ denote the set of germs of holomorphic functions
$h$ at $0\in\cc^n$
so that $h\circ f=h$ for each $f\in \cL F$.
Let $[\mathfrak M_n]_1^{L\cL F}$ be the subset of linear functions   of  
$\mathfrak M_n^{L\cL F}$.
\begin{lemma}\label{sd}
Let $G$ be an abelian group of
 holomorphic $($resp. formal$)$ involutions fixing $0\in\cc^n$. Then $G$ has $2^\ell$ elements 
 which are simultaneously
diagonalizable by a holomorphic $($resp. formal$)$ transformation.
If $k=\dim_\cc[\mathfrak M_n]_1^{LG}$ then  $\ell\leq n-k$.  Assume furthermore that $\ell=n-k$ then, in suitable holomorphic $(z_1,\ldots, z_n)$ coordinates, the group $G$ is generated by $Z_{k+1},\ldots, Z_n$ with
\eq{zjzjp}
Z_j\colon z_j'=-z_j, \quad z_i'=z_i, \quad i\neq j, \quad 1\leq i\leq n.
\eeq
In the $z$ coordinates,
 the set of convergent $($resp. formal$)$
  power series in $z_1,\ldots, z_k$, $z_{k+1}^2,\ldots,z_{n}^2$  is equal
 to $\cL O_n^{G}$
 $($resp. $\widehat{\cL O}_n^{G})$, and with $Z=Z_{n-k}\cdots Z_n$,
 \eq{Mn1g}
 [\mathfrak M_n]_1^{G}=[\mathfrak M_n]_1^{Z}, \quad
 \fix(Z)=\bigcap_{j=k+1}^n\fix(Z_j).\eeq
\end{lemma}
\begin{proof}
We first want to show that $G$ has $2^\ell$ elements.
  Suppose that it has more than one element and
we have already found a subgroup of $G$ that
has $2^{i}$ elements $f_1,\ldots, f_{2^i}$. Let $g$ be an element
in $G$ that is different from the $2^i$ elements. 
Since $g$ is an involution and commutes
with  each $f_j$, then
$$
f_1,\ldots, f_{2^{i}},\quad  gf_1,
\ldots, gf_{2^{i}}
$$
form a group of   $2^{i+1}$ elements.
We have proved that every finite subgroup of $G$ has exactly $2^\ell$
elements. Moreover, if $G$ is infinite then it contains a subgroup of $2^\ell$ elements   for every $\ell\geq0$. Let $\{f_1,\ldots, f_{2^\ell}\}$ be such a subgroup of $G$. It suffices to show that $\ell\leq n-k$.
We first linearize  all $f_j$ simultaneously.
We know that $Lf_1, \ldots, Lf_{2^\ell}$  commute pairwise. Note that $ I +f_1'(0)^{-1}f_1$ linearizes $f_1$.
Assume that $f_1$ is linear. Then $f_1=Lf_1$ and $Lf_2$ commute, and
$ I+f_2'(0)^{-1}f_2$ commutes with $f_1$ and linearizes $f_2$. Thus $f_j$ can be simultaneously linearized by a holomorphic (resp. formal) change of coordinates. Without loss of generality, we may assume that each $f_j$ is linear. We
want to diagonalize all $f_j$ simultaneously. Let $E_i^{\,1}$ and
$E_i^{-1}$ be the eigenspaces of $f_i$
with eigenvalues $1$ and $-1$, respectively. Since $f_i=f_j^{-1}f_if_j$, each eigenspace of $f_i$
is invariant under $f_j$. Then we can decompose
\eq{ccn}
\cc^n=\bigoplus_{(i_1, \ldots, i_s)}E_1^{i_1}\cap\cdots\cap E_s^{i_s}.
\eeq
Here $
(i_1,\ldots, i_s)$ runs over $\{-1,1\}^s$ with subspaces $E^{(i_1,\ldots, i_s)}:= E_{1}^{i_1}\cap\cdots\cap E_s^{i_s}\neq\{0\}$. On each of these subspaces,
$f_j= I$ or $- I$.
We are ready to choose
 a new basis for $\cc^n$ whose elements are in the subspaces. Under
the new basis, all $f_j$ are diagonal.

Let us rewrite \re{ccn} as
$$
\cc^n=V_1\oplus V_2\oplus\cdots\oplus V_{\ell}.
$$
Here  $V_j=E^{I_j}$ and $I_1=(1,\ldots, 1)$. Also,
$I_j\neq (1,\ldots, 1)$
and $\dim V_j>0$ for $j>1$. We have $\dim_\cc\fix (G)=\dim_\cc V_1=\dim_\cc[\mathfrak M_n]_1^{LG}=k$.  Therefore, $\ell\leq
n-\dim_\cc V_1\leq n-k$.
  We have proved that in suitable coordinates $G$ is generated by $Z_{k+1}, \ldots, Z_n$. The remaining assertions follow easily.
\end{proof}

We will need an elementary result about invariant functions. 
\le{twosetin} Let  $Z_{k+1}, \ldots, Z_n$ be defined by \rea{zjzjp}.
Let $F=\{f_{k+1}, \dots, f_n\}$ be a family of germs of  holomorphic  
mappings at the origin $0\in\cc^n$. Suppose that the family $F$ is holomorphically  
equivalent to $\{Z_{k+1}, \dots, Z_n\}$. 
Let $b_{1}(z), \dots, b_n(z)$ be germs of holomorphic functions that are invariant under $F$. 
Suppose that  for $1\leq j\leq k$, $b_j(0)=0$ and the linear part of $b_j$ at the origin is $\tilde b_j$.
Suppose that 
for $i>k$, 
$b_i(z)=O(|z|^2)$ and the quadratic part of $b_i$ at the origin is $b_i^*$.  Suppose that $\tilde b_1, \dots, \tilde b_k$ are linear independent,  and that $b_{k+1}^*,
\dots, b_n^*$ are linearly independent modulo $\tilde b_1, \dots, \tilde b_k$,  i.e.
$$
\sum c_ib_i^*( z)=\sum d_j(z)\tilde b_j(z) + O(|z|^3)
$$
holds for some constants $c_i$ and formal power series $d_j$,
 if and only if all $c_i$ are zero.  Then  invariant functions
of $F$ are power series in $b_1, \dots, b_n$. Furthermore, $F$ is uniquely determined by $b_1,\ldots, b_n$.
The same conclusion holds if $F$ and $b_j$ are given by formal power series. 
\ele
\begin{proof} Without loss of generality, we may assume that $F$ is $\{Z_{k+1}, \dots, Z_n\}$.  Hence, for all $j$, there is a formal power series $a_j$ such that $b_j(z)= a_j(z_1, \dots, z_k, z_{k+1}^2, \dots, z_n^2)$. Let us show that the map $w\to a(w)=(a_1(w),\ldots, a_n(w))$ is invertible.

By \rl{sd},
$\tilde b_1(z), \dots,\tilde b_k(z)$ are linear combinations of $z_1,\ldots, z_k$, and
vice versa. By \rl{sd} again, $b_{k+1}^*,
\dots, b_n^*$ are linear combinations of $z_{k+1}^2, \dots, z_n^2$ modulo $z_1,\ldots, z_k$. This shows
that 
\eq{bjsz}\nonumber
b_i^*(z)=\sum_{j>k} c_{ij}z_j^2+\sum_{\ell\leq k} d_{i\ell}(z)\tilde b_\ell(z), \quad i>k.
\eeq
Since $b_{k+1}^*, \dots, b_n^*$ are linearly independent modulo $\tilde b_1,\dots, \tilde b_k$. Then $(c_{ij})$
is invertible; so is the linear part of $a$.

To show that $F$ is uniquely determined by its invariant functions, let $\tilde F$ be another such 
family that is equivalent to $\{Z_{k+1},\dots, Z_n\}$. 
 Assume that $F$ and $\tilde F$ 
have the same invariant functions.  Without loss of generality, assume that $\tilde F$
is $\{Z_{k+1},\dots, Z_n\}$.
 Then $z_1,\ldots, z_k$ are invariant by each $F_j$, i.e. the $i$th component of $F_j(z)$
is $z_i$ for $i\leq k$. Also $F_{j,\ell}^2(z)=z_\ell^2$ for $\ell>k$. We get $F_{j,\ell}=\pm z_\ell$. Since $z_\ell$
is not invariant by $\tilde F$, then it is not invariant by $F$ either. Then $F_{j_\ell,\ell}(z)=-z_{\ell}$ for
some $\ell_j>k$.  Since $F_{j_\ell}$ is equivalent to some   $Z_i$, the set of fixed points of $F_{j_\ell}$
is a hypersurface. This shows that $F_{j_\ell}=Z_{\ell}$. So the family $F$ is $\{Z_{k+1},\dots, Z_n\}$. 
\end{proof}

 We now want to find a special set of generators
 for the deck transformations  and its basic properties, which will be important
 to our study of the normal form problems.
\begin{lemma}\label{sd1} Let $M$ be defined by \rea{variete-orig} and
\rea{variete-orig+} with $q_*=0$.
Suppose that $\cL T_i$, the group of deck transformations
 of $\pi_i\colon \cL M\to\cc^p$,  has
 exactly $2^p$ elements. Then the followings hold.
 \bppp
\item $\cL T_1$ is generated by $p$
 distinct involutions $\tau_{1j}$ such that $\fix(\tau_{11})$,
 $\dots$, $ \fix(\tau_{1p})$ are hypersurfaces
 intersecting transversally at $0$.
And $\tau_1=\tau_{11}\cdots\tau_{1p}$   is the unique deck transformation
 of which the set of fixed points
 has dimension $p$. Moreover, $\fix(\tau_1)=\bigcap\fix(\tau_{1j})$.
 \item $\cL O_n^{\cL T_1}$ $($resp. $\widehat{\cL O}_n^{\cL T_1})$
is precisely  the set of convergent $($resp. formal$)$ power series in
 $z'$ and $ E(z',w')$. $\cL O_n^{\cL T_2}$ $($resp. $\widehat{\cL O}_n^{\cL T_2})$
is the set of convergent $($resp. formal$)$ power series in
 $w'$ and $ \ov E(w',z')$. In particular, in $(z',w')$ coordinates   of $\cL M$, $\cL T_1$ and $\cL T_2$ satisfy
 \ga\label{i1i1}
 [\mathfrak M_n]_1^{L\cL T_1}\cap[\mathfrak M_n]_1^{L\cL T_2}
 =\{0\},\\  
 \dim\fix(\tau_i)=p, \quad \fix(\tau_1)\cap\fix(\tau_2)=\{0\}.\nonumber
 \end{gather}
 Here $[\mathfrak M_n]_1$ is the set of linear functions in $z',w'$ without constant terms.
  \eppp
\end{lemma}
\begin{proof} (i). 
%
 Since $z_1,\ldots, z_p$ are invariant under deck transformations of $\pi_1$, we have $p'=\dim_\cc  [\cL O_n]_1^{L\cL T_1}\geq p$. By \rl{sd}, $\pi_1$ has at most $2^{2p-p'}$
deck transformations. Therefore, $p'=p$.
By \rl{sd} again,  we may assume that
 in suitable
 $(\xi,\eta)$ coordinates, the deck transformations are generated by
\eq{t1j}
Z_j\colon (\xi,\eta)\to(\xi,\eta_1,\ldots, \eta_{j-1},-\eta_j,\eta_{j+1},\ldots, \eta_{p}),
 \quad1\leq j\leq p.
 \eeq
It follows that $Z=Z_{1}\cdots Z_{p}$ is the unique  deck transformation of $\pi_1$,  of which the set of fixed points has dimension $p$.

(ii). We have proved that in $(\xi,\eta)$ coordinates the deck transformations are generated by the above $Z_{1}, \ldots,Z_{p}$. Thus, the
 invariant holomorphic functions of  $Z_{1}, \ldots, Z_{p}$ are precisely the holomorphic
functions in $\xi_1, \ldots, \xi_p, \eta^2_{1}, \ldots, \eta^2_{p}$.
Since $z_1, \ldots, z_p$ and $E_i(z',w')$ are invariant under deck transformations,
then on $\cL M$
\eq{zpfx}
z'=f(\xi, \eta_1^2, \ldots, \eta_p^2), \quad
E(z',w')=g(\xi,\eta_1^2, \ldots, \eta_p^2).
\eeq
Since $(z',w')$ are local coordinates of $\cL M$,
  the differentials of $z_1,\ldots, z_p$ under any coordinate 
  system of $\cL M$
 are linearly independent. Computing the differentials of $z'$ in variables $\xi,\eta$ by using \re{zpfx}, we see that
the mapping
 $\xi \to f(\xi,0)$ is a local biholomorphism. Expressing both sides of
  the second identity in \re{zpfx} as power series in $\xi,\eta$, we obtain
\gan
E(f(\xi,0),w')=g(\xi,\eta_1^2,\ldots, \eta_p^2)+O(|(\xi,\eta)|^3).
\end{gather*}
We  set $\xi=0$, compute the left-hand
side, and rewrite the identity as
\ga\label{g0et} g(0,\eta_1^2,\ldots,\eta_p^2)
=q(w')+O(|(\xi,\eta)|^3).
\end{gather}
 As coordinate systems, $(z',w')$ and $(\xi,\eta)$
 vanish at $0\in\cL M$. We now use $(z',w')=O(|(\xi,\eta)|)$. By \re{zpfx}, $f(0)=g(0)=0$ and
 $g(\xi,0)=O(|\xi|^2)$.
Let us   verify that  the linear parts of $g_1(0,\eta), \ldots, g_p(0,\eta)$ are
linearly independent. Suppose that
$\sum_{j=1}^pc_jg_j(0,\eta)=O(|\eta|^2)$. Replacing  $\xi,\eta$ by $O(|(z',w')|)$ in \re{g0et} and setting $z'=0$, we obtain
$$
\sum_{j=1}^pc_jq_j(w')=O(|w'|^3),\quad i.e. \quad\sum_{j=1}^pc_jq_j(w')=0.
$$
As remarked after condition B was introduced, $q_*=0$ implies that $q_1(w'), \ldots, q_p(w')$ are linearly independent. Thus all $c_j$ are $0$. We have verified that $\xi\to f(\xi,0)$
is biholomorphic near $\xi=0$. Also $\eta\to g(0,\eta)$ is biholomorphic near $\eta=0$
and  $g(\xi,0)=O(|\xi|^2)$.
Therefore, $(\xi,\eta)\to(f,g)(\xi,\eta)$ is invertible near $0$. By solving \re{zpfx},
 the functions $\xi,\eta_1^2,\ldots, \eta_p^2$ are expressed as power series
  in $z'$ and $E(z',w')$.

It is clear that $z_1,\ldots, z_p$ are invariant under $\tau_{1j}$. From linearization of $\cL T_1$,
we know that the space of invariant linear functions of $L\cL T_1$ is the same
as the space  of linear invariant functions of $L\tau_1$, which
has dimension $p$. This shows that $z_1,\ldots, z_p$ span the space of linear invariant functions of $L\tau_1$. Also $w_1,\ldots, w_p$ span
the space of linear invariant functions of $L\tau_2$. We obtain $[\mathfrak M_n]_1^{L\cL T_1}\cap[\mathfrak M_n]_1^{L\cL T_2}=\{0\}$.  We have verified   \re{i1i1}.

In view of the linearization of $\cL T_1$ in (i), we obtain $\dim\fix(\tau_1)=\dim\fix(\cL T_1)=p$.
Moreover,  $\fix(\tau_i)$ is a smooth submanifold of which the tangent space at the origin
is   $\fix(L\tau_i)$.
We choose a basis $u_1,\ldots, u_p$ for  $\fix(L\tau_1)$. Let $v_1,\ldots, v_p$ be any $p$
vectors such that $u_1,\ldots, u_p$,  $v_1,\ldots, v_p$ form a basis of $\cc^n$. In new coordinates defined by
$\sum\xi_iu_i+\eta_iv_i$, we know that linear invariant functions of $L\tau_1$ are spanned by $\xi_1,\ldots, \xi_p$. The linear invariant functions in $(\xi,\eta)$
that are invariant by $L\tau_2$ are spanned by
$f_j(\xi,\eta)=\sum_k (a_{jk}\xi_k+b_{jk}\eta_k)$
for $1\leq j\leq p$. Since $[\mathfrak M_n]^{L\tau_1}\cap[\mathfrak M_n]^{L\tau_2}=\{0\}$,
 then $\xi_1,\ldots, \xi_p, f_1,\ldots$,   $ f_p$ are linearly
independent. Equivalently, $(b_{jk})$ is non-singular.  Now $\fix(L\tau_2)$ is spanned by vectors  $\sum_k (a_{jk}u_k+b_{jk}v_k)$.
This shows that $\fix(L\tau_1)\cap\fix(L\tau_2)=\{0\}$. Therefore, $\fix(\tau_1)$ intersects
 $\fix(\tau_2)$  transversally at the origin and the intersection must
be the origin.
\end{proof}

We remark that the proof of the above lemma actually gives us a more general result.

\begin{cor} \label{sd2}  Let $0\leq p\leq n$.
Let $\mathfrak{I}$ be a group of commuting holomorphic $($formal$)$
involutions on $\cc^n$.
 \bppp
\item $\fix(L  \mathfrak{I})=\{0\}$ if and only if  $[\mathfrak M_n]_1^{L  \mathfrak{I}}$  has dimension $0$.
 \item Let $\widetilde{ \mathfrak{ I}}$ be another family of commuting holomorphic $($resp. formal$)$ involutions such that
 $[\mathfrak M_n]^{L  \mathfrak{I}}\cap[\mathfrak M_n]^{L\widetilde{  \mathfrak{I}}}=\{0\}$. Then
 $\fix(L  \mathfrak{I} )\cap\fix(L\widetilde{  \mathfrak{I}})=\{0\}$. Moreover,  $\fix(  \mathfrak{I} )\cap\fix(\widetilde{  \mathfrak{I}})=\{0\}$ if $  \mathfrak{I}$
 and $\widetilde{  \mathfrak{I}}$ consist of convergent involutions.
   \eppp
\end{cor}

In view of  \rl{sd1},
we will refer   to
 $$
 \{\tau_{1j}, \tau_{2j}, \rho;\; 1\leq j\leq p\} $$
as the Moser-Webster involutions, while the two groups of the $2^p$
 involutions intertwined by $\rho$
will be called the extended family of Moser-Webster involutions.
 Recall that $\tau_{2j}=\rho\tau_{1j}\rho$. Let us denote
$$
\cL T_1:=\{\tau_{11}, \ldots, \tau_{1p}\},
\quad \cL T_2:=\{\tau_{21}, \ldots, \tau_{2p}\}.
$$
Thus the sets of involutions are uniquely determined by
$$
\{\cL T_1,\rho\}=\{\tau_{11}, \ldots, \tau_{1p},\rho\}.
$$

  The significance of the two sets of
involutions is the following proposition that transforms the normalization of the real manifolds
into that 
of two families of commuting involutions.

 For clarity, recall the anti-holomorphic involution $\rho_0\colon(z',w')\to(\ov{w'},\ov{z'})$. 
\begin{prop}\label{inmae}
Let $M$ and $\widetilde M$ be two real analytic submanifolds of the form \rea{variete-orig} 
and \rea{variete-orig+}
 that
admit Moser-Webster involutions $\{\cL T_1,\rho_0\}$ 
and $\{\widetilde{\cL T_1},\rho_0\}$,
respectively. Then $M$ and $\widetilde M$ are holomorphically equivalent
if and only if $\{\cL T_1,\rho_0\}$ and $\{\widetilde{\cL T_1},\rho_0\}$
are holomorphically   equivalent, i.e. if there is a biholomorphic map $f$ commuting with $\rho_0$
such that $f\cL T_1f^{-1}=\widetilde{\cL T_1}$, that is that $f\tau_{1j}f^{-1}=\tilde\tau_{1i_j}$ for $1\leq j\leq p$. Here $\{i_1,\ldots, i_p\}=\{1,\ldots, p\}$.

Let  $\cL T_1=\{\tau_{11},\ldots,\tau_{1p}\}$ be a family of $p$ distinct
commuting   holomorphic
involutions.  Suppose that $\fix(\tau_{11}), \ldots, \fix(\tau_{1p})$ are hypersurfaces
 intersecting transversely at the origin.
 Let $\rho$ be an anti-holomorphic involutions and let $\cL T_2$ be the family of
involutions $\tau_{2j}=\rho\tau_{1j}\rho$ with $1\leq j\leq p$.
Suppose that
\eq{ilii+}
 [\mathfrak M_n]_1^{L\cL T_1}\cap[\mathfrak M_n]_1^{L\cL T_2}
 =\{0\}.
 \eeq
There exists a real analytic real $n$-submanifold
\eq{realM2}
M\subset\cc^{2p}
\colon z_{p+j}=A_j^2(z',\bar z'), \quad1\leq j\leq p
\end{equation} such that the set 
of Moser-Webster involutions
$\{\widetilde{\cL T_1},\rho_0\}$ of $M$
is holomorphically equivalent to
 $\{\cL T_1,\rho\}$.
\end{prop}
\begin{proof} We  recall from \re{pi1M} the branched  covering
\eq{pi1M+}\nonumber
\pi_1\colon \cL M_1:= \mathcal M\cap((\Delta_{\delta}^{p}\times \Delta_{\delta^2}^{p})
\times(\Delta_{C\delta}^{p}\times \Delta_{C\delta^2}^{p}))\longrightarrow \Delta_{\delta}^{p}\times \Delta_{\delta^2}^{p}.
\eeq
Here $C\geq1$.
  Let $\pi_1$
be restricted to $\cL M_1$.  
  Then $\pi_2=\ov{\pi_1\circ\rho}$ is defined on $\rho(\cL M_1)$. Note that
\eq{pi2M}\nonumber
\pi_2\colon \cL \rho(\cL M_1)
\longrightarrow \Delta_{\delta}^{p}\times \Delta_{\delta^2}^{p}.
\eeq
We have $\pi_1^{-1}(z)\cap \fix(\rho)=\{(z,\ov z)\}$ for $z\in M$
and $\pi_1(\fix(\rho))=M$.
Let   $\cL B_0\subset\Delta_{\delta}^{p}\times \Delta_{\delta^2}^{p}$
be the branched  locus. Take $\cL B=\pi_1^{-1}(\cL B_0)$.
We will denote by $\widetilde {\cL M}_1,\tilde{\cL B}$ and $\tilde{\cL B_0}$ the corresponding
  data for $\widetilde M$. Here $\widetilde {\cL M}_1$ is an analogous
   branched  covering over $\pi_1(\widetilde{\cL M}_1)$. We assume that
  the latter contains $f(\pi_1(\cL M_1))$ if $\widetilde M$ is equivalent to $M$
  via $f$.

Assume that $f$ is a biholomorphic map sending $M$ into $\widetilde M$. Let
$f^c$ be the restriction of biholomorphic map
$f^c(z,w)=(f(z),\ov f(w))$ to $\cL M$. Let $M$ be defined by $z''=E(z',\ov z')$
 and $\widetilde M$ be defined by $z''=\tilde E(z',\ov z')$.
 By $f(M)\subset\widetilde M$, $f=(f',f'')$ satisfies
$$
f''(z',E(z',\ov z'))=\tilde E(f'(z', E(z',\ov z')),\ov{f'}(\ov z',\ov E(\ov z',z'))).
$$
Using the defining equations for $\cL M$,  we get
$f^c(\cL M)\subset\widetilde{\cL M}$
and $\rho f^c=f^c\rho$ on $\cL M\cap\rho(\cL M)$.
We will also assume that $ f^c(\cL M_1)$ is contained in $\widetilde{\cL  M}_1$. It is clear that $f^c$
sends a fiber $\pi_1^{-1}(z)$ onto the fiber $\pi_1^{-1}(f(z))$ for $z\in \Om=
\cL \pi_1(\cL M_1)\setminus(\cL B_0\cup f^{-1}
(\tilde B_0))$, since the two fibers have the same
number of points and $f$ is injective.
 Thus $f^c\tau_{1j}=\tilde\tau_{1i_j}f^c$ on $\pi_1^{-1}(\Om)$. Here $i_j$ is of course locally determined on
 $\pi_1^{-1}(\Om)$.
 Since $\cL B$ has positive codimension in $\cL M_1$
then $\cL M_1\setminus\cL B$ is connected. Hence $i_j$ is well-defined on $\pi_1^{-1}(\Om)$.
Then $f^c\tau_{1j}=\tilde\tau_{1i_j} f^c$ on $\cL M_1\setminus B$. This
shows that
$f^c$ conjugates simultaneously the deck transformations of ${\mathcal M}$ to the deck transformations of $\widetilde {\mathcal M}$ for $\pi_1$.  The same conclusion holds for
 $\pi_2$.

Conversely, assume that there is a biholomorphic map $g\colon \cL M\to \widetilde{\cL  M}$
such that $\rho g=g\rho$ and $g\tau_{1i}=\tilde\tau_{1j_i}g$. Since $\tau_{11}, \ldots, \tau_{12^p}$ are distinct
and $M_1\setminus\cL B$ is connected, then $\bigcup_{j\neq i}\{ x\in\cL M_1\setminus\cL B\colon\tau_{1i}(x)=\tau_{1j}(x)\}$
is a complex subvariety  of positive codimension in
$\cL M_1\setminus\cL B$. Its image under the proper projection $\pi_1$ is  a subvariety of positive codimension in $
\Delta_{\delta}^p\times\Delta_{\delta^2}^p\setminus\cL B_0$. This shows that the latter contains a non-empty open subset $\omega$ such that $\{\tau_{11}(x),\ldots,\tau_{12^p}(x)\}=\pi_1^{-1}\pi_1(x)$ has $2^p$ distinct points
 for each $\pi_1(x)\in\omega$. Therefore, $\tau_{11},\ldots, \tau_{12^p}$ are
  all deck transformations of $\pi_1$ over $\om$. Hence they are
   all deck transformations of $\pi_1\colon\cL M_1\setminus\cL B\to\Delta_{\delta}^p\times\Delta_{\delta^2}^p\setminus\cL B_0$, too.
 This shows that $\pi_1^{-1}(\pi_1(x))=\{\tau_{1j}(x)\colon1\leq j\leq 2^p\}$ for $x\in\cL M_1\setminus\cL B$.
Now, $g$ sends $\tau_{1j}(x)$ to $\tilde\tau_{1i_j}(g(x))$ for each $j$. Hence
  $f(z)=\pi_1g\pi_1^{-1}(z)$
is well-defined and holomorphic for $z\in\Delta_{\delta}^p\times\Delta_{\delta^2}^p\setminus\cL B_0$.
By the Riemann extension for bounded holomorphic functions, $f$ extends to a holomorphic mapping, still denoted by $f$,
which is defined near the origin.   We know that
 $f$ is invertible and  in fact the inverse
  can be obtained by extending  the mapping $z\to
 \pi_1g^{-1}\pi_1(z)$.
 If $z=(z',E(z',w'))\in M$, then $w'=\ov z'$
and
$f(z)=\pi_1g\pi_1^{-1}(z)=\pi_1g(z,\ov z)$ with $(z,\ov z)\in \fix(\rho)$. Since $\rho g=g\rho$,
then $g(z,\ov z)\in \fix(\rho)$. Thus $f(z)=\pi_1g(z,\ov z)\in \widetilde M$.

Assume that $\{\tau_{1j}\}$ and $\rho$
are germs of involutions defined at the origin of $\cc^n$. Assume that they satisfy the conditions in the proposition.   From 
\rl{sd} it follows that $\tau_{11}, \ldots, \tau_{1p}$ generate a group of $2^p$ involutions, while the $p$ generators are the only elements of which each
fixes a hypersurface pointwise.
To realize them as deck transformations of the complexification of a real analytic submanifold, we apply   \rl{sd}
to find  a
coordinate map $(\xi,\eta)\to \phi(\xi,\eta)=(A,B)(\xi,\eta)$  such that invariant holomorphic functions of $\{\tau_{1j}\}$ are precisely holomorphic functions in
$$
 z'=(A_1(\xi,\eta),\ldots,
A_p(\xi,\eta)), \quad  z''=(B_1^2(\xi,\eta), \ldots, B_p^2(\xi,\eta)).
$$
Note that $B_j$ is skew-invariant under $\tau_{1j}$ and
is invariant under $\tau_{1i}$
for $i\neq j$ and $A$ is invariant under all $\tau_{1j}$.
Set $$
w_j'=\ov{A_j\circ\rho(\xi,\eta)}, \quad w''_j=\ov{B_j^2\circ\rho(\xi,\eta)}.
$$
Since $\tau_{2j}=\rho\tau_{1j}\rho$, the holomorphic functions invariant under all $\tau_{2j}$ are precisely the holomorphic functions in the above $w',w''$. We now draw conclusions for the linear parts of invariant functions and
involutions. Since $\phi$ is biholomorphic, then
$LA_1,\ldots, LA_p$ are linearly independent. They are also invariant under $L\tau_{1j}$. Since $\tau_{2j}=\rho\tau_{1j}\rho$,  the $p$ functions $L\ov {A_i\circ\rho}$
 are linearly independent and   invariant under $L\tau_{2j}$. Thus $$LA_1,\ldots, LA_p,\
 L\ov {A_1\circ\rho}, \ldots, L\ov{A_p\circ\rho}$$ are linearly independent,
since $[\mathfrak M_n]_1^{L\cL T_1}\cap[\mathfrak M_n]_1^{L\cL T_2}=\{0\}$. This shows that
  the map $(\xi,\eta)\to  (z',w')=(A(\xi,\eta),\ov{A\circ\rho(\xi,\eta)})$
  has
an inverse $(\xi,\eta)=\psi(z',w')$. Define
$$
M\colon z''=(B_1^2,\ldots, B_p^2)\circ\psi(z',\ov z').
$$
The  complexification of $M$ is given by
$$
\cL M\colon
z''=(B_1^2,\ldots, B_p^2)\circ\psi(z',w'),\quad
w''=(\ov B_1^2,\ldots, \ov B_p^2)\circ\ov\psi(w',  z').
$$
Note that $\phi\circ\psi(z',w')=(z',B\circ\psi(z',w'))$ is biholomorphic. In particular, we can write
$$
B_j^2\circ\psi(z',\ov z')=h_j(z',\ov z')+q_j(\ov z')+b_j(z')+O(|(z',\ov z')|^3).
$$
Here $q_j(\ov z')=\tilde q^2_j(\ov z')$, and $\tilde q(w')$
is the linear part of $w'\to B\circ\psi(0,w')$. Therefore, $|q(w')|
\geq c|w'|^2$ and $q_*=0$.  By \rl{2p1}, $\pi_1\colon\cL M\to\cc^p$
is a $2^p$-to-$1$ branched covering defined near
$0\in\cL M$.
Since  $B^2$ is invariant by $\tau_{1j}$, then $z''=B^2\circ\psi(z',w')$ is invariant by $\psi^{-1}\tau_{1j}\psi(z',w')$.
Also $A$ is invariant under $\tau_{1j}$. Then $z'=
A\circ\psi(z',w')$ is invariant by $\psi^{-1}\tau_{1j}\psi(z',w')$.
 This show that $\{\psi^{-1}\tau_{1j}\psi\}$ has the same invariant functions
as of the deck transformations of $\pi_1$. By \rl{twosetin}, $\{\psi^{-1}\tau_{1j}\psi\}$ 
agrees with the set of deck transformations of $\pi_1$.
For $\rho_0(z',w')=(\ov{w'},\ov{z'})$ we have $\rho_0\psi^{-1}=\psi^{-1}\rho.$
This shows that $M$ is a realization for    $\{\tau_{11},\ldots, \tau_{1p},\rho\}$.
\end{proof}

\begin{rem} (i)
We choose the realization in such a way that $z_{p+j}$ are square functions. Of course, the choice is not unique. In fact, we can
replace $z''$ by  $f(z)=(f_{p+1}(z),\ldots, f_{2p}(z))$ as long as  the mapping  $z\to(z',f(z))$ is biholomorphic.  However, this particular holomorphic equivalent form of $M$ will be crucial
to study the asymptotic manifolds in section~\ref{secideal}.   In fact,  by 
Example~\ref{mgae}, \re{reformu} provides a general equation
for $M$ to admit $2^p$ deck transformations. 
(ii)
An interesting  case is when $f(z)$ can be so chosen that $M$ is holomorphically flattened, i.e, $M$ is contained in $\IM z''=0$.
In \cite{MW83},  such a choice is always possible at least at the formal level. We will discuss the holomorphic flatness  in \rt{abelm}.
\end{rem}

Next we want to compute the deck transformations for a product quadric. 
We will  first recall the Moser-Webster involutions for elliptic and hyperbolic complex tangents.  We will then
compute the deck transformations for complex tangents of complex type.

Let us first recall involutions in \cite{MW83} where the complex tangents
are elliptic (with non-vanishing Bishop invariant) or hyperbolic. 
When $\gaa_1\neq0$, the non-trivial
 deck transformations of 
\eq{z2z12}
\nonumber
z_2=|z_1|^2+\gaa_1(z_1^2+\ov z_1^2)
\eeq
 for $\pi_1,\pi_2$ are $\tau_1,\tau_2$,  respectively. They are
\eq{tau10e}
\nonumber
\tau_1\colon  z_1' = z_1, \quad w_1'=-w_1 -\gamma_1^{-1}z_1; \quad \tau_2=\rho\tau_1\rho
\eeq
with $\rho$ being defined by (\ref{antiholom-invol}). Note that $\tau_1$ and $\tau_2$
do not commute and $\sigma=\tau_1\tau_2$ satisfies
 \eq{} \nonumber
 \sigma^{-1}=\tau_i\sigma\tau_i=\rho\sigma\rho, \quad \tau_i^2= I,\quad\rho^2= I.
 \eeq
When the complex tangent is not parabolic, the eigenvalues of $\sigma$ are   $\mu,\mu^{-1}$ with
$\mu=\la^2$ and
$$
\gaa\la^2 -\la +\gaa=0.
$$
For the elliptic complex tangent, we can choose a solution  $\la>1$, and in suitable coordinates we obtain
\begin{gather}
\label{tau1e}
\tau_1\colon\xi'=\la\eta+O(|(\xi,\eta)|^2), \quad \eta'=\la^{-1}\xi+O(|(\xi,\eta)|^2),\\
\tau_2=\rho\tau_1\rho, \nonumber
\quad 
\nonumber
\rho(\xi,\eta)=(\ov\eta,\ov\xi),\\
\sigma\colon\xi'=\mu\xi+O(|(\xi,\eta)|^2), \quad \eta'=\mu^{-1}\eta+O(|(\xi,\eta)|^2),\quad \mu=\la^2.\nonumber
\end{gather}
When the complex tangent is hyperbolic, i.e.
$1/2<\gaa<\infty$,    $\tau_i$ and  $\sigma$ still have the above form, while $|\mu|=1=|\la|$ and 
\eq{rhohy}
\nonumber
\rho(\xi,\eta)
=(\ov\xi,\ov\eta).
\eeq
When the complex tangent is parabolic, i.e. $\gaa=1/2$, the pair of involutions
still exists. However, $L\sigma$ is not diagonalizable and $1$ is its only eigenvalue.

 For the complex type, new situations arise.   Recall that such a quadric has the form
 \eq{z3z42} Q_{\gaa_s}\colon  z_3=z_1\ov z_2+\gaa_s\ov z_2^2+(1- \gaa_s)z_1^2, \quad z_4=\ov z_3.
\eeq
Here $\gaa_s$ is a complex number. 
Let us first check  that such a quadric   is not the product of two Bishop quadrics~:
 Its CR singular set is defined by
$$
(z_1+2\gaa_s \ov z_2)(z_2+2(1-\ov\gaa_s)\ov z_1)=0. 
$$
It is the union of a complex line and a totally real plane, or two totally real planes. 
The CR singular set of  a quadric defined by $z_3=|z_1|^2+\gaa_1(z_1^2+\ov z_1^2)$
and $z_4=|z_2|^2+\gaa_2(z_2^2+\ov z_2^2)$ is given by
$$
(z_1+2\gaa_1 \ov z_1)(z_2+2\gaa_2\ov z_2)=0. 
$$
It is the union of two complex lines, or one complex line and a $3$ dimensional plane.

By condition B, we know that $\gaa_s\neq0,1$.   Let us compute the deck transformations of the complexification of  \re{z3z42}. 
 According to   \rl{sd1} (i), the deck transformations for $\pi_1$ are generated by two involutions
\ga\tau_{11}
\label{tau1112}
\nonumber
\colon \begin{cases}
z_1'=z_1,\\
z_2'=z_2,\\
w_1'=-w_1-\gaa_s^{-1} z_2,\\
w_2'=w_2;
\end{cases}\quad
\tau_{12}\colon \begin{cases}
z_1'=z_1,\\
z_2'=z_2,\\
w_1'=w_1,\\
w_2'=-w_2-(1-\ov\gaa_s)^{-1} z_1.
\end{cases}
\end{gather}
We still have $\rho$ defined by (\ref{antiholom-invol}).  Let $\tau_{2j}=\rho\tau_{1j}\rho$. Then $\tau_{21},\tau_{22}$
generate the deck transformations of $\pi_2$. Note that
\ga
\nonumber
\tau_{21}\colon \begin{cases}
z_1'=-z_1-\ov\gaa_s^{-1} w_2,\\
z_2'=z_2,\\
w_1'=w_1,\\
w_2'=w_2;
\end{cases}\quad
\tau_{22}\colon \begin{cases}
z_1'=z_1,\\
z_2'=-z_2- (1-\gaa_s)^{-1} w_1,\\
w_1'=w_1,\\
w_2'=w_2.
\end{cases}
\end{gather}
Recall that $\tau_i=\tau_{i1}\tau_{i2}$ is the unique deck transformation of $\pi_i$ that has the smallest dimension of the set of fixed-points among all deck transformations. They are
\ga\tau_{1}\colon \begin{cases}
z_1'=z_1,\\
z_2'=z_2,\\
w_1'=-w_1-\gaa_s^{-1} z_2,\\
w_2'=-w_2-(1-\ov\gaa_s)^{-1} z_1;
\end{cases}
\tau_{2}\colon   \begin{cases}
z_1'=-z_1-\ov\gaa_s^{-1} w_2,\\
z_2'=-z_2- (1-\gaa_s)^{-1} w_1,\\
w_1'=w_1,\\
w_2'=w_2.
\end{cases}
\nonumber
\end{gather}
  And $\tau_1\tau_2$ is given by
\ga\sigma_s\colon \begin{cases}
z_1'=-z_1-\ov\gaa_s^{-1}w_2,\\
z_2'=-z_2- (1-\gaa_s)^{-1} w_1,\\
w_1'=\gaa_s^{-1} z_2+((\gaa_s-\gaa_s^2)^{-1}-1)w_1,\\
w_2'= (1-\ov\gaa_s)^{-1}z_1+((\ov\gaa_s-\ov\gaa_s^2)^{-1}-1)w_2.
\end{cases}
\nonumber
\end{gather}

In contrast to the elliptic and hyperbolic cases, $\tau_{11}$  and $\rho\tau_{11}\rho$ commute; in other words, $\tau_{11}\rho\tau_{11}\rho$ is actually an involution.
And $\tau_{12}$ and $\rho\tau_{12}\rho$ commute, too.  However, $\tau_{11}$ and $ \tau_{22}$ do not commute,  and $\tau_{12}, \tau_{21}$ do not commute either.  Thus, we form compositions $$\sigma_{s1}=\tau_{11}\tau_{22}, \quad \sigma_{s2}=\tau_{12}\tau_{21}, \quad\sigma_{s2}^{-1}=\rho\sigma_{s1}\rho.$$
By a simple computation, we have
\ga  \nonumber
\sigma_{s1}\colon \begin{cases}
z_1'=z_1,\\
z_2'=-z_2- (1-\gaa_s)^{-1} w_1,\\
w_1'=\gaa_s^{-1} z_2+((\gaa_s-\gaa_s^2)^{-1}-1)w_1,\\
w_2'=w_2;
\end{cases}\\
\sigma_{s2}\colon \begin{cases}
z_1'=-z_1-\ov\gaa_s^{-1}w_2,\\
z_2'=z_2,\\
w_1'=w_1,\\
w_2'= (1-\ov\gaa_s)^{-1}z_1+((\ov\gaa_s-\ov\gaa_s^2)^{-1}-1)w_2.
\end{cases}
\nonumber
\end{gather}
We verify that $$
\sigma_{s1}\sigma_{s2}=\sigma_s=\tau_1\tau_2.
$$
This allows us to compute the eigenvalues of $\sigma_{s1}\sigma_{s2}$ easily: 
\begin{gather}\label{msms-}
\mu_s, 
\quad  \mu_s^{-1}, 
\quad
  \ov \mu_s^{-1}, \quad 
  \ov\mu_s, \\  
\mu_s=(\gaa_s^{-1}-1)^{-1}.
\nonumber
\end{gather} 
In fact we compute them by observing that the first two  in \re{msms-} and $1$ with multiplicity are eigenvalues of $\sigma_{s1}$, while the last two  in \re{msms-} and $1$ with multiplicity are eigenvalues of $\sigma_{s2}$. Therefore,  for $\gaa_s\neq1/2$, i.e.
 $\mu_s\neq1$, we can find a linear transformation of the form
$$
\psi\colon (z_1,w_2)\to(\xi_2,\eta_2)=\ov\phi(z_1,w_2), \quad
(z_2,w_1)\to(\xi_1,\eta_1)=\phi(w_1,z_2)$$
 such that $\sigma_{s1},\sigma_{s2},\sigma_s=\sigma_{s1}\sigma_{s2}$ are simultaneously diagonalized as
\begin{equation}\label{linears}
\begin{array}{rrclrclrclrcl}
\sigma_{s1}\colon & \xi_1'&\!\!=\!\!&\mu_s\xi_1, \quad & \eta_1'&\!\!=\!\!&\mu_s^{-1}\eta_1,& \quad \xi_2'&\!\!=\!\!&\xi_2,\quad& \eta_2'&\!\!=\!\!&\eta_2, \\
\sigma_{s2}\colon &\xi_1'&\!\!=\!\!&\xi_1,\quad &\eta_1'&\!\!=\!\!&\eta_1,\quad&\xi_2'&\!\!=\!\!&\ov\mu_s^{-1}\xi_2,\quad&\eta_2'&\!\!=\!\!&\ov\mu_s\eta_2,\\
\sigma_s\colon&\xi_1'&\!\!=\!\!&\mu_s\xi_1, \quad&\eta_1'&\!\!=\!\!&\mu_s^{-1}\eta_1, \quad& \xi_2'&\!\!=\!\!&\ov\mu_s^{-1}\xi_2,\quad&\eta_2'&\!\!=\!\!&\ov\mu_s\eta_2.
\end{array}
\end{equation}
Under the transformation $\psi$, the involution $\rho$, defined by (\ref{antiholom-invol}),
takes the form
 \eq{rhos}
\rho(\xi_1,\xi_2,\eta_1,\eta_2)=(\ov\xi_2,\ov\xi_1,\ov\eta_2,\ov\eta_1).\eeq
  Moreover, for $i,j=1,2$, we have
\ga
\tau_{ij}\colon\xi_j'=\la_j\eta_j, \quad \eta_j'=\la_j^{-1}\xi_j; \quad\xi_i'=\xi_i, \quad\eta_i'=\eta_i, \quad i\neq j;\label{taus}\\
\la_1=\la_s, \quad \la_2=\ov\la_s^{-1}, \quad\mu_s=\la_s^2.\nonumber 
\end{gather}
 When $\gaa_s=1/2$,   the only eigenvalue of $\sigma_{s1}$ is $1$.     We can choose a suitable $\phi$ such that $\psi$ transforms $\sigma_{s1},\sigma_{s2},\sigma_s$ into
\begin{equation}\label{jsig}
\begin{array}{rrclrclrclrcl}
\sigma_{s1}\colon &\xi_1'&\!\!=\!\!&\xi_1, \quad &\eta_1'&\!\!=\!\!&\eta_1+\xi_1, \quad &\xi_2'&\!\!=\!\!&\xi_2,\quad &\eta_2'&\!\!=\!\!&\eta_2\\
\sigma_{s2}\colon &\xi_1'&\!\!=\!\!&\xi_1,\quad &\eta_1'&\!\!=\!\!&\eta_1,\quad  &\xi_2'&\!\!=\!\!&\xi_2,\quad&\eta_2'&\!\!=\!\!&-\xi_2+\eta_2,\\
\sigma_s\colon      &\xi_1'&\!\!=\!\!&\xi_1, \quad &\eta_1'&\!\!=\!\!&\xi_1+\eta_1, \quad &\xi_2'&\!\!=\!\!& \xi_2,\quad&\eta_2'&\!\!=\!\!&-\xi_2+\eta_2.
\end{array}
\end{equation}
Note that  eigenvalues formulae \re{msms-} and  the Jordan normal form \re{jsig} tell us that $\tau_1$ and $\tau_2$ do not commute, while $\sigma_{s1},\sigma_{s2}$ commute as mentioned earlier.

\begin{rem}
The mappings $\sigma_{s1},\sigma_{s2}$ behave like a hyperbolic mapping  when $|\mu_s|>1$,   an elliptic mapping when $|\mu_s|=1$,  or
a parabolic mapping when $\mu_s=1$. Recall that $\sigma$ has $4$ distinct eigenvalues for the first case,
$2$ distinct eigenvalues with multiplicity for the second case, and only eigenvalue of 1 for the last
case. The $\sigma$ is diagonalizable for the first two cases, but it has a Jordan bock with multiplicity
for the last case.
In this paper, we will only study the first case of complex type, i.e.
$$
|\mu_s|>1,
$$
 which follows from  condition E. 
\end{rem}

For later purpose, we summarize some   facts for complex type in the following.
\begin{prop}\label{sigs} Let $Q_{\gaa_s}\subset\cc^4$ be the   quadric defined by \rea{z3z4} and \rea{gsra}.  Then $\pi_1$ admits two deck transformations $\tau_{11},\tau_{12}$ such that 
the set of fixed points of each $\tau_{1j}$ has dimension $3$. Also, $\tau_{2j}=\rho\tau_{1j}\rho$ are the deck transformations of $\pi_2$ and \ga
\tau_{11}\tau_{21}=\tau_{21}\tau_{11}, \quad
\tau_{12}\tau_{22}=\tau_{22}\tau_{12}.
 \nonumber\end{gather}
Let $\sigma_{s1}=\tau_{11}\tau_{22}$, $\sigma_{s2}=\tau_{12}\tau_{21}$, $\tau_i=\tau_{i1}\tau_{i2}$, and $\sigma_s=\tau_1\tau_2$. Then
\eq{}  \nonumber
\sigma=\sigma_{s1}\sigma_{s2}=\sigma_{s2}\sigma_{s1}, \quad \sigma_{s2}^{-1}=\rho\sigma_{s1}\rho, \quad \sigma_s^{-1}=\rho\sigma_s\rho.
\eeq
In suitable coordinates $\sigma_{s1},\sigma_{s2},\sigma,\rho_s$ are given by \rea{linears}-\rea{rhos}
when   $\gaa_s\neq1/2$; when $\gaa_s=1/2$, they are given by \rea{rhos}
and
\rea{jsig}.
If $$\gaa_s\in \{z\in\cc\colon \RE z>1/2, \IM z>0\},$$ 
then $\sigma_s$ admits $4$ distinct eigenvalues  $(\gaa_s^{-1}-1)^{-1},  \ov\gaa_s^{-1}-1,
\gaa_s^{-1}-1$, and $(\ov\gaa_s^{-1}-1)^{-1}$.
\end{prop}

The commutativity of $\sigma_h,\sigma_e,\sigma_{s1},\sigma_{s2}$ will be important to understand the convergence of normalization for the  abelian CR singularity to be introduced in section~\ref{abeliancr}.

 Let us summarize some facts in this section. 

Let $\tau_i=\tau_{i1}\cdots\tau_{ip}$ for $i=1,2$.
 Note
that  they
are intertwined by the anti-holomorphic involution via $\tau_2=\rho \tau_1\rho$.
Each
$\tau_i$ is the unique deck transformation  for $\pi_i$ whose set of fixed points has minimum dimension $p$. Then
$\sigma=\tau_1\tau_2$ is {\it reversible} by   $\tau_i$ and $\rho$ in the sense that
$$
\tau_i\sigma\tau_i=\sigma^{-1}, \quad\rho\sigma\rho=\sigma^{-1},
\quad\tau_i^2= I,\quad \rho^2= I.
$$
The reversible map $\sigma$
will play a central role to the study of the submanifolds $M$, as we will demonstrate this
in the classification of quadratic manifolds. In particular, they carry
 some geometry and dynamics  associated to the real manifolds; for instance the attached complex submanifolds are closely related the 
 invariant submanifolds of $\sigma$, which is discussed in section~\ref{secideal}.
We will also call $\tau_{11},\ldots, \tau_{1p}$ the generators of the deck transformations, which
are unique as  each $\fix(\tau_{1j})$
has codimension $1$.

For various reversible mappings and their relations with general mappings, the reader is referred  to \cite{OZ11} for recent results and references therein.

To derive our normal forms, we shall transform  $\{\tau_1,\tau_2,\rho\}$ into 
a 
normal form first.
We will  further normalize $\{\tau_{1j},\rho\}$ by using the group of biholomorphic maps
that preserve the normal form of $\{\tau_1,\tau_2,\rho\}$, i.e. the centralizer of the normal form  of $\{\tau_1,\tau_2,\rho\}$.

\setcounter{thm}{0}\setcounter{equation}{0}

\section{
Quadrics with the maximum number of deck transformations}
\label{secquad}
 In section~\ref{secinv}, 
we  establish  the basic relation between the classification of real manifolds and that of  two families of involutions intertwined
by an antiholomorphic involution; see \rp{inmae}.
As a first application, we obtain  in this section a normal form for  two families of linear involutions
and use it to construct the normal form for their associated quadrics. This section
also serves an introduction to our approach to find the normal forms of the real submanifolds   at least at
the formal level.

\subsection{Normal form of two families of linear involutions}

To formulate our results, we first discuss the normal forms  which we are seeking for the involutions.
We are given two families of commuting linear involutions
$ \cL T_1
=\{T_{11}, \ldots, T_{1p}\}$ and $\cL T_2=\{T_{21},\ldots, T_{2p}\}$ with $T_{2j}=\rho T_{1j}\rho$. Here $\rho$ is a linear anti-holomorphic involution. We 
 set
$$  T_1=T_{11}\cdots T_{1p}, \quad T_2=\rho T_1\rho. $$
Recall that our involutions satisfy the additional   \re{Mn1g} and \re{ilii+}. Thus
\ga\label{ilii++}
  \dim [\mathfrak M_n]_1^{\cL T_i}=p,\quad  [\mathfrak M_n]_1^{\cL T_i}= [\mathfrak M_n]_1^{T_i},
 \\
 [\mathfrak M_n]_1^{ T_1}\cap[\mathfrak M_n]_1^{T_2}
 =\{0\}.\label{bilii++}
 \end{gather}
  Recall that  $[\mathfrak M_n]_1$ denotes the linear functions   without constant terms.   We would like to find
a change of coordinates $\var$ such that $\var^{-1} T_{1j}\var$ 
and
$\var^{-1}\rho\var$ have a simpler form. We would like
to show that two such families of involutions $\{\cL T_1,\rho\}$ and
$\{\widetilde{\cL T_1},\tilde\rho\}$ are holomorphically equivalent, if there are
normal forms are equivalent under a much smaller set of changes of coordinates,
or if they are identical in the ideal situation.

Next, we describe our plans to derive the normal forms for linear involutions.
The scheme  to derive the linear normal forms
turns out to be essential to understand the derivation of normal forms for non-linear involutions and the perturbed quadrics. We define 
$$
 S=T_1T_2.
$$
Besides   conditions   \re{ilii++}-\re{bilii++}, we will soon impose condition E below that $S$ has $2p$ distinct eigenvalues. 

We  first use a linear map $\psi$ to
 diagonalize $S$ to its normal form
$$
\hat S\colon \xi_j'=\mu_j\xi, \quad \eta_j'=\mu_j^{-1}\eta_j,\quad 1\leq j\leq p.
$$
The choice of $\psi$ is not unique.
We   further normalize $T_1,T_2, \rho$ under linear transformations commuting with $\hat S$, i.e.
the   invertible mappings in the 
{\it linear centralizer} of $\hat S$. We   use a linear map that commutes with $\hat S$ to
 transform $\rho$ into a
normal form too, which is still
 denoted by $\rho$. 
 We then use a transformation $\psi_0$ in the linear centralizer of $\hat S$ and $\rho$
  to normalize the   $T_1,T_2$ into the normal form
$$
\hat T_i\colon \xi_j'=\lambda_{ij}\eta_j, \quad \eta_j'=\lambda_{ij}^{-1}\xi_j, \quad 1\leq j\leq p.
$$
Here we require $\lambda_{2j}=\lambda_{1j}^{-1}$. Thus $\mu_j=\lambda_{1j}^2$ for $1\leq j\leq p$, and
 $\lambda_{11}, \ldots, \lambda_{1p}$ form a complete set of invariants
of $T_1,T_2,\rho$, provided they are normalized into the regions 
$$
\la_{1e}>1, \quad \IM\la_{1h}>0, \quad \arg\la_{1s}\in(0,\pi/2), \quad |\la_s|>1.
$$

 Next we   normalize the family  $\cL T_1$ of linear involutions
 under mappings in the linear centralizer of $\hat T_1, \rho$. Let us assume that   $T_1,\rho$ 
 are in the normal forms $\hat T_1,\rho$.
To normalize the families $\{\cL T_1,\rho\}$,   we use the crucial property that $T_{11},\ldots, T_{1p}$
commute pairwise and each $T_{1j}$  fixes a hyperplane. 
 This allows us to express the family of involutions via a single linear mapping $\phi_1$:
$$
T_{1j}=\var_1\phi_1Z_j\phi_1^{-1}\var_1^{-1}.
$$
Here the linear mapping $\var_1$ depends only on $\la_1,\ldots, \la_p$  and 
$$
Z_j\colon \xi'=\xi, \quad \eta_i'=\eta_i \ (i\neq j), \quad \eta_j'=-\eta_j.
$$
Expressing  $\phi_1$
 in  a  non-singular $p\times p$  constant
matrix $\mathbf B$, the normal form for $\{T_{11}, \ldots, T_{1p},\rho\}$ consists of invariants $\la_1,\ldots, 
\la_p$ and a normal form  of $\mathbf B$.  After we obtain the normal form for  $\mathbf B$,
we will construct the normal form of the quadrics by using the realization procedure  in the proof
of \rp{inmae}.

\medskip

We now carry out the details.

Let  $T_1=T_{11}\cdots T_{1p}$,
   $T_2=\rho T_1\rho$ and
$
S=T_1T_2.
$
   Since $T_i$ and $\rho$ are involutions, then $S$ is reversible with respect to $T_i$
   and $\rho$, i.e. $$S^{-1}=T_i^{-1}ST_i, \quad S^{-1}=\rho^{-1} S\rho, \quad T_i^2= I, \quad\rho^2= I.
   $$
   Therefore,
  if $\kappa$ is an eigenvalue of $S$ with a  (non-zero)   eigenvector $u$, then
\eq{} 
Su=\kappa u, \quad S(  T_iu)=\kappa^{-1}T_iu, \quad S(\rho u)=\ov\kappa^{-1}\rho u,\quad
S(\rho T_iu )=\ov\kappa\rho T_iu. \nonumber
\eeq
  Following \cite{MW83} and  [St07], we will divide eigenvalues into $4$ types: $\mu$ is {\it elliptic} if $\mu\neq\pm1$
and $\mu$ is real, $\mu$ is {\it hyperbolic} if $|\mu|=1$ and $\mu\neq1$,  $\mu$ is
{\it parabolic } if $\mu=1$, and $\mu$ is {\it complex} otherwise.   The classification of $\sigma$ into the types 
corresponds to the classification of the types of complex tangents described in section~\ref{secinv}; namely,  an elliptic (resp. hyperbolic)  complex tangent is tied to 
 a hyperbolic (resp. elliptic) mapping $\sigma$. A complex tangent of parabolic (reps. complex) type is tied to a mapping
of parabolic (resp.
complex)  type.

  To classify the families of linear involutions,  we   need  a mild assumption to
exclude multiplicity in $\gaa_1,\ldots, \gaa_p$
and also parabolic complex tangent at the origin.  We therefore impose the following condition on quadrics.

\medskip
\noindent
{\bf Condition E.}
The composition $S$ has $2p$ distinct eigenvalues.
 \medskip

\begin{lemma}\label{incomp}
Under conditions E and \rea{bilii++},  neither $1$ nor $-1$ is an eigenvalue of $S$.
\end{lemma}
\begin{proof}
 Assume for the sake
of contradiction that $1$ is an eigenvalue. We have seen that
eigenvalues arrive in pairs $\mu, \mu^{-1}$ if $\mu\neq\pm1$.
Since there are $n=2p$ eigenvalues by condition E,  both $-1$ and $1$ are eigenvalues. Let $u, v$ be
eigenvectors such that
\begin{gather*}
Su=u, \quad T_1u=\e_1 u, \quad T_2u=\e_1u, \quad \e_1=\pm1;\\
Sv=-v, \quad T_1v=\e_2 v, \quad T_2v=-\e_2v, \quad\e_2=\pm1.
\end{gather*}
Since $\fix(T_1)\cap\fix(T_2)=\{0\}$, then $\e_1=-1$. Without loss of generality, we
may assume that $\e_2=-1$. Let $V$ be the span of eigenvectors of $S$ with
eigenvalues other than $\pm1$.
Thus $T_i$ preserves $V$, $\dim V=2p-2$,
$\dim\fix(T_1|_V)=p$, and
 $\dim\fix(T_2|_V)=p-1$. Since $p+(p-1)>2p-2$, then $\fix(T_1)\cap\fix(T_2)$ has dimension
 at least one, a contradiction.
 \end{proof}

We now assume conditions E and     \re{ilii++}-\re{bilii++} for the rest of the section to derive a normal form for $T_{1j}$ and $\rho$.

We need to choose the eigenvectors of $S$  and their  
eigenvalues in such a way that $T_1,T_2$ and $\rho$ are in a normal form. We will first
choose eigenvectors to put
$\rho$ into a normal form. After normalizing $\rho$, we will then choose eigenvectors to normalize $T_1$
and $T_2$.

First, let us consider  an elliptic eigenvalue $\mu_e$. Let $u$ be an eigenvector of $\mu_e$. Then $u$ and $v=\rho (u)$ satisfy
\begin{equation}\label{svmN}
S(v)=\mu_e^{-1}v, \quad T_j(u)=\la_j^{-1}v, \quad \mu_e=\la_1\la_2^{-1}.
\end{equation}
Now $T_2(u)=\rho T_1\rho(u)$ implies that
$$\lambda_2=\ov\la_1^{-1},\quad \mu_e=|\la_1|^2.
$$
 Replacing $(u,v)$ by $(cu,\ov cv)$, we may assume that $\la_1>0$
 and $\la_2=\la_1^{-1}$. Replacing $(u,v)$ by $(v,u)$ if necessary,
  we may further  achieve
 $$
 \rho(u)=v, \quad \la_1=\la_e>1, \quad\mu_e=\la_e^2>  1.
 $$
We still have the freedom to replace $(u,v)$ by $(ru,rv)$ for $r\in\rr^*$, while
preserving the above conditions. 

Next,  let $\mu_h$ be a
 hyperbolic eigenvalue of $S$ and $S(u)=\mu_h u$. Then $u$ and
$v=T_1(u)$ satisfy
$$
\rho (u)=au, \quad \rho (v)=bv, \quad |a|=|b|=1.
$$
Replacing $(u,v)$ by $(cu,v)$, we may assume that $a=1$.
Now $T_2(v)=\rho T_1\rho(v)=\ov bu$.   To obtain $b=1$, we replace $(u,v)$ by $(u,\sqrt b^{-1}v)$.
This give us \re{svmN} with $|\la_j|=1$. Replacing $(u,v)$ by $(v,u)$ if necessary,
we may further achieve
\eq{}
\nonumber
\rho (u)=u, \quad \rho( v)=v, \quad \la_1=\la_h, \quad \mu_h=\la_h^2, \quad  \arg\la_h\in(0,\pi/2).
\eeq
Again, we  have the freedom to replace $(u,v)$ by $(ru,rv)$ for $r\in\rr^*$, while
preserving the above conditions.

 Finally, we consider a complex eigenvalue $\mu_s$. Let $S(u)=\mu_s u$. Then $\tilde u=\rho (u)$
satisfies $S(\tilde u)=\ov\mu_s^{-1}\tilde u$.  Let $u^*=T_1(u)$
and $\tilde u^*=\rho(u^*)$. Then $S(u^*)=\mu_s^{-1}u^*$ and $S(\tilde u^*)=\ov\mu_s \tilde u^*$. We change eigenvectors by
$$(u,\tilde u,u^*,\tilde u^*)\to (u,\tilde u, cu^*,\ov c\tilde u^*)$$
so that
\gan
\rho(u)=\tilde u, \quad\rho(u^*)=\tilde u^*,
\\ T_j(u)=\la_j^{-1}u^*,\quad
T_j(\tilde u)=\ov\la_j\tilde u^*,\quad \la_2=\la_1^{-1}.
\end{gather*}
Note that $S(u)=\la_1^2 u$, $S(u^*)=\la_1^{-2}u^*$, $S(\tilde u)=\ov\la_1^{-2}\tilde u$, and $S(\tilde u^*)=\ov\la_1^2\tilde u^*$.
Replacing $(u,\tilde u,u^*,\tilde u^*)$ by $(u^*,\tilde u^*,u,\tilde u)$ changes the argument and the modulus of $\la_1$ as  $\la_1^{-1}$
becomes $\la_1$. Replacing them by $(\tilde u,u,\tilde u^*,u^*)$ changes only the modulus as
 $\la_1$ becomes $ \bar \la_1^{-1}$ and then replacing them by $(u^*,\tilde u^*, -u,-\tilde u)$ changes the sign of $\la_1$. Therefore, we
may achieve
\eq{las}\nonumber
\mu_s=\la_s^2, \quad\la_1=\la_s, \quad  \arg\gaa_s\in(0,\pi/2), \quad |\la_s|>1.
\eeq
 We still have the freedom to replace $(u,u^*,\tilde u,\tilde u^*)$ by $(cu,cu^*,\ov c\tilde u,\ov c\tilde u^*)$.

We  summarize  the above choice of eigenvectors and their corresponding coordinates.
First, $S$
has distinct eigenvalues
$$
\lambda_e^2=\ov\la_e^2, \quad\lambda_e^{-2};\qquad
\lambda_h^2,\quad\ov\lambda_h^2=\lambda_h^{-2}; \qquad \lambda_s^2,\quad \lambda_s^{-2},
\quad\overline\lambda_s^{-2},\quad\overline\lambda_s^2.
$$
Also, $S$ has linearly independent  eigenvectors satisfying
\gan
Su_e=\lambda_e^2u_e, \quad Su_e^*=\lambda_e^{-2}u^*_e,\\
 Sv_h=\la_h^2v_h,\quad Sv_h^*=\la_h^{-2}v_h^*,\\
  Sw_s=\lambda_s^2w_s,
\quad   Sw_s^*=\lambda_s^{-2}w_s^*,
\quad S\tilde w_s=\ov\lambda_s^{-2}\tilde w_s,\quad S\tilde w_s^*=\ov\lambda_s^2 \tilde w_s^*.
\end{gather*}
Furthermore,  the $\rho$, $T_1$, and the chosen eigenvectors of $S$ satisfy
\begin{gather*}
 \rho u_e= u^*_e,\quad  T_1u_e=\lambda_e^{-1}u_e^*;\\
 \rho v_h= v_h,\quad \rho v_h^*=v_h^*,
 \quad T_1v_h=\lambda_h^{-1}v_h^*; \\
 \rho w_s= \tilde w_s,\quad \rho w_s^*=\tilde w_s^*,
 \quad T_1w_s=\lambda_s^{-1}  w_s^*,\quad T_1\tilde w_s=\ov\lambda_s\tilde w_s^*.
 \end{gather*}

 For normalization, we collect elliptic eigenvalues
$\mu_e$ and $ \mu_e^{-1}$,    hyperbolic eigenvalues 
$\mu_h$ and $\mu_h^{-1}$, and   complex eigenvalues
in $\mu_s,\mu_s^{-1}, \ov\mu_s^{-1}$ and $\ov\mu_{s}$.
We put them in the order
\gan
\mu_e=\ov\mu_e,\quad \mu_{p+e}=\mu_e^{-1},  \\
\mu_h,\quad\mu_{p+h_*+h}=\ov\mu_h,   \\
 \mu_s,\quad\mu_{s+s_*}=
\overline\mu_s^{-1},\quad \mu_{p+s}= \mu_s^{-1},
\quad\mu_{p+s_*+s}=\overline\mu_s .
\end{gather*}
Here and throughout the paper the ranges of subscripts $e,h, s$ are restricted to
$$
1\leq e\leq e_*, \quad e_*< h\leq e_*+h_*, \quad e_*+h_*<s\leq p-s_*.
$$
 Thus
$
e_*+h_*+2s_*=p.
$
Using the new coordinates
$$
\sum(\xi_eu_e+\eta_eu_e^*)+\sum(\xi_hv_h+\eta_hv_h^*)+
\sum (\xi_{s} w_s+\xi_{s+s_*}\tilde w_s+\eta_{s} w^*_s+ \eta_{s+s_*}\tilde w^*_s),
$$
we have normalized $\sigma, T_1, T_2$ and $\rho$.  In summary, we have the following normal form.
\begin{lemma}\label{t1t2sigrho}Let $T_1,T_2$ be linear holomorphic involutions on $\cc^{n}$ that satisfy \rea{bilii++}.
Then $n=2p$ and $\dim [\mathfrak M_n]_1^{T_i}=p$. Suppose that $T_2= \rho_0 T_1\rho_0$ for some anti-holomorphic linear involution $\rho_0$.
Assume that $S=T_1T_2$ has $n$ distinct eigenvalues.
There exists a linear change of holomorphic coordinates that transforms 
 $T_1,T_2, S,\rho_0$  simultaneously   into the normal forms $\hat T_1,\hat T_2,\hat S,\rho:$
\ga
\label{hatt1}
\hat T_1\colon \xi_j'=\la_j\eta_j,\quad \eta_j'=\la_j^{-1}\xi, \quad 1\leq j\leq p;
\\
\hat T_2\colon \xi_j'=\la_j^{-1}\eta_j,\quad \eta_j'=\la_j\xi_j, \quad 1\leq j\leq p;\\
\label{hsxie}\hat S\colon\xi_j'=\mu_j\xi_j, \quad\eta_j'=\mu_j^{-1}\eta_j, \quad 1\leq j\leq p;\\
\label{eqrh}
\rho\colon \left\{\begin{array}{lr}
\xi_e'=\ov\eta_e, &  \qquad \eta_e'= \ov\xi_e,\vspace{.75ex}\\
\xi_h'=\ov\xi_h, & \qquad \eta_h'=\ov\eta_h,\\
\xi_{s}'=\ov\xi_{s+s_*}, &\xi_{s+s_*}'=\ov\xi_{s},\\
\eta_{s}'=\ov\eta_{s+s_*}, &\eta_{s+s_*}'=\ov\eta_{s}.
\end{array}\right.
\end{gather}
Moreover, the eigenvalues $\mu_1,\ldots, \mu_p$ satisfy
\ga\label{mula}
\mu_j=\la_j^2, \quad 1\leq j\leq  p;\\
\label{muor}
\la_e>1,\quad |\la_h|=1,\quad |\la_s|>1,\quad\la_{s+s_*}=\ov\la_{s}^{-1};\\
\arg\la_{h}\in(0,\pi/2), \quad\arg\la_s\in(0,\pi/2);\\
\label{laee} \la_{e'}<\la_{e'+1}, \quad 0<\arg\la_{h'}<\arg\la_{h'+1}<\pi/2;\\
 \label{lass}\text{$\arg\la_{s'}<\arg\la_{s'+1}$, or $\arg\la_{s'}=\arg\la_{s'+1}$
 and $|\la_{s'}|<|\la_{s'+1}|$.}
\end{gather}
Here $1\leq e'<e_*$, $e_*<h'<e_*+h_*$, and $e_*+h_*<s'<p-s_*$. And $1\leq e\leq e_*$, $e_*<h\leq e_*+h_*$, and $e_*+h_*<s\leq p-s_*$.
 If $\tilde S$
is also in the normal form \rea{hsxie} for possible different eigenvalues $\tilde \mu_1,\dots, \tilde \mu_p$
satisfying \rea{mula}-\rea{lass},  then $S$ and $\tilde S$
are equivalent if and only if their eigenvalues are identical.
\end{lemma}
The above normal form of $\rho$ 
 will be fixed
for the rest of paper.  
 Note that in case of non-linear involutions
$\{\tau_{11},\ldots,\tau_{1p},\rho\}$ of which the
linear part are given by $\{T_{11},\dots,T_{1p},\rho\}$
we can always linearize $\rho$ first
under a holomorphic  map of which the
linear part at the origin is described  in above 
normalization 
for the linear part of $\{\tau_{11},\ldots, \tau_{1p},\rho\}$.
Indeed, we may assume that the linear part of the latter family is already 
in the normal form. Then $\psi=\f{1}{2}(I+(L\rho)\circ\rho)$ is tangent
to the identity and $(L\rho)\circ\psi \circ\rho=\psi$, i.e. $\psi$
transforms $\rho$ into $L\rho$ while preserving the linear parts 
of $\tau_{11},\dots, \tau_{1p}$.
  Therefore in the non-linear case, we can 
assume that $\rho$ is given by the above normal form.
 The above lemma   tells us  the ranges
of eigenvalues $\mu_e,\mu_h$ and $\mu_s$ that can be realized by  quadrics that satisfy conditions E and     \re{ilii++}-\re{bilii++}.

Having normalized $T_1$ and $\rho$, we want to  further
 normalize $\{T_{11},\ldots, T_{1p}\}$ under linear maps that preserve the normal forms   of
$   \hat T_1$ and $\rho$.
 We know that the composition of $T_{1j}$ is in the normal form, i.e.
\eq{t11h}
T_{11}\cdots T_{1p}= \hat T_1
\eeq
   is given in \rl{t1t2sigrho}.
We  first need to find an expression for  all   $T_{1j}$ that  commute pairwise and satisfy \re{t11h}, by using
 invariant and skew-invariant functions of $\hat T_1$.   Let
\eq{var1z}\nonumber
(\xi,\eta)=\var_1(z^+,z^-)\eeq
 be defined  by
\begin{alignat}{3}\label{zep1}
&z_e^+ &&=\xi_e+\la_e\eta_e, \quad &&
z_e^-=\eta_e-\la^{-1}_e\xi_e, \\
&z_h^+ &&=\xi_h+\la_h\eta_h, \quad
&&z_h^-=\eta_h-\ov\la_h\xi_h, \\
&z_{s}^+&&=\xi_{s}+\la_s\eta_{s}, \quad
&&z_{s}^-=\eta_{s}-\la^{-1}_s\xi_{s}, \\
 &z_{s+s_*}^+&&=\xi_{s+s_*}+\ov\la^{-1}_s\eta_{s+s_*}, \quad
 &&z_{s+s_*}^-=\eta_{s+s_*}-\ov\la_s\xi_{s+s_*}.
 \label{zep4}
\end{alignat}
 In $(z^+,z^-)$ coordinates, $\var_1^{-1}\hat T_1\var_1$ becomes
$$
Z\colon z^+\to z^+, \quad z^-\to -z^-.
$$
We decompose $Z=Z_{1}\cdots Z_{p}$ by using
$$
Z_{j}\colon (z^+,z^-)\to (z^+, z_1^-,\ldots, z^-_{j-1},-z_j^-,z_{j+1}^-, \ldots,
z_{p}^-).
$$

To keep simple notation, let us use the same notions $x,y$ for a linear transformation $y=A(x)$
and its matrix representation:
$$
A\colon x\to \mathbf Ax.
$$
The following lemma, which can be verified immediately, shows the advantages of coordinates $z^+, z^-$.
\begin{lemma}\label{zcent}The   linear centralizer of $Z$ is the set of  mappings of the form
\eq{vaph}
\phi\colon (z^+,z^-)\to (\mathbf{A}z^+,\mathbf{B}z^-),
\eeq
where $\mathbf{A},\mathbf{B}$ are constant and possibly singular 
matrices.
 Let $\nu$ be a permutation of $\{1,\ldots, p\}$.
Then $Z_j\phi=\phi Z_{\nu(j)}$ for all $j$ if and only  if $\phi$ has the above form with
$\mathbf{B}=\diag_\nu \mathbf{d}$. Here
\eq{dfndiag}
\diag_\nu(d_1,\ldots, d_p):= (b_{ij})_{p\times p}, \quad b_{j\nu(j)}=d_j, \quad b_{jk}=0 \ \text{if $k\neq\nu(j)$}.
\eeq
In particular, the linear centralizer of
$\{Z_{1},\ldots, Z_{p}\}$ is   the set of  mappings \rea{vaph} in  which  $\mathbf{B}$ are diagonal.
 \end{lemma}

 To continue our normalization for the family $\{T_{1j}\}$, we note that
 $\var_1^{-1}T_{11}\var_1,\ldots$, $\var_1^{-1}T_{1p}\var_1$  generate an abelian group of $2^p$  involutions
 and each of these $p$ generators fixes
a hyperplane. By \rl{sd}
 there is a linear transformation $\phi_1$ such that
\eq{phi1-}\nonumber
\phi_1^{-1}\var_1^{-1}T_{1j}\var_1\phi_1=Z_{j}, \quad 1\leq j\leq p.
\eeq
 Computing two compositions on both sides, we see that   $\phi_1$  must be in the linear centralizer of $Z$. Thus, it is in the form \re{vaph}.
 Of course,
$\phi_1$ is not unique; 
$\tilde\phi_1$ is another such linear map for the same $T_{1j}$ if and only if  $\tilde\phi_1=\phi_1\psi_1$
with $\psi_1\in\cL C(Z_{1},\ldots, Z_{p})$.
By \re{vaph}, we may restrict ourselves to $\phi_1$  given by
\eq{phi1z}
\phi_1\colon (z^+,z^-)\to (z^+,\mathbf{B}z^-).
\eeq
Then $\tilde \phi_1$ yields  the same
$T_{1j}$ if and only if its corresponding matrix
 $\widetilde{\mathbf{B}}=\mathbf{B}\mathbf{D}$ for a diagonal matrix $\mathbf{D}$.

In the above we have expressed all $T_{11},\ldots, T_{1p}$ via equivalence classes of matrices. It will be convenient to
restate them via matrices.

For simplicity,  $T_i$ and $S$ denote $\hat T_i,\hat S$, respectively.  In matrices, we write
$$
 T_1\colon 
\left(\begin{array}{c}\xi \\ \eta \end{array}\right)
 \to \mathbf{T}_1\left(\begin{array}{c}\xi \\ \eta \end{array}\right),\quad
 \rho\colon \left(\begin{array}{c}\xi \\ \eta \end{array}\right)\to \boldsymbol{\rho}\left(\begin{array}{c}
 \ov\xi \\ 
 \ov\eta \end{array}\right), \quad
S\colon \left(\begin{array}{c}\xi \\ \eta \end{array}\right)\to \mathbf{S}\left(\begin{array}{c}\xi \\ \eta \end{array}\right).
$$
Recall that  the bold faced  $\mathbf{ A}$  represents a linear map $A$.  Then
\gan
\mathbf{T_1}=\begin{pmatrix}
\mathbf{ 0} &\mathbf{\Lambda}_1  \\
\mathbf{\Lambda}_1^{-1}  &\mathbf{ 0}
\end{pmatrix}_{2p\times2p}, \quad \mathbf{S}=\begin{pmatrix}
\mathbf{\Lambda}_1^2 & \mathbf{ 0} \\
\mathbf{ 0} & \mathbf{\Lambda}_1^{-2}
\end{pmatrix}_{2p\times2p}.
\end{gather*}
We will abbreviate
$$
{\boldsymbol\xi}_{e_*}=(\xi_1,\ldots,\xi_{e_*}), \quad
{\boldsymbol\xi}_{h_*}=(\xi_{e_*+1},\ldots, \xi_{e_*+h_*}), \quad
{\boldsymbol\xi}_{2s_*}=(\xi_{e_*+h_*+1},\ldots,\xi_p).
$$
We use the same abbreviation for $\eta$. Then $({\boldsymbol\xi}_{e_*},{\boldsymbol\eta}_{e_*})$, $({\boldsymbol\xi}_{h_*},{\boldsymbol\eta}_{h_*})$, and
$({\boldsymbol\xi}_{2s_*},{\boldsymbol\eta}_{2s_*})$ subspaces are invariant under $T_{1j}$,
$T_1$, and $\rho$. We also denote by
$T^{e_*}_1, T^{h_*}_1,T^{s_*}_1$ the restrictions of $T_1$ to these
subspaces. Define analogously for the restrictions of $\rho, S$ to these subspaces. 
Define diagonal matrices $\mathbf{\Lambda}_{1e_*}, \boldsymbol{\Lambda}_{1h_*}, \mathbf{\Lambda}_{1s_*}$, of size $e_*\times e_*,h_*\times h_*$ and $ s_*\times s_*$ respectively, by
$$
\boldsymbol{\Lambda}_{1}=\begin{pmatrix}
\mathbf{\Lambda}_{1e_*} &\mathbf{ 0 } &\mathbf{ 0} &\mathbf{ 0} \\
\mathbf{ 0} & \mathbf{\Lambda}_{1h_*}&  \mathbf{ 0}&\mathbf{ 0}\\
\mathbf{ 0} & \mathbf{ 0} &\mathbf{\Lambda}_{1s_*} &\mathbf{ 0}\\
\mathbf{ 0}&\mathbf{ 0}&\mathbf{ 0}&\ov {\mathbf{\Lambda}}_{1s_*}^{-1}
\end{pmatrix} ,\quad
\ov{\boldsymbol{\Lambda}_{1}}=\begin{pmatrix}
\mathbf{\Lambda}_{1e_*} &\mathbf{ 0}  &\mathbf{ 0} &\mathbf{ 0} \\
\mathbf{ 0} & \mathbf{\Lambda}_{1h_*}^{-1}& \mathbf{  0}&\mathbf{0}\\
\mathbf{0} & \mathbf{0} &\ov \Lambda_{1s_*} &\mathbf{0}\\
\mathbf{0}&\mathbf{\mathbf{0}}&\mathbf{0}&\mathbf{\Lambda}_{1s_*}^{-1}
\end{pmatrix}.
$$
Thus, we can
express  $T_1^{s_*}$ and $S^{s_*}$ in $(2s_*)\times(2s_*)$
matrices
$$
\mathbf{T}_1^{s_*}
=\begin{pmatrix}
\mathbf{0}& \mathbf{0} &\mathbf{\Lambda}_{1s_*} &  \mathbf{0} \\
 \mathbf{0}& \mathbf{0} &\mathbf{0}&\ov{\mathbf{\Lambda}}_{1s_*}^{-1}\\
 \mathbf{\Lambda}_{1s_*}^{-1}&\mathbf{0}&\mathbf{0} &\mathbf{0}\\
\mathbf{0} &\ov{\mathbf{\Lambda}}_{1s_*}&\mathbf{0} & \mathbf{0}
\end{pmatrix},\quad
\mathbf{S}^{s_*}=\begin{pmatrix}
\mathbf{\Lambda}_{1s_*}^2 & \mathbf{0}& \mathbf{0} & \mathbf{0} \\
\mathbf{0} &\ov{\mathbf{\Lambda}}_{1s_*}^{-2} & \mathbf{0} & \mathbf{0}\\
\mathbf{0}&\mathbf{0} & \mathbf{\Lambda}_{1s_*}^{-2}&\mathbf{0}\\
\mathbf{0}&\mathbf{0} & \mathbf{0} &\ov{\mathbf{\Lambda}}_{1s_*}^{2}
\end{pmatrix}.
$$
Let $\mathbf{ I}_k$ denote the $k\times k$
identity matrix.
With the abbreviation, we can express $\rho$ as
\begin{gather*}
\boldsymbol{\rho}^{e_*}=\begin{pmatrix}
\mathbf{0} & \mathbf{I}_{e_* }\\
\mathbf{I} _{e_* }& \mathbf{0}
\end{pmatrix},\quad
\boldsymbol{\rho}^{h_*}=\mathbf{I}_{2h_*},\quad
\\
\boldsymbol{\rho}^{s_*}
=\begin{pmatrix}
\mathbf{0} & \mathbf{I}_{s_*}
& \mathbf{0} & \mathbf{0} \\
\mathbf{I}_{s_*} & \mathbf{0} & \mathbf{0} & \mathbf{0}\\
\mathbf{0}&\mathbf{0} & \mathbf{0}& \mathbf{I}_{s_*}\\
\mathbf{0}&\mathbf{0} & \mathbf{I} _{s_*}&\mathbf{0}
\end{pmatrix}.
\end{gather*}
Note that   $\rho$ is anti-holomorphic linear transformation.  
If $A$ is a complex linear transformation,  in $(\xi,\eta)$ coordinates the matrix of $\rho A$ is 
$\boldsymbol{\rho}\ov{\mathbf{A}}$, i.e. 
$$
\rho A\colon
\begin{pmatrix}\xi\\ \eta\end{pmatrix}\to \boldsymbol{\rho}\ov{\mathbf A}\begin{pmatrix}\ov \xi\\ \ov\eta\end{pmatrix}
$$
 with  
\ga\label{lrho}
\nonumber
\boldsymbol{\rho}=
\begin{pmatrix}
\mathbf{0} &\mathbf{0}  &\mathbf{0}   &\mathbf{0}&\mathbf{I}_{e_*}  & \mathbf{0} & \mathbf{0} &\mathbf{0}  \\
\mathbf{0}&\mathbf{I}_{h_*}   & \mathbf{0} & 0&\mathbf{0}&\mathbf{ 0}  &\mathbf{ 0}  &\mathbf{ 0}  \\
\mathbf{ 0} & \mathbf{ 0} &  \mathbf{  0} & \mathbf{I}_{s_*} &\mathbf{ 0 } & \mathbf{ 0} & \mathbf{ 0} &\mathbf{0}\\
\mathbf{0} & \mathbf{0} &  \mathbf{I}_{s_*} & \mathbf{0}&\mathbf{0} & \mathbf{0} & \mathbf{0} & \mathbf{0} \\
\mathbf{I}_{e_*} & \mathbf{0} &\mathbf{0} &\mathbf{0} & \mathbf{0} & \mathbf{0} & \mathbf{0} & \mathbf{0} \\
\mathbf{0} &\mathbf{0}  & \mathbf{0}  & \mathbf{0}&\mathbf{0} &\mathbf{I}_{h_*}  &\mathbf{0}  & \mathbf{0} \\
\mathbf{0} & \mathbf{0} & \mathbf{0}  &\mathbf{0}&\mathbf{0}  &\mathbf{0}  & \mathbf{0} &  \mathbf{I}_{s_*} \\
 \mathbf{0}&\mathbf{0}  & \mathbf{0}  &\mathbf{0}& \mathbf{0} & \mathbf{0} & \mathbf{I}_{s_*}  &\mathbf{0}
\end{pmatrix}.
\end{gather}

For an invertible $p\times p$ matrix $\mathbf A$,  let us define an $n\times n$ matrix $\mathbf E_{\mathbf A}$ by 
\eq{Elam}
\mathbf{E}_{\mathbf A}:=\f{1}{2}\begin{pmatrix}
\mathbf{I}_p & -\mathbf{A} \\
\mathbf{A}^{-1} & \mathbf{I}_p
\end{pmatrix},\quad
\mathbf{E}_{\mathbf A}^{-1}=\begin{pmatrix}
\mathbf{I}_p & \mathbf{A} \\
-\mathbf{A}^{-1} & \mathbf{I}_p
\end{pmatrix}.
\eeq
For a $p\times p$ matrix $\mathbf{B}$,  we define
\eq{bstar}
\mathbf{B}_*:=\begin{pmatrix}
\mathbf{I} _p&\mathbf{0} \\
\mathbf{0} & \mathbf{B}
\end{pmatrix}.
\eeq
Therefore,   we can express
\ga\label{t1jd}
\mathbf{T}_{1j}=\mathbf{E}_{\mathbf \Lambda_1} \mathbf{B}_*\mathbf{Z}_j\mathbf{B}_*^{-1}\mathbf{E}_{\mathbf \Lambda_1}^{-1}, \quad \mathbf{T}_{2j}=\boldsymbol{\rho} \ov{\mathbf{T}_{1j}}\boldsymbol{\rho}, \\
\mathbf{Z}_j=\diag(1,\ldots,1,-1,1,\ldots, 1).\label{t1jd+}
\end{gather}
Here $-1$ is at the $(p+j)$-th place. Moreover, $\mathbf{B}$ is uniquely determined up to equivalence relation via diagonal matrices $\mathbf{D}$:
\eq{bsdb}
\mathbf{B}\sim \mathbf{B}\mathbf{D}.
\eeq
We have  expressed all $\{T_{11},\ldots, T_{1p},\rho\}$    for which
  $\hat T_1=T_{11}\cdots T_{1p}$ and $\rho$ are in the normal forms  in \rl{t1t2sigrho} and we have found an equivalence relation
 to classify the involutions. Let us summarize the results in a lemma.
\begin{lemma}\label{sett}
 Let $\{T_{11},\ldots, T_{1p},\rho\}$ be the involutions of a quadric manifold $M$. Assume that $
 S =T_1\rho T_1\rho$ has
distinct eigenvalues. Then in suitable linear $(\xi,\eta)$
 coordinates, $T_{11},\ldots, T_{1p}$ are given by   \rea{t1jd}, while $T_{11}\cdots T_{1p}=\hat T_1$ and $ \rho$ are given by
\rea{hatt1} and \rea{eqrh}, respectively.
Moreover, $\mathbf{B}$ in \rea{t1jd} is uniquely determined by the equivalence relation \rea{bsdb} for diagonal matrices $\mathbf{D}$.
\end{lemma}
  We remind the reader that we divide the classification for $\{T_{11},\ldots, T_{1p},\rho\}$ into two steps.
We have obtained the classification for the composition $T_{11}\cdots T_{1p}=\hat T_1$ and $\rho$ in \rl{t1t2sigrho}.
Having found   all $\{T_{11},\ldots, T_{1p},\rho\}$ and an equivalence relation,  
we are ready to reduce their classification  
to an equivalence problem that involves two dilatations and a coordinate permutation. 
\begin{lemma}\label{unsol}    Let $\{T_{i1},\ldots, T_{ip},\rho\}$ be given by  \rea{t1jd}.  Suppose that $\hat T_1=T_{11}\cdots T_{1p}$,  $\rho$,   $\hat T_2=\rho\hat T_1\rho$, and
$\hat S=\hat T_1\hat T_2$ 
have the form in \rla{t1t2sigrho}.  Suppose that $\hat S$ has distinct eigenvalues. 
 Let $\{\hat{T}_{11},\ldots,\hat{ T}_{1p},\rho\}$ be given by \rea{t1jd}
where $\la_j$ are unchanged  and $\mathbf B$ is replaced by   $\hat{\mathbf B}$. Suppose that
 $R^{-1}T_{1j}R=\widehat T_{1\nu(j)}$ for all $j$ and $R\rho=\rho R$.
  Then the matrix of $R$ is
$\mathbf R=\diag( \mathbf{a}, \mathbf{a})$ with $\mathbf{a}=(\mathbf{ a}_{e_*},\mathbf{ a}_{h_*},\mathbf{ a}_{s_*},\mathbf{ a}_{s_*}')$,
while $\mathbf a$ satisfies  the reality condition
\ga
\label{aeae} \mathbf{ a}_{e_*}  \in(\rr^*)^{e_*}, \quad \mathbf{ a}_{h_*}\in(\rr^*)^{h_*}, \quad
\ov{\mathbf{ a}_{s_*}}=\mathbf{ a}_{s_*}'\in(\cc^*)^{s_*}.
\end{gather}
Moreover, there exists $\mathbf{d}\in(\cc^*)^p$  such that
\ga\label{bcbd-}
\hat {\mathbf{B}}=(\diag \mathbf{a} )^{-1}\mathbf{B}(\diag_{ \nu} \mathbf{d}), \quad i.e.,
\quad a_i^{-1}b_{i\nu^{-1}(j)}d_{\nu^{-1}(j)}=\hat b_{ij}, 
 \quad
1\leq i,j\leq p.
\end{gather}
Conversely, if $\mathbf{ a},\mathbf{ d}$ satisfy \rea{aeae} and \rea{bcbd-}, then $R^{-1}T_{1j}R=\hat T_{1\nu(j)}$
and $R\rho=\rho R$.
\end{lemma}
\begin{proof}
Suppose that $R ^{-1}  T_{1j} R=\widehat T_{1\nu(j)}$ and $ R \rho=\rho R $. Then $ R^{-1}\hat T_1 R= \hat T_1$ and $R^{-1}\hat SR=\hat S$.  The latter implies that the matrix of
 $R$ is diagonal. The former   implies that
\eq{phi0}
\nonumber
 R \colon \xi_j'=a_j\xi_j,\quad \eta_j'=a_j\eta_j
\eeq
with $  a_j\in\cc^*$.  Now $ R \rho=\rho R $ implies \re{aeae}.
We express $R ^{-1}  T_{1j} R=\widehat T_{1\nu(j)}$ via matrices:
\eq{ela1t}
\mathbf{E}_{\mathbf \Lambda_1}  \widehat{\mathbf{B}}_*\mathbf{Z}_{\nu(j)}\widehat{\mathbf{B}}_*^{-1} \mathbf{E}_{\mathbf \Lambda_1}^{-1}
=\mathbf{R}^{-1}\mathbf{E}_{\mathbf \Lambda_1} \mathbf{B}_*\mathbf{Z}_j\mathbf{B}_*^{-1}\mathbf{E}_{\mathbf \Lambda_1}^{-1}\mathbf{R}.
\eeq
In view of formula \re{Elam}, we see that $\mathbf{E}_{\mathbf \Lambda_1}$ commutes with $\mathbf{R}=\diag(\mathbf{a},\mathbf{a})$.
 The above is equivalent to that $\boldsymbol{\psi}:=\mathbf{B}_*^{-1}\mathbf{R}\widehat{\mathbf{B}}_*$ satisfies $ \mathbf{Z}_{\nu(j)}={\boldsymbol{\psi}}^{-1}\mathbf{Z}_j\boldsymbol{\psi}$.
   By \rl{zcent}  we obtain $\boldsymbol{\psi}=\diag(\mathbf{A},\diag_{\nu}\mathbf{d})$. This shows that
$$
\begin{pmatrix}
\mathbf{A}&\mathbf{ 0}\\
 \mathbf{0}  &  \diag_{\nu}\mathbf{d}
\end{pmatrix} =\begin{pmatrix}
\mathbf{I } &\mathbf{0}\\
 \mathbf{0}  &\mathbf{B}
\end{pmatrix}^{-1}\begin{pmatrix}
\diag \mathbf{a}  &\mathbf{0}\\
 \mathbf{0}  &\diag \mathbf{a}
\end{pmatrix}\begin{pmatrix}
\mathbf{ I} &\mathbf{0}\\
 \mathbf{0}  &  \widehat{\mathbf{B}}
\end{pmatrix}.
$$
The matrices on diagonal yield $\mathbf{A}=\diag \mathbf{a}$  and \re{bcbd-}.
The lemma is proved.
\end{proof}

\rl{unsol}  does not give us an explicit description of the normal form  for the families of
involutions $\{T_{11},\ldots, T_{1p},\rho\}$.  Nevertheless by the lemma, we can always choose a $\nu$ and $\diag {\mathbf d}$ such that the diagonal elements of $\mathbf{ \tilde B}$,
corresponding to $\{\tilde T_{1\nu(1)},\ldots, \tilde T_{1\nu(p)}, \rho\}$,  are $1$.

\begin{rem}
In what follows, we will fix a $\mathbf{B}$ and its associated  $\{\cL T_1,\rho\}$ to further
study our normal form problems.
\end{rem}

\subsection{Normal form of the  quadrics}

  We now use  the matrices $\mathbf{B}$  to express
the normal form for the quadratic submanifolds. Here we follow the realization procedure in the proof of \rp{inmae}.
We will use the coordinates $z^+, z^-$ again to express invariant functions of $T_{1j}$ and
use them to construct the corresponding quadric.
We will then pull back  the quadric to the $(\xi,\eta)$ coordinates and then to the $z,\ov z$ coordinates
to achieve the final normal form of the quadrics.

We return to the construction of invariant and skew-invariant functions $z^+,z^-$ in \re{zep1}-\re{zep4}.
when $\mathbf{B}$ is the identity matrix. For a general $\mathbf{B}$, we define $\Phi_1$ and the matrix
$\mathbf{\Phi}_1^{-1}$ by
$$
\Phi_1(Z^+,Z^-)=(\xi,\eta), \quad
\mathbf{\Phi}_1^{-1}:=
 \mathbf{B}_*^{-1}\mathbf{E}_{\mathbf \Lambda_1}^{-1}= \begin{pmatrix}
\mathbf{I} &\mathbf{\Lambda}_1 \\
 -\mathbf{B}^{-1}\mathbf{\Lambda}_1^{-1} & \mathbf{B}^{-1}
\end{pmatrix}.
$$
Note that $Z^+=z^+$ and $\Phi_1^{-1}T_{1j}\Phi_1=Z_j$.  The  $Z^+, Z_i^-$ with $i\neq j$
are invariant functions of $T_{1j}$, while $Z_j^-$
is a skew-invariant function
of $T_{1j}$. They can be written as
$$
Z^+=\xi+\mathbf{\Lambda}_1\eta, \quad Z^-=\mathbf{B}^{-1}(-\mathbf{\Lambda}_1^{-1}\xi+\eta).
$$
Therefore, the invariant functions of $\cL T_1$ are generated
by
$$
Z_j^+=\xi_j+\la_{j}\eta_j, \quad (Z_j^-)^2=(\tilde {\mathbf{ B}}_j(-\mathbf{\Lambda}_1^{-1}\xi+\eta))^2, \quad 1\leq j\leq p.
$$
Here 
$\tilde {\mathbf{ B}}_j$ is the $j$th row of $\mathbf{B}^{-1}$.
The invariant (holomorphic)
 functions of $\cL T_2$ are generated
by
\eq{wjpm}
W_j^+=\ov{Z_j^+\circ\rho}, \quad( W_j^-)^2=(\ov{Z^-_j\circ\rho})^2,
\quad 1\leq j\leq p.
\eeq
  Here $W_j^-=\ov{Z^-_j\circ\rho}$.
We will soon verify that
$$
m\colon (\xi,\eta)\to (z',w')=(Z^+(\xi,\eta), W^+(\xi,\eta))
$$
is biholomorphic.  A straightforward computation shows that
$m\rho m^{-1}$ equals 
$$
\rho_0\colon (z',w')\to(\ov {w'},\ov {z'}).
$$
We define
\ga\label{m0zp}
\nonumber
M\colon z_{p+j}''=(Z_j^-\circ m^{-1}(z',\ov{z'}))^2.\end{gather}
We want to find a simpler expression for $M$.
We first separate $B$ from $Z^-$ by writing
\ga\label{Zhmi}
\mathbf{  \hat Z}^-:= (-{\mathbf \Lambda}_1^{-1}\  \mathbf{ I}),
 \quad \mathbf{ Z}^-=\mathbf{ B}^{-1}\mathbf{ \hat Z}^-.
\end{gather}
Note that $m$ does not depend on $\mathbf B$.  To compute $\hat Z^-\circ m^{-1}$, we will use
matrix expressions for  $({\boldsymbol\xi}_{e_*},{\boldsymbol\eta}_{e_*})$, $({\boldsymbol\xi}_{h_*},{\boldsymbol\eta}_{h_*})$
and $({\boldsymbol\xi}_{2s_*}, {\boldsymbol\eta}_{2s_*})$ subspaces. 
Let $m_{e_*},m_{h_*},m_{s_*}$ be the restrictions  $m$ to these
subspaces. In the matrix form,
we have   by \re{wjpm}
$$
\mathbf{ W}^+=\ov {\mathbf{ Z}^+\boldsymbol{\rho}}, \quad \mathbf{ W}^-=\ov{\mathbf{  Z}^-
\boldsymbol{\rho}}.
$$
Recall that $\mathbf{\Lambda}_1=\diag({\mathbf \Lambda}_{e_*},{\mathbf \Lambda}_{h_*},{\mathbf \Lambda}_{1s_*}, \ov{\mathbf \Lambda}_{1s_*}^{-1})$. Thus
\aln
\mathbf{ m}_{e_*}&=\begin{bmatrix}
\mathbf{ I} & {\mathbf \Lambda}_{1{e_*}} \\
{\mathbf \Lambda}_{1{e_*}} & \mathbf{ I}
\end{bmatrix},\quad
\mathbf{ m}_{e_*}^{-1}=
\begin{bmatrix}
\mathbf{ I} &- {\mathbf \Lambda}_{1{e_*}} \\
-{\mathbf \Lambda}_{1{e_*}} &\mathbf{  I}
\end{bmatrix}\begin{bmatrix}
(\mathbf{ I}-{\mathbf \Lambda}_{1{e_*}}^2)^{-1} &\mathbf{0} \\
\mathbf{0} & (\mathbf{ I}-{\mathbf \Lambda}_{1{e_*}}^2)^{-1}
\end{bmatrix},\\
\mathbf{ m}_{h_*}&=\begin{bmatrix}
\mathbf{ I} & {\mathbf \Lambda}_{1{h_*}} \\
\mathbf{ I} & {\mathbf \Lambda}_{1{h_*}}^{-1}
\end{bmatrix},\hspace{5ex}
\mathbf{ m}_{h_*}^{-1}=
\begin{bmatrix}
\mathbf{  I }& -{\mathbf \Lambda}_{1{h_*}}^2 \\
-{\mathbf \Lambda}_{1{h_*}} &{\mathbf \Lambda}_{1{h_*}}
\end{bmatrix}\begin{bmatrix}
(\mathbf{ I}-{\mathbf \Lambda}_{1{h_*}}^2)^{-1}&\mathbf{0} \\
\mathbf{0} &(\mathbf{ I}-{\mathbf \Lambda}_{1{h_*}}^2)^{-1}
\end{bmatrix},
\\
\mathbf{ m}_{s_*}&= 
\begin{bmatrix}
{\mathbf  I} & {\mathbf  0}&{\mathbf \Lambda}_{1{s_*}} &{\mathbf  0}\\
{\mathbf  0}&{\mathbf  I} & {\mathbf  0}&\ov{\mathbf \Lambda}_{1{s_*}}^{-1}\\
{\mathbf  0}&{\mathbf  I}&{\mathbf  0}&\ov{\mathbf \Lambda}_{1{s_*}}\\
{\mathbf  I}&{\mathbf  0}&{\mathbf \Lambda}_{1{s_*}}^{-1}&{\mathbf  0}
\end{bmatrix},\\
\mathbf{ m}_{s_*}^{-1}&=
\begin{bmatrix} 
{\mathbf \Lambda}_{1{s_*}}^{-1}&{\mathbf  0}&{\mathbf  0}&-{\mathbf \Lambda}_{1{s_*}}\\
{\mathbf  0}&\ov{\mathbf \Lambda}_{1{s_*}} &-\ov{\mathbf \Lambda}_{1{s_*}}^{-1}&{\mathbf  0}\\
-{\mathbf  I}&{\mathbf  0}&{\mathbf  0}&{\mathbf  I}\\
{\mathbf  0}&-{\mathbf  I}&{\mathbf  I}&{\mathbf  0}
\end{bmatrix}
\begin{bmatrix}
\mathbf{ L}_{s_*}&{\mathbf  0}\\
{\mathbf  0}&-\ov{\mathbf L}_{s_*}
\end{bmatrix},\\
\mathbf{ L}_{s_*}&=\begin{bmatrix}
({\mathbf \Lambda}_{1{s_*}}^{-1}-{\mathbf \Lambda}_{1{s_*}})^{-1}&{\mathbf  0}\\
{\mathbf  0}& (\ov{\mathbf \Lambda}_{1{s_*}}-\ov{\mathbf \Lambda}_{1{s_*}}^{-1})^{-1}
\end{bmatrix}.
\end{align*}
Note that $\mathbf{ I}-{\mathbf \Lambda}_{1}^2$ is diagonal.  Using \re{Zhmi} and
the above formulae, the matrices of $\hat Z_{e_*}^{-1}\circ m^{-1}$,
$\hat Z_{h_*}^-\circ m^{-1}$, and $\hat Z_{s_*}^{-1}\circ m^{-1}$ are respectively given by
\begin{align*}
\mathbf{\hat Z}^-_{e_*}\mathbf{ m}_{e_*}^{-1}
&=\mathbf{ L}_{e_*}
\begin{bmatrix}
\mathbf{ I}  &-2({\mathbf \Lambda}_{1{e_*}}+{\mathbf \Lambda}_{1{e_*}}^{-1})^{-1}
\end{bmatrix},\\
\mathbf{ L}_{e_*}&=
(\mathbf{ I}-{\mathbf \Lambda}_{1{e_*}}^2)^{-1}(-{\mathbf \Lambda}_{1{e_*}}-{\mathbf \Lambda}_{1{e_*}}^{-1}),\\
\mathbf{ \hat Z}_{h_*}^-\mathbf{ m}_{h_*}^{-1}
&=\mathbf{ L}_{h_*}
\begin{bmatrix}
\mathbf{ I} & -2{\mathbf \Lambda}_{1{h_*}}({\mathbf \Lambda}_{1{h_*}}+{\mathbf \Lambda}_{1{h_*}}^{-1})^{-1}
\end{bmatrix},\\
\mathbf{ L}_{h_*}&=(\mathbf{ I}-{\mathbf \Lambda}_{1{h_*}}^2)^{-1}
(-{\mathbf \Lambda}_{1{h_*}}-{\mathbf \Lambda}_{1{h_*}}^{-1}),
\\ 
\mathbf{ \hat Z}_{s_*}^- \mathbf{ m}_{s_*}^{-1}&=
\begin{bmatrix}
-\mathbf{ I}-{\mathbf \Lambda}_{1{s_*}}^{-2}&{\mathbf  0}&{\mathbf  0}&2\mathbf{ I}\\
{\mathbf  0}&-\mathbf{ I}-\ov{\mathbf \Lambda}_{1{s_*}}^{2} &2\mathbf{ I} & {\mathbf  0}
\end{bmatrix}
\begin{bmatrix}
\mathbf{ L}_{s_*}&{\mathbf  0}\\
{\mathbf  0}&-\ov{ \mathbf{ L}}_{s_*}
\end{bmatrix}
\\
&=\mathbf{ \tilde L}_{s_*}
\begin{bmatrix}
\mathbf{ I}&{\mathbf  0}&{\mathbf  0}  &-  2(\mathbf{ I}+{\mathbf \Lambda}_{1{s_*}}^{-2})^{-1}\\
{\mathbf  0}&\mathbf{ I}&-  2 (\mathbf{ I}+\ov{\mathbf \Lambda}_{1{s_*}}^{2})^{-1} &{\mathbf  0}
\end{bmatrix},\\
\mathbf{ \tilde L}_{s_*}&=
\begin{bmatrix}
(\mathbf{ I}+{\mathbf \Lambda}_{1{s_*}}^{-2})({\mathbf \Lambda}_{1{s_*}}-{\mathbf \Lambda}_{1{s_*}}^{-1})^{-1}&{\mathbf  0}\\
{\mathbf  0}&(\mathbf{ I}+\ov{\mathbf \Lambda}_{1{s_*}}^{2})(\ov{\mathbf \Lambda}_{1{s_*}}^{-1}-\ov{\mathbf \Lambda}_{1{s_*}})^{-1}
\end{bmatrix}.
\end{align*}
Combining the above identities, we obtain
$$
\mathbf{ \hat Z}^{-1}\mathbf{ m}^{-1}=
\diag(\mathbf{ L}_{e_*}, \mathbf{ L}_{h_*},\mathbf{ \tilde L}_{s_*})\left(\mathbf{ I}_p, 
  -2\diag\left(\boldsymbol{ \Gamma}_{e_*},{\mathbf \Lambda}_{1{h_*}}\boldsymbol{ \Gamma}_{h_*},
   \begin{bmatrix}0&{\boldsymbol{ \Gamma}}_{s_*}\\
  \tilde{\boldsymbol{ \Gamma}}_{s_*} &0\end{bmatrix}\right)\right)
$$
with  $\tilde{\mathbf{ \Gamma}}_{s_*}=\mathbf{ I}-\ov{\mathbf{ \Gamma}}_{1{s_*}}$  and  \ga\label{Gehs}
\mathbf{ \Gamma}_{e_*}=({\mathbf \Lambda}_{1{e_*}}+{\mathbf \Lambda}_{1{e_*}}^{-1})^{-1},
\quad\mathbf{  \Gamma}_{h_*}=({\mathbf \Lambda}_{1{h_*}}+{\mathbf \Lambda}_{1{h_*}}^{-1})^{-1},\quad
\mathbf{ \Gamma}_{s_*}= (\mathbf{ I}+{\mathbf \Lambda}_{1{ s_*}}^{-2})^{-1}.
\end{gather}
  We define $\tilde{\mathbf B}_j$ to be the $j$-th row of
\eq{deftb}
\tilde {\mathbf B}:=\mathbf{ B}^{-1}\diag(\mathbf{ L}_{e_*}, \mathbf{ L}_{h_*},\mathbf{ \tilde L}_{s_*}).
\eeq
 
With $\mathbf  z_{s_*}'=(z_{p-s_*+1}, \ldots, z_p)$,  the defining equations of $M$ are given by
$$
z_{p+j}''=\left\{\mathbf{ \tilde B}_j\diag(\mathbf z_{{e_*}}-2\mathbf{ \Gamma}_{e_*}\ov {\mathbf z}_{e_*}, 
\mathbf z_{h_*}
-   
2\mathbf{ \Gamma}_{h_*}\mathbf{\Lambda}_{1h_*}\ov{ \mathbf  z}_{h_*}, \mathbf z_{s_*}-2\mathbf{ \Gamma}_{s_*}\ov{ \mathbf  z}_{s_*}', 
{\mathbf  z}_{s_*}'-2(\mathbf{ I}-\ov{\mathbf{ \Gamma}_{s_*}})\ov {\mathbf  z}_{s_*})\right\}^2.
$$
 Let us replace   $z_j$ with $j\neq h$,  $z_h$ by $iz_j$ and  $i\sqrt{\la_h}^{-1}z_h$, respectively.
We also multiply the $h$-th column of $\tilde
{\mathbf B}  
$ by $-i\sqrt{\la_h}$ and its $j$-th column, $j\neq h$,  by $-i$. In the new coordinates, $M$ is given
by 
$$
z_{p+j}''=\left\{\hat{\mathbf{B}}_j\diag(\mathbf z_{{e_*}}+2\mathbf{ \Gamma}_{e_*}\ov {\mathbf z}_{e_*}, 
\mathbf z_{h_*}
+  2\mathbf{ \Gamma}_{h_*}\ov{ \mathbf  z}_{h_*}, \mathbf z_{s_*}+2\mathbf{ \Gamma}_{s_*}\ov{ \mathbf  z}_{s_*}', 
{\mathbf  z}_{s_*}'+2(\mathbf{ I}-\ov{\mathbf{ \Gamma}}_{s_*})\ov {\mathbf  z}_{s_*})\right\}^2.
$$
 Explicitly, we have
\begin{align}\label{qbga}
Q_{\mathbf B,\gamma}\colon
z_{p+j}&=\Bigl(  
\sum_{\ell=1}^{e_*+h_*} \hat b_{j\ell}(z_\ell+2\gamma_\ell\ov z_\ell)
\\ &\qquad  
   +\sum_{s=e_*+h_*+1}^{p-s_*}
 \hat b_{js}(z_s+2\gamma_s\ov z_{s+s_*})  +
 \hat b_{j(s+s_*)}(z_{s+s_*}+2
 \gamma_{s+s_*}\ov z_{s})\Bigr)^2  \nonumber
\end{align}
for $ 1\leq j\leq p$. Here   
\eq{gsss}
\gaa_{s+s_*}=1-\ov\gaa_s.
\eeq
By \re{deftb}, we also obtain the following identity
\eq{defhb}
\nonumber
\hat{\mathbf B}= -i\mathbf B^{-1} \diag(\mathbf{ L}_{e_*}, \mathbf{ L}_{h_*},\mathbf{ \tilde L}_{s_*})
\diag(\mathbf I_{e_*}, \mathbf{\Lambda}_{1h_*}^{1/2}, \mathbf I_{2s_*})
\eeq
The equivalence relation \re{bcbd-} on the set of non-singular matrices $\mathbf B$ now takes the form
\eq{hbsim}
 \widehat{\tilde{\mathbf B}}= (\diag_{ \nu}\mathbf d)^{-1}\hat{\mathbf B}\diag\mathbf a, 
\eeq
where $\mathbf a$ satisfies \re{aeae} and   $\diag_{ \nu}\mathbf d$  is defined in \re{dfndiag}.

Therefore, by \rp{inmae} we obtain the following classification for
the quadrics.

\begin{thm}\label{quadclass}
Let $M$ be a quadratic submanifold defined by \rea{variete-orig} and
\rea{variete-orig+} with $q_*=0$. Assume that the branched covering
 $\pi_1$ has   $2^p$  deck transformations.
 Let $T_1,T_2$ be the pair of Moser-Webster involutions
 of $M$. Suppose that  $S=T_1T_2$
has $2p$ distinct eigenvalues.
Then $M$ is holomorphically equivalent to
\rea{qbga} with  $\hat{\mathbf{ B}}\in GL(p,\cc)$ being uniquely determined by   the equivalence
  relation \rea{hbsim}.
\end{thm}

 When $\hat{\mathbf{ B}}$  is the identity, we obtain
 the product of 3 types of quadrics
\gan
\cL Q_{\gamma_e}\colon z_{p+e}= (z_e+2\gaa_e\ov z_e)^2;\\
\cL Q_{\gamma_h}\colon z_{p+h}= (z_{h}+2\gaa_{h}\ov z_h)^2;\\
\cL Q_{\gamma_s}\colon z_{p+s}= (z_s+2\gamma_s\ov z_{s+s_*})^2,
\quad z_{p+s+s_*}=( z_{s+s_*}+2
(1-\ov\gamma_{s})  \ov z_{s})^{2}
\end{gather*}
with
\eq{gens}
\gamma_e=\frac{1}{\la_e+\la_e^{-1}}, \quad
\gamma_h=\frac{ 1}{\la_h+\ov\la_h},\quad
\gamma_s=\frac{1}{1+\la_s^{ -2}}.
\eeq
Note that $\arg\la_s
\in(0,\pi/2)$ and $|\la_s|>1$. Thus
\eq{gehs}
\nonumber
 0<\gaa_e<1/2, \quad
 \gaa_h>1/2, \quad
\gaa_s\in \{z\in\cc\colon \RE z>1/2, \IM z>0\}.
\eeq

We define the following invariants.
\begin{defn}We call $\mathbf{ \Gamma}=\diag(\mathbf{ \Gamma}_{e_*},\mathbf{ \Gamma}_{h_*},\mathbf{ \Gamma}_{s_*}, \mathbf I_{s_*}-\bar{\boldsymbol \Gamma}_{s_*})$,  given by formulae \re{Gehs},
 the {\em Bishop invariants} of the quadrics. The  equivalence classes $\hat{\mathbf B}$ of non-singular matrices
 $\mathbf{ B}$ under the equivalence relation \re{bcbd-} are called the {\it extended   Bishop invariants}
  for the quadrics.
\end{defn}
Note that $\mathbf{ \Gamma}_{e_*}$ has
diagonal elements in $(0,1/2)$, and $\mathbf{ \Gamma}_{h_*}$
has   diagonal elements in $(1/2,\infty)$,
and $\mathbf{ \Gamma}_{s_*}$ has diagonal elements
in $(1/2, \infty)+ i (0,\infty)$.

We remark that $Z_j^-$ is
skew-invariant by $T_{1i}$ for $i\neq j$ and invariant by
$\tau_{1j}$. Therefore, the square of a linear combination
of $Z_1^-, \ldots, Z_p^-$ might not be invariant by all
$T_{1j}$. This explains the presence of $\mathbf B$ as invariants
in the normal form.

It is worthy stating the following normal form
for two families of linear holomorphic involutions
which may not satisfy the reality
condition.  

\begin{prop}\label{2tnorm}
Let $\cL T_i=\{T_{i1},\ldots, T_{ip}\},i=1,2$ be two families
of distinct and commuting
linear holomorphic involutions on $\cc^n$.
Let $T_i=T_{i1}\cdots T_{ip}$. Suppose that for each $i$, $\fix(T_{i1}), \ldots$, $\fix(T_{ip})$
are hyperplanes intersecting transversally. Suppose that  $T_1,T_2$ satisfy \rea{bilii++} and 
$S=T_1T_2$ has $2p$ distinct eigenvalues.
In suitable linear coordinates, the matrices of $T_i,S$ are 
\gan
\mathbf{ T}_i=\begin{pmatrix}
                 \mathbf{  0 }& {\mathbf \Lambda}_i \\
                  {\mathbf \Lambda}_i^{-1} & \mathbf{ 0} \\
                \end{pmatrix},\quad
\mathbf{ S}=\begin{pmatrix}
               {\mathbf \Lambda}_1^2   & \mathbf{ 0} \\
               \mathbf{  0} & {\mathbf \Lambda}_1^{-2} \\
                \end{pmatrix}
\end{gather*}
with ${\mathbf \Lambda}_2={\mathbf \Lambda}_1^{-1}$  being diagonal matrix whose entries
do not contain $\pm1, \pm i$. The  ${\mathbf \Lambda}_1^2$ is uniquely determined
up to a permutation in diagonal entries.  Moreover, the matrices of $T_{ij}$ are
\eq{btij}
\mathbf{ T}_{ij}=\mathbf{ E}_{{\mathbf \Lambda}_i}(\mathbf{ B}_i)_*\mathbf{ Z}_j(\mathbf{ B}_i)_*^{-1}\mathbf{ E}_{{\mathbf \Lambda}_i}^{-1}
\eeq
for some non-singular complex matrices $\mathbf{ B}_1, \mathbf{ B}_2$ 
uniquely determined by   the equivalence
relation
\eq{tb1tb2}
(\mathbf{ B}_1,\mathbf{ B}_2)\sim (\mathbf{ \tilde B}_1,\mathbf{ \tilde B}_2):=
(\mathbf{ R}^{-1}\mathbf{ B}_1\diag_{  \nu_1}\mathbf d_1,\mathbf{R}^{-1}\mathbf{ B}_2\diag_{ \nu_2}\mathbf d_2),
\end{equation}
where $\diag_{\nu_1}\mathbf{ d}_1,\diag_{ \nu_2}\mathbf{ d}_2$ are defined as in \rea{dfndiag},
and $\mathbf{ R}$ is a non-singular diagonal  complex matrix   representing  the linear transformation $\var$ such that
\eq{}
\nonumber
\var^{-1}T_{ij}\var=\tilde T_{i\nu_i(j)}, \quad i=1,2, j=1, \ldots, p.
\eeq
 Here $\tilde {\cL T}_i$ is the family of  the involutions  associated to the matrices $\mathbf {\tilde B}_i$, and
 $\mathbf E_{\mathbf \Lambda_i}$ and $\mathbf B_*$ are defined by  \rea{Elam} and
\rea{bstar}.
\end{prop}
\begin{proof}  Let $\kappa$ be an eigenvalue of $S$ with (non-zero) eigenvector $u$. Since $T_iST_i=S^{-1}$. 
Then $S(T_i(u))=\kappa^{-1}T_i(u)$. This shows that $\kappa^{-1}$ is also an eigenvalue of $S$. By \rl{incomp}, $1$ and $ -1$
are not eigenvalues of $S$. Thus, we can list the eigenvalues of $S$ as $\mu_1, \ldots, \mu_p, \mu_1^{-1}, \ldots, \mu_p^{-1}$. 
Let $u_j$ be an eigenvector of $S$ with eigenvalue $\mu_j$. Fix $\la_j$ such that $\la_j^2=\mu_j$.
Then $v_j:=\la_j T_1(u_j)$ is an eigenvector of $S$ with eigenvalue $\mu_j^{-1}$. The $\sum\xi_ju_j+\eta_jv_j$ defines
a coordinate system on $\cc^n$ such that $T_i,S$ have  the above matrices $\mathbf{\Lambda}_i$ and $\mathbf S$,
respectively.  By \re{phi1z} and  \re{t1jd}, 
$T_{ij}$ can be expressed in \re{btij}, where each $\mathbf B_i$ is uniquely determined up to
$\mathbf B_i\diag\mathbf d_i$. Suppose that $  \{\tilde T_{1j}\}, \{\tilde T_{2j}\}$ 
are another pair of
families of linear involutions of which the corresponding matrices are $ \mathbf{\tilde B}_1, \mathbf {\tilde B}_2$.
  If there is a linear change of coordinates $\var$ such
that $\var^{-1} T_{ij}\var=\tilde T_{i\nu_i(j)}${, 
then  in} 
the matrix $\mathbf R$ of $\var$, we obtain \re{tb1tb2}; see a similar computation
for \re{bcbd-} by using \re{ela1t}.
Conversely, \re{ela1t} implies that the corresponding pairs of families
of involutions are equivalent.
\end{proof}

\setcounter{thm}{0}\setcounter{equation}{0}

\setcounter{thm}{0}\setcounter{equation}{0}
\section{
Formal deck
transformations  and centralizers}\label{fsubm}

In section~\ref{secinv} we  show 
 the equivalence of   the classification of real analytic submanifolds $M$ that admit the maximum number of deck transformations and the classification of the families of
involutions $\{\tau_{11}, \ldots, \tau_{1p},\rho\}$ that satisfy some mild conditions  (see \rp{inmae}). To classify the families of involutions and to find their normal forms, we will first study  normal forms at the formal level. The main purpose of this section is to show that
at the formal level, the classification of the formal submanifolds of the desired CR singularity and the classification of 
$\{\tau_{11}, \ldots, \tau_{1p},\rho\}$ are equivalent
under these mild conditions.

We will also study the centralizers of various linear maps to deal with  resonance.  
This is relevant as the normal form of $\sigma$ will belong to the centralizer of its linear part
and any further normalization will also be performed by transformations that are
in the centralizer. 

\subsection{Formal submanifolds and formal deck transformations}

We  first need some notation.
Let $I$ be an ideal
of the ring $\rr[[x]]$ of formal power series in   $x=(x_1,\ldots, x_N)$.
Since $\rr[[x]]$ is noetherian, then $I$ and its radical $\sqrt I$ are finitely generated. We say that $I$ defines a formal  submanifold $M$ of dimension $N-k$ if $\sqrt I$ is generated by $r_1,\ldots, r_k$ such that at the origin all $r_j$ vanish and $dr_1,\ldots, dr_k$
  are linearly independent. For such an $M$, let $I(M)$ denote $\sqrt I$ and let   $T_0M$ be defined by $dr_1(0)=\cdots=dr_k(0)=0$. If $F=(f_1,\ldots, f_N)$ is a formal mapping
  with $f_j\in\rr[[x]]$, we say that its set of (formal) fixed points is a submanifold if
  the ideal generated by $f_1(x)-x_1, \ldots, f_N(x)-x_N$ defines a submanifold.
Let $I,\tilde I$ be   ideals of $\rr[[x]], \rr[[y]]$ and
 let $\sqrt I, \sqrt{\tilde I}$ define two formal submanifolds $M, \tilde M$,
 respectively. We say
  that a formal map $y=G(x)$ maps $M$ into $\widetilde{M}$ if $\tilde I\circ G\subset \sqrt I$.
If $M,\tilde M$ are   in the same space, we write $M\subset\tilde M$ if $\tilde I\subset\sqrt I$. We say that a formal map $F$ fixes $M$
pointwise if $I(M)$ contain each component of the mapping $F-\id$.

We now  consider a   formal $p$-submanifold in $\cc^{2p}$
defined by
\eq{fmzp}
M\colon
z_{p+j} = E_{j}(z',\bar z'), \quad 1\leq j\leq p.
\eeq
Here $E_j$ are   formal power series in $z',\ov z'$.  We  assume that
\eq{fmzp+}
E_j(z',\bar z')=h_j(z',\ov z')+q_j(\ov z')+O((|(z',\ov z')|^3)
\eeq
and $h_j, q_j$ are homogeneous quadratic
polynomials. The formal complexification of $M$ is  defined by
\begin{equation}\nonumber
\label{variete-complex}
\begin{cases}
z_{p+i} = E_{i}(z',w'),\quad i=1,\ldots, p,\\
w_{p+i} = \bar E_i(w',z'),\quad i=1,\ldots, p.\\
\end{cases}
\end{equation}
 We define a {\it formal deck transformation}
of $\pi_1$ to be a formal biholomorphic map
$$
\tau\colon (z',w')\to (z',f(z',w')), \quad \tau(0)=0
$$
such that $\pi_1\tau=\pi_1$, i.e. $E\circ\tau=E$. Recall that condition B says that $q_*=\dim\{z'\in\cc^n\colon q(z')=0\}$ is zero, i.e.  $q$ vanishes only at
the origin
in $\cc^p$.
\begin{lemma} 
Let $M$ be a formal submanifold defined by \rea{fmzp}-\rea{fmzp+}. Suppose that
$M$ satisfies condition {\rm B}. Then formal
deck transformations of $\pi_1$ are commutative  involutions. Each formal deck transformation $\tau$ of $\pi_1\colon \cL M\to \cc^p$ is 
uniquely determined by its linear part $L\tau$
in the $(z',w')$ coordinates, while $L\tau$ is a deck transformation for the complexification for $\pi_1\colon \cL Q\to\cc^p$, where $\cL Q$ is the complexification of
the quadratic part  $Q$ of $M$. 
 If $M$ is real analytic, all
formal
deck transformations of $\pi_1$ are convergent. 
\end{lemma}
\begin{proof} Let us recall some results about the quadric $Q$.
We already know that $q_*=0$ implies that $\pi_1$ for the complexification of
$Q$ is a branched   covering. As used in the proof of \rl{2p1},  $\pi_1$  is an open mapping near
 the origin and its regular values are dense.
In particular, we have
\eq{dwph}
\det\pd_{w'}\{h(z',w')+  q(w')\}\not\equiv0.
\eeq

Let $\tau$ be a formal deck transformation for  $M$. To show that $\tau$ is an involution, we note that its linear part at the origin, $L\tau$, is a deck transformation
of $Q$. Hence $L\tau$ is an involution.
Replacing $\tau$ by the deck transformation $\tau^2$, we may assume that
$\tau$ is tangent to the identity. Write
$$
\tau(z',w')=(z',w'+u(z',w')).
$$
We want to show that $u=0$. Assume that $u(z',w')=O(|(z',w')|^k)$  and let $u_k$ be
 homogeneous and of degree $k$ such that $u(z',w')=u_k(z',w')+O(|(z',w')|^{k+1})$. We have
$$
E(z',w'+u(z',w'))=E(z',w').
$$
Comparing terms of order $k+1$, we get
$$
\pd_{w'}\{h(z',w')+ q(w')\}u_k(z',w')=0.
$$
By \re{dwph}, $u_k=0$. This shows that each formal deck transformation $\tau$ of $\pi_1$ for $M$ is an involution. As mentioned above, $L\tau$ is a deck transformation of $\pi_1$ for $Q$.
Also if $\tau,\tilde\tau$  are commuting formal involutions then $\tau^{-1}\tilde\tau$ is an involution and $\tau=\tilde\tau$ if and only if $L\tau= L\tilde\tau$.

 Assume now that $M$ is real analytic. We want to show that  each formal deck transformation $\tau$ is convergent.
By a theorem of Artin~\cite{artin68}, there is a convergent $\tilde\tau(z',w')=\tau(z',w')
+O(|(z',w')|^2)$ such that $E\circ\tilde\tau=E$, i.e. $\tilde\tau$ is a deck
transformation. Then $\tilde\tau^{-1}\tau$ is a deck transformation   tangent to the identity. Since it is a
formal involution by the above argument,
then it must be identity. Therefore,
 $\tau=\tilde\tau$ converges.
\end{proof}

Analogous to real analytic submanifolds, we say that a formal manifold defined
by \rea{fmzp}-\rea{fmzp+} satisfies condition D if its  formal branched  covering $\pi_1$ admits
$2^p$ formal deck transformations.

 Recall from section~\ref{secinv}
that
it is crucial to distinguish a special set of generators
for the deck transformations in order to relate the classification
of real analytic manifolds to the classification of
  certain $\{\tau_{11}, \ldots, \tau_{1p},\rho\}$.
The set of generators is uniquely determined by
the dimension of  fixed-point sets.
We want to extend these results at the
formal level.
 
\begin{prop} \label{mmtp}
Let $M,\tilde M$ be formal $p$-submanifolds 
in $\cc^n$ of the form \rea{fmzp}-\rea{fmzp+}.
Suppose that   $M,\tilde M$ satisfy condition  {\rm D}. Then the following hold~$:$
\bppp
\item $M$ and $\tilde M$ are
formally equivalent
if and only if  their associated families of involutions 
$\{\tau_{11}, \ldots, \tau_{1p},
\rho\}$ and $\{\tilde\tau_{11}, \ldots, \tilde\tau_{1p},\rho\}$
are   formally equivalent.
\item
Let  $\cL T_1=\{\tau_{11}, \ldots,\tau_{1p}\}$ be a family of  
 formal holomorphic
involutions which commute pairwise. Suppose that 
 the tangent spaces of $\fix(\tau_{11}),\ldots,
 \fix(\tau_{1p})$  are hyperplanes
   intersecting  transversally at the origin. Let $\rho$
 be an anti-holomorphic formal involution and let $\cL T_2=\{\tau_{21},\ldots,\tau_{2p}\}$   with $\tau_{2j}=\rho\tau_{1j}\rho$.
Suppose that  $\sigma=\tau_1\tau_2$ has distinct eigenvalues for $\tau_i=\tau_{i1}\cdots\tau_{ip}$, and 
$$
 [\mathfrak M_n]_1^{L\cL T_1}\cap[\mathfrak M_n]_1^{L\cL T_2}
 =\{0\}.
 $$
There exists a formal  submanifold  defined by 
\begin{equation}\label{masym}
 z''=(B_1^2,\ldots, B_p^2)(z',\ov z')
\end{equation}
for some formal power series $B_1,\ldots, B_p$
such that $M$ satisfies condition  {\rm D}.  The set
of involutions  $\{ \tilde\tau_{11},\ldots, \tilde\tau_{1p},   \rho_0\}$  of $M$ is formally equivalent to
 $\{\tau_{11}, \ldots, \tau_{1p},\rho\}$.  
 \eppp
\end{prop}
\begin{proof} (i) Let $M$ and $\tilde M$ be given by $z''=E(z',\ov{z'})$ and $\tilde z''=\tilde E(\tilde z',\ov{\tilde z'})$, respectively. 
Suppose that $f$ is a formal holomorphic
 transformation sending $M$ into $\tilde M$.  We have
\eq{fppz}
f''(z',E(z',w'))=\tilde E(f'(z',E(z',w')),\ov f'(w',\ov E(w',z'))).
\eeq
Here $f=(f',f'')$. Recall that 
 $\rho_0(z',w')=(\ov{w'},\ov{z'})$. Define a formal mapping $(z',w')\to (\tilde z',\tilde w')=F(z',w')$ by
\eq{Fzpwp}
F(z',w'):= (f'(z',E(z',w')),\ov {f'}(w',\ov E(w',z'))).
\eeq
It is clear that $ F\rho_0=\rho_0 F$.
  By \rl{twosetin}, we know that $\tilde z'$ and $\tilde z''=\tilde E(\tilde z', \tilde w')$ generate invariant formal
power series of $\{\tilde\tau_{1j}\}$.
Thus, $\tilde z'\circ F(z',w')=f'(z',E(z',w'))$ and $\tilde E\circ F(z',w')$  are invariant by  $F^{-1}\circ\tilde\tau_{1j}\circ F$. By \re{fppz} and the definition of $F$,
 $$\tilde E\circ F(z',w')=f''(z',E(z',w')).$$
This shows that $f(z', E(z',w'))$  is invariant under
 $F^{-1}\circ\tilde\tau_{1j}\circ F$. Since $f$ is invertible, then $z'$ and $E(z',w')$
  are invariant under $F^{-1}\circ\tilde\tau_{1j}\circ F$.
 Therefore, $\{\tau_{1j}\}$ and $\{F^{-1}\circ\tilde\tau_{1i}\circ F\}$  are the same by \rl{twosetin} as
 they
 have the same invariant functions.  

Assume now that $\{\tau_{1j}\}=\{F^{-1}\circ\tilde\tau_{1i}\circ F\}$ for some
formal biholomorphic map $F$ commuting with $\rho_0$.  Recall that $\tilde z',\tilde z''$
are invariant by $ \tilde\tau_{1j}$. Then $\tilde z'\circ F$ and $\tilde E\circ F$ are
 invariant by $\{\tau_{1j}\}$.    By \rl{twosetin}, invariant power series of ${\tau_{1j}}$
are generated by $z',E(z',w')$. Thus
we can write 
\ga \nonumber
\tilde z'\circ F(z',w')=f'(z',E(z',w')), \\
\label{hefp} \tilde E\circ F(z',w')=f''(z',E(z',w'))
 \end{gather}
for some formal power series map $f=(f',f'')$.   Since $\rho_0 F=F\rho_0$, 
then by \re{Fzpwp} $$F(z',w') 
=(f'(z',0),\ov f'(w',0))+O(|(z',w')|^2).$$ 
Since $F$ is (formal) biholomorphic then $z'\to f'(z',0)$ is biholomorphic. 
Then
$$
f''(0,E(0,w'))=\tilde E(0,\ov f'(w',0))+O(|w'|^3).
$$
 We have $E(0,w')=q(w')+O(|w'|^3)$ and $\tilde E(0,w')=\tilde q(w')+O(|w'|^3)$. Here $q(w'),\tilde q(w')$ are quadratic. By condition $q_*=0$, we know that $\tilde
 q_1, \ldots,\tilde q_p$ and hence $\tilde q_1\circ L, \ldots, \tilde q_p\circ L$ are linearly independent. 
 Here $L$ is the linear part of the mapping $w'\to \ov f'(w',0)$, which is invertible. This shows that
the linear part of $w'\to f''(0,w')$
 is biholomorphic. 
 By \re{hefp}, $f''(z',0)=O(|z'|^2)$. Hence $f=(f',f'')$ is biholomorphic.
By a simple computation, we have $f(M)=\tilde M$, i.e.
$$
\tilde E(f'(z),\ov{f'(z)})=f''(z)
$$
for $z''=E(z',\ov z')$.

(ii)
Assume that $\{\tau_{1j}\}$ and $\rho$
are given in the $(\xi,\eta)$
space.  We want to show that a formal holomorphic
 equivalence class of $\{\tau_{1j},\rho\}$ can be realized by a formal submanifold satisfying condition D.
The proof is almost identical to the realization proof of \rp{inmae} and we will be brief.  Using a formal, instead of convergent,
change of coordinates,  we
know that invariant formal power series of $\{\tau_{1j}\}$ are generated by
$$
 z'=(A_1(\xi,\eta),\ldots,
A_p(\xi,\eta)), \quad  z''=(B_1^2(\xi,\eta), \ldots, B_p^2(\xi,\eta)),
$$
where $B_j$ is skew-invariant by $  \tau_{1j}$, and $A,B_i$ are invariant under $\tau_{1j}$
for $i\neq j$.  Moreover, $\phi(\xi,\eta)=(A,B)(\xi,\eta)$ is formal biholomorphic.
Set $$
w_j'=\ov{A_j\circ\rho(\xi,\eta)}, \quad w''_j=\ov{B_j^2\circ\rho(\xi,\eta)}.
$$
Then $(\xi,\eta)\to (A(\xi,\eta),\ov{A\circ\rho(\xi,\eta)})$
has an inverse $\psi$.  Define
$$
M\colon z''=(B_1^2,\ldots, B_p^2)\circ\psi(z',\ov z').
$$
The complexification of $M$ is given by
$$
\cL M\colon
z''=(B_1^2,\ldots, B_p^2)\circ\psi (z',w'),\quad
w''=(\ov B_1^2,\ldots, \ov B_p^2)\circ\ov\psi(w',  z').
$$
Note that $\phi\circ\psi (z',w')=(z',B\circ\psi(z',w'))$. Since $\phi\psi$ is invertible, 
the linear part $D$ of $B\circ\psi$ satisfies
$$
|D(0,w')|\geq |w'|/C.
$$
This shows that $q_*=0$.
As in the proof of \rp{inmae}, we can verify that $M$ is a realization for $\{\tau_{1j},\rho\}$.\end{proof}

\subsection{Centralizers and normalized transformations}
In this subsection, we   describe several centralizers regarding $\hat S, \hat T_1$ and $\hat{\mathcal T}_1$. 
We will also describe the complement sets of the centralizers, i.e. the sets
of mappings which satisfy suitable normalizing conditions. Roughly speaking, our normal forms are   in 
the centralizers and coordinate transformations that achieve the normal forms are normalized, while
an arbitrary formal transformation admits a unique  decomposition of a mapping in a centralizer and a mapping in the complement of
the centralizer. The description of the centralizer of $\{\mathcal T_1,\rho\}$ is more complicated and it will be given in 
section~\ref{secideal}. We will also deal with the convergence for the decomposition. 

Recall that
\ga\label{sxij}
\hat S\colon\xi_j'=\mu_j\xi_j, \quad\eta_j'=\mu_j^{-1}\eta_j, \quad 1\leq j\leq p,\\
\hat T_i\colon \xi_j'=\la_{ij}\eta_j,\quad \eta_j'=\la_{ij}^{-1}\xi, \quad 1\leq j\leq p
\label{tixi}
\end{gather}
with $\mu_j=\la_{1j}^2$ and $\la_{2j}^{-1}=\la_{1j}=\la_j$.

\begin{defn}Let $\cL  F$ be a family of formal mappings  on $\cc^n$  fixing
the origin.
Let $\cL C(\cL  F)$ 
 be the  {\it centralizer} of $ \cL   F$, i.e.
 the set of formal  holomorphic mappings $g$ that fix  the origin and
  commute with each element $f$ of $\cL  F$, i.e., $f\circ g=g\circ f$.
\end{defn}

 Note that we do not require that elements in $\cL C(\cL  F)$   be invertible or convergent.

We first compute the centralizers.
\begin{lemma}\label{cents}
Let $\hat S$  be given by \rea{sxij} with  $\mu_1,\ldots, \mu_p$
being non-resonant.
Then ${\cL C}(\hat S)$ consists of mappings of the form
\eq{pajb}
\psi\colon \xi_j'= a_j(\xi\eta)\xi_j,\quad \eta_j'= b_j(\xi\eta)\eta_j, \quad 1\leq j\leq p.
\eeq
Let $\tau_1,\tau_2$ be  formal  holomorphic
involutions  such that $\hat S=\tau_1\tau_2$. Then
$$
\tau_{i}\colon\xi_j'=\Lambda_{ij}(\xi\eta)\eta_j, \quad \eta_j'=\Lambda_{ij}^{-1}(\xi\eta)\xi_j, \quad 1\leq j\leq p
$$
with $\Lambda_{1j}\Lambda_{2j}^{-1}=\mu_j$. 
The   centralizer of   $\{\hat T_1,\hat T_2\}$ consists of the above transformations satisfying
 \eq{bjaj}
 b_j=a_j, \quad 1\leq j\leq p.
 \eeq
\end{lemma}
\begin{proof} Let $e_j=(0,\ldots, 1,\ldots, 0)\in \nn^p$, where $1$ is at the $j$th place.  Let $\psi$ be given by $$\xi_j'=\sum a_{j,PQ}\xi^P\eta^Q,\quad\eta_j'=\sum b_{j,PQ}\xi^P\eta^Q.
$$
By the non-resonance condition, it is straightforward that if $\psi \hat S=\hat S\psi$, then $a_{j,PQ}=b_{j,QP}=0$
if $P-Q\neq e_j$.
Note that $\hat S^{-1}=T_0\hat ST_0$ for $T_0\colon(\xi,\eta)\to(\eta,\xi)$. Thus $\tau_1T_0$ commutes with
$\hat S$. So $\tau_1T_0$ has the form \re{pajb} in which we rename $a_j,b_j$ by $\Lambda_{1j},\tilde\Lambda_{1j}$, respectively. Now $\tau_1^2=\id$ implies that
$$
\Lambda_{1j}((\Lambda_{11}\tilde\Lambda_{11})(\zeta)\zeta_1,\ldots, (\Lambda_{1p}\tilde\Lambda_{1p})(\zeta)\zeta_p)\tilde\Lambda_{1j}(\zeta)=1, \quad 1\leq j\leq p.
$$
Then $\Lambda_{1j}(0)\tilde\Lambda_{1j}(0)=1$. Applying induction on $d$, we verify that for all $j$
 $$\Lambda_{1j}(\zeta)\tilde\Lambda_{1j}(\zeta)=1+O(|\zeta|^d), \quad d>1.
 $$
 Having found the formula for
   $\tau_1T_0$, we obtain the desired formula of $\tau_1$ via composition $(\tau_1T_0)T_0$.
\end{proof}

Let ${\mathbf D}_1:=\text{diag}(\mu_{11},\ldots,\mu_{1n}),\ldots, {\mathbf D}_\ell:=\text{diag}(\mu_{\ell 1},\ldots,\mu_{\ell n})$ be diagonal invertible matrices of $\cc^n$. Let us set $D:=\{{\mathbf D}_iz\}_{i=1,\ldots \ell}$.
\begin{defn}\label{ccst}  Let $F$ be a formal mapping of $\cc^n$ that is tangent to the identity. 
\bppp
\item Let $n=2p$.   $F$   is  {\it normalized} with respect to $\hat S$,  if $F=(f,g)$ is tangent to the identity 
 and $F$ contains no resonant terms, i.e.
 \eq{fj0g}
 \nonumber
 f_{j,(A+e_j)A}=0=g_{j,A(A+e_j)},  \quad  |A|>1. 
 \eeq
 \item Let $n=2p$. 
$F$ is {\it normalized} with respect to $\{\hat T_1,\hat T_2\}$, if $F=(f,g)$ is tangent to
the identity and
\eq{fjmg}
\nonumber
 f_{j,(A+e_j)A}=-g_{j,A(A+e_j)},\quad |A|>1.
\eeq
\item
$F$ is {\it normalized} with respect to $D$ if it does not have components along the centralizer of $D$, i.e.   for each 
$Q$ with $|Q|\geq2$,
\eq{norm-D}
\nonumber
f_{ j,Q}=0,\quad\text{if}\; \mu_i^Q=\mu_{ij}\; \text{for all}\; i. 
\eeq
\eppp
Let ${\cL C}^{\mathsf{ c}}(\hat S)$ (resp. ${\cL C}^{\mathsf{ c}}(\hat T_1,\hat T_2)$, ${\cL C}^{\mathsf{ c}}(D)$) denote the set of formal mappings  normalized with respect to $\hat S$ (resp. $\{\hat T_1,\hat T_2\}$, the family $D$).
\end{defn}

 For convenience, we 
let ${\cL C}^{\mathsf{ c}}_2(\hat S)$ (resp. ${\cL C}^{\mathsf{ c}}_2(\hat T_1,\hat T_2)$, ${\cL C}^{\mathsf{ c}}_2(D)$) denote the set of formal mappings  $F-\I$ with 
$F\in {\cL C}^{\mathsf{ c}}(\hat S)$ (resp. ${\cL C}^{\mathsf{ c}}(\hat T_1,\hat T_2)$, ${\cL C}^{\mathsf{ c}}(D)$).

\begin{rem}\label{rem-rho}
 Note that if $f\in {\cL C}^{\mathsf{ c}}(\hat S)$  (resp. $ {\cL C}^{\mathsf{ c}}(\hat T_1,\hat T_2)$),
 then $\rho f\rho$ is in $ {\cL C}^{\mathsf{ c}}(\hat S)$ (resp. ${\cL C}^{\mathsf{ c}}(\hat T_1,\hat T_2))$.
 \end{rem}

We now deal with the following decomposition problem: Let $\mathcal C$ be a set of analytic mappings. We shall decompose an arbitrary invertible
 mapping into the composition of an element of a centralizer of $\mathcal C$ and an element which is normalized with respect to $\mathcal C$. We shall also deal with the convergence issue.
The following lemma, which deals with a general situation,  will be used several times. 
\begin{defn}
Let $\cL A$ be a group of permutations of $\{1,\ldots, n\}$. Then 
$\cL A$ acts on the right (resp. on the left) on $\widehat{\cL O}_n^n$ by permutation of variables $z=(z_1,\ldots, z_n)$ as follows: Let $F(z)=\sum_{|Q|>0}F_Qz^Q$ be a formal mapping from $\cc^n$ to $\cc^n$,   and let $\nu, \mu\in \cL A$; set
$$\nu\circ F\circ\mu(z):= \sum_{Q\in \nn^n}F_{\nu(i),\mu^{-1}(Q)}z^Q.
$$
Define the components $(\cL AF)_i, (F\cL A)_i$, and consequently $(\cL AF\cL A)_i$ by
\aln
(\cL A F)_i(z)& :=\sum_{Q\in \nn^n}\max_{\nu\in \cL A}|F_{\nu(i),Q}|z^Q,\\
(F\cL A)_i(z) &:=\sum_{Q\in \nn^n}\max_{\mu\in\cL A}|F_{i,\mu^{-1}(Q)}|z^Q,
\\
(\cL AF\cL A)_i(z) &\, =\sum_{Q\in \nn^n}\max_{(\nu,\mu)\in \cL A^2}|F_{\nu(i),\mu^{-1}(Q)}|z^Q.
 \end{align*}
\end{defn}
We see that $F\cL A$ is the smallest (w.r.t. $\prec$) power series mapping that majorizes $F$
and is right-invariant under   $\cL A$, while $\cL AF$ is the smallest power series
mapping that majorizes $F$ and is left-invariant under   $\cL A$. 
In particular, if $F, G$ are mappings without constant or linear terms, then
\ga
\label{FiGp}
\cL A(F\circ (I+G))\cL A\prec (\cL AF\cL A)(\cL A I\cL A+\cL AG\cL A),
\end{gather}
where the last relation holds if the composition is well-defined.

To simply our notation, we will take $\cL A$ to be the full permutation group of $\{1, \ldots, n\}$. We will denote
$$
F_{sym}=\cL AF\cL A. 
$$


\begin{lemma}\label{fhg-}  Let  $\hat {\cL H}$ be  a real subspace of $({\widehat { \mathfrak M}}_n^2)^n$.
Let $\pi : ({\widehat { \mathfrak M}_n^2)^n}\rightarrow \hat  {\cL H}$ be
a $\rr$ linear projection $($i.e. $\pi^2=\pi)$
 that preserves the degrees of the mappings and let $\hat  {\cL G}:= (\I-\pi)({\widehat { \mathfrak M}}_n^2)^n$.
Suppose that there is a positive constant $C$ such that
\eq{pifp}
\pi(E)\prec  CE_{sym}  
\eeq
for any $E\in (\widehat { \mathfrak M}_n^2)^n$.
 Let $F$ be a formal map tangent to the identity. 
There exists a unique decomposition
\eq{decompo}
F=HG^{-1}
\eeq
with $G-I\in\hat  {\cL G}$ and $H-I\in \hat  {\cL H}$.
  If $F$ is convergent, then $G$ and $H$ are also convergent.
\end{lemma}
\begin{proof} If $f$ is a formal mapping, we define the  $k$-jet: 
 $$
 J^kf(z)=\sum_{|Q|\leq k}f_Qz^Q.
 $$
 Write $F=I +f$,  $G=I +g$ and $H=I+h$. We need to solve $FG=H$, i.e to solve
$$
h-g=f(I+g).
$$
Since $f'(0)=0$, then for any $k\geq 2$, the $k$-jet of $f(I +g)$ depends only on the $(k-1)$-jet of $g$. Since $\pi$ is linear and preserves degrees,  \re{pifp} implies that $J^k$ commutes with $\pi$.
 Hence we can define, for all $k\geq 2$,
$$
-J^{k}(g):=\pi\left(J^{k}(f(I +g)\right),\quad J^{k}(h):=(I-\pi)\left(J^{k}(f(I +g)\right).
$$
This solves the formal decomposition uniquely. Assume that $F$ is a germ of holomorphic mapping.
Hence, we have
$$
 g\prec C(f(I +g))_{sym}\prec C f_{sym}(I_{sym}+ g_{sym}).
$$
Since $g_{sym}$ is the smallest left and right $\cL A$ invariant power series that dominates $g$, we have
$$
g_{sym} \prec C f_{sym}(I_{sym}+ g_{sym}).
$$
Therefore, $g_{sym}$ is dominated by the solution $u$ to
$$
u=Cf_{sym}(I_{sym}+ u),\quad u(0)=0.
$$
Notice that $u$  is real analytic near the origin by the implicit function theorem. So, $g_{sym}$ is convergent,  and both $g$ and $h
=g+f(I +g)$ are convergent in a neighborhood of the origin.
\end{proof}
\begin{cor}\label{fhg}
The previous decomposition \rea{decompo} is valid with $\hat  {\cL G}:= {\cL C}_2(\hat S)$ and $\hat  {\cL H}:= {\cL C}_2^{\mathsf{ c}}(\hat S)$ $($resp. $\hat  {\cL G}:={\cL C}_2(\hat T_1,\hat T_2)$ and $\hat  {\cL H}:={\cL C}^{\mathsf{ c}}_2(\hat T_1,\hat T_2)$; $\hat  {\cL G}:={\cL C}_2(D)$ and $\hat  {\cL H}:={\cL C}^{\mathsf{ c}}_2(D))$.
\end{cor}
\begin{proof}
We apply the previous lemma by finding $\pi$. 
The first case is obvious since
$K$ is in ${\cL C}_2(\hat S)$ (resp. ${\cL C}^{\mathsf{ c}}_2(\hat S)$) if and only if $K_Qz^Q\in \cL C_2(\hat S)$
(resp. ${\cL C}^{\mathsf{ c}}_2(\hat S)$)
for all $Q$. So we take 
$$
(I-\pi)(K)=\sum_{j=1}^{n}\sum_{e_jz^Q\in\widehat {\cL G}
} K_{j,Q}z^Qe_j.
$$

Next, we consider the case where $\hat  {\cL G}= {\cL C}_2(\hat T_1,\hat T_2)$ and $\hat  {\cL H}= {\cL C}^{\mathsf{ c}}_2(\hat T_1,\hat T_2)$. We need to find a projection such that $\hat  {\cL H}=\pi({\widehat { \mathfrak M}}_n^2)^n$ and $ \hat  {\cL G}=(\id-\pi)({\widehat { \mathfrak M}}_n^2)^n.$
Note that  $g\in\cL C_2(\hat T_1,\hat T_2)$ and $h\in \cL C^{\mathsf{ c}}_2(\hat T_1,\hat T_2)$ are determined by conditions
\gan
g_{j,(\gamma+e_j)\gamma}=g_{(j+p),\gamma(\gamma+e_j)}, \quad
h_{j,(\gamma+e_j)\gamma}=-h_{(j+p),\gamma(\gamma+e_j)}, \quad 1\leq j\leq p,
\\
g_{j,PQ}=g_{(j+p),QP}=0,\quad P-Q\neq e_j.
\end{gather*}
Thus, if $h-g=K$, we determine $g$ uniquely by combining the above identities with
\aln
g_{j,(\gamma+e_j)\gamma}&=\frac{-1}{2}\left\{K_{j,(\gamma+e_j)\gamma}
+K_{(j+p),\gamma(\gamma+e_j)}\right\},\\ 
h_{j,(\gamma+e_j)\gamma}&=
\frac{1}{2}\left\{K_{j,(\gamma+e_j)\gamma}-
K_{(j+p),\gamma(\gamma+e_j)}\right\} 
\end{align*}
for $1\leq j\leq p$.
For the remaining coefficients of $h$, set $h_{i,PQ}=K_{i,PQ}$.
 Therefore, $\pi(K):=h\prec K_{sym}$  and the proof is complete.
\end{proof}

\begin{rem}Let $\cL A,\cL B$ be two subgroups of permutations. Instead of using the full permutations group, we could have used $G_{sym}:=\cL AG\cL B$. We have
$$
G\prec \cL AG\cL B\prec C\cL A(F\circ (I+G))\cL B\prec (\cL AF\cL A)(\cL AI\cL B+\cL AG\cL B).
$$
\end{rem}
\begin{rem}We do not know if  there are convergent $G\in \cL C(\hat S)$ and $H\in\cL C^{\mathsf{ c}}(\hat S)$ such
that $F=GH$ when $F$ is convergent.  Note that the formal decomposition exists.
\end{rem}

Recall that for $j=1,\ldots, p$, we define
$$
Z_j\colon\xi'=\xi, \quad \eta_k'=\eta_k, \ k\neq j, \quad \eta_j'=-\eta_j.
$$
We have seen in section~\ref{secquad} how invariant functions of $Z_j$ play a role in constructing normal form of quadrics. 
In section~\ref{nfin}, we will also need 
a centralizer for non linear maps (see~\rl{cnnl}) to obtain normal forms for two families of involutions. Therefore, let us
first record here  the following description
of centralizer of $Z_1,\ldots, Z_p$. 
\begin{lemma}\label{lehphi}The  centralizer, ${\cL C}(Z_1,\ldots, Z_p)$,  consists of  formal  mappings $$(\xi,\eta)\to (U(\xi,\eta),
 \ldots, \eta_1V_1(\xi,\eta), \ldots, \eta_pV_p(\xi,\eta))$$ such that $U(\xi,\eta), V(\xi,\eta)$
 are even in each $\eta_j$. Let
 $
 {\cL C}^{\mathsf{ c}}(Z_1,\ldots, Z_p) $
 denote the set of mappings $I+ (U, V)$
 which are tangent to the identity such that
 \eq{vjpej}
 U_{j,PQ}=  V_{j,P(e_j+Q')}=0, \quad Q, Q'\in 2\nn^p, \ 
|P|+|Q|>1, \  |P|+|Q'|>1. \eeq
 Let $\psi\in{\cL C}(Z)$  be tangent to the identity. There exist  unique
$\psi_0\in {\cL C}(Z_1,\ldots, Z_p)$
and $\psi_1\in{\cL C}^{\mathsf{ c}}(Z_1,\ldots, Z_p)$ such that
$ 
 \psi=\psi_1\psi_0^{-1}.
$ 
Moreover, if $\psi$ is convergent, then $\psi_0$ and $\psi_1$ are convergent.
\end{lemma}
\begin{proof} The lemma follows immediately from \rl{fhg-} in which 
$\hat H$ is the $\rr$ linear space of mappings $(U,V)$ without constant or linear terms, which  satisfy  \re{vjpej}. The projection
$\pi$ is the unique projection onto $\hat H$ (i.e. $\pi^2=\pi$, and $\pi$ is the identity on $\hat H$) such that 
 $\pi$ is linear and preserves degrees,  and 
$\pi(E)=0$ if $E(\xi,\eta)=O(|(\xi,\eta)|^2)$ and   $E\in{\cL C}_2(Z_1,\ldots, Z_p)$.
\end{proof}

\setcounter{thm}{0}\setcounter{equation}{0}

\setcounter{thm}{0}\setcounter{equation}{0}
\section{Formal normal forms of the reversible map $\sigma$ 
}\label{secnfs}

Let us first describe our plans to derive the normal forms of $M$.
We would like
to show that two  families of involutions $\{\tau_{1j},\tau_{2j},\rho\}$ and
$\{\tilde\tau_{1j},\tilde\tau_{2j},\tilde\rho\}$ are holomorphically equivalent, if their corresponding
normal forms are equivalent under a much smaller set of changes of coordinates.
Ideally, we would like to conclude that  $\{\tilde\tau_{1j},\tilde\tau_{2j},\tilde\rho\}$ are holomorphically equivalent
if and only if their corresponding normal forms are the same, or if they are the same under a change of coordinates with finitely many parameters. For instance the Moser-Webster normal form for real analytic surfaces ($p=1$) with non-vanishing elliptic Bishop invariant falls into
  the former situation, while the Chern-Moser theory \cite{chern-moser} for real analytic hypersurfaces with non-degenerate Levi-form is an example for the latter.
Such a normal form will tell us if
 the real manifolds have infinitely many invariants or not. One of our goals is to understand if the normal form so achieved can
 be realized by a convergent normalizing transformation.
We will see soon 
that we can achieve our last goal  under  some assumptions on the  family of involutions. Alternatively and perhaps for simplicity of the normal form theory, we would
like to seek normal forms which are dynamically or geometrically significant.

Recall that for each real analytic manifold that
has  $2^p$, the maximum number of,  commuting
deck transformations $\{\tau_{1j}\}$,
we have found
 a unique set of generators $\tau_{11},\ldots, \tau_{1p}$  so that each $\fix(\tau_{1j})$
has codimension $1$. More importantly $\tau_1=\tau_{11}\cdots\tau_{1p}$ is the unique deck
transformation of which the set of fixed points  
has dimension $p$. Let $\tau_2=\rho\tau_1\rho$ and $\sigma=\tau_1\tau_2$.
To normalize $\{\tau_{1j}, \tau_{2j},\rho\}$, we will choose $\rho$ to be the standard anti-holomorphic
involution determined by the linear parts of $\sigma$. Then we
normalize $\sigma=\tau_1\tau_2$ under formal mapping commuting with $\rho$. This will determine a 
normal form for $\{\tau_1^*,\tau_2^*,\rho\}$. This part of normalization is analogous to the Moser-Webster normalization.
  When $p=1$, Moser and Webster obtained a unique normal form by a simple argument.
However, this last step of simple normalization is not available when $p>1$.   By assuming
$\log \hat M$ associated to  $\hat\sigma$
 is tangent to the identity, 
  we
will obtain a unique formal normal form $\hat\sigma, \hat\tau_1,\hat\tau_2$  for $\sigma,\tau_1,\tau_2$. 
Next, we need to construct the normal form for the families of involutions. We first ignore the reality condition, by
finding $\Phi$ which transforms $\{\tau_{1j}\}$ into a set of
involutions  $\{\hat \tau_{1j}\}$ which is decomposed canonically according to
$\hat\tau_1$.  This allows us to express $\{\tau_{11},\ldots, \tau_{1p},\rho\}$
via $\{\hat\tau_1,\hat\tau_2,\Phi,\rho\}$,   as  in the classification of the families of linear involutions. 
Finally, we further normalize $\{\hat\tau_1,\hat\tau_2,\Phi,\rho\}$ to get our normal form.

\

\begin{defn}\label{notation}
 Throughout this section and next, we denote $\{h\}_{d}$ the set of coefficients
of $h_P$ with $|P|\leq d$ if $h(x)$ is a map or function in $x$ as power series. We  
  denote by $\cL A_{P}(t), \cL A(y;t)$, etc.,   a universal {\it polynomial} whose
coefficients and degree depend  on a multiindex. 
The variables in these polynomials will involve a collection of Taylor coefficients of various
mappings. The collection will also depend on $|P|$. 
As such dependency (or independency to  coefficients of higher degrees) is crucial to our computation,
we will remind the reader the dependency when  emphasis is necessary. 
\end{defn}

For instance, let us take two formal mappings $F,G$ from $\cc^n$ into itself. Suppose
that  $F=\id+f$ with $f(x)=O(|x|^2)$ and $ G= LG+g$ 
with $g(x)=O(|x|^2)$ and $LG$ being linear. For $P\in\nn^n$ with $ |P|>1$, we can express
\ga\label{F-1P5}
(F^{-1})_{P}=-f_{P}+\cL F_{P}(\{f\}_{|P|-1}), \\
 (G\circ F)_{P}=g_P+((LG)\circ f)_P+\cL G_{P}(LG;\{f,g\}_{|P|-1}),
\label{GFPg}\\
\label{F-1GF}
 (F^{-1}\circ G\circ F)_{P}=g_P-(f\circ (LG))_P+((LG)\circ f)_P+\cL H_{P}(LG;\{f,g\}_{|P|-1}).
\end{gather}

\subsection{
Formal normal forms of  pair of involutions $\{\tau_1,\tau_2\}$}

We first find a normal form for $\sigma$ in $\cL C(S)$.

\begin{prop}\label{ideal0} Let $\sigma$ be a holomorphic map. Suppose that $\sigma$ has  the  linear part $$
\hat S\colon\xi_j'=\mu_j\xi_j, \quad \eta_j=\mu_j^{-1}\eta_j, \quad 1\leq j\leq p
$$
and $\mu_1,\ldots, \mu_p$ are non-resonant. Then there exists a unique normalized formal map $\Psi\in {\cL C}^{\mathsf{ c}}(\hat S)$ such that $\sigma^*=\Psi^{-1}\sigma\Psi\in{\cL C}(\hat S)$.
Moreover,  $\tilde \sigma=\psi_0^{-1}\sigma^*\psi_0\in\cL C(\hat S)$,
if and only if $\psi_0\in\cL C(\hat S)$  and it is invertible. Let
\gan
\sigma^*\colon\xi_j'=M_j(\xi\eta)\xi_j, \quad \eta_j'=N_j(\xi\eta)\eta_j,\\
\tilde \sigma\colon\xi_j'=\tilde  M_j(\xi\eta)\xi_j, \quad \eta_j'=\tilde  N_j(\xi\eta)\eta_j,\\
\psi_0\colon\xi_j'= a_j(\xi\eta)\xi_j,\quad \eta_j'=  b_j(\xi\eta)\eta_j. 
\end{gather*}
\bppp
\item
Assume that    $\tau_1,\tau_2$
are holomorphic involutions and $\sigma=\tau_1\tau_2$.  Then $\sigma^*=\tau_1^*\tau_2^*$ with
\ga\label{tauis}
\tau_{i}^*=\Psi^{-1}\tau_i\Psi\colon\xi_j'=\Lambda_{ij}(\xi\eta)\eta_j, \quad \eta_j'=\Lambda_{ij}^{-1}(\xi\eta)\xi_j;\\
N_j=M_j^{-1}, \quad  M_j=\Lambda_{1j}\Lambda_{2j}^{-1}.
\nonumber
\end{gather}
Let the linear part of $\tau_i$ be given by
$$
\hat T_i\colon\xi_j'=\lambda_{ij}\eta_j,\quad\eta_j'=\lambda_{ij}^{-1}\xi_j.
$$
Suppose that $\lambda_{2j}^{-1}=\lambda_{1j}$.
There exists a unique $\psi_0\in   {\cL C}^{\mathsf{ c}}(\hat T_1,\hat T_2)$ such that
\ga \nonumber
\tilde\tau_{i}=\psi_0^{-1}\tau_i^*\psi_0\colon\xi_j'=\tilde\Lambda_{ij}(\xi\eta)\eta_j, \quad \eta_j'=\tilde
\Lambda_{ij}^{-1}(\xi\eta)\xi_j;\\
\label{tl21}
\tilde  M_{j}=\tilde \Lambda_{1j}^2  =\tilde N_j^{-1}, \quad \tilde \Lambda_{2j}=\tilde \Lambda_{1j}^{-1}.
\end{gather}
 Let
 $\psi_1$ be a formal biholomorphic map. Then  $\{\psi_1^{-1}\tilde\tau_1\psi_1,
 \psi_1^{-1}\tilde\tau_2\psi\}$ has the same form as of $\{\tilde \tau_{1},\tilde \tau_{2}\}$  
    if and only
if $\psi_1\in{\cL C}(\hat T_1,\hat T_2)$;   moreover, $\tilde\Lambda_{ij}(\xi\eta)$, $\tilde M_j(\xi\eta)$ are transformed
into 
\eq{lamtpsi}
\tilde\Lambda_{ij}\circ\tilde\psi_1, \quad \tilde M_j\circ\tilde\psi_1.
\eeq
Here $\tilde\psi_1(\zeta)=(\diag c(\zeta))^2\zeta$ and
$\psi_1(\xi,\eta)=((\diag c(\xi\eta))\xi,(\diag c(\xi\eta))\eta)$. 
\item
Assume further that $\tau_2=\rho\tau_1\rho$, where    $\rho$
is defined by \rea{eqrh}. Let
\eq{rhze}
\nonumber
{\rho_z}\colon\zeta_j\to\ov\zeta_j, \quad 1\leq j\leq e_*+h_*;
\quad \zeta_{s}\to\ov\zeta_{s+s_*}, \quad e_*+h_*<s\leq p-s_*.
\eeq
 Then $\rho\Psi=\Psi\rho$,   $\tau_2^*=\rho\tau_1^*\rho$,  and $(\sigma^*)^{-1}=\rho\sigma^*\rho$. The 
  last two identities are equivalent to
\begin{alignat}{5} \label{a2e-}
&\Lambda_{2e}^{-1}
&&=\ov{\Lambda_{1e}\circ{\rho_z}}, \quad &&\ov {M_e\circ{\rho_z}}=M_e,
\quad &&1\leq e\leq e_*;\\
&\Lambda_{2h}&&=\ov{\Lambda_{1h}\circ{\rho_z}},\quad&&
\ov {M_h\circ{\rho_z}}=M_h^{-1},\quad&& e_*< h\leq h_*+e_*;\\
  & \Lambda_{2(s)}&&=\ov{\Lambda_{1(s_*+s)}\circ{\rho_z}}, &&
\\
& \Lambda_{2(s_*+s)}&&=\ov{\Lambda_{1s}\circ{\rho_z}},\quad&&
\ov {M_s^{-1}\circ{\rho_z}}=M_{s_*+s}, \quad && h_*+e_*< s\leq p- s_*.
\label{aiss}\end{alignat}
Let   $\psi_0$ and $\tilde\tau_i=\psi_0^{-1}\tau_i^*\psi_0$ be as in (i). Then $\rho\psi_0=\psi_0\rho$, and $\hat\tau_1,\hat\tau_2$ satisfy
\ga\label{ohai}
\tilde\Lambda_{ie}=\ov{\tilde\Lambda_{ie}\circ{\rho_z}}, \quad \tilde\Lambda_{ih}^{-1}=\ov{\tilde\Lambda_{ih}\circ{\rho_z}},\quad
\tilde\Lambda_{i{s+s_*}}=\ov{\tilde\Lambda_{is}^{-1}\circ{\rho_z}}.
\end{gather}
\eppp
\end{prop}
\begin{proof} We will use the Taylor formula
$$
f(x+y)=f(x)+\sum_{k=1}^m\frac{1}{k!}D_kf(x;y)+R_{m+1}f(x;y)
$$
with $D_kf(x;y)=\{\pd_t^kf(x+ty)\}|_{t=0}$ and 
\eq{taylor}
 R_{m+1}f(x;y)=(m+1)\int_0^1(1-t)^m\sum_{|\alpha|=m+1}\frac{1}{\alpha!}\partial^{\alpha}f(x+ty)y^{\alpha}\, dt.
\eeq
Set $D=D_1$.
Let $\sigma$ be given by
\eq{idfg-}
\nonumber
\xi_j'=M_j^0(\xi\eta)\xi_j+f_j(\xi,\eta), \quad \eta_j'=N_j^0(\xi\eta)\eta_j+g_j(\xi,\eta)
\eeq
with
\eq{idfg}
 (f,g)\in  {\cL C}^{\mathsf{ c}}_2(\hat S). 
\eeq

We need to find    $\Phi\in  {\cL C}^{\mathsf{ c}}(S)$
such that $\Psi^{-1}\sigma\Psi=\sigma^*$ is given by
$$
\xi_j'=M_j(\xi\eta)\xi_j, \quad\eta_j'=N_j(\xi\eta)\eta_j.
$$
By definition, $\Psi$ has the form
$$
\xi'_j=\xi_j+U_j(\xi,\eta), \quad\eta_j'= \eta_j+V_j(\xi,\eta), \quad U_{j,(P+e_j)P}=V_{j,P(P+e_j)}=0.
$$
The components of $\Psi\sigma^*$  are
\al\label{xjmj}
\xi_j'&=M_j(\xi\eta)\xi_j+U_j(M(\xi\eta)\xi,N(\xi\eta)\eta)
,\\
\eta_j'&=N_j(\xi\eta)\eta_j+V_j(M(\xi\eta)\xi,
N(\xi\eta)\eta).
\label{ejmj}\end{align}
To derive the normal form, we only need Taylor theorem in order one. This can also demonstrate  small divisors in the normalizing transformation; however, one cannot see the small divisors in the normal forms.  Later we will  show the existence of divergent normal forms. This requires us to use Taylor formula whose remainder
has  order two.
By the Taylor theorem,
we write the components of $\sigma\Psi$ as
\al\label{mj0x}
\xi_j'&=(M^0_j(\xi\eta)+DM_j^0(\xi\eta)(\eta U +\xi V  +UV))(\xi_j+U_j)\\
\nonumber
&\quad +f_j(\xi,\eta)+ Df_j(\xi,\eta)(U, V)+A_j(\xi,\eta),\\
\label{nj0x}\eta_j'&=(N_j^0(\xi\eta)+DN_j^0(\xi\eta)(\eta U +\xi V+UV))(\eta_j+V_j)\\
\nonumber
&\quad +g_j(\xi,\eta)+Dg_j(\xi,\eta)(U,V)+B_j(\xi,\eta).
\end{align}
Recall our notation that
 $UV=(U_1(\xi,\eta)V_1(\xi,\eta),\ldots, U_p(\xi,\eta)V_p(\xi,\eta))$. The second order remainders are
\al \label{ajpq}
A_j(\xi,\eta)&=
 R_2M_j^0(\xi\eta;\xi U+\eta V+UV)(\xi_j+U_j) +R_2f_j(\xi,\eta;U,V),\\
  B_j(\xi,\eta)&=
 R_2N_j^0(\xi\eta;\xi U+\eta V+UV)(\eta_j+V_j) +R_2g_j(\xi,\eta;U,V).\label{bjpq--}
 \end{align}
Note that the remainder $R_2M^0$ is independent of the linear part of $M^0$. Thus
\gan
R_2M_j^0=R_2(M_j^0-LM_j^0), \quad R_2N_j^0=R_2(N_j^0-LN_j^0).
\end{gather*}
Let us calculate the largest degrees of coefficients of $M^0-LM^0, (U,V,f,g)$ on which $A_{j,PQ}$ depend.
We denote the two degrees by $w,d$, respectively.
Since $\ord (f,g,U,V)\geq2$, we have
$$
 2(w- 2)+(d+4)+1\leq |P|+|Q|,\quad \text{or} \quad (d-2)+2d\leq |P|+|Q|,
$$
where the first inequality is obtain from the first term on the right-hand side of \re{ajpq} and the second
term yields the second inequality.
Since $M^0-LM^0, (U,V)$,  and $(f,g)$ do not have linear terms, we have $w\geq2$ and $d\geq2$.
Thus, we have  crude bounds
$$
d\leq |P|+|Q|-1, \quad w\leq\f{|P|+|Q|-1}{2}.
$$
Analogously, we can estimate the degrees of coefficients of $N^0$. We obtain
\aln
A_{j,PQ}&=\cL A _{j,PQ}( \{M^0-LM^0\}_{\f{|P|+|Q|-1}{2}};\{f, U, V\}_{|P|+|Q|-  1}),\\
B_{j, QP}&=\cL B _{j,QP}(\{N^0-LN^0\}_{\f{|P|+|Q|-1}{2}}; \{g, U, V\}_{|P|+|Q|-1}).
\end{align*}
  Recall our notation that $\{f,U,V\}_d$ is the set of coefficients of $f_{PQ}, U_{PQ}, V_{PQ}$
with $|P|+|Q|\leq d$.  Here
 $\cL A_{j,PQ}(t';t''),\cL B_{j,QP}(t';t'')$ are polynomials  of which each has  coefficients that depend only on $j,P,Q$
 and they vanish at $t''=0$.

To finish the proof of the proposition, we will not need the explicit expressions
involving $DM_j^0, DN_j^0$, $Df_j$, $Dg_j$. We will use these derivatives
in the proof of \rl{mj0mj}. So we derive derive these expression in this proof too.

We apply the projection \re{xjmj}-\re{ejmj} and \re{mj0x}-\re{nj0x} onto $\cL C^{\mathsf c}_{2}(S)$,  via monomials in each 
component of both sides of the identities. 
  The images of the mappings
\begin{align*}(\xi,\eta)&\mapsto (U(M(\xi\eta)\xi,N(\xi\eta)\eta)
,V(M(\xi\eta)\xi,N(\xi\eta)\eta)),\\
(\xi,\eta)
&\mapsto (M^0(\xi\eta)U(\xi,\eta), N^0(\xi,\eta)V(\xi,\eta))\end{align*}
under the projection are $0$.
We obtain from \re{xjmj}-\re{nj0x}   and \re{ajpq}-\re{bjpq--}
\al\label{mp-q}
(\mu^{P-Q}-\mu_j)
U _{j,PQ} &=f _{j,PQ}+\cL U _{j,PQ}( \{M^0\}_{\f{|P|+|Q|-1}{2}}; \{f, U, V\}_{|P|+|Q|-1}),\\
(\mu^{Q-P}-\mu_j^{-1})
V _{j,QP}&=g _{j,QP}+\cL V _{j,QP}(\{N^0\}_{\f{|P|+|Q|-1}{2}}; \{g, U, V\}_{|P|+|Q|-1})
\label{mp-q2}\end{align}
for $\mu^{P-Q}\neq\mu_j$, which is always solvable. 
Next, we project \re{xjmj}-\re{ejmj} and \re{mj0x}-\re{nj0x} onto $\cL C_{2}(\hat S)$,  via monomials in each 
component of both sides of the identities. 
  Using \re{idfg} we obtain
\al
\label{mpmp}
M_P&=M^0_P+\cL M_P(\{M^0\}_{|P|-1}; \{f, U, V\}_{2|P|-1}), \\ N_P&=N^0_P+\cL N_P(\{N^0\}_{|P|-1}; \{f, U, V\}_{2|P|-1}).\label{npnp}
\end{align}
  Here $\cL M_{P}, \cL N_{P}$ are polynomials  of which each has  coefficients  that depend only on $P$, and $\{M^0\}_d$
stands for the set of coefficients $M^0_{P}$ with $|P|\leq d$. 
Note that $\cL U _{j,PQ}=\cL V _{j, QP}=0$ when $|P|+|Q|=2$, or $\ord (f,g)> |P|+|Q|$. And $\cL M_P=\cL N_P=0$ when $\ord (f,g)>2|P|$, by \re{idfg}. Inductively, by using
\re{mp-q}-\re{mp-q2} and \re{mpmp}-\re{npnp}, we obtain unique solutions $U,V, M, N$. Moreover, the solutions and their dependence on the coefficients of $f,g$ and small divisors have the form
\al
\label{ujpq}
U _{j,PQ}&=(\mu^{P-Q}-\mu_j)^{-1}\left\{f _{j,PQ}+\cL U _{j,PQ}^*(\del_{d-1},
\{M^0,N^0\}_{[\f{d-1}{2}]};\{ f,g\}_{d-1})\right\},\\
\label{vjpq}
V _{j,QP}&=(\mu^{Q-P}-\mu_j^{-1})^{-1}\left\{g _{j,QP}+\cL V _{j,QP}^*(\del_{d-1},
\{M^0,N^0\}_{[\f{d-1}{2}]};\{ f,g\}_{d-1})\right\},
\end{align}
where $d=|P|+|Q|$ and $\mu^{P-Q}\neq\mu_j$, and
$
\del_{d-1}
$ is the union of $\{\mu_1,\mu_1^{-1},\ldots,
\mu_p,\mu_p^{-1}\}$ and
$$
\left\{\frac{1}{\mu^{A-B}-\mu_j}\colon |A|+|B|\leq d-1,
  A,B\in \nn^p\right\}.
$$
This shows that for any $M^0, N^0$ there exists  a unique mapping $\Psi$ transforms $\sigma$ into $\sigma^*$.
  Furthermore,  $\cL U^*_{j,PQ}(t';t''), \cL V^*_{j,QP}(t';t'')$ are polynomials  of which each has  coefficients that depend only on $j,P,Q$, and they vanish at $t''=0$.

  For later purpose,  let us express $M,N$ in terms of $f,g$. We substitute expressions \re{ujpq}-\re{vjpq}
for $U,V$ in \re{mpmp}-\re{npnp} to obtain
\al
\label{mpmpnsd}
M_P&=M^0_P+\cL M_P^*(\del_{2|P|-1},\{M^0,N^0\}_{|P|-1}; \{f, g\}_{2|P|-1}), \\ N_P&=N^0_P+\cL N_P^*(
\del_{2|P|-1},\{M^0,N^0\}_{|P|-1}; \{f, g\}_{2|P|-1})\label{npnpnsd}
\end{align}
with $f,g$ satisfying \re{idfg}.

Assume that   $\tilde\sigma=\psi_0^{-1}\sigma^*\psi_0$ commutes with $\hat S$. By \nrc{fhg},
we can decompose  
$\psi_0=HG^{-1}$ with $G\in\cL C(\hat S)$ and $H\in\cL C^{\mathsf{ c}}(\hat S)$. Furthermore,
 $G^{-1}\tilde \sigma G$ commutes with $\hat S$ and
  $H^{-1}\sigma^* H$. By the uniqueness conclusion for the above $\psi_0$,   $H$ must be the identity.  This shows that $\psi_0\in\cL C(\hat S)$.

(i). Assume that we have normalized $\sigma$. We now use it to normalize the pair of involutions.   Assume that $\sigma=\tau_1\tau_2$ and $\tau_j^2=I$. Then $\sigma^*=\tau_1^*\tau_2^*$.
Let $T_0(\xi,\eta):=(\eta,\xi)$.
We have 
 $T_0(\sigma^*)^{-1}T_0=T_0\tau_1^*\sigma^*\tau_1^*T_0$.
 By the above normalization,  $T_0(\sigma^*)^{-1}T_0$ commutes with $\hat S$. Therefore, $\tau_1^*T_0$ belongs to the centralizer of $\hat S$ and  it must be of the form
 $(\xi,\eta)\to(\xi \Lambda_1(\xi\eta),\eta  \Lambda_1^*(\xi\eta))$. Then $(\tau_1^*)^2=I$ implies that
 $$
 \Lambda_1(\xi\eta (\Lambda_1 \Lambda_1^*)(\xi\eta)) \Lambda_1^*(\xi\eta)=1.
 $$
  The latter implies, by induction on $d>1$, that $\Lambda_1 \Lambda_1^*=1+O(d)$ for all $d>1$, i.e. $\Lambda_1 \Lambda_1^*=1$.

Let $\tau_i^*$ be given by \re{tauis}. We want to achieve
$\tilde\Lambda_{1j}\tilde\Lambda_{2j}=1$ for $\tilde\tau_i=\psi_0^{-1}\tau_i^*\psi_0$
by applying
 a transformation $\psi_0$ in $\cL C^{\mathsf{ c}}(\hat T_1,\hat T_2)$ that commutes with $\hat S$. According to Definition \ref{ccst}, it has the form
 $$
 \psi_0\colon \xi_j=\tilde\xi_j(1+a_j(\tilde\zeta)), \quad
 \eta_j=\tilde\eta_j (1-a_j(\tilde\zeta))
  $$
with $a_j(0)=0$. Here $\tilde \zeta_j:=\tilde\xi_j\tilde \eta_j$ and
 $\tilde \zeta:=(\tilde \zeta_1,\ldots, \tilde \zeta_p)$. Computing the products $\zeta$ in $\tilde\zeta$ and solving $\tilde\zeta$ in $\zeta$, we obtain
$$
\psi_0^{-1}\colon \tilde\xi_j=\xi_j(1+b_j(\zeta))^{-1}, \quad
 \tilde\eta_j=\eta_j (1-b_j(\zeta))^{-1}.
$$
Note that $(a_j^2)_P=\cL A _{j,P}(\{a\}_{|P|-1})$, and
$$\xi_j\eta_j=\tilde\xi_j\tilde\eta_j(1-a_j^2(\tilde\zeta)),\quad
\tilde\xi_j\tilde\eta_j=\xi_j\eta_j(1-b_j^2(\zeta))^{-1}.$$
>From $\psi_0^{-1}\psi_0=I$, we get
\ga
b_j(\zeta)=a_j(\tilde\zeta), \quad
\label{bjpq}
b _{j,P}=a _{j,P}+\cL B _{j,P}(\{a\}_{|P|-1}).
 \end{gather}
By a simple computation we see that $\tilde\tau_i=\psi_0^{-1}\tau_i^*\psi_0$ is given by
$$
\tilde\xi_j'=\tilde\eta_j\tilde\Lambda_{ij}(\tilde\zeta),
\quad\tilde\eta_j'=\tilde\xi_j\tilde\Lambda_{ij}^{-1}(\tilde\zeta)
$$
with
$$
\tilde\Lambda_{1j}\tilde\Lambda_{2j}(\tilde\zeta)=(\Lambda_{1j}\Lambda_{2j})(\zeta)
(1+b_j(\zeta'))^{-2}(1- a_j(\tilde\zeta))^{2}.
$$
Here $\zeta_j'=\zeta_j(1-a_j^2(\tilde\zeta))$.
 Using \re{bjpq} and 
 the implicit function theorem, we determine $a_j$ uniquely to achieve $\tilde\Lambda_{1j}\tilde\Lambda_{2j}=1$.

To identify the transformations that preserve the form of $\tilde\tau_1,\tilde\tau_2$,
 we first verify that each element $\psi_1\in\cL C(\hat T_1,\hat T_2)$ preserves that form.
According to \re{bjaj}, we have
\gan
\psi_1\colon\xi_j=\tilde\xi_j  \tilde a_j(\tilde\zeta), \quad\eta_j=\tilde\eta_j\tilde a_j(\tilde\zeta),\\
\psi_1^{-1}\colon \tilde\xi_j= \xi_j\tilde b_j( \zeta), \quad\tilde\eta_j= \eta_j\tilde b_j( \zeta),\\
\tilde b_j(\zeta)\tilde a_j(\tilde\zeta)=1.
\end{gather*}
This shows that $\psi_1^{-1}\tilde\tau_i$ is given by  
$$
\tilde\xi_j'=\tilde\Lambda_{ij}(\zeta)\tilde b_j(\zeta)\eta_j,\quad \tilde\eta_j'=\tilde\Lambda_{ij}^{-1}(\zeta)\tilde b_j(\zeta)\xi_j.
$$
Then $\psi_1^{-1}\tilde\tau_i\psi_1$ is given by
$$
\tilde\xi_j'=\tilde\Lambda_{ij}(\zeta)\tilde\eta_j,\quad \tilde\eta_j'=\tilde\Lambda_{ij}^{-1}(\zeta)\tilde\xi_j.
$$
Since $\zeta_j=\tilde\zeta_j\tilde a_j^2(\tilde\zeta)$, then $\psi_1^{-1}\tilde\tau_i\psi_1$
still satisfy \re{tl21}.
Conversely,
suppose  that  $\psi_1$ preserves the forms of $\tilde\tau_1,\tilde\tau_2$.
We apply \nrc{fhg} to decompose  
$\psi_1=\phi_1\phi_0^{-1}$ with $\phi_0\in\cL C(\hat T_1,\hat T_2)$
 and $\phi_1\in\cL C^{\mathsf{ c}}(\hat T_1,\hat T_2)$. Since we just proved that each element in $\cL
 C(\hat T_1,\hat T_2)$ preserves the form of $\tilde\tau_{i}$, then $\phi_1=\psi_1\phi_0$
 also preserves the forms of $\tilde\tau_1,\tilde\tau_2$. On the other hand, we have shown that there exists a unique
 mapping   in $\cL C^{\mathsf c}(\hat T_1,\hat T_2)$ which transforms $\{\tau_1^*,\tau_2^*\}$ into $\{\tilde\tau_1,\tilde\tau_2\}$. This shows that $\phi_0=I$.
We have verified all assertions in (i).

 (ii).  
 According to Remark~\ref{rem-rho}, $\cL C^{\mathsf{ c}}(\hat S)$ and $\cL C^{\mathsf{ c}}(\hat T_1,\hat T_2)$
 are invariant under conjugacy by $\rho$.
 We have $\Psi^{-1}\sigma\Psi=\sigma^*$ and $\Psi\in\cL C^{\mathsf{ c}}(\hat S)$.   Note that $\rho\sigma\rho
 =\sigma^{-1}$ and $\rho\sigma^*\rho$ have the same form as of    $(\sigma^*)^{-1}$, i.e. they
 are in $\cL C(\hat S)$   and have the same linear part.
 We have $\rho\Psi\rho\sigma\rho\Psi^{-1}\rho=\rho(\sigma^*)^{-1}\rho$.
 The uniqueness of $\Psi$ implies that $\rho\Psi\rho
=\Psi$ and
$\tau_2^*=\rho\tau_1^*\rho$. Thus, we obtain relations \re{a2e-}-\re{aiss}.
 Analogously, $\rho\psi_0\rho$ is still in $\cL C^{\mathsf{ c}}(\hat T_1,\hat T_2)$, and
 $\rho\phi_0\rho$ preserves the form of $\tilde\tau_1,\tilde\tau_2$. Thus $\rho\psi_0\rho=\psi_0$ and
 $\tilde\tau_2=\rho\tilde\tau_1\rho$, which gives us \re{ohai}.
\end{proof}

  We will also need the following uniqueness result. 
\begin{cor}\label{finitever} Suppose that $\sigma$ has linear part $\hat S$ with nonresonant $\mu_1, \ldots, \mu_p$. Let $\Psi$ be the 
unique
formal mapping in $\cL C^{\mathsf c}(\hat S)$ such that $\Psi^{-1}\sigma\Psi\in \cL C(\hat S)$. If $\tilde\Psi\in\cL C^{\mathsf c}(\hat S)$ is a polynomial
map of degree at most $d$ such that $\tilde\Psi^{-1}\sigma\tilde\Psi(\xi,\eta)=\tilde\sigma(\xi,\eta)+O(|(\xi,\eta)|^{d+1})$ and $\tilde\sigma\in\cL C(\hat S)$,
then $\tilde\Psi$ is unique. In fact, $\Psi-\tilde\Psi=O(d+1)$. 
\end{cor}
\begin{proof} The proof is contained in the proof of \rp{ideal0}. 
 Let us recap  it by using \re{ujpq}-\re{vjpq} and the proposition.
We take a unique normalized mapping  $\Phi$
such that $\Phi^{-1}\tilde\Psi^{-1}\sigma\tilde\Psi\Phi\in\cL C(\hat S)$. By \re{ujpq}-\re{vjpq},  $\Phi=I+O(d+1)$. From \rp{ideal0} it follows that
$\psi_0:=\tilde\Psi\Phi\Psi^{-1}\in \cL C(\hat S)$.  We obtain $\tilde\Psi\Phi=\psi_0\Psi$. Thus $\psi_0\Psi=\tilde\Psi+O(d+1)$. Since $\psi_0\in\cL C(\hat S)$,
and $\Psi$, $\tilde\Psi$ are in $\cL C^{\mathsf c}(\hat S)$, we conclude that   $\Psi=\tilde\Psi+O(d+1)$. 
\end{proof}

When $p=1$, \rp{ideal0} is due to Moser and Webster. In fact, they  achieved
$$
\tilde M_1(\zeta_1)=e^{\delta (\xi_1\eta_1)^s}. 
$$
Here $\delta=0, \pm1$  
for the elliptic case and  $\delta=0,\pm i$ for the hyperbolic case when $\mu_1$ is not a root of unity,   i.e. 
$\gaa$ is {\em non-exceptional}.
In particular the normal form is always convergent, although the normalizing transformations
are generally divergent for the hyperbolic case.

Let us find out further normalization  that can be performed to preserve the form of $\sigma^*$. 
In \rp{ideal0}, we have proved that if $\sigma$ is tangent to $\hat S$, there exists a unique
$\Psi\in\cL C^{\mathsf{ c}}(\hat S)$ such that $\Psi^{-1}\sigma\Psi$ is an element $\sigma^*$ in the  
centralizer of $\hat S$. Suppose now that $\sigma=\tau_1\tau_2$
while $\tau_i$ is tangent to $\hat T_i$. Let   
$\tau_i^*=\Psi^{-1}\tau_i\Psi$.
 We have also proved that there is a unique $\psi_0\in\cL C^{\mathsf{ c}}(\hat T_1,\hat T_2)$ such that
 $\tilde\tau_i=\psi_0^{-1}\tau_i^*\psi_0$, $i=1,2$, are of the form \re{tl21}, i.e.
 \gan
 \tilde\tau_i\colon\xi_j'=\tilde\Lambda_{ij}(\zeta)\eta_j,
 \quad\eta_j'=\tilde\Lambda_{ij}^{-1}(\zeta)\xi_j;\\
 \tilde\sigma\colon\xi_j'=\tilde M_{j}(\zeta)\xi_j,\quad\eta_j'
 =\tilde M_{j}^{-1}(\zeta)\eta_j.
 \end{gather*}
Here $\zeta=(\xi_1\eta_1,\ldots, \xi_p\eta_p)$,  $\tilde\Lambda_{2j}=\tilde\Lambda_{1j}^{-1}$ and $\tilde M_j=\tilde\Lambda_{1j}^2$. We still have freedom to further normalize $\tilde\tau_1,\tilde\tau_2$ and to preserve their forms. However, any new coordinate transformation must be in $\cL C(\hat T_1,\hat T_2)$, i.e. it must have the form
$$
\psi_1\colon \xi_j\to a_j(\xi\eta)\xi_j, \quad\eta_j\to a_j(\xi\eta)\eta_j.
$$
When $\tau_{2j}=\rho\tau_{1j}\rho$, we require that
  $\psi_1$   commutes with $\rho$, i.e.
$$
a_e=\ov a_e,\quad a_h=\ov a_h, \quad a_s=\ov a_{s+s_*}.
$$
In $\zeta$ coordinates, the transformation $\psi_1$ has the form
\eq{zjaj}
\var\colon \zeta_j\to b_j(\zeta)\zeta_j, \quad 1\leq j\leq p
\eeq
with $b_j=a_j^2$.
Therefore, the mapping $\var$ needs to satisfy
$$
b_e>0, \quad b_h>0, \quad b_s=\ov b_{s+s_*}.
$$
Recall from \re{a2e-}-\re{aiss} the reality conditions on $\tilde M_j$
\begin{alignat*}{4} 
& \ov {\tilde M_e\circ{\rho_z}}&&=\tilde M_e,
\quad &&1\leq e\leq e_*;  \\
&\ov {\tilde M_h\circ{\rho_z}}&&=\tilde M_h^{-1},\quad && e_*< h\leq h_*+e_*;\\
&\tilde M_{s_*+s}&&=\ov {\tilde M_s^{-1}\circ{\rho_z}}, \quad && h_*+e_*< s\leq p- s_*.
\end{alignat*}
Here
\eq{rhoz5}
{\rho_z}\colon\zeta_j\to\ov\zeta_j,
\quad\zeta_s\to\ov\zeta_{s+s_*}, \quad\zeta_{s+s_*}\to\ov\zeta_s
\eeq
for $1\leq j\leq e_*+h_*$ and $e_*+h_*<s\leq p-s_*$.

Therefore, our normal form problem leads to another normal form problem which is interesting in its own right. 
To formulate a new normalization problem, let us  define
\ga\label{logm}
(\log \tilde M)_j(\zeta):=\begin{cases}
\log(\tilde M_j(\zeta)/{\tilde M_j(0)}), & 1\leq j\leq e_*,\\
-i\log(\tilde M_j(\zeta)/{\tilde M_j(0)}), & e_*< j\leq   p. 
\end{cases}
\end{gather}
Let $F=\log\tilde M:=((\log \tilde M)_1,\ldots, (\log \tilde M)_p)$.
Then   the reality conditions on $\tilde M$ become
\ga\label{reaf}
F={\rho_z} F{\rho_z}.
\end{gather} The transformations \re{zjaj}
will then satisfy
\eq{reality-phi}
{\rho_z}\var{\rho_z}=\var, \quad b_j(0)>0, \quad 1\leq j\leq e_*+h_*.
\eeq
Therefore,  when $F'(0)$ is furthermore 
diagonal and invertible and its $j$th diagonal entry is positive for $j=e,h$,
we apply a dilation $\varphi$ satisfying the above condition so that
$F$ is tangent to the identity. Then any further change of coordinates
must be tangent to the identity too.
Thus, we need to normalize the formal holomorphic mapping $F$ by composition $F\circ\var$,
for which we study in next subsection.

\subsection{A normal form
for maps  tangent to the identity  
}
%

Let us consider a germ of holomorphic mapping $F(\zeta)$ in ${\mathbf
C}^p$ with an invertible linear part ${\mathbf A}\zeta$ at the origin.
According to the inverse function theorem, there exists a holomorphic
mapping $\Psi$ with $\Psi(0)=0$, $\Psi'(0)=I$ such that $F\circ
\Psi(\zeta)=\mathbf A\zeta$.  On the other hand, if we impose some
restrictions on $\Psi$, we can no longer linearize
$F$ in general.

To focus on applications to CR singularity and to limit the scope
of our investigation,  we now deliberately restrict
 our analysis to the simplest case : $F$ is tangent to the identity.
We shall apply our result to $F=\log \tilde M$ as defined in the previous subsection. 
In what follows, we shall devise  a normal
form of such an $F$ under right composition by $\Psi$ that
{\it preserve all coordinate hyperplanes}, i.e.
$\Psi_j(\zeta)=\zeta_j\psi_j(\zeta)$, $j=1,\ldots, p$.

%

\begin{lemma}\label{fcfp}
Let $F$ be a  
formal holomorphic map of $\cc^p$ that is tangent to the identity  at the origin.
\bppp\item
There exists a  unique formal biholomorphic map $\psi$
which preserves all $\zeta_j=0$
such that   $\hat F \colonequals F\circ\psi$    has the form
\ga\label{hfj1}
\hat F= I+ 
\hat f, \quad\hat f(\zeta)=O(|\zeta|^2); \quad \pd_{\zeta_j}\hat f_j=0, \quad 1\leq j\leq p.
\end{gather}
\item If $F$ is convergent,  the $\psi$  in $(i)$ is convergent.
 If $F$ commutes with ${\rho_z}$,  so does  the 
 $\psi$.  
 \item   The formal normal form in (i) has the form 
 \eq{hfjqf} 
 \hat f_{j,Q}=f_{j,Q}+\cL F_{j,Q}(\{f\}_{|Q|-1}), \quad q_j=0, 
\quad |Q|>1.
 \eeq
 Here $\cL F_{j,Q}$ are universal polynomials depending only on $F'(0)$ and they vanish at $0$.
\eppp
\end{lemma}
\begin{proof} (i)
Write  $F=I+f$ and 
\gan
\psi\colon \zeta_j' =\zeta_j+\zeta_jg_j(\zeta), \quad g_j(0)=0.
 \end{gather*}
 For $\hat F=F\circ\psi$,  we need to solve for $\hat f, g$ from
 $$
\hat f_j(\zeta) =\zeta_jg_j(\zeta)+f_j\circ\psi(\zeta). 
 $$
Fix $Q=(q_1,\ldots, q_p)\in \nn^p$ with $|Q|>1$.
We obtain unique solutions
\ga
\label{jina} g_{j,Q-e_j}=-\{  f_j(\psi(\zeta))\}_Q,\quad q_j>0,\\
\hat f _{j,Q}:=\{  f_j(\psi(\zeta))\}_Q, \quad q_j=0.
\label{fjpk}
\end{gather} 

(ii)  Assume that $F$ is convergent.   Define
$  \ov h(\zeta)=\sum|h_Q|\zeta^Q.
$
We obtain for every multi-index  $Q=(q_1,\ldots, q_p)$ and for every $j$ satisfying  $q_j\geq 1$
$$
\ov g_{j,Q-e_j}\leq   \left\{\ov{  f_j}(\zeta_1+\zeta_1\ov g_1(\zeta), \ldots, \zeta_p+\zeta_p\ov g_p(\zeta))\right\}_Q.
$$
  Set $w(\zeta)=\sum\zeta_k\ov g_k(\zeta)$.  We obtain
$$
w(\zeta)\prec \sum \ov{f_j}(\zeta_1+  w(\zeta), \ldots, \zeta_p+w(\zeta)). 
$$
Note that $  f_j(\zeta)=O(|\zeta|^2)$ and $w(0)=0$.
By the Cauchy majorization and the implicit function theorem,  $w$ and hence $g, \psi, \hat f$ are convergent.

  Assume that ${\rho_z} F{\rho_z}=F$. Then $\rho_z LF\rho_z$ is   normalized,  $\rho_z\psi\rho_z$
is   tangent to the identity,  and  the $j$th component of $\rho_z\hat F\rho_z(\zeta)-LF(\zeta)$ is   independent of $\zeta_j$. Thus $\rho_z\psi\rho_z$ normalizes $F$ too.
   By the uniqueness of $\psi$, we obtain $\rho_z\psi\rho_z=\psi$.
  
(iii)   By rewriting  \re{fjpk},  we obtain 
\eq{fjqfjq}
\hat f_{j,Q}=f_{j,Q}+\{f_j(\psi)-f_j\}_Q=f_{j,Q}+\cL F'_{j,Q}(\{f\}_{|Q|-1}, \{g\}_{|Q|-2}).
\eeq 
 From \re{jina}, it follows that
 $$g_{k,Q-e_k}=-f_{k,Q}+\cL G_{k,Q-e_k}(\{f\}_{|Q|-1}, \{g\}_{|Q|-2}), \quad |Q|>1.$$
Note that $ \{g\}_0=0$ and $\{f\}_1=0$. Using the identity repeatedly, we obtain
$g_{k,Q-e_k}=-f_{k,Q}+\cL G_{k,Q-e_k}^*(\{f\}_{|Q|-1}).$ Therefore, we can rewrite \re{fjqfjq} as \re{hfjqf}.
\end{proof}

\subsection{A unique formal normal form of a reversible map $\sigma$}

We now state a normal form for $\{\tau_1,\tau_2,\rho\}$
under a 
 condition on the third-order invariants  of $\sigma$.

\begin{thm}\label{ideal5}  Let $\tau_{1}$, $\tau_{2}$
be a pair of holomorphic involutions with linear parts $\hat T_i$. Let $\sigma=\tau_1\tau_2$.
Assume that the linear part of $\sigma$ is $$
\hat S\colon\xi_j'=\mu_j\xi_j, \quad \eta_j=\mu_j^{-1}\eta_j, \quad 1\leq j\leq p
$$
and $\mu_1,\ldots, \mu_p$ are non-resonant. Let
 $\Psi \in   {\cL C}^\mathsf{c}(\hat S)$ be the unique
 formal mapping such that
\gan
\tau_{i}^*=  \Psi^{-1}\tau_i\Psi\colon\xi_j'=\Lambda_{ij}(\xi\eta)\eta_j, \quad \eta_j'=\Lambda_{ij}(\xi\eta)^{-1}\xi_j;\\
\sigma^*=\Psi^{-1}\sigma\Psi\colon\xi_j'=M_j(\xi\eta)\xi_j, \quad \eta_j'=M_j(\xi\eta)^{-1}\eta_j
\end{gather*}
with $M_j=  \Lambda_{1j}\Lambda_{2j}^{-1}$.
Suppose that $\sigma$ satisfies the 
 condition that   $\log M$
is  tangent to the identity. 
\bppp\item
Then there exists an invertible   formal map $\psi_1\in  {\cL C}(\hat S)$   such that
\ga\label{htai}
\hat\tau_{i}=\psi_1^{-1}\tau_i^*\psi_1\colon\xi_j'
=\hat\Lambda_{ij}(\xi\eta)\eta_j, \quad \eta_j'=\hat\Lambda_{ij}(\xi\eta)^{-1}\xi_j;\\
\hat\sigma=\psi_1^{-1}\sigma^*\psi_1\colon\xi_j'=\hat
M_j(\xi\eta)\xi_j, \quad \eta_j'
=\hat M_j(\xi\eta)^{-1}\eta_j.\label{hsig}
\end{gather}
Here $\hat\Lambda_{2j}=\hat\Lambda_{1j}^{-1}$,   and  $\hat T_i$ is the linear part of $\hat\tau_i$.  Moreover, 
 $\log\hat M$ is tangent to the identity  at the origin.
\item    The centralizer of $\{\hat\tau_1,\hat\tau_2\}$ consists of $2^p$ dilations
  $(\xi,\eta)\to (a\xi,a\eta)$ with $a_j=\pm1$. And $\hat\Lambda_{ij}$ are unique.
   If   $ \Lambda_{ij}$ are
convergent, then $\psi_1$ is convergent too.
\item Suppose that    $\hat\sigma$ is
divergent. If $\sigma$ is formally equivalent
to a mapping $\tilde\sigma\in  {\cL C}(\hat S)$ then $\tilde\sigma$ must be divergent too.
\item  Let $\rho$ be given by \rea{eqrh} and let
 $\tau_2=
\rho\tau_1\rho$. Then the above $\Psi$ and $\psi_1$
  commute with $\rho$.
    Moreover, $\hat\tau_i$, $\hat\sigma$ are unique.  
\eppp
\end{thm}
\begin{proof}  Assertions in (i) and (ii)   are direct consequences of \rp{ideal0} and \rl{fcfp} 
 in which $F$ is the $\tilde M$ in \rp{ideal0}.  The assertion on   the centralizer of $\{\hat\tau_1,\hat\tau_2\}$ is obtained 
from   \re{lamtpsi}   of \rp{ideal0} in which    $\tilde \Lambda_{ij}=\hat \Lambda_{ij}$.
Now (iii) follows from (ii) too. 
Indeed, suppose $\sigma$ is formally equivalent to some convergent
$$
\tilde\sigma\colon\xi_j=\tilde M_j(\xi\eta)\xi_j,\quad\eta'_j=\tilde M_j(\xi\eta)^{-1}\eta_j.
$$
Then  by the assumption on the linear part of $\log
M$,  
 we can apply a dilation to achieve that $(\log \tilde M)'(0)$ is tangent to the identity. 
 By \rl{fcfp}, there exists a unique convergent mapping $\var\colon\zeta_j'
=b_j(\zeta)\zeta_j$ ($1\leq j\leq p$) with $b_j(0)=1$  such that $\log\tilde M\circ\var$ is in the normal form $\log M_*$.
Then 
$$
(\xi'_j,\eta'_j)=(b_j^{1/2}(\xi\eta)\xi_j,b_j^{1/2}(\xi\eta)\eta_j), \quad 1\leq j\leq p
$$
 transforms $\tilde\sigma$ into
a convergent mapping  $\sigma_*$. Since the normal form for $\log M$ is unique, 
then $\hat\sigma=\sigma_*$. In particular, $\hat\sigma$   is convergent.

(iv). Note that $\rho\sigma\rho=\sigma^{-1}$. Also $\rho(\sigma^*)^{-1}\rho$ has the same form as $\sigma^*$. 
By $(\rho\Psi^{-1}\rho)\sigma(\rho\Psi\rho)=(\rho\sigma^*\rho)^{-1}$, we conclude that $\rho\Psi\rho=\Psi$. 
The rest of assertions 
can be verified easily. 
\end{proof}
Under the condition that  $\log M$ is 
tangent to the identity,
the above theorem completely
settles  the formal
classification of $\{\tau_1,\tau_2,\rho\}$.
It also says that {\bf the normal form $\hat\tau_1,\hat\tau_2$ can be achieved by a convergent transformation, if and only if $\sigma^*$ can be achieved by some convergent transformation}, i.e.  the $\Psi$ in the theorem is convergent.

 However, we would like state clear that our results do not rule out the case where
a refined normal form for $\{\tau_1^*,\tau_2^*,\rho\}$ is achieved  by  convergent transformation, while $\Psi$ is divergent, when 
 $\log M$ is tangent to the identity. 

\subsection{An algebraic manifold with linear $\sigma$}

We conclude the section showing that when $\tau_1,\tau_2$ are normalized as   in this section,
$\{\tau_{ij}\}$ might still be very general; 
  in particular $\{\tau_{1j},\rho\}$ cannot always be simultaneously linearized 
even at the formal level. 
This is one of main differences between $p=1$
and $p>1$.

\begin{exmp}\label{texpl} Let $p=2$.
Let $\phi$ be a holomorphic mapping of the form  
$$
\phi\colon\xi'_i=\xi_i+q_i(\xi,\eta), \quad \eta'_i=\eta_i+\la^{-1}_iq_i(T_1(\xi,\eta)),\; i=1,2.
$$
Here $q_i$ is a homogeneous quadratic polynomial map and 
$$T_1(\xi,\eta)=(\la_1\eta_1,\la_2\eta_2,\la^{-1}_1\xi_1,\la^{-1}_2\xi_2).$$
Let $\tau_{1j}=\phi T_{1j}\phi^{-1}$ and $\tau_{2j}=\rho\tau_{1j}\rho$. Then $\phi$ commutes with $T_1$
and $\tau_1=T_1$. In particular $\tau_2=\rho T_1\rho$ and $\sigma=\tau_1\tau_2$ are in linear normal forms.
However,  $\tau_{11}$ is given by
\aln
 \xi_1'&=\la_1\eta_1-q_1(\la\eta,\la^{-1}\xi)+q_1(\la_1\eta_1,\xi_2,\la_1^{-1}\xi_1,\eta_2)+O(3),\\
 \xi_2'&=\xi_2-q_2(\xi,\eta)+q_2(\la_1\eta_1,\xi_2,\la_1^{-1}\xi_1,\eta_2)+O(3),\\
 \eta_1'&=\la_1^{-1}\xi_1-\la_1^{-1}q_1(\xi,\eta)+\la_1^{-1}q_1(\xi_1,\la_2\eta_2,\eta_1,\la_2^{-1}\xi_2)+O(3),\\
 \eta_2'&=\eta_2-\la_2^{-1}q_2(\la\eta,\la^{-1}\xi)+\la_2^{-1}q_2(\xi_1,\la_2\eta_2,\eta_1,\la_2^{-1}\xi_2)+O(3).
 \end{align*}
 Notice that the common zero set $V$ of $\xi_1\eta_1$ and $\xi_2\eta_2$ is invariant under $\tau_1,\tau_2,\sigma$ and $\rho$.
 In fact, they are linear on $V$.
 However, for $(\xi',\eta')=\tau_{11}(\xi,\eta)$, we have
\aln
\xi_1'\eta_1'&=-\eta_1q_1(0,\xi_2,\eta)+\eta_1q_1(0,\la_2\eta_2,\eta_1,\la_2^{-1}\xi_2) -\la_1^{-1}\xi_1q_1(0,\la_2\eta_2,\la^{-1}\xi)\\
&\quad+\la_1^{-1}\xi_1q_1(0,\xi_2,\la_1^{-1}\xi_1,\eta_2)\mod (\xi_1\eta_1,\xi_2\eta_2
, O(4)).
\end{align*}
For a generic $q$, $\tau_{11}$ does not preserve $V$.
\end{exmp}By a simple computation, we can verify that $\sigma_j=\tau_{1j}\tau_{2j}$ for $j=1,2$ do not commute with each other. In fact,
we will prove in section~\ref{abeliancr} that  if the  $\mu_1,\ldots, \mu_p$ are nonresonant, $\sigma_j$ commute pairwise, and $\sigma$ is linear as above, then $\tau_{1j}$ must be linear.

\setcounter{thm}{0}\setcounter{equation}{0}

\section{
Divergence of all 
normal forms of 
 a reversible map $\sigma $}\label{div-sect}

Unlike the Birkhoff normal form for a Hamiltonian system, the Poincar\'{e}-Dulac normal form 
is not unique for a general $\sigma$;   it just belongs to the centralizer of the linear part $S$ of $\sigma$. One can obtain a divergent normal form  easily from any non-linear Poincar\'{e}-Dulac normal form of 
$\sigma=\tau_1\tau_2$ 
by conjugating with a divergent transformation in the centralizer of $S$; see \re{lamtpsi}.
We have seen  how
the small divisors   enter in the computation of the normalizing transformations via \re{ujpq}-\re{vjpq}, but they have not yet
appeared in  \re{mpmp}-\re{npnp}  in the computation of the normal forms. 
To see the effect of small divisors on normal forms,  we  first assume a condition,  to be achieved later, 
 on the third order invariants of $\sigma$ and then we shall need   to modify the
normalization procedure.
We will use two sequences of normalizing mappings to normalize $\sigma$.
The composition of
normalized mappings might not be normalized. Therefore, the new normal form $\tilde \sigma$
 might not be the
$\sigma^*$ in \rp{ideal0}. 
We will show that this $\tilde\sigma$,
  after it is transformed into the normal form $\hat\sigma$
  in  \rt{ideal5} (i),  is divergent. Using the divergence
of  $\hat\sigma$, we will then show that any other normal forms of $\sigma$ that are in the centralizer of $S$ must be divergent too. This last step requires a convergent solution
given by \rl{fcfp}.

Our goal is to see a small divisor in a normal form $\tilde\sigma$; however they appear as a product. This is more complicated 
than the situation for the normalizing transformations, where a small divisor appears in a much simple
way.  In essence, a small divisor problem occurs naturally when one applies a Newton iteration scheme for
a convergence proof.
For a small divisor to show up in the normal form, we have to go beyond the Newton iteration scheme, measured
in the degree or order of approximation in power series. 
Therefore, we first  refine the formulae \re{mpmp}. 
\begin{lemma}\label{mj0mj} Let $\sigma$ be a holomorphic mapping,
given by
$$
\xi_j'=M_j^0(\xi\eta)\xi_j+f_j(\xi,\eta), \quad
\eta_j'=N_j^0(\xi\eta)\eta_j+g_j(\xi,\eta), \quad 1\leq j\leq p.
$$
Here $ M_j^0(0)=\mu_j=N^0_j(0)^{-1}.$
Suppose that $\ord (f,g)\geq d$,  $d\geq4$, and $I+(f,g)\in \cL C^{\mathsf c}(S)$.
Assume that $\mu_1, \ldots, \mu_p$ are non-resonant. Assume that
 \eq{dijm}
 \nonumber
(M^0)'(0)=\diag(\mu_1, \ldots, \mu_p). 
 \eeq
There exist unique polynomials $U,V$
 of degree at most $2d-1$ such that $\Psi=I+(U,V)\in\cL C^{\mathsf c}(S)$  transforms $\sigma$
into
$$
\sigma^*\colon\xi'= M(\xi\eta)\xi+\tilde f(\xi,\eta),\quad
\eta'=N(\xi\eta)\eta+\tilde g(\xi,\eta)
 $$
 with $I+(\tilde f,\tilde g)\in\cL C^{\mathsf c}(S)$ and $\ord(\tilde f,\tilde g)\geq 2d$.
Moreover,  
\al
\label{ujpqcc}
U _{j,PQ}&=(\mu^{P-Q}-\mu_j)^{-1}\left\{f _{j,PQ}+\cL U _{j,PQ}^*(\del_{ \ell-1},
\{M^0,N^0\}_{[\f{\ell-1}{2}]};\{ f,g\}_{\ell-1})\right\},\\
V _{j,QP}&=(\mu^{Q-P}-\mu_j^{-1})^{-1}\left\{g _{j,QP}+\cL V _{j,QP}^*(\del_{\ell-1},
\{M^0,N^0\}_{[\f{\ell-1}{2}]};\{ f,g\}_{\ell-1})\right\},
\label{ujpqcc6}
\end{align}
for  $2\leq |P|+|Q|=\ell\leq 2d-1$ and $\mu^{P-Q}\neq\mu_j$.  In particular,  $\ord(U,V)\geq d$.   For $|P|=d$ and $|P'|<d$,
\ga
M_{j,P'}=M_{j,P'}^0, 
 \label{mjpp0}  \\
\label{mjpp--}
 M_{j,P}=M_{j,P}^0+ \mu_j\left\{2(U_j V_j)_{PP}+(U_j^2)_{(P+e_j)(P-e_j)}\right\}+\{Df_j(\xi,\eta)(U,V)\}_{(P+e_j)P}.
\end{gather} 
\end{lemma}
\begin{rem}\label{keyrem}  Formula \re{mjpp--} gives us an effective way to compute the Poincar\'e-Dulac normal form. 
It tells us that under the above conditions, 
the coefficients of  $M_{jP}(\xi\eta)\xi_j$ of degree $2|P|+1$  do not depend on coefficients of $f(\xi,\eta),g(\xi,\eta)$ of degree $\geq 2|P|$, if $2|P|>3$.\end{rem}
\begin{proof}  Identities \re{ujpqcc}-\re{mjpp0} follow directly from \re{ujpq}-\re{mpmpnsd},
where  by notation 
in Definition~\ref{notation}
 $$\cL U_{j,PQ}^*(\cdot; 0)=\cL V^*_{j,QP}(\cdot;0)=\cL M_P^*(\cdot; 0)=0.$$
 
 Let $D_i$ denote $\partial_{\zeta_i}$. Let $Du(\xi,\eta)$ and $Dv(\zeta)$ denote the gradients of two functions.
The right-hand sides of \re{xjmj} and \re{mj0x} give us
\al\label{mjxexc}
M_j(\xi\eta)\xi_j&+U_j(M(\xi\eta)\xi,N(\xi\eta)\eta)=f_j(\xi,\eta)+ Df_j(\xi,\eta)(U, V)+A_j(\xi,\eta)\\
\nonumber&\quad +(M^0_j(\xi\eta)+DM_j^0(\xi\eta)(\eta U +\xi V  +UV))(\xi_j+U_j). 
\end{align}
  We recall  from \re{ajpq} the   remainders
\aln 
A_j(\xi,\eta)&=
 R_2M_j^0(\xi\eta;\xi U+\eta V+UV)(\xi_j+U_j) +R_2f_j(\xi,\eta;U,V). 
 \end{align*}
Here by \re{taylor}, we have the Taylor remainder formula
\eq{}
\nonumber
 R_{2}f(x;y)=2\int_0^1(1-t)\sum_{|\alpha|=2}\frac{1}{\alpha!}\partial^{\alpha}f(x+ty)y^{\alpha}\, dt.
\eeq
Since $\ord(U,V)\geq d$, $\ord(f,g)\geq d$, and  $d\geq4$, then $A_j$, 
defined by \re{ajpq},
 satisfies
\gan
A_j(\xi,\eta)=O(|(\xi,\eta)|^{2d+2}). 
\end{gather*}
Recall that $f_j(\xi,\eta)$ and $ U_j(\xi,\eta)$ do not contain terms of the form $\xi_j\xi^P\eta^P$, while
 $g_j(\xi,\eta)$ and $ V_j(\xi,\eta)$ do not contain terms of the form $\eta_j\xi^P\eta^P$.
Assume that $i\neq j$.  Then $D_iM_j^0(\xi\eta)=O(|\xi\eta|)$. We  see  that
$D_iM_j^0(\xi\eta)\eta_iU_i(\xi,\eta)$ and $ D_iM_j^0(\xi\eta)\xi_iV_i(\xi,\eta)$ do not
contain terms of $\xi^P\eta^P$, and 
$$D_iM_j^0(\xi\eta)\xi_iU_i(\xi,\eta)V_i(\xi,\eta)=O(2d+3).$$
Comparing both sides of \re{mjxexc} for coefficients of $\xi_j\xi^P\eta^P$, we get  \re{mjpp--}.
\end{proof}

Set  $|\delta_N(\mu)|\colonequals\max\left\{|\nu|\colon\nu\in\delta_N(\mu)\right\}$ for
\eq{delnmu}
\delta_N(\mu)=\bigcup_{j=1}^p\left\{\mu_i, \mu_i^{-1}, \f{1}{\mu^{P}-\mu_j}\colon P\in\zz^p, P\neq e_j,|P|\leq N\right\}.
\eeq
 \begin{defn}\label{sdon}We 
say that $\mu^{P-Q}-\mu_j$ and
 $\mu^{Q-P}-\mu_j^{-1}$ are {small divisors} of   {\it height} $N$, if there exists a partition
\gan
\bigcup_j\Bigl\{|\mu^{P-Q}-\mu_j|\colon P,Q\in\nn^p, |P|+|Q|\leq N, \mu^{P-Q}\neq\mu_j\Bigr\}=S_N^0\cup S^1_N
\end{gather*}
with $
 |\mu^{P-Q}-\mu_j|\in S_N^0$ and $
S^1_N\neq\emptyset$ such that
\gan
 \max S_N^0<C\min S_N^0,\\
\max S_N^0<  
(\min S^1_N)^{L_N}<1.
\end{gather*}
Here $C$ depends only on an upper bound of  $|\mu|$ and $|\mu|^{-1}$ and
$$
L_N\geq N.
$$
  If $|\mu^{P-Q}-\mu_j|$ is in $S_N^0$ and if $P,Q\in\nn^p$,  we call $  |P-Q|$ the   {\it degree} of the small
divisors $\mu^{P-Q}-\mu_j$ and $\mu^{Q-P}-\mu_j^{-1}$.  
\end{defn}

  To avoid confusion, let
us call $\mu^{P-Q}-\mu_j$ that appear in $S_N^0$ the  {\it exceptional} small divisors. These small divisors have played important roles in Siegel's works \cite{siegel-divergent, siegel-integral}. Siegel's small divisors technic was extended to a construction of divergent Birkhoff normal form in \cite{Go12} (see also \cite{PM03} for related problems). 
The degree and height play different roles in computation. The height serves as the maximum degree of all small  divisors that 
need to be considered in computation.  

Roughly speaking, the quantities in $S^0_N$ are comparable but they are much smaller than the ones in $S^1_N$.    
We will construct $\mu$ for any prescribed
sequence of positive integers $L_N$ so that 
$$
\max S_N^0<(\min S^1_N)^{L_N}<1
$$
for a subsequence $N=N_k$ tending to $\infty$. Furthermore, to use the small divisors we will  
 identify  all   exceptional small divisors of   height  $2N_k+1$ and  all    degrees of the exceptional small divisors
with $N_k$ being the smallest. 

We start with the following lemma which gives us small divisors that decay as rapidly as we wish.
\begin{lemma}\label{smallvec}
Let  $L_k$ be an increasing
 sequence of positive integers   such that $L_k$ tends to $\infty$
as $k\to\infty$. There exist a real number $\nu\in(0,1/2)$
  and a sequence $(p_k,q_k)\in \nn^2$  such that
$e,1,\nu$ are linearly independent over $\qq$, and
\begin{gather}\label{epq-}
|q_k\nu-p_k-e|\leq\Delta(p_k,q_k)^{(p_k+q_k)L_{k}},  \\
\label{dpkqk}\Delta(p_k,q_k)=\min\Bigl\{\frac{1}{2}, |q\nu-p-re|\colon  0<|r|+|q|<3(q_k+1), 
\\
\nonumber  (p,q,r)\neq  0, \pm(p_k,q_k,1), \pm2(p_k,q_k,1)\Bigr\}.
\end{gather}
\end{lemma}
\begin{proof} We consider two increasing sequences  $  \{m_k\}_{k=1}^\infty, \{n_k\}_{k=1}^\infty$ of positive integers, which are
to be chosen. For $k=1,2,\ldots$, we set
\gan
\nu=\nu_k+\nu_k',\quad \nu_k=\sum_{\ell=1}^{ k}\frac{1}{m_\ell!}\sum_{j=0}^{  n_\ell}\f{1}{j!}, \quad
\nu_k'=\sum_{\ell>k}\frac{1}{m_\ell!}\sum_{j=0}^{  n_\ell}\f{1}{j!},\\
q_k=m_k!. \end{gather*}
We choose $m_k>(m_\ell)!(  n_\ell!)$ for $\ell<k$ and decompose
\begin{gather*}
q_k\nu=p_k+e_k+e_k',\\
p_k=m_k!\nu_{k-1}\in\nn,\quad
e_k=\sum_{\ell=0}^{  n_k}\f{1}{k!},\quad e_k'=m_k!\nu_k'.
\end{gather*}
We have   $e_k'<m_k!\sum_{\ell>k}\f{e}{m_\ell!}$ and 
\ga
\nonumber
q_k\nu=p_k+e+e_k'-\sum_{\ell=  n_k+1}^\infty\frac{1}{\ell!},\\
\label{ekpl}
|q_k\nu-p_k-e|\leq m_k!\nu_k'+\sum_{\ell=n_k+1}^\infty\f{1}{\ell!}<\left\{12(3(q_k+1)^3)!\right\}^{-(p_k+q_k)  L_k}.
\end{gather}
Here $\re{ekpl}_k$ is achieved by  choosing   $(m_2,n_1)$, \ldots, $(m_{k+1},n_{k})$ successively.
 Clearly we can get $0<\nu<1/2$ if $m_1$ is sufficiently large. 

Next, we want to show that $re+p+q\nu\neq0$ for all integers $p, q,r$ with $(p,q,r)\neq(0,0,0)$. Otherwise, we  rewrite  $-m_k!p=
m_k!(q\nu+re)$ as
$$
-m_k!p=qp_k+r\sum_{j=  0}^{m_k}\f{m_k!}{j!}+qe+q\left(e_k'-\sum_{\ell=  n_k+1}^\infty\f{1}{\ell!}\right)+r\sum_{j>m_k}\frac{m_k!}{j!}.
$$
The left-hand side is an integer. On the right-hand side,     the first two terms are integers, $qe$ is a fixed
irrational number, and the rest terms tend to $0$
as $k\to\infty$. We get a contradiction.

To verify \re{epq-}, we need to show that for each tuple $(p,q,r)$ satisfying \re{dpkqk},
\eq{qn-p}
|q\nu-p-re|\geq |q_k\nu-p_k-e|^{\f{1}{(p_k+q_k)L_k}}.
\eeq
We first note the following elementary inequality
\eq{elemineq}
|p+qe|\geq \frac{1}{(q-1)!}\min\left\{3-e, \frac{1}{q+1}\right\}, \quad p, q\in \zz, \quad q\geq1.
\eeq
Indeed,  the inequality holds for $q=1$. For $q\geq2$ we have $q!e=m+\e$ with $m\in \nn$ and 
$$
\e:=\sum_{k=q+1}^\infty\frac{q!}{k!}>\frac{1}{q+1}.
$$
Furthermore,  $1-\e>1-\frac{2}{q+1}=\frac{q-1}{q+1}$ as
$$
\e<\frac{1}{q+1}+\sum_{k\geq q+2}\frac{1}{k(k+1)}=\frac{2}{q+1}.
$$
We may assume
that $q\geq0$.  
If $q=0$, then $|r|<3q_k+3$ and hence
$|p+re|\geq \frac{1}{(3q_k+4)!}$. Now \re{qn-p} follows from \re{ekpl}.
  Assume that $q>0$.
We have
\al \label{qnpq}
|-q\nu+p+re|&\geq|-q\f{p_k+e}{q_k}+p+re|-q
\f{|e+p_k-q_k\nu|}{q_k}\\
&=\left|\f{q_kp-qp_k}{q_k}+\f{rq_k-q}{q_k}e\right|-q
\f{|e+p_k-q_k\nu|}{q_k}.
\nonumber\end{align}
We first verify that $q_kp-qp_k$ and $q-rq_k$ do not vanish simultaneously. Assume that both are zero. Then $(p,q,r)=r(p_k,q_k,1)$. Thus $|r|\neq 1,2$,  and   $|r|\geq3$ by conditions in \re{dpkqk}; 
we obtain  $|r|+|q|\geq3(|q_k|+1)$, a contradiction.  Therefore,
 either $q_kp-qp_k$ or $rq_k-q$
is not zero. By \re{elemineq} and \re{qnpq},  
\begin{align*}
&|-q\nu+p+re|\geq  \f{1}{q_k}\cdot \frac{1}{3}\cdot \f{1}{(|rq_k-q|+1)!}-q
\f{|e+p_k-q_k\nu|}{q_k}\\
&\qquad\geq  \f{1}{(3q_k+4)^2!}-4|e+p_k-q_k\nu|.
\end{align*}
Using \re{ekpl} twice,  we obtain the next two inequalities:
$$|-q\nu+p+re|\geq\yt
\left\{(3q_k+4)^2!\right\}^{-1}\geq|p_k+e-q_k\nu|^{\f{1}{(p_k+q_k)L_k}}.$$  The two ends give us \re{qn-p}.
\end{proof}

We now reformulate the above lemma as follows.
\begin{lemma}\label{smallvec+} Let $L_k$ be an increasing sequence of positive integers
such that $L_k$ tends to $+\infty$ as $k\to\infty$. Let  $\nu\in(0,1/2)$, and let $p_k$ and $ q_k$ be
positive integers as in \rla{smallvec}.  
Set 
$(\mu_1,\mu_2,\mu_3):=(  e^{-1},e^\nu, e^{e})$.  Then
\begin{align}\label{epqcc}
|\mu^{P_k}-\mu_3|&\leq (C\Delta^*(P_k))^{|P_k|L_{k}},\quad  P_k= (p_k,q_k,0),\\
\label{delspk} \Delta^*(P_k)&=\min_j\Bigl\{|\mu^{R}-\mu_j|\colon R\in\zz^3,  |R|\leq 2(  q_k+p_k)+1,  
\\
&\quad  R-e_j\neq 0,  \pm(  p_k,q_k,-1), \pm2(p_k,q_k,-1)\Bigr\}.
\nonumber\end{align}
Here $C$ does not depend on $k$.  Moreover, 
  all   exceptional small divisors   of height $2|P_k|+1$ have   degree at least $|P_k|$. Moreover, 
   $\mu^{P_k}-\mu_3$ is the only   exceptional small divisor
of  degree $|P_k|$ and
 height $2|P_k|+1$.
\end{lemma}
In the definition of $\Delta^*(P_k)$, equivalently we require that
$$ 
R \neq  P_k,\     R_k^1, \ R_k^2,\  R_k^3
$$ 
with  $R_k^1:=- P_k+2e_3,  R_k^2:=2 P_k-e_3,$ and $ R_k^3:=-2P_k+3e_3$.
Note that $ |R_k^1|=|P_k|+2, |R_k^2|=2|P_k|+1$, and $|R_k^3|=2|P_k|+3$  are bigger 
than  $|P_k|$, i.e.  the degree of the exceptional small divisor $\mu^{P_k}-\mu_3$.
 Each $\mu^{R_k^i}-\mu_3$
is a small divisor comparable with $ \mu^{P_k}-\mu_3$. Finally,  $\Del^*(P_k)$ tends to zero as $|P_k|\to\infty$.
Let us set 
$N:=2|P_k|+1$, 
 and  
\aln
S_N^0:&=\left\{  |\mu^{P_k}-\mu_3|,|\mu^{R_k^1}-\mu_3|,|\mu^{R_k^2}-\mu_3|,|\mu^{R_k^3}-\mu_3|\right\},\\
S_N^1: &= \bigcup_j\Bigl \{ |\mu^{R}-\mu_j|\colon R\in\zz^3,  |R|\leq 2(q_k+p_k)+1,  
\\
\nonumber&\qquad\quad R-e_j\neq 0,\   \pm(p_k,q_k,-1),\  \pm2(p_k,q_k,-1)\Bigr\}.\end{align*}
This
implies that the last paragraph of Lemma~\ref{smallvec+} holds when 
the $L_N$ in Definition~\ref{sdon}, denoted it by $L_N'$, takes the value
$L_N'=\frac{1}{2}|P_k|L_k$ and $k$ is sufficiently large, while  $L_k$ is given in  Lemma~\ref{smallvec+}. 
\begin{proof}
By \rl{smallvec}, we find a real number $\nu\in(0,1/2)$ and positive integers $p_k,q_k$ such that $e,1,\nu$ are linearly independent over $\qq$ and \begin{gather}\label{epq}
|q_k\nu-e-p_k|\leq\Delta(p_k,q_k)^{|P_k|L_{k}},\\
\Delta(p_k,q_k)=\min\left\{|q\nu-re-p|\colon  0<|r|+|q|<3(q_k+1),\right.\nonumber
\\
 \qquad \qquad\qquad\left. (p,q,r)\neq  0, \pm(  p_k,q_k, 1), \pm2(p_k,q_k,1)\right\}.
\nonumber\end{gather}
 Note that $\mu_1,\mu_2,\mu_3$
 are   positive 
  and non-resonant.  We have
  $$
 |\mu^{P_k}-\mu_3|=|\mu_3|\cdot|e^{q_k\nu-p_k-e}-1|. 
 $$
Let $\nu^*:=(  -1,\nu,e)$. If $|R\cdot\nu^*-\nu_j^*|<2$, then by the intermediate value theorem
$$
  e^{-2}|\mu_j||R\cdot\nu^*-\nu_j^*|\leq |\mu^R-\mu_j|\leq  e^2|\mu_j||R\cdot\nu^*-\nu_j^*|.
$$
If $R\cdot\nu^*-\nu_j^*>2$ or $R\cdot\nu^*-\nu_j^*<-2$, we have
$$
|\mu^R-\mu_j|\geq   e^{-2}|\mu_j|.
$$
Thus,  we can restate the properties of $\nu^*$ as   follows:
\begin{align*} 
\nonumber
&|\mu^{-(p_k,q_k,0)}-\mu_3|\leq  C' (C'\tilde\Delta(p_k,q_k))^{|P_k|L_{k}},\\
& \tilde\Delta(p_k,q_k)=\min\left\{ |\mu^{(p,q,r)}-1|\colon   0<|r|+|q|<3(q_k+1),\right.
\\
&\qquad\qquad\qquad  \left. (p,q,r)\neq 0,\pm(  p_k,q_k,-1), 
\pm2(p_k,q_k,-1)\right\}.
\end{align*}
Recall that $0<\nu<1/2$. By \re{epq}, we have $|q_k \nu - e - p_k|<1$. Since
$p_k, q_k$ are positive,  then $p_k<\nu q_k<q_k/2$.
Assume that $|  \mu^{R}-\mu_j|=\Del^*(P_k)$,  $|R|\leq 2(p_k+q_k)+1$, and
$$R-e_j\neq 0,\pm(p_k,q_k,-1), \pm 2(p_k,q_k,-1).$$
Set $R':=R-e_j$ and $(p,q,r):=R'$. Then $\Del^*(P_k)=|\mu_j||\mu^{R'}-1|$.
Also, $|r|+|q|\leq|R'|\leq |R|+1\leq 
2(p_k+q_k)+2\leq q_k+2q_k+2<3(q_k+1)$.  This shows that $|\mu^{Q'}-1|\geq\tilde\Del(p_k,q_k)$.
We obtain $\Del^*(P_k) \geq\mu_j\Del(p_,q_k)$.  We have verified \re{epqcc}. 
  For the remaining assertions, see the remark following the lemma.
\end{proof}

  In the above we have retained $\mu_j>0$ which
are sufficient to realize $\mu_1, \mu_2, \mu_3$, $\mu_1^{-1},\mu_2^{-1},\mu_3^{-1}$ as eigenvalues of $\sigma$ for an elliptic
complex tangent.  Indeed, with $0<\mu_1<1$, interchanging $\xi_1$ and $\eta_1$ preserves $\rho$ and changes the $(\xi_1,\eta_1)$
components of $\sigma$ into   $(\mu_1^{-1}\xi_1,\mu_1\eta_1)$. 

We are ready to prove   \rt{divsig}, which is restated here:
\begin{thm}\label{divnf1} 
There exists a non-resonant elliptic  real analytic $3$-submanifold $M$ in $\cc^6$ such
that $M$
  admits the maximum number of deck transformations 
 and all Poincar\'e-Dulac normal forms of  the $\sigma$ associated to $M$ are divergent.  
 \end{thm}
\begin{proof} 
We will not construct  the real analytic submanifold $M$ directly. Instead, we will
construct a family  of involutions $\{\tau_{11}, \ldots, \tau_{1p},\rho\}$ so that all Poincar\'e-Dulac
normal forms of $\sigma$ are divergent. By the realization
 in \rp{inmae}, we get the desired submanifold.

We first give an outline of the proof. 
To prove the theorem, we first deal with the associated  $\sigma$ and its normal form $\tilde \sigma$,  which belongs to  
the centralizer of $S$,   the linear part of $\sigma$ at the origin.
Thus $\sigma^*
$ has the form
$$
\sigma^*\colon \xi'= M(\xi\eta)\xi, \quad \eta'= N(\xi\eta)\eta.
$$
We assume that $\log M$ is tangent to identity at the origin.
We then normalize $
\sigma^*$ into the normal form
$\hat \sigma$ stated in \rt{ideal5}  (i). (In \rl{mj0mj} we take $F=\log M$
and $\hat F=\log\hat M$.)
We will show that $\hat\sigma$ is divergent if $\sigma$ is well chosen. By \rt{ideal5} (iii),   all normal forms
 of $\sigma$ in the centralizer of  $S$  
  are divergent. 
To get $\sigma^*$, we use the normalization  of \rp{ideal0} (i). To get $\hat\sigma$, 
we  normalize further  using 
\rl{fcfp}.  To find a divergent $\hat\sigma$,  we need  to tie the normalizations of
two formal normal forms together, by keeping
track of the small divisors  in the two normalizations.

We will start with our initial pair of involutions $\{\tau_1^0,\tau_2^0\}$
 satisfying $\tau_2^0=\rho\tau_1^0\rho$ such that $\sigma^0$
 is  a  third order  perturbation of $S$. We
 require that $\tau_1^0$ be the composition of $\tau_{11}^0, \ldots, \tau_{1p}^0$. The latter
can be realized by a real analytic submanifold by using \rp{inmae}. We will then perform a sequence
of holomorphic  changes of coordinates $\var_k$ such that $\tau_1^{k}=\var_k^{-1}\tau_1^{k-1}\var_k$, $\tau_2^{k}=\rho\tau_1^{k}\rho$, and  $\sigma^{k}=\tau_1^{k}\tau_2^{k}$.
Each $\var_k$  is tangent to the identity
to order $d_k$.
For a suitable choice of $\var_k$, we want to show that
the coefficients of order $d_k$ of the normal form of $\sigma^k$
increase rapidly to the effect that the coefficients of the normal form of the limit
mapping $\sigma^\infty$  increase rapidly too.
Here we will use the  exceptional small divisors to achieve the rapid growth of the coefficients of
the normal forms. 

We now present the proof. Let $\sigma^0=\tau_1^0\tau_2^0$, $\tau_2^0=\rho\tau_1^0\rho$, and
\begin{alignat*}{3}
&\tau_1^0\colon\xi_j'=\Lambda_{1j}^0(\xi\eta)\eta_j, \quad &&\eta_j'=(\Lambda_{1j}^0(\xi\eta))^{-1}\xi_j,\\
&\sigma^0\colon\xi_j'=(\Lambda_{1j}^0(\xi\eta))^{2}\xi_j, \quad &&\eta_j'=(\Lambda_{1j}^0(\xi\eta))^{-2}(\xi\eta)\eta_j.
\end{alignat*}
Since we consider the elliptic case, we require that $(\Lambda_{1j}^0(\xi\eta))^2=\mu_je^{\xi_j\eta_j}$.
 So $\zeta\to (\Lambda_{1}^0)^{2}(\zeta)$ is  biholomorphic. Recall that $\sigma^0$ can be realized by $\{\tau_{11}^0, \ldots, \tau_{1p}^0,\rho\}$. 
We will take
\ga\label{rest1}
\var_k\colon\xi_j'=(\xi+h^{(k)}(\xi),\eta),\quad  \ord h^{(k)}=d_k>3,\\
\label{rest2} d_{k}\geq   2d_{k-1},   \quad |h^{(k)}_P|\leq 1.
\end{gather}
 We will also choose each $h_j^{(k)}(\xi)$ to have one monomial only.  Let $\Del_r:=\Del_r^3$ denote the polydisc of radius $r$. Let $\|\cdot\|$
be the sup norm on $\cc^3$. 
  Let $H^{(k)}(\xi)=\xi+h^{(k)}(\xi)$ and we first verify that $H_k=H^{(k)}\circ \cdots\circ H^{(1)}$ converges to a holomorphic function
 on the polydisc $\Delta_{r_1}$ for $r_1>0$ sufficiently small; consequently, $\var_k\circ\cdots\circ\var_1$ converges to a germ of  holomorphic map at the
 origin. Note that $H^{(k)}$ sends   $\Del_{r_k}$ into $\Del_{r_{k+1}}$ for $r_{k+1}=r_k+r_k^{d_k}$. We want to show that when $r_1$ 
 is sufficiently small,
 \eq{rksk}
 r_k\leq s_k:=(2-\f{1}{k})r_1.
 \eeq
 It holds for $k=1$. 
 Let us show that  ${r_{k+1}}/{r_k}-1\leq \theta_k:= s_{k+1}/s_k-1$, i.e.
 $$
 r_k^{d_k-1}\leq \theta_k=\f{1}{(k+1)(2k-1)}. 
 $$
 We have $(2r_1)^{d_k-1}\leq (2r_1)^k$ when $0<r_1<1/2$.  Fix $r_1$  sufficiently small such that  $(2r_1)^k<\f{1}{(k+1)(2k-1)}$
 for all $k$.  By induction, we obtain \re{rksk} for all $k$. In particular, we have $\|h^{(k)}(\xi)\|\leq \|\xi\|+\|H^{(k)}(\xi)\|\leq
 2r_{k+1}$ for $\|\xi\|<r_k$.   To show the convergence of $H_k$, we write
 $
H_{k}- H_{k-1}(\xi)=h^{(k)}\circ   H_{k-1}.$  By the Schwarz lemma,  we obtain $$
\| h^{(k)}\circ   H_{k-1}(\xi)\|\leq \f{ 2r_{k+1}}{r_1^{d_k}}\|\xi\|^{d_k}, \quad \|\xi\|<r_1.
 $$
Therefore, $  H_k$ converges to a holomorphic function on $\|\xi\|<r_1$. 
   
  Throughout the proof, we make initial assumptions that
  $d_k$ and $h^{(k)}$ satisfy \re{rest1}-\re{rest2},   $e^{-1}\leq \mu_j\leq e^e$, and $\mu^Q\neq1$ for $Q\in\zz^3$ with $Q\neq0$. 
Set $\sigma^k=\tau_1^k\tau_2^k$, $\tau_2^k=\rho\tau_1^k\rho$, and
$$
\tau_1^k=\var_k^{-1}\tau_1^{k-1}\var_k.
$$
We want $\sigma^k$ not to be holomorphically equivalent to $\sigma^{(k-1)}$. Thus we have chosen  
a $\varphi_k$ that does not commute with $\rho$ in general.  Note that $\sigma^{k}$ is still generated by a real analytic submanifold; indeed, when $\tau_{i}^{k-1}
=\tau_{i1}^{(k-1)}\cdots\tau_{ip}^{(k-1)}$ and $\tau_{2j}^{k-1}=\rho\tau_{1j}^{k-1}\rho$, we
still have the same identities if the superscript $(k-1)$ is replaced by $k$ 
and $\tau_{1j}^{(k)}$ equals $\varphi_k^{-1}\tau_{1j}^{(k-1)}\varphi_k$.
It is clear that $\sigma^k=\sigma^{k-1}+O(d_k)$.  As power series, we have
\eq{sinf}
\nonumber
\sigma^{\ell}=\sigma^{k-1}+O(d_k), \quad  k \leq \ell\leq \infty.
\eeq

We know that $\sigma^\infty$ does not have a unique normal form in the centralizer $  S$.
Therefore, we will choose a procedure that arrives at a unique formal normal form in $  S$. We show that this unique normal form is divergent; and hence
by \rt{ideal5} (iii) any normal form of  $\sigma$ that is  in the centralizer of $  S$ must diverge.

We now describe the procedure. For a formal mapping $F$, we have a unique decomposition
$$
F=NF+N^{\mathsf c}F, \quad NF\in \cL C(  S), \quad N^{\mathsf c}F\in\cL C^{\mathsf c}(  S).
$$
Set $\hat\sigma_{0}^\infty=\sigma^\infty$.   For $k=0,1,\ldots,$  we take a normalized polynomial map $\Phi_{k}\in \cL C_2^{\mathsf c}(  S)$  
of degree less than $d_k$  such that $\sigma_{k}^\infty:=\Phi_k^{-1}\hat\sigma_{k}^\infty\Phi_{k}$ is normalized
up to degree $d_k-1$.   Specifically, we require that
\gan
\deg\Phi_k\leq d_k-1, \quad\Phi_k\in \cL C^{\mathsf c}(  S);\quad \quad N^c\sigma_k^\infty(\xi,\eta)=O({d_k}).
\end{gather*}
Take a normalized polynomial map $\Psi_{k+1}$   such that $\Psi_{k+1}$ and
 $\hat\sigma_{k+1}^\infty:=\Psi_{k+1}^{-1}\sigma_k^\infty\Psi_{k+1}$ satisfy
 \ga\label{ncsigm}
 \nonumber
\deg\Psi_{k+1}\leq 2d_k-1; \quad\Psi_{k+1}\in \cL C_2^{\mathsf c}(  S),\quad
N^c\hat\sigma_{k+1}^\infty=O(  2d_k).
\end{gather}
We can repeat this for $k=0,1,\ldots$. Thus we apply two sequences of normalization as follows
 \eq{} \nonumber
   \hat\sigma_{k+1}^\infty=\Psi_{k+1}^{-1}\circ\Phi_k^{-1}\cdots\Psi_1^{-1}
\circ \Phi_0^{-1}\circ \sigma^\infty\circ \Phi_0\circ\Psi_1\cdots\Phi_{k}\circ\Psi_{k+1}.
 \eeq
 We will show that 
 $\Psi_{k+1}=I+O(d_k)$ and $\Phi_k=I+O(  2d_{k-1})$. This shows that  the 
  sequence $\Phi_0\Psi_1\cdots\Phi_k\Psi_{k+1}$
  defines  a formal biholomorphic mapping $\Phi$ so that 
 \eq{hsiginf}
  \hat\sigma^\infty:=\Phi^{-1}\sigma^\infty\Phi 
 \eeq
 is in a normal form. Finally, we need to combine the above normalization with 
the normalization for the unique normal form   
 in \rl{fcfp}.  We will show that the unique normal form    diverges. 
 
Let us recall previous results to show that $\Phi_k,\Psi_{k+1}$ are uniquely determined. 
Set \ga\label{hsink}
\hat\sigma^{\infty}_k\colon\left\{
\begin{split} \xi'=\hat M^{(k)}(\xi\eta)\xi+\hat f^{(k)}(\xi,\eta), \\
\eta'=\hat N^{(k)}(\xi\eta)\eta+\hat g^{(k)}(\xi,\eta),
\end{split}\right.
\\
 (\hat f^{(k)},\hat g^{(k)})\in \cL C_2^{\mathsf c}(  S).
 \end{gather}
 Recall that $\hat\sigma_0
=\sigma^\infty$.  Assume that we have achieved
 \eq{2dk-1}
 (\hat f^{(k)},\hat g^{(k)})=O(2d_{k-1}).
  \eeq
  Here we take $d_{-1}=2$ so that \re{hsink}-\re{2dk-1} hold for $k=0$.
By \rp{ideal0}, there is  a unique normalized  polynomial  mapping $\tilde\Phi_k$ that
transforms $\hat\sigma_k^\infty$ into a normal form.  We denote by $\Phi_k$ 
the truncated polynomial mapping of $\tilde\Phi_k$ of degree $d_k-1$. We write    
\gan
\Phi_{k}\colon\xi'=\xi+U^{(k)}(\xi,\eta), \quad \eta'=\eta+V^{(k)}(\xi,\eta),\\
(U^{(k)},V^{(k)})=O(2),
\quad \deg (U^{(k)}, V^{(k)})\leq d_k-1.\end{gather*}
  By Corollary~\ref{finitever},  $\Phi_k$ satisfies
\ga
\nonumber
\sigma^\infty_k=\Phi_{k}^{-1}\hat\sigma^\infty_{k}\Phi_{k}\colon\left\{
\begin{split}\xi'=M^{(k)}(\xi\eta)\xi+f^{(k)}(\xi,\eta), \\
\eta'=N^{(k)}(\xi\eta)\eta+g^{(k)}(\xi,\eta),
\end{split}\right.
\\
 (f^{(k)},g^{(k)})\in \cL C_2^{\mathsf c}(  S), \quad \ord(f^{(k)},g^{(k)})\geq d_k.\label{fkgkd}
 \end{gather}
  In fact, by \re{ujpq}-\re{vjpq} (or \re{mp-q}-\re{mp-q2}), we have 
 \al \label{ujpq6}
U _{j,PQ}^{(k)}&=(\mu^{P-Q}-\mu_j)^{-1}\left\{\hat f ^{(k)}_{j,PQ}+\cL U _{j,PQ}(\del_{d-1},
\{\hat M^{(k)},\hat N^{(k)}\}_{[\f{d-1}{2}]};\{ \hat f^{(k)},\hat g^{(k)}\}_{d-1})\right\},\\
\label{vjpq6}
V^{(k)} _{j,QP}&=(\mu^{Q-P}-\mu_j^{-1})^{-1}\left\{\hat g^{(k)}_{j,QP}+\cL V _{j,QP}(\del_{d-1},
\{\hat M^{(k)},\hat N^{(k)}\}_{[\f{d-1}{2}]};\{ \hat f^{(k)},\hat g^{(k)}\}_{d-1})\right\},
\end{align}
for  $|P|+|Q|=d< d_k$ and $\mu^{P-Q}\neq\mu_j$. 
 By \re{mpmpnsd}-\re{npnpnsd} (or \re{mpmp}-\re{npnp}), we have
\al
\label{mpmpnsd6}
M_P^{(k)}&=\hat M^{(k)}_P+\cL M_P(\del_{2|P|-1},\{\hat M^{(k)},\hat N^{(k)}\}_{|P|-1}; \{\hat f^{(k)}, \hat g^{(k)}\}_{2|P|-1}), \\ 
N_P^{(k)}&=\hat N^{(k)}_P+\cL N_P(
\del_{2|P|-1},\{\hat M^{(k)},\hat N^{(k)}\}_{|P|-1}; \{\hat f^{(k)}, \hat g^{(k)}\}_{2|P|-1})  \label{npnpnsd6}
\end{align}
  for $2|P|-1< d_k$. 
Recall that  ${\mathcal U}_{j,PQ}, {\mathcal V}_{j,QP}, {\mathcal M}_{j,P}$, 
and ${\mathcal N}_{j,P}$  are universal polynomials in their variables. In  notation defined by Definition~\ref{notation}, 
\eq{uvmn=0}
\nonumber
\cL U_{j,PQ}(\cdot; 0)= \cL V_{j,QP}(\cdot; 0)=0, \quad \cL M_{P}(\cdot; 0)= 
\cL N_{P}(\cdot; 0)=0.
\eeq
Since $d_k\geq   2d_{k-1}$, we apply \re{ujpq6}-\re{vjpq6}   for $d<2d_{k-1}\leq d_k$
and \re{mpmpnsd6}-\re{npnpnsd6}  for $2|P|-1< 2d_{k-1}\leq d_k$ to obtain
\ga
\Phi_k-I=(U^{(k)}, V^{(k)})=O(2d_{k-1}),
 \label{phiki} \\ \label{mphmp}
 M_P^{(k)}=\hat M_P^{(k)}, \quad N_P^{(k)}=\hat N_P^{(k)}, \quad |P|\leq d_{k-1}.
 \end{gather}
  
By \rl{mj0mj},  there is a unique normalized polynomial mapping
 \gan
 \Psi_{k+1}(\xi,\eta)=(\xi+\hat U^{(k+1)}(\xi,\eta), \eta+\hat V^{(k+1)}(\xi,\eta)),\\
 (\hat U^{(k+1)}, \hat V^{(k+1)})\in \cL C_2^{\mathsf c}(  S),\\
 (\hat  U^{(k+1)}, \hat V^{(k+1)})=O(2),\quad
  \deg (\hat U^{(k+1)},\hat V^{(k+1)})\leq 2d_k-1 \end{gather*}
such that
$\hat\sigma^\infty_{k+1}=
\Psi_{k+1}^{-1}\Phi_{k}^{-1}\sigma^\infty_{k}\Phi_{k}\Psi_{k+1}$ satisfies  the following:
\ga
\nonumber
\hat\sigma^\infty_{k+1}\colon\xi'=\hat M^{(k+1)}(\xi\eta)\xi+  \hat f^{(k+1)}, \quad
\eta'=\hat N^{(k+1)}(\xi\eta)\eta+  \hat g^{(k+1)},\\
(\hat f^{(k+1)},\hat g^{(k+1)})\in \cL C^{\mathsf c}_2(  S), \quad
\ord(\hat f^{(k+1)},\hat g^{(k+1)} )\geq2d_k.
\label{hfk1}
\nonumber
\end{gather}
 By     \re{ujpqcc}-\re{ujpqcc6}, we know that
\al
\label{ujpqc}
\hat U _{j,PQ}^{(k+1)}&=(\mu^{P-Q}-\mu_j)^{-1}\left\{f _{j,PQ}^{(k)}+\cL U _{j,PQ}^*(\del_{  \ell-1},
\{M^{(k)},N^{(k)}\}_{[\f{\ell-1}{2}]};\{ f^{(k)},g^{(k)}\}_{\ell-1})\right\},\\
\hat V _{j,QP}^{(k+1)}&=(\mu^{Q-P}-\mu_j^{-1})^{-1}\left\{g _{j,QP}^{(k)}+\cL V _{j,QP}^*(\del_{\ell-1},
\{M^{(k)},N^{(k)}\}_{[\f{\ell-1}{2}]};\{ f^{(k)},g^{(k)}\}_{\ell-1})\right\},
\label{ujpqc6} \end{align}
for $d_k\leq |P|+|Q|=\ell\leq 2d_k-1$ and $\mu^{P-Q}\neq\mu_j$. 
Recall that  ${\mathcal U}^*_{j,PQ}$ and ${\mathcal V}^*_{j,QP}$  are universal polynomials in their variables. In  notation defined by Definition~\ref{notation}, 
$
\cL U_{j,PQ}^*(\cdot; 0)= \cL V_{j,QP}^*(\cdot; 0)=0$. Thus
\ga
\label{psik1}\Psi_{k+1}-I=(\hat U^{(k+1)}, \hat V^{(k+1)})=O(d_k), \\
\label{hUk+1}\hat U^{(k+1)}_{j,PQ}=\f{f^{(k)}_{j,PQ}}{\mu^{P-Q}-\mu_j},\quad
\hat V^{(k+1)}_{j,QP}=\f{g^{(k)}_{j,QP}}{\mu^{Q-P}-\mu_j^{-1}}, \quad |P|+|Q|=d_k.
\end{gather}
Here $ \mu^{P-Q}\neq
\mu_j.$
By \re{mjpp0}-\re{mjpp--}, we have
\al
\label{mjpp-2}
\hat M^{(k+1)}_{j,P'}&=M^{(k)}_{j,P'}, \quad |P'|<d_k,\\
\label{mjpp2}
 \hat M^{(k+1)}_{j,P}&=M^{(k)}_{j,P}+   \mu_j \left\{
 2(\hat U^{(k+1)}_j \hat V^{(k+1)}_j)_{PP}+(
 (\hat U^{(k+1)}_j)^2)_{(P+e_j)(P-e_j})  \right\}
 \\
 &\quad+\left\{Df^{(k)}_j(\xi,\eta)(\hat U^{(k+1)},\hat V^{(k+1)})\right\}_{(P+e_j)P}, \quad |P|=d_k.
\nonumber\end{align}
 As mentioned in \rrem{keyrem}, 
 one of consequence of the above formula is that the coefficients of $\hat M^{(k+1)}_j(\xi\eta)\xi_j$
 of degree $2d_k+1$ do not depend on the coefficients of $f^{(k)}, g^{(k)}$ of degree $\geq 2d_k$, provided
 $(f^{(k)},g^{(k)})=O(d_k)$  is in $\cL C_2^{\mathsf c}(S)$ as we assume.

 Next, we need to estimate the size of coefficients of $M^{(k)}$ 
  that appear
  in \re{mjpp-2}-\re{mjpp2}. Recall that we apply  two sequences of normalization. We have
 \eq{}  
 \nonumber
  \hat\sigma_{k+1}^\infty=\Psi_{k+1}^{-1}\circ\Phi_k^{-1}\cdots\Psi_1^{-1}
\circ \Phi_0^{-1}\circ \sigma^\infty\circ \Phi_0\circ\Psi_1\cdots\Phi_{k}\circ\Psi_{k+1}.
 \eeq
 Thus, $M^{(k)}, N^{(k)}$ depend  only
 on $\sigma^\infty$, $\Phi_0, \Psi_1, \Phi_1, \ldots, \Psi_{k-1},\Phi_k$. 
 
 Recall that if $u_1,\ldots, u_m$ are power seres, then
 $\{u_1,\ldots, u_m\}_d$ denotes the set of their coefficients of degree at most $m$,
 and $|\{u_1,\ldots, u_m\}_d|$ denotes the sup norm. 
We choose $\sigma^\infty$ in such a way that its coefficients of degree $m$ satisfy
 \eq{siginfP}
 \nonumber
|\{ \sigma^k\}_{m}|+|\{\sigma^\infty\}_{m}|\leq C^m. 
 \eeq
Here $C$ does not dependent on $k$,    $\mu_1,\mu_2,\mu_3$,  $d_k$, and  $h^{(k)}$.
 We also need some crude estimates on the
 growth of Taylor coefficients.  
 If $F=I+f$ and 
  $f=O(2)$ is a map in formal power series, then   \re{F-1P5}-\re{F-1GF} imply 
 \al\label{F-1P}
 |(F^{-1})_{P}|&\leq  |F_{P}|+ (2+|\{f\}_{m-1}|)^{\ell_m},\\
\label{GFP} |(G\circ F)_{P}|&\leq  |((LG)\circ F,G)_P|+(2+|\{f,G\}_{m-1}|)^{\ell_m},  \\
 |(F^{-1}\circ G\circ F)_P|&\leq |(G,(LG)\circ F,F\circ LG)_P|+(2+|\{f,G\}_{m-1}|)^{\ell_m},
\label{FGPF} \end{align}
 for $m=|P|>1$ and some positive integer $\ell_m$. 
 Inductively, let us  show that  for $k=0,1, 2, \ldots$, 
\begin{alignat}{2}\label{hmk1-6}
|\{\hat M^{(k)},\hat N^{(k)}\}_{P}|&\leq \del_{d_{k-1}-1}^{L_m}, \quad& &m=  2|P|+1<2d_{k-1},
\\
|\{\hat\sigma^\infty_{k}\}_{PQ}|&\leq \del_{  2d_{k-1}-1}^{L_m}, \quad &&m=|P|+|Q|\geq 2d_{k-1}.
\label{hsigke}\end{alignat}
  We emphasize that here and in what follows 
$L_m$ does not depend on the choices of $\mu_j, d_k,h^{(k)}$ which  satisfy the initial conditions but are arbitrary otherwise. 
The above estimates hold trivially for $k=0$ and $d_{-1}=2$, since $\hat\sigma_0^\infty=\sigma^\infty$ is convergent. For induction, we assume that \re{hmk1-6}-\re{hsigke} hold.  
We need  to  find possibly larger $L_m$ for $m\geq  2 d_{k-1}$  in order to verify
\begin{alignat}{3}\label{hmk1-6+1}
|\{\hat M^{(k+1)},\hat N^{(k+1)}\}_{P}|&\leq \del_{d_{k}-1}^{L_m}, \quad &&m=2|P|+1<2d_{k}, \\
|\{\hat\sigma^\infty_{k+1}\}_{PQ}|&\leq \del_{2d_{k}-1}^{L_m}, \quad &&m=|P|+|Q|\geq 2d_{k}.
\label{hsigke1}
\end{alignat} 
The  $\Phi_k=I+(U^{(k)}, V^{(k)})$  is a polynomial mapping. Its  degree is  at most $d_k-1$ and   its coefficients are polynomials in  
  $\{\hat\sigma_k\}_{d_k-1}$ and $\del_{d_k-1}$; see \re{ujpq6}-\re{vjpq6}. 
  Hence
 \ga\label{fjpqg}
  |U^{(k)}_{j,PQ}|+|V^{(k)}_{j,QP}|\leq \del_{  d_k-1}^{L_m}, \quad m=|P|+|Q|.
 \end{gather}
 Note that the inequality holds trivially for $m<2d_{k-1}$, in which case the left-hand side
 is zero and we do not change the value of $L_m$. For $m\geq 2d_{k-1}$, we might have to increase $L_m$ if necessary to obtain \re{fjpqg}. 
 Applying \re{FGPF} to $\sigma_k^\infty=\Phi_k^{-1}\hat\sigma_k^\infty\Phi_k$,  we obtain from \re{hsigke} and \re{fjpqg} that
 \ga
 |M^{(k)}_{j,P}|+|N^{(k)}_{j,P}|\leq \del_{d_k-1}^{L_m}, \quad m=2|P|+1,
\label{mkjpn} \\ \label{fjpqg+}  |f_{j,PQ}^{(k)}|+|g^{(k)}_{j,QP}|
 \leq \del_{  d_k-1}^{L_m}, \quad m=|P|+|Q|.
 \end{gather}
 Here we use that fact that since $d_k  \geq 2d_{k-1}$, the small divisors in $\del_{2d_{k-1}-1}$ appear in $\del_{d_k-1}$ too. Furthermore, \re{fjpqg+} holds trivially when $m<2d_{k-1}$, as in this case its left-hand side is $0$.  If $|P|<d_{k-1}$ (i.e. $m=2|P|+1<2d_{k-1}+1$),  from \re{mphmp} and \re{hmk1-6} it follows 
   \re{mkjpn}      for the same $L_m$ in \re{hmk1-6}. 
 For \re{fjpqg+} with $m\geq 2d_{k-1}$ and for \re{mkjpn} with $m\geq 2d_{k-1}+1$, we   might have
 to increase the value of $L_m$  so that they are valid.  
We further remark that for possibly  increased $L_m$,
 \re{hmk1-6}-\re{hsigke} remain valid. 
To obtain \re{hmk1-6+1}-\re{hsigke1},  we note that $\Psi_{k+1}$ is a polynomial map that has 
 degree  at most $2d_k-1$ and 
 the coefficients of degree $m$   bounded by $\del_{2d_k-1}^{L_m}$; 
see \re{ujpqc}-\re{ujpqc6}.  This shows that
\eq{hUVk}
|\hat U^{(k+1)}_{j,PQ}|+|\hat V^{(k+1)}_{j,QP}|\leq \del_{2d_k-1}^{L_m}, \quad |P|+|Q|=m.
\end{equation}
The argument to obtain \re{hmk1-6+1}-\re{hsigke1} for $\hat\sigma_{k+1}^\infty=\Psi_{k+1}^{-1}\sigma_k^\infty\Psi_{k+1}$ is similar to the one
to obtain \re{mkjpn}-\re{fjpqg+}  for $\sigma_{k}^\infty=\Phi_k^{-1}
\hat\sigma_k^\infty\Phi_k$. Of course,  we still use \re{FGPF}, while replacing \re{phiki},  \re{mphmp},  \re{hmk1-6}, and \re{hsigke} by \re{psik1}, 
\re{mjpp-2}, \re{mkjpn}, and \re{fjpqg+}, respectively. 
  We emphasize that the sequence $L_m$ can be chosen consistently,  as 
for  $d_k\to\infty$, we only    increase each $L_m$ for  finitely many times.

%
%

Let us summarize the above computation for $\hat\sigma^\infty$ defined by \re{hsiginf}. 
We know that $\hat\sigma^\infty$ is the unique power series such that $\hat\sigma^\infty-
\hat\sigma^\infty_k=O(d_k)$ for all $k$, and $\hat\sigma^\infty$ is a formal formal form of $\sigma^\infty$.
Let us write
\gan
\hat\sigma^{\infty}\colon\left\{
\begin{split} \xi'=\hat M^{\infty}(\xi\eta)\xi, \\
\eta'=\hat N^{\infty}(\xi\eta)\eta.
\end{split}\right.
 \end{gather*}
Let $|P|\leq d_k$. By \re{mphmp}, we get $\hat M^{(k+1)}_P=M^{(k+1)}_P$; by \re{mjpp-2} in which $k$ is replaced by $k+1$, 
we get $\hat M^{(k+2)}_P=M^{(k+1)}_P$ as $|P|\leq d_k<d_{k+1}$.  Therefore,
\eq{mkmdk}\hat M^{\infty}_P=\hat M^{(k+1)}_P, \quad |P|\leq d_k.
\eeq
For $|P|<d_k$,  \re{mjpp-2} says that $\hat M^{(k+1)}_{j,P}=M_P^{(k)}$; by \re{mkjpn} that holds for any $P$,  we obtain
\eq{hmk1}
|\hat M^{\infty}_P|=|\hat M^{(k+1)}_{j,P}|\leq \del_{d_k-1}^{L_{m}},\quad |P|<d_k, \quad m=2|P|+1.
\eeq

To obtain rapid increase of coefficients of $\hat M^{(k+1)}_{j,P}$, we want to use both small divisors
hidden in $\hat U^{(k)}_{j,PQ}$ and $\hat V^{(k)}_{j,QP}$ in \re{mjpp2}.   Therefore, if $M^{(k)}_{j,P}$  is already sufficiently
large for $|P|=d_k$ that will be specified later, we take $\var_k$ to be the identity, i.e. $\tau_1^k=\tau_1^{k-1}$. Otherwise, we need to achieve it by choosing
$$
\tau_1^k=\var_k^{-1}\tau_1^{k-1}\var_k.
$$
Therefore, we  examine the effect of a coordinate change by $\var_k$ on these coefficients.

Recall that we are in the elliptic case. We have 
$
\rho(\xi,\eta)=(\ov\eta,\ov\xi)
$
and $\tau_2^k=\rho\tau_1^k\rho.$ Recall that
$$
\var_k\colon\xi_j'=(\xi+h^{(k)}(\xi),\eta),\quad  \ord h^{(k)}=d_k>3.
$$
By a simple computation, we obtain   
\gan
\tau_1^{k}(\xi,\eta)=\tau_1^{k-1}(\xi,\eta)+(-h^{(k)}(\la\eta),\la^{-1}h^{(k)}(\xi))+O(|(\xi,\eta)|^{d_k+1}),\\
\tau_2^{k}(\xi,\eta)=\tau_2^{k-1}(\xi,\eta)+
(\la^{-1}\ov{h^{(k)}}(\eta),-\ov{h^{(k)}}(\la\xi))+O(|(\xi,\eta)|^{d_k+1}).
\end{gather*}
Then we have
\al
\label{sigkk}\sigma^{k}&=\sigma^{k-1}+(r^{(k)},s^{(k)})+O(d_k+1);\\
r^{(k)}(\xi,\eta)&=
 -\la\ov{h^{(k)}} (\la\xi)-h^{(k)}(\la^{2}\xi), 
 \nonumber
  \\  
s^{(k)}(\xi,\eta)&=
 \la^{-2}\ov{h^{(k)}}(\eta)+\la^{-1}h^{(k)}(\la^{-1}\eta). 
\nonumber
\end{align}
 Since $\sigma^k$ converges to $\sigma^\infty$, from \re{sigkk} it follows that
\eq{siginfk}
\sigma^\infty=\sigma^{k-1}+(r^{(k)}, s^{(k)})+O(d_k+1).
\eeq
For $|P|+|Q|=d_k$, we have
\eq{}\begin{array}{l}
r^{(k)}_{j,PQ}=\left\{ -\la_j\ov{h_j^{(k)}}(\la\xi)-h_j^{(k)}(\la^{2}\xi)\right\}_{PQ},
\vspace{1ex}\\ 
s^{(k)}_{j,QP}=\left\{ \la_j^{-2}\ov{h_j^{(k)}} (\eta)+\la_j^{-1}h_j^{(k)}(\la^{-1}\eta)\right\}_{QP}.
\end{array}\nonumber
\eeq
We obtain
\begin{alignat}{4} \label{gjgj}
r^{(k)}_{j, P0} &=-\la^{P+e_j}\ov{h^{(k)}_{j,P}}-\la^{2P} h^{(k)}_{j,P},   &&
\\
s^{(k)}_{j,0P} &=\la_j^{-2}\ov{h^{(k)}_{j,P}}+\la^{-P-e_j} h^{(k)}_{j,P},\quad && |P|=d_k,\\
\label{gjgj3}
r^{(k)}_{j, PQ}&=s^{(k)}_{j,QP}=0,\quad                                                    & &|P|+|Q|=d_k,  \ Q\neq0. 
\end{alignat}

The above computation is actually sufficient to construct a divergent normal form $\tilde\sigma
\in\cL C(  S)$. To show that all normal forms of $\sigma$ in $\cL C(  S)$ are divergent,
We need to related it to the normal form $\hat\sigma$ in \rt{ideal5}, which is unique. This requires us to
keep track of the small divisors in the normalization 
procedure in the proof of \rl{fcfp}.

Recall that  $  F^{(k+1)}=\log \hat M^{(k+1)}$ is defined by 
\eq{fjlogm}
F^{(k+1)}_j(\zeta)=\log(\mu_j^{-1}\hat M^{(k+1)}_j(\zeta))=\zeta_j+a_j^{(k+1)}(\zeta), \quad 1\leq j\leq 3.
\eeq
We also have $F^{\infty} =\log \hat M^{\infty}$ with $
F^{\infty}_j(\zeta)= \zeta_j+a_j^{\infty}(\zeta)$.
Then  by \re{mkmdk}, 
\eq{hajpa}
 a_{j,P}^{\infty}= a^{(k+1)}_{j,P}, \quad |P|\leq d_k.
\eeq
By \re{GFPg} and \re{fjlogm}, we have
\eq{fjpmuj} \nonumber
a^{(k+1)}_{j,P}(\zeta)= \mu_j^{-1} \hat M^{(k+1)}_{j,P}+ {\cL A}_{j,P}(\{ \hat M^{(k+1)}_j\}_{|P|-1}), \quad|P|>1.
\eeq
By \re{hmk1}, we have
\eq{Ajph}
| {\cL A}_{j,P}(\{ \hat M^{(k+1)}_j\}_{|P|-1})|\leq \del_{d_k-1}^{L_m}, \quad |P|=d_k, \quad m=2|P|+1.
 \eeq
 Recall from the formula \re{hfjqf}
  that  $F^{(k+1)}$, $F^\infty$ have the normal forms
$\hat F^{(k+1)}=I+\hat a^{(k+1)}$ and $\hat F^\infty =I+\hat a^\infty$, respectively.
The coefficients of $\hat a_{j,Q}^{(k+1)}$ and $\hat a_{j,Q}^{\infty}$
   are zero, except the ones given by
\aln
 \hat a^{(k+1)}_{j,Q}&=a^{(k+1)}_{j,Q}+ {\cL B}_{j,Q}(\{a^{(k+1)}\}_{|Q|-1}), 
 \\
  \hat a^{(\infty)}_{j,Q}&=a^{(\infty)}_{j,Q}+ {\cL B}_{j,Q}(\{a^{(\infty)}\}_{|Q|-1}), 
 \end{align*} 
  for $Q=(q_1,\ldots, q_p), q_j=0$, and $ |Q|>1$.
 Derived from the same normalization, the $ {\cL B}_{j,Q}$ in both formulae 
 stands for the same polynomial. 
 Hence $\hat a_{j,P}^{(\infty)}=\hat a_P^{(k+1)}$ for $|P|\leq d_k$, by \re{hajpa}.
Combining \re{mjpp2} and \re{mkmdk} yields 
\al\label{uvprod}
\hat a^{\infty}_{3,P_k} &= \hat a^{(k+1)}_{3,P_k}= 2(\hat U^{(k+1)}_3 \hat V^{(k+1)}_3)_{P_kP_k}  +((\hat U^{(k+1)}_3)^2)_{(P_k+e_3)(P_k-e_3)}
+\mu_3^{-1}M_{3,P_k}^{ (k)}
\\
&\qquad +\{Df^{( k)}_{3}(\xi,\eta)(\hat U^{(k+1)},\hat V^{(k+1)})\}_{(P_k+e_3)P_k} 
+\cL {A}_{  k, P_k}(\{
  \hat M^{(k+1)}\}_{|P_k|-1}).\nonumber
\end{align}
We regard $\hat a^{\infty}_{3,P_k}$ as polynomials in $(\mu^{P-Q}-\mu_j)^{-1}$.    The above
formula    holds for any $P_k$ with $|P_k|=d_k$.  

To examine the effect
of small divisors,  we assume  that $P_k=(p_k,q_k,0)$ are given by \rl{smallvec+}, so are $\mu_1,\mu_2$, 
and $\mu_3$.    Then the second term in \re{uvprod} is $0$ as the third component of $P_k-e_3$ is negative.
We apply the above computation to $$|P_k|=d_k.$$
  Taking a subsequence of $P_k$ if necessary, we may assume that $d_k\geq 2d_{k-1}$ and $d_{k-1}>3$ for all $k\geq1$.
 The $4$ exceptional small divisors of height $2|P_k|+1$ in \re{delspk} are
$$
\mu^{P_k}-\mu_3, \quad \mu^{-P_k}-\mu_3^{-1}, \quad \mu^{2P_k-e_3}-\mu_3, \quad \mu^{-2P_k+e_3}-\mu_3^{-1}.
$$
 The last two  cannot show up in $\hat a_{3,P_k}^\infty$, since   their degree,  $2d_k+1$,  is larger than 
 the degrees of Taylor coefficients in 
$\hat a_{3,P_k}$ . 
We have $3$ products of   the   two exceptional small divisors of height 
$2|P_k|+1$ and   degree $|P_k|$, which
are 
$$
(\mu^{P_k}-\mu_3)(\mu^{-P_k}-\mu_3^{-1}), \quad
(\mu^{P_k}-\mu_3)(\mu^{P_k}-\mu_3), \quad
(\mu^{-P_k}-\mu_3^{-1})(\mu^{-P_k}-\mu_3^{-1}).
$$
The first product, but none of  the other two, appears in  $(\hat U^{(k+1)}_3 \hat V^{(k+1)}_3)_{P_kP_k}$.
The third  term and $f_3^{(k)}$ in $\hat a_{3,P_k}^\infty$ 
do not contain  exceptional small divisors
of degree $|P_k|=d_k>2d_{k-1}-1$. 
  Since $f_3^{(k)}=O(d_k)$ by \re{fkgkd}, the  exceptional small divisors of height $2|P_k|+1$
can show up at most once in the fourth term of  $\hat a_{3,P_k}^\infty$.  
Therefore, we arrive at
\aln
&\hat a^{\infty}_{3,P_k}=2\hat U^{(k+1)}_{3,P_k0}\hat V^{(k+1)}_{3, 0P_k}+\hat{\cL A}_{k}^1(\del_{d_k-1},\f{1}{\mu^{P_k}-\mu_3}; \{f^{(k)},g^{(k)}\}_{d_k})\\
&\quad\quad \quad \  \ +\hat {\cL A}_{ k}^2(\del_{d_k-1};\{f^{(k)},g^{(k)}\}_{d_k})+\mu_3^{-1}M_{3,P_k}^{ (k)}+ \cL {A}_{ k, P_k}( \{
  \hat M^{(k+1)}\}_{|P_k|-1}),\\
&\hat{\cL A}_{k}^1(\del_{d_k-1},
  \f{1}{\mu^{P_k}-\mu_3}; \{f^{(k)},g^{(k)}\}_{d_k})
=(\hat U^{(k+1)}_{3,P_k0},\hat V^{(k+1)}_{3,0P_k})\cdot \hat {\cL A}_{k}^3(\del_{d_k-1};\{f^{(k)},g^{(k)}\}_{d_k}).
\end{align*}
By \re{mkjpn} and \re{Ajph}, we obtain $|M_{3,P_k}^{ (k)}|+| \cL {A}_{ k, P_k}( \{
  \hat M^{(k+1)}\}_{|P_k|-1})|\leq \del_{d_k-1}^{L_m}$ for $m=2d_k+1$. 
Omitting  the arguments in the polynomial functions, we obtain from \re{fjpqg}-\re{hUVk}, 
and \re{mkmdk} that
$$
|\hat{\cL A}_{k}^1|+|\hat{\cL A}_k^2|+|M_{3,P_k}^{(k)}|+|\cL A_{k,P_k}|
\leq\f{|\del_{d_k-1}(\mu)|^{ L_{  m}}}{|\mu^{P_k}-\mu_3|}, \quad m=2|P_k|+1,
$$
for a possibly larger $L_{m}$. 
  We remark that although each term in the inequality depends on the choices of the sequences 
$\mu_i,d_j,h^{(\ell)}$, the $L_m$   does not
depend on the choices, provided that $\mu_j,d_k,h^{(i)}$ satisfy our initial conditions. 
Therefore, we have
\eq{}
\nonumber
|\hat a^{\infty}_{3,P_k}|\geq2 |\hat U^{(k+1)}_{3,P_k0}\hat V^{(k+1)}_{3, 0P_k}|-|\del_{d_k-1}(\mu)|^{ L_{2|P_k|+1}}|\mu^{P_k}-\mu_3|^{-1}.
\eeq
Recall that $\sigma^\infty_k=
\Phi_k^{-1}\Psi_{k-1}^{-1}\cdots\Phi_0^{-1}\sigma^\infty\Phi_0\Psi_1\cdots\Phi_k$.
Set 
\eq{} 
\nonumber
\tilde\sigma^\infty_k:=
\Phi_k^{-1}\Psi_{k-1}^{-1}\cdots\Phi_0^{-1}\sigma^{k-1}\Phi_0\Psi_1\cdots\Phi_k.
\eeq
By \re{siginfk}, we get
\eq{sigtsig} 
\sigma^\infty_k=\tilde\sigma_k^\infty+(r^{(k)}, s^{(k)})+O(d_k+1).
\eeq
  Recall that  $\Phi_k$  depends only on coefficients of
$\hat\sigma^\infty_{k-1}=\Psi_{k-1}^{-1}\sigma^\infty_{k-2}\Psi_{k-1}$ of degree less than $d_k$, while $\Psi_{k-1}$
 depends only on coefficients of $\sigma_{k-1}^\infty=\Phi_{k-1}^{-1}\hat\sigma_{k-1}\Phi_{k-1}$ of
degree at most $2d_{k-1}-1$ which is less than $d_k$ too. Therefore,
 $\Phi_k, \Psi_{k-1}, \ldots, \Phi_0$   depend only on
coefficients of $\sigma^\infty$ of degree less than $d_k$. On the other hand, $\sigma^\infty=\sigma^{k-1}+O(d_k)$. 
Therefore,   $\tilde\sigma_k^\infty$  
depends only on  $\sigma^{k-1}$, and hence it depends only on $h^{(\ell)}$ for $\ell<k$. By \re{sigtsig}, 
we can express
\eq{fktfk} 
f^{(k)}_{j,PQ}=\tilde f^{(k)}_{j,PQ}+r^{(k)}_{j,PQ}, \quad 
g^{(k)}_{j,QP}=\tilde g^{(k)}_{j,QP}+s^{(k)}_{j,QP},\eeq
where $|P|+|Q|=d_k$ and  $\tilde f^{(k)}_{j,PQ}, \tilde g^{(k)}_{j,QP}$ depend only on 
$h^{(\ell)}$ for $\ell<k$.
Collecting  \re{hUk+1},   \re{fktfk}, and \re{gjgj}-\re{gjgj3}, we obtain
\eq{}
\nonumber
|\hat a^{\infty}_{3,P_k}|\geq2 \f{|T_k
|}{|\mu^{P_k}-\mu_3||\mu^{-P_k}-\mu_3^{-1}|}-\f{|\del_{d_k-1}(\mu)|^{ L_{2d_k+1}}}{|\mu^{P_k}-\mu_3|}
\eeq
with
\aln
&T_k =
(-\la^{P_k+e_3}\ov{h^{(k)}_{3,P_k}}-\la^{2P_k} h^{(k)}_{3,P_k} + \tilde f^{(k-1)}_{3,P_k0})
(\la_3^{-2}\ov{h^{(k)}_{3,P_k}}+\la^{-P_k-e_3} h^{(k)}_{3,P_k}+\tilde g^{(k-1)}_{3,0P_k})\\
&=-\la^{2P_k-2e_3}
(\la^{e_3-P_k}\ov{h^{(k)}_{3,P_k}}+ h^{(k)}_{3,P_k} -\la^{-2P_k}\tilde f^{(k-1)}_{3,  P_k0})
(\la^{e_3-P_k} h^{(k)}_{3,P_k}+\ov{h^{(k)}_{3,P_k}}+\la_3^2\tilde g^{(k-1)}_{3,0P_k}).\nonumber
\end{align*}
Set  $\tilde T_k(h^{(k)}_{3,P_k}):=-\la^{2e_3-2P_k}T_k$.
We are ready to choose $h^{(k)}_{3,P_k}$ to get a divergent normal form. We have either
$|\la^{P_k-e_3}+1|\geq1$ or $|\la^{P_k-e_3}-1|\geq1$. 
When the first case occurs, 
 one of $|\tilde T_k(0)|, |\tilde T_k(1)|, |\tilde T_k(-1)|$ is at least $1/4$; otherwise, 
 we would have
\aln
2|\la^{P_k-  e_3}+1|^2=|\tilde T_k(1)+\tilde T_k(-1)-2\tilde T_k(0)|<1,
\end{align*}
which is a contradiction.  
Here the first identity follows from the fact that $\tilde f^{(k)}_{j,PQ}, \tilde g^{(k)}_{j,QP}$ depend only on 
$h^{(\ell)}$ for $\ell<k$.
When the second case occurs, we conclude that one of $|\tilde T_k(0)|$, $|\tilde T_k(i)|$, $|\tilde T_k(-i)|$ is at least $1/4$. 
This shows that by taking $h^{(k)}_{3,P_k}$ to be one of $0, 1,-1,i$,  $-i$, we can achieve 
$$
|T_k|\geq\f{1}{4}\mu^{2P_k -2e_3}.
$$
Therefore,
\eq{divhf}
\nonumber
|\hat a^{\infty}_{3,P_k}|\geq\f{\mu^{2P_k-2e_3} }{2|\mu^{P_k}-\mu_3||\mu^{-P_k}-\mu_3^{-1}|} -\f{|\del_{d_k-1}(\mu)|^{ L_{2d_k+1}}}{|\mu^{P_k}-\mu_3|}.
\eeq
Recall that $\mu_3=e^e$. If $|\mu^{P_k}-\mu_3|<1$ then $1/2<\mu^{P_k-e_3}<2$. The above inequality implies
\eq{divhf+}
|\hat a^{\infty}_{3,P_k}|\geq \f{ \mu^{3P_k-3e_3} }{4|\mu^{P_k}-\mu_3|^2},
\eeq
provided
$$
|\mu^{P_k}-\mu_3|\leq  \f{1}{32} |\del_{d_k-1}(\mu)|^{-L_{2d_k+1}}, \quad |P_k|=d_k.
$$
For the last inequality  to hold, it suffices have
\eq{epqccc}
|\mu^{P_k}-\mu_3|\leq  |\del_{d_k-1}(\mu)|^{-L_{2d_k+1}-  1}, \quad |\del_{d_k-1}(\mu)|^{-1}<1/4.
\eeq
Note that the sequence $L_m$ does not depend on the choice of $\la$.   The existence of $\mu_1,\mu_2,\mu_3$
is ensured by \rl{smallvec+} as follows: We choose the sequence
$L_m$ in \rl{smallvec+}, denoted by $L_m'$ now,
so that  $|P_k|L_{k}'>  2L_{2d_k+1}+2$. Then   \re{epqccc} follows from \re{epqcc}, the definitions of $\del_{d_k-1}(\mu)$ by \re{delnmu}
and    of $\Delta^*(P_k)$ by \re{delspk}; indeed
\aln
|\mu^{P_k}-\mu_3|&\leq  (C\Delta^*(P_k))^{|P_k|L_k'}\leq  (\Delta^*(P_k)^{1/2})^{|P_k|L_k'}\\
&\leq (\delta_{d_k-1}(\mu))^{-|P_k|L_k'/2}
\leq |\del_{d_k-1}(\mu)|^{-L_{  2d_k+1}  -1}.\end{align*}
  Here the second inequality follows from       $C(\Delta^*(P_k))^{1/2}<1$ when $k$ is sufficiently large.
   The third inequality is obtained as follows.   The definition of $\Delta^*(P_k)$ and $|P_k|=d_k$
   imply that any small divisor
  in $\del_{d_k-1}(\mu)$ is contained in $\Delta^*(P_k)$. Also, $\Delta^*(P_k)<\mu_i^{-1}$
  for $i=1,2,3$ and $k$ sufficiently large.  Hence, $\Delta^*(P_k)
  \leq\del^{-1}_{d_k-1}(\mu)$, which gives us the third inequality. 
Without loss of generality, we may assume that $L_k>  k$. From \re{divhf+} and \re{epqccc}
it follows that 
$$ 
|\hat a_{3,P_k}^\infty|>\del_{d_{k-1}}^{d_k+1}(\mu)=\del_{d_{k-1}}^{|P_k|+1}(\mu),
$$
for $k$ sufficiently large.  As $\del_{d_k}(\mu)\to+\infty$, this shows that
the divergence of $ \hat F_3$   and  the
divergence of  the normal form $\hat\sigma$.

As mentioned earlier, \rt{ideal5} (iii) implies
that any normal form of  $\sigma$ that is  in the centralizer of $\hat S$ must diverge.
  \end{proof}

\setcounter{thm}{0}\setcounter{equation}{0}
\section{A unique formal normal form of  a real submanifold}\label{nfin}
 
 Recall that we consider submanifolds of which the complexifications  admit the maximum number of deck transformations.
   The deck transformations of $\pi_1$  are generated by 
   $\{\tau_{i1},\ldots, \tau_{1p}\}$. 
 We also set  $\tau_{2j}=\rho\tau_{1j}\rho$.  Each of $\tau_{i1},\ldots, \tau_{ip}$ fixes a hypersurface and
  $\tau_i=\tau_{11}\cdots\tau_{1p}$ is the unique deck transformation whose set of fixed points has
 the smallest dimension. We first normalize  the composition $\sigma=\tau_1\tau_2$.
 This normalization is reduced to   two normal form problems.
  In \rp{ideal0} we obtain  a transformation  $\Psi$  to transform $\tau_1,\tau_2,$
 and $\sigma$ into
 \begin{alignat*}{4}
 \tau_i^*\colon\xi_j'&=\Lambda_{ij}(\xi\eta)\eta_j,\quad &&\eta_j'=\Lambda_{ij}^{-1}(\xi\eta)\xi_j,\\
 \sigma^*\colon\xi_j'&=M_j(\xi\eta)\xi_j, \quad &&\eta_j'=M_j^{-1}(\xi\eta)\eta_j, \quad
 1\leq j\leq p.
 \end{alignat*}
 Here $\Lambda_{2j}=\Lambda_{1j}^{-1}$ and $M_j=\Lambda_{1j}^2$ are power series in the product
  $\zeta=(\xi_1\eta_1,\ldots, \xi_p\eta_p)$.
  We also  normalize the map $M\colon
 \zeta\to M(\zeta)$ by a transformation  $\var$ which preserves all coordinate
 hyperplanes. 
 This is the  second normal form problem, which is solved formally   in \rt{ideal5} 
under the condition on the normal form of $\sigma$, namely,
 that  $\log \hat M$ is tangent to the identity. 
 This   gives us a map $\Psi_1$
 which transforms  $\tau_1,\tau_2$, and $\sigma$
 into $\hat\tau_1,\hat\tau_2,\hat\sigma$ of the above form
  where $\Lambda_{ij}$ and $M_j$
become  $\hat\Lambda_{ij},\hat M_j$. 

In this section, we derive a {\bf unique formal normal form for
  $\{\tau_{11},\ldots, \tau_{1p},\rho\}$ under the above 
  condition on $\log\hat M$}.
In this     case, 
we know from \rt{ideal5} that $ {\cL C}(\hat\sigma)$ consists of   only  $2^p$ dilatations
\eq{upsi}
R_\e\colon (\xi_j,\eta_j)\to(\e_j\xi_j,\e_j\eta_j), \quad \e_j=\pm1, \quad 1\leq j\leq p. 
\eeq
  We will consider two cases. In the first case, we impose no restriction on the linear parts of $\{\tau_{ij}\}$ but
the coordinate changes are restricted to mappings that are tangent to the identity. The second is for the family $\{\tau_{ij}\}$
that arises from a higher order perturbation of a product quadric, while no restriction is imposed on the changes of coordinates. 
We will   show that in both cases, if the normal form  of $\sigma$ can be achieved
by a convergent transformation,  the normal form of $\{\tau_{11}, \ldots,\tau_{1p},\rho\}$
can be achieved by a convergent transformation too.

We now restrict our real submanifolds to some classes. First, we assume that $\sigma$ and $\tau_1,\tau_2$
are already in the normal form $\hat\sigma$ and $\hat\tau_{1}, \hat\tau_2$ such that
 \ga\label{htaix}
 \hat\tau_i\colon\xi'=\hat\Lambda_{i}(\xi\eta)\eta, \quad \eta'=\hat\Lambda_i(\xi\eta)^{-1}\xi, \quad\hat\Lambda_{2}=\hat\Lambda_1^{-1},\\
\label{hsixi} \hat\sigma\colon \xi'=\hat M(\xi\eta)\xi,\quad \eta'=\hat M(\xi\eta)^{-1}\eta, \quad \hat M=\hat\Lambda_1^2. 
\end{gather}

Let us start with the general situation   without imposing the restriction on  the linear part of $\log M$.
Assume that 
$\hat\sigma$ and $\hat\tau_i$ are in the above forms. 
Recall that $\mathbf Z_j=\diag(1,\ldots, -1,\ldots, 1)$ with $-1$ at the $(p+j)$-th place, and $\mathbf
Z:=\mathbf Z_1\cdots \mathbf Z_p$.   Let $Z_j$ (resp. $Z$) be the linear transformation with the matrix $\mathbf Z_j$
(resp. $\mathbf Z$).   We also use notation
\ga
\mathbf B_*=\begin{pmatrix}
\mathbf I & \mathbf 0 \\
\mathbf 0 &\mathbf  B
\end{pmatrix},\quad
\mathbf E_{\mathbf{\hat{ \Lambda}}_i}=\begin{pmatrix}
\mathbf I &\mathbf {\hat\Lambda}_i  \\
-\mathbf {\hat\Lambda}_i^{-1} &\mathbf  I
\end{pmatrix}.
\label{t1jd+l}
\end{gather}
Here $\mathbf B$, as well as $\mathbf {\hat\Lambda}_i$   given by \re{htaix},   is a non-singular complex $(p\times p)$ matrix. Define two transformations
\eq{bise}
(B_i)_*\colon \begin{pmatrix}\xi\\ \eta\end{pmatrix}\to
 (\mathbf B_i)_* \begin{pmatrix}\xi\\ \eta\end{pmatrix}, \quad
  E_{\mathbf {\hat\Lambda}_i}\colon \begin{pmatrix}\xi\\ \eta\end{pmatrix}\to
\begin{pmatrix}
\mathbf I &\mathbf {\hat\Lambda}_i(\xi\eta)  \\
-\mathbf {\hat\Lambda}_i^{-1}(\xi\eta) &\mathbf  I
\end{pmatrix} \begin{pmatrix}\xi\\ \eta\end{pmatrix}. 
\eeq
  Let us assume that 
in suitable linear coordinates,   the linear parts of two families
of involutions $\{\tau_{i1}, \ldots, \tau_{ip}\}$ for $i=1,2$ are given by 
\eq{ltij}
L\tau_{ij}=T_{ij}, \quad
{ T}_{ij}= { E}_{\mathbf{\Lambda}_i}({ B}_i)_*{ Z}_j({ B}_i)_*^{-1}{ E}_{\mathbf {\Lambda}_i}^{-1}, \quad
 \mathbf{\Lambda_i}=\hat{\mathbf\Lambda}_i(0).
\eeq
Here $T_{ij}$ are in the normal forms described in  \rl{unsol} or in \rp{2tnorm}.
 
Note that $(B_i)_*$ commutes with $Z$.   Also,
$E_{\mathbf {\hat\Lambda}_i}\circ\hat\tau_i=Z\circ E_{\mathbf {\hat\Lambda}_i}$.
 Let us set
\ga
\label{t1jdl}
\hat \tau_{ij}:=  E_{\mathbf {\hat\Lambda}_i }\circ ( B_i)_*\circ Z_j\circ
( B_i)_*^{-1}\circ E_{\mathbf {\hat\Lambda}_i }^{-1}
\end{gather}
and we have   $\hat\tau_1=\hat\tau_{11}\cdots\hat\tau_{1p}$. 
The following lemma is analogous to the scheme used to classify the quadrics with the maximum number of  deck transformations. 
The lemma provides a way to represent all involutions $\{\tau_{11}, \ldots, \tau_{12^p}, \rho\}$ provided that we already have
a normal form for $\sigma$. 
\begin{lemma}\label{clas}  Let $\{\tau_{1j}\}$ and
$\{\tau_{2j}\}$
be two families  of formal holomorphic commuting involutions.
Let $\tau_i=\tau_{i1}\cdots\tau_{ip}$ and $\sigma=\tau_1\tau_2$. Suppose that
\gan
\tau_{i}=\hat\tau_i\colon\xi_j'
=\hat\Lambda_{ij}(\xi\eta)\eta_j, \quad \eta_j'=\hat\Lambda_{ij}(\xi\eta)^{-1}\xi_j;\\
\sigma=\hat\sigma\colon\xi_j'=\hat
M_j(\xi\eta)\xi_j, \quad \eta_j'=\hat M_j(\xi\eta)^{-1}\eta_j
\end{gather*}
with $\hat M_j=\hat\Lambda_{1j}^2$. Assume further that the linear parts $T_{ij}$ of $\tau_{ij}$,  given by \rea{ltij}
are in normal forms in \rla{unsol} or  \rpa{2tnorm}.
 Then we have the following~$:$
\bppp
\item For $i=1,2$
there exists $\Phi_i\in\cL C(\hat{\tau}_{i})$, which is tangent to
the identity,
 such that
$
\Phi_i^{-1}\tau_{ij}\Phi_i=\hat \tau_{ij}
$
for $1\leq j\leq p$.
\item 
Let $\{\tilde\tau_{1j}\}$ and
$\{\tilde\tau_{2j}\}$
be two families  of formal holomorphic commuting involutions.
Suppose that $\tilde\tau_{i}=\hat\tau_i$ and $\tilde\sigma=\hat\sigma$ and
$
\tilde\Phi_i^{-1}\tilde\tau_{ij}\tilde\Phi_i=\widehat{\tilde\tau}_{ij}$ with $\tilde\Phi_i\in\cL C(\hat\tau_i)
$ being tangent to the identity and  
\ga
\label{t1jdls}
\nonumber
\widehat{\tilde\tau}_{ij}=  E_{{\mathbf{\hat\Lambda}}_i}\circ (\tilde B_i)_*\circ Z_j\circ
(\tilde B_i)_*^{-1}\circ E_{\mathbf{\hat\Lambda}_i}^{-1}.
\end{gather}
  Here for $i=1,2$, the matrix $\mathbf {\tilde B}_i$ of $\tilde B_i$ is non-singular. 
Then  
\ga
\Upsilon^{-1}\tau_{ij}\Upsilon=\tilde\tau_{i\nu_i(j)},
\label{bs-1}
\nonumber
\end{gather}
 if and only if 
there exist $\Upsilon\in\cL C(\hat{\tau}_{1},\hat{\tau}_{2})$ and    $\Upsilon_i\in\cL C(\hat\tau_i)$ such
that
\ga\label{tpup}
\tilde\Phi_i=\Upsilon^{-1}\circ\Phi_i\circ\Upsilon_i,\quad i=1,2,\\
\label{el1-}
\nonumber
\Upsilon_i^{-1}\hat\tau_{ij}\Upsilon_i=\widehat {\tilde\tau}_{i\nu_i(j)},\quad 1\leq j\leq p.
\end{gather}
Here each $\nu_i$ is a permutation of $\{1,\ldots, p\}$. 
\item
Assume further that $\tau_{2j}=\rho\tau_{1j}\rho$ with $\rho$ being  defined by  \rea{eqrh}.
Define $\hat \tau_{1j}$ by \rea{t1jdl} and  let
\eq{t2rd}
\nonumber
\hat\tau_{2j}\colonequals\rho\hat\tau_{1j}\rho.
\eeq
Then
we can choose $\Phi_2=\rho\Phi_1\rho$ for $(i)$. Suppose that $\tilde\Phi_2=\rho\tilde\Phi_1\rho$ where $\tilde\Phi_1$ is as in $(ii)$. Then $\{\tilde\tau_{1j},\rho\}$ is equivalent to
$\{\tau_{1j},\rho\}$  if and only if
there exist $\Upsilon_i$,   $\nu_i$ with $\nu_2=\nu_1$,  and $\Upsilon$ satisfying the
conditions in $(ii)$ and
 $\Upsilon_2=\rho\Upsilon_1\rho$. The latter
 implies that $\Upsilon\rho=\rho\Upsilon$.
\eppp
\end{lemma}
\begin{proof}
(i) 
Note that $\hat\tau_{ij}$ is conjugate to
$Z_j$ via the   map  
$E_{\boldsymbol{\hat{\Lambda}_i}}\circ (B_i)_*$. Fix $i$.  Each  $\hat\tau_{ij}$
is an involution
and its set of fixed-point is a hypersurface. Furthermore,  $\fix(\tau_{11}), \ldots, \fix(\tau_{1p})$
intersect
transversally at the origin.  By \rl{sd} there exists a formal mapping $\psi_i$
 such that
$\psi_i^{-1}\tau_{ij}\psi_i=L\tau_{ij}$.  Now  $L\psi_i$ commutes with $L\tau_{ij}$,
Replacing $\psi_i$ by $\psi_i(L\psi_i)^{-1}$, we may assume that $\psi_i$ is tangent
to the identity. We also find a formal mapping $\hat\psi_i$, which is tangent to the identity,  such
that $\hat\psi^{-1}_i\hat\tau_{ij}\hat\psi_i=L\hat\tau_{ij}=L\tau_{ij}$. Then  
 $\Phi_i= \psi_i\hat\psi_i^{-1}$ fulfills the requirements.

(ii) Suppose that
\eq{tijp}\nonumber
\tau_{ij}=\Phi_i\hat \tau_{ij}\Phi_i^{-1},\quad
\tilde\tau_{ij}=\tilde\Phi_i\widehat {\tilde\tau}_{ij}\tilde\Phi_i^{-1}.
\eeq
Assume that there is a formal biholomorphic mapping $\Upsilon$ that transforms $\{\tau_{ij}\}$ into $\{\tau_{ij}\}$
for $i=1,2$. Then  
\eq{UtUt}
\Upsilon^{-1}\tau_{ij}\Upsilon=\tilde\tau_{i\nu_i(j)}, \quad j=1,\ldots, p,\  i=1,2.
\eeq
Here $\nu_i$ is a permutation of $\{1,\ldots, p\}$. Then
\eq{upsh}
\Upsilon\hat\tau_i=\hat\tau_i\Upsilon, \quad \Upsilon\hat\sigma=\hat\sigma\Upsilon.
\eeq
  Set
$\Upsilon_i\colonequals\Phi_i^{-1}\Upsilon \tilde \Phi_i$. We obtain
\ga\label{el1}
\Upsilon_i^{-1}\hat\tau_{ij}\Upsilon_i=\widehat {\tilde\tau}_{i\nu_i(j)},\quad 1\leq j\leq p,\\
\label{pitp}
\tilde\Phi_i= \Upsilon^{-1}\Phi_i\Upsilon_i, \quad i=1,2.
\end{gather}
Conversely, assume that \re{upsh}-\re{pitp}
 are valid. 
 Then \re{UtUt} holds as
\ga
\label{el2}
\nonumber
\Upsilon^{-1}\tau_{ij}\Upsilon=\Upsilon^{-1}\Phi_i
\hat\tau_{ij}\Phi_i^{-1}\Upsilon= \tilde\Phi_i \Upsilon_i^{-1}
\hat \tau_{ij} \Upsilon_i\tilde\Phi_i^{-1} =\tilde\tau_{\nu_i(j)}. 
\end{gather}
  

(iii) Assume that we have the reality assumption
 $\tau_{2j}=\rho\tau_{1j}\rho$ and
$\tilde\tau_{2j}=\rho\tilde\tau_{1j}\rho$. As before,
we take $\Phi_1$,  tangent to the identity, such that  $\tau_{1j}=\Phi_1\hat \tau_{1j}\Phi_1^{-1}$. Let $\Phi_2
=\rho\Phi_1\rho$. By $\hat \tau_{2j}=\rho \hat \tau_{1j}\rho$, we get $\tau_{2j}=\rho\tau_{1j}\rho=\Phi_2\hat \tau_{2j}\Phi_2^{-1}$
  for $\nu_2=\nu_1$.
Suppose that $\tilde\Phi_i$ satisfy the analogous properties
for $\tilde \tau_{1j}$ and $\rho$.
Suppose that $\Upsilon^{-1}\tau_{ij}\Upsilon=\tilde \tau _{i\nu_i(j)}$, $\nu_2=\nu_1$, 
 and $\Upsilon\rho=\rho\Upsilon$. Letting
$\Upsilon_i= \Phi_i^{-1}\Upsilon\tilde\Phi_i$  we get
$\Upsilon_2=\rho\Upsilon_1\rho$. Conversely, if   $\Upsilon_1$ and $
\Upsilon_2$ satisfy 
 $\Upsilon_2=\rho\Upsilon_1\rho$, then
$$\rho\Upsilon\rho=\rho\Phi_1\Upsilon_1\tilde\Phi_1^{-1}\rho=
\Phi_2\Upsilon_2\tilde\Phi_2^{-1}=\Upsilon.$$
This shows that $\Upsilon$ satisfies the reality condition.
\end{proof}

  Now we assume that 
    $\hat F=\log\hat M$ is 
  tangent to the identity and is in the 
      normal form   
   \re{hfj1}. Recall the latter means that
  the $j$th component of $\hat F-I$ is independent of the $j$ variable. 
We assume that the linear part $T_{ij}$ of $\tau_{ij}$ are given by \re{ltij}, where the non-singular matrix $\mathbf B$ is arbitrary.
  As mentioned earlier in this section,  the group of formal biholomorphisms that preserve  the form of $\hat\sigma$ consists of only linear involutions
$R_\e$ defined by \re{upsi}.  This   restricts  the holomorphic equivalence classes of the quadratic parts of $M$. 
By \rp{2tnorm}, such quadrics are classified 
by a more restricted equivalence relation,   namely, 
$(\mathbf {\tilde B_1},\mathbf {\tilde B}_2)\sim (\mathbf B_1,\mathbf B_2)$, if and only if
$$ 
\tilde {\mathbf B}_i= \mathbf R_\e^{-1}\mathbf B_i\diag_{  \nu_i}\mathbf d, \quad i=1,2.
$$ 
For simplicity, we will now fix a representative $\mathbf B_1,\mathbf B_2$ for its equivalence class. 

Using the normal form $\{\hat\tau_{1},\hat\tau_2\}$ and the matrices $\mathbf B_1,\mathbf B_2$, we first decompose
$\hat\tau_{i}=\hat\tau_{11}\cdots\hat\tau_{1p}$. By \rl{clas} (i), we then find $\Phi_i$   such
that 
\eq{tauPh}
\nonumber
\tau_{ij}=\Phi_i\hat\tau_{ij} \Phi_i^{-1}, \quad 1\leq j\leq p.
\eeq
For each $i$,  $\Phi_i$  commutes with $\hat\tau_i$. It is within this family of $\Phi_i\in\cL C(\hat\tau_i)$
for $i=1,2$ that we will find a normal form for $\{\tau_{ij}\}$.  When restricted  to $\tau_{2j}=\rho\tau_{1j}\rho$,  the classification
 of   the real submanifolds is within the family of  $\{\{\tau_{1j}\},\{\tau_{2j}\}\}$ as described above and such that
$$
\Phi_2=\rho\Phi_1\rho. 
$$

 From \rl{clas} (ii),    the equivalence relation on $\cL C(\hat\tau_i)$ is given by
\ga
\label{pitp+}
\nonumber
\tilde\Phi_i= \Upsilon^{-1} \Phi_i\Upsilon_i, \quad i=1,2.
\end{gather}
Here $\Upsilon_i$ and $\Upsilon$   satisfy
\eq{el1+}
\nonumber
\Upsilon_i^{-1}\hat\tau_{ij}\Upsilon_i=\hat \tau_{i\nu_i(j)},\quad 1\leq j\leq p;
\quad\Upsilon\hat\tau_i\Upsilon^{-1}=\hat\tau_i,\quad i=1,2.
\eeq

We now construct  a normal form for $\{\tau_{ij}\}$ within the above family.
Let us first use the centralizer of $\cL C^{\mathsf c}(Z_1,\ldots, Z_p)$, described in \rl{lehphi}, 
to define the
complement of the
centralizer of the family of non-linear commuting involutions 
$\{\hat\tau_{11}, \ldots, \hat\tau_{1p}\}$.
Recall that 
the mappings  $E_{\mathbf{\hat\Lambda}_i}$ and $(B_i)_*$ are defined by   \re{bise}.
According to \rl{lehphi}, we have the following.
\begin{lemma}\label{cnnl} Let $i=1$,  $2$.
Let $\{\hat\tau_{i1},\ldots, \hat\tau_{ip}\}$ be given by \rea{t1jdl}. Set
\ga
\nonumber
\cL E_i := E_{\boldsymbol{\hat\Lambda}_i}\circ (B_i)_*.
\end{gather}
Then ${\cL C}(\hat\tau_{i1}, \ldots, \hat\tau_{ip})
=\left\{\cL E_i \phi_0 \cL E_i^{-1}\colon
\phi_0\in{\cL C}(Z_1, \ldots,  Z_p)\right\}$. Set 
\ga
{\cL C}^{\mathsf c}(\hat\tau_{i1},\ldots, \hat\tau_{ip})
\colonequals\left\{\cL E_i \psi_1 \cL E_i^{-1}\colon
\psi_1\in{\cL C}^{\mathsf c}(Z_1,\ldots, Z_p)\right\}.\label{ccht}
\nonumber
\end{gather}
Each formal biholomorphic mapping $\psi$ admits a unique decomposition $\psi_1\psi_0^{-1}$
with $$\psi_1\in{\cL C}^{\mathsf c}(\hat\tau_{i1}, \ldots, \hat\tau_{ip}),\quad \psi_0\in
{\cL C}(\hat\tau_{i1}, \ldots, \hat\tau_{ip}).$$ If $\hat\tau_{ij}$ and $\psi$
are convergent, then $\psi_0,\psi_1$ are convergent.
  Assume further that $\tau_{2j}=\rho\tau_{1j}\rho$ with $\rho$ being 
given by \rea{eqrh}. Then $\rho\phi_1\rho\in \cL C^{\mathsf c}(\hat\tau_{11}, \ldots, \hat\tau_{1p})$ for $\phi_1\in \cL C^{\mathsf c}(\hat\tau_{21}, \ldots, \hat\tau_{2p})$.
\end{lemma}

\begin{prop}\label{ideal6}  
Let $\hat\tau_i,\hat\sigma$ be given by \rea{htaix}-\rea{hsixi} in which  $\log\hat M$ is in
 the 
  formal normal form \rea{hfj1}.
 Let $\{\hat\tau_{ij}\}$ be given by \rea{t1jdl}.  
  Assume further that the linear parts $T_{ij}$ of $\hat \tau_{ij}$ 
are in normal forms in \rla{unsol} or  \rpa{2tnorm}.
 Suppose that
 \ga\label{tauPhp}
\tau_{ij}=\Phi_i\hat\tau_{ij} \Phi_i^{-1}, \quad\tilde \tau_{ij}=\tilde\Phi_i\hat\tau_{ij}\tilde \Phi_i^{-1}
 \quad 1\leq j\leq p,\\
 \Phi_i\in\cL C(\hat\tau_i),\quad  \tilde\Phi\in \cL C(\hat\tau_i), \quad
 \tilde\Phi_i'(0)=\Phi_i'(0)= \mathbf I, \quad i=1,2.
\end{gather}
Then   $\{\Upsilon^{-1}\tau_{ij}\Upsilon\}=\{\tilde\tau_{ij}\}$
for $i=1,2$ and for some $\Upsilon\in {\mathcal C}(\hat \tau_1,\hat\tau_2)$,
if and only if there exist formal biholomorphisms  $\Upsilon, \Upsilon_1^*,\Upsilon_2^*$ such that 
\ga\label{quaeqp}
\Upsilon^{-1}\circ(( B_i)_*\circ Z_j\circ
(B_i)_*^{-1})\circ \Upsilon^{-1}=( B_i)_*\circ Z_{\nu_i(j)}\circ
(B_i)_*^{-1},\\
 \label{tpup+-}
\tilde \Phi_i= \Upsilon^{-1}\Phi_i\Upsilon_i^*\Upsilon,  \quad \Upsilon_i^*\in\cL C(\hat\tau_{i1},\ldots,\hat\tau_{ip}), \quad i=1,2,
\\
\Upsilon\hat\sigma\Upsilon^{-1}=\hat\sigma,
\label{el1-+}
\end{gather} 
where each $\nu_i$ is a permutation of $\{1, \ldots, p\}$.  Assume further that
 $\hat\tau_{2j}=\rho\hat\tau_{1j}\rho$ and $\Phi_2=\rho\Phi_1\rho$
and $\tilde\Phi_2=\rho\tilde\Phi_1\rho$. If $\Upsilon$    commutes with $\rho$, one can take
 $\Upsilon_2^*=\rho\Upsilon_1^*\rho$   and $\nu_2=\nu_1$.  
 \end{prop}
\begin{proof} Recall that
\eq{tijp+}
\nonumber
\tau_{ij}=\Phi_i\hat \tau_{ij}\Phi_i^{-1},\quad\Phi_i\in \cL C(\hat\tau_i); \qquad
\tilde\tau_{ij}=\tilde\Phi_i\widehat { \tau}_{ij}\tilde\Phi_i^{-1}, \quad\tilde\Phi_i\in \cL C(\hat\tau_i).
\eeq
Suppose that
\eq{UtUt+}
\Upsilon^{-1}\tau_{ij}\Upsilon=\tilde \tau_{i\nu_i(j)}, \quad j=1,\ldots, p,\  i=1,2.
\eeq 
By \rl{clas}, there are invertible $\Upsilon_i$ such that
\ga\label{el1++}
\Upsilon_i^{-1}\hat\tau_{ij}\Upsilon_i=\widehat { \tau}_{i\nu_i(j)},\quad 1\leq j\leq p,\\
\label{tpup+}
\nonumber
\tilde\Phi_i=\Upsilon^{-1}\circ\Phi_i\circ\Upsilon_i,\quad i=1,2. 
\end{gather}
Let us simplify the equivalence relation. 
  By \rt{ideal5}, $\cL C(\hat\tau_1,\hat\tau_2)$
consists of $2^p$ dilations $\Upsilon$ of the form
$(\xi,\eta)\to(a \xi,a\eta)$ with $a_j=\pm1$.
Since $\Phi_i,\tilde\Phi_i$ are tangent to the identity, then $D\Upsilon_i(0)$ is diagonal too. In fact   the linear part of $\Upsilon$ at the origin is
\eq{}
\nonumber
L\Upsilon_i=\Upsilon.
\eeq
   Clearly, $\Upsilon$ commutes with each  non-linear
transformation $  E_{\boldsymbol{\hat\Lambda_i}}$.    
Simplifying the linear parts of both sides of  \re{el1++}, we get
\eq{quaeq}
\Upsilon^{-1}\circ(( B_i)_*\circ Z_j\circ
(B_i)_*^{-1})\circ \Upsilon^{-1}=(\tilde B_i)_*\circ Z_{\nu_i(j)}\circ
(\tilde B_i)_*^{-1}.\eeq
 From the commutativity of $\Upsilon$ and $E_{\hat{\Lambda}_i}$ again and  the above identity, it follows that 
\eq{LUht}
\Upsilon^{-1}\circ\hat\tau_{ij}\circ \Upsilon=\widehat{\tau}_{\nu_i(j)}, \quad j=1,\ldots, p,\  i=1,2.
\eeq
Using \re{tauPhp} and \re{LUht}, we can rewrite \re{UtUt+} as
\eq{}
\nonumber
\Upsilon^{-1}\Phi_i\hat\tau_{ij}\Phi_i^{-1}\Upsilon=\tilde\Phi_i\Upsilon^{-1} \hat{\tau}_{ij}\Upsilon\tilde\Phi_i^{-1}.
\eeq
It is equivalent to  $\Upsilon_i^*\hat\tau_{ij}=\hat\tau_{ij}\Upsilon_i^*$, where  
\eq{dfUp}\nonumber
\Upsilon_i^*:=\Phi_i^{-1}\Upsilon\tilde\Phi_i   \Upsilon^{-1}.
\eeq
 Therefore, by \re{tpup},  in
$\cL C(\hat\tau_i)$, $\tilde\Phi_i $ and $\Phi_i $ are equivalent, if and only if
\eq{tppi}
\nonumber
\tilde \Phi_i= \Upsilon^{-1}\Phi_i\Upsilon_i^*\Upsilon,  \quad \Upsilon_i^*\in\cL C(\hat\tau_{i1},\ldots,\hat\tau_{ip}), \quad i=1,2.
\eeq
Conversely, if $\Upsilon_i^*$ satisfy the above identities, we take $\Upsilon_i=\Upsilon_i^*\Upsilon$. Note that
\re{el1-+} ensures that $\Upsilon$ commutes with $\hat\tau_i$ and $E_{\mathbf{\hat\Lambda_i}}$. Then \re{LUht}, or equivalently
\re{quaeq} as $\Upsilon$ commutes with $E_{\mathbf{\hat\Lambda}_i}$, 
gives us \re{el1++}.
\end{proof}
%
 
\pr{caseI} Let $\{\tau_{ij}\}$, $\{\tilde \tau_{ij}\}$, $\Phi_i$, and $\tilde\Phi_i$ be as in \rpa{ideal6}. 
Decompose $\Phi_i=\Phi_{i1}\circ\Phi_{i0}^{-1}$   with $\Phi_{i1}\in\cL C^{\mathsf c}(\hat\tau_{i1}, \ldots, \hat\tau_{1p})$
and $\Phi_{i0}\in\cL C (\hat\tau_{i1}, \ldots, \hat\tau_{1p})$, and decompose $\tilde\Phi_i$ analogously.
Then $\{\{\tau_{1j}\}, \{\tau_{2j}\}\}$ and $\{\{\tilde\tau_{1j}\}, \{\tilde\tau_{2j}\}\}$ are equivalent under a mapping
that is tangent to the identity  if and only if $\Phi_{i1}=\tilde\Phi_{i1}$ for $i=1,2$.  Assume further that
 $\tau_{2j}=\rho\tau_{1j}\rho$ and $\tilde\tau_{2j}=\rho\tilde\tau_{1j}\rho$. 
Then two families   
are equivalent under a mapping that is tangent to the identity and commutes with $\rho$ if and only if $\Phi_{i1}=\tilde\Phi_{i1}$. 
\end{prop}
\begin{proof} When restricting to changes of coordinates that are tangent to the identity, we have $\Upsilon=I$ in \re{UtUt+}.
Also \re{quaeqp} is the same as $\nu_i$ being the identity. 
By the uniqueness
of the decomposition $\Phi_i=\Phi_{i1}\Phi_{i0}^{-1}$, \re{tpup+-} becomes $\Phi_{i1}=\tilde\Phi_{i1}$. 
\end{proof}
We consider the following special case   without restriction on coordinate changes.  We will
assume that $M$ is a higher order perturbation of non-resonant product quadric. Let us recall $\hat \sigma$
be given by \re{hsixi} and define $\hat\tau_{ij}$ as follows:
\ga\label{7convnfsi}
\hat\sigma:\begin{cases}\xi'_j = \hat M_j(\xi\eta)\xi_i\\ \eta'_j=\hat M_j^{-1}(\xi\eta)\eta_j,\end{cases}\quad
\hat \tau_{ij}:\begin{cases}\xi'_i  =\hat \Lambda_{ij} (\xi\eta)\eta_j\\ \eta'_j = \hat\Lambda_{ij}^{-1}(\xi\eta)\xi_j\\ \xi'_k = \xi_k\\
 \eta'_k = \eta_k,\quad k\neq j\end{cases}
\end{gather}
with $\hat\Lambda_{2j}=\hat\Lambda_{1j}^{-1}$ and $\hat M_j=\hat\Lambda_{1j}^2$. 
Let $\hat\tau_i=\hat\tau_{i1}\cdots\hat\tau_{1p}$.  Recall that $E_{\boldsymbol{\hat\Lambda}_i}$ in \re{bise}. Set
\eq{}
{\cL C}^{\mathsf c}(\hat\tau_{11},\ldots, \hat\tau_{1p})
\colonequals\left\{E_{\boldsymbol{\hat\Lambda}_1} \psi E_{\boldsymbol{\hat\Lambda}_1}^{-1}\colon
\psi\in{\cL C}^{\mathsf c}(Z_1,\ldots, Z_p)\right\}. \label{7ccht}
\eeq
\begin{prop}\label{ideal6+} 
  Let $\{\tau_{11},\dots, \tau_{1p},\rho\}$ be the family of involutions associated with a real analytic submanifold
 $M$ that is a higher order perturbation of a non-resonant product quadric. 
 Assume further that the linear parts $T_{ij}$ of $\tau_{ij}$ 
are in normal forms in \rla{unsol}.
Let $\hat\sigma$ be the formal normal form $\hat\sigma$ of the $\sigma$ associated to $M$ that is 
given by  \rea{hsixi} in which  $\log\hat M$ is 
in the 
 formal normal form  \rea{hfj1}.
  In suitable formal coordinates
  the   involutions $\tau_{ij}$ of $M$
  have the form 
  \eq{tPtP}
  \tau_{1j}=\Psi\hat\tau_{ij} \Psi^{-1},\quad \tau_{2j}=\rho \tau_{1j}\rho, \quad
  \quad \Psi\in\cL C(\hat\tau_1)\cap C^c(\hat\tau_{11}, \ldots, \hat\tau_{1p}).
   \eeq 
Moreover,  if $\tilde\tau_{11}, \dots, \tilde\tau_{1p}$ have the form \rea{tPtP} in which $\Phi$
is replaced by $\tilde\Phi$. Then there exists a formal mapping $R$ commuting with $\rho$ and transforms
the family 
 $\{\tilde\tau_{11},\dots,\tilde\tau_{1p}\}$ into 
$\{ \tau_{11},\dots, \tau_{1p}\}$  if and only if  $R$ is an $R_\e$ defined by \rea{upsi}
and \eq{redequiv}
\tilde\Psi=R_\e^{-1}\Psi R_\e.
\eeq
In particular, $\{\tau_{11},\dots,\tau_{1p},\rho\}$ is formally equivalent to $\{\hat\tau_{11},\ldots,\hat\tau_{1p},\rho\}$
if and only if $\Psi$ in \rea{tPtP} is the identity map.
 \end{prop}
\begin{proof} We apply \rp{ideal6} with $\mathbf B_1=\mathbf B_2=\mathbf I$.
We need to refine the equivalence relation \re{quaeqp}-\re{el1-+}.  First we know that $\re{el1-+}$
means that $\Upsilon$ is some $R_\e$. Since $R_\e$ is diagonal, then \re{quaeqp} is always true
for $\nu_1=\nu_2=I$.  It remains to refine \re{tpup+-}.
 We have $\Phi_2=\rho\Phi_1\rho$.  
We know that  $\Upsilon$ is  a dilation of the form 
$$\xi_j\to\e_j\xi_j,\quad\eta_j\to\e_j\eta_j, \quad 1\leq j\leq p, \quad \e_j=\pm1.
$$
Since  $\mathbf B_1=\mathbf I$,  
then    $\Phi_1\in\cL C^{\mathsf c}(\hat\tau_{11}, \ldots, \hat\tau_{1p})$ implies that 
    $\Upsilon^{-1}\Phi_1\Upsilon\in \cL C^{\mathsf c}(\hat\tau_{11}, \ldots, \hat\tau_{1p})$;  and $\Upsilon$
commutes with each $\hat\tau_{1j}$.   
 By the uniqueness of decomposition,  \re{tpup+-} becomes
\eq{}
\nonumber
\tilde\Phi_{11}=\Upsilon^{-1}\Phi_{11}\Upsilon, \quad \tilde\Phi_{10}^{-1}=\Upsilon^{-1} \Phi_{10}^{-1}\Upsilon^*\Upsilon.
\eeq
The second equation defines $\Upsilon_1^*$ that is in $\cL C(\hat\tau_{11}, \ldots, \hat\tau_{1p})$
as $\Upsilon, \Phi_{i0}, \tilde\Phi_{i0}$ are in the centralizer. 
Rename $\Phi_{11},\tilde\Phi_{11}$ by $\Psi,\tilde\Psi$. 
This shows that the equivalence relation is reduced to  \re{redequiv}.  \end{proof}

We now derive the following formal normal form.
\begin{thm}\label{mnorm} Let $M$ be a real analytic submanifold   that is a higher order perturbation of a non-resonant 
product quadric.
Assume that the formal normal form $\hat\sigma$ of the $\sigma$ associated to $M$ is 
given by  \rea{hsixi} in which  $\log\hat M$ is tangent to the identity and in the 
formal normal form  \rea{hfj1}. Let $ E_{\boldsymbol{\hat\Lambda_1}}$ be defined by \rea{t1jd+l}.
Then $M$
is formally equivalent to a formal submanifold in the $(z_1, \ldots, z_{2p})$-space defined by
\begin{equation}\label{tMzp}
\nonumber
 \tilde M\colon   z_{p+j}=(\la_j^{-1}U_j(\xi,\eta)-V_j(\xi,\eta))^2, \quad1\leq j\leq p, \end{equation} 
where  $(U,V)=E_{\boldsymbol{\hat\Lambda_1}(0)} E_{\boldsymbol{\hat\Lambda}_1}^{-1}\Psi^{-1}$, 
 $\Psi$ is in $\cL C(\hat\tau_1)$ and 
$\cL C^{\mathsf c}(\hat\tau_{11},\ldots, \hat\tau_{1p})$, defined by \rea{7ccht}, and 
$\xi,\eta$ are solutions to
$$
z_j=U_j(\xi,\eta)+\la_jV_j(\xi,\eta), \quad \ov z_j=\ov{U_j\circ\rho(\xi,\eta)}+\ov\la_j\ov{V_j\circ\rho(\xi,\eta)},
\quad 1\leq j\leq p.$$
Furthermore, 
the $\Psi$ is uniquely determined up to   conjugacy  $R_\e\Psi R_\e^{-1}$ by an 
 involution  $R_\e\colon \xi_j\to\e_j\xi_j,
\eta_j\to\e_j\eta_j$ for $1\leq j\leq p$.
The formal holomorphic automorphism group of $\hat M$ consists of
 involutions of the form
$$
L_\e\colon z_j\to\e_j z_j, \quad  z_{p+j}\to z_{p+j}, \quad 1\leq j\leq p
$$
with $\e$ satisfying  $R_\e\Psi=\Psi R_\e$.  If the $\sigma$   associated to  $M$ is holomorphically equivalent to  a Poincar\'e-Dulac normal
form, then  $\tilde M$ can be achieved by a holomorphic transformation too.
\end{thm}
\begin{proof}  We fist choose linear coordinates so that the linear parts of $\{\tau_{11},\dots, \tau_{1p},\rho\}$ are in the
normal form in
\rla{t1t2sigrho}. 
We apply \rp{ideal6+} and assume that $\tau_{ij}$ are already in the normal form. The rest of proof is essentially in \rp{mmtp}
and we will be brief. Write $T_{1j}=E_{\boldsymbol{\hat\Lambda}_1(0)} \circ Z_j\circ
 E_{\boldsymbol{\hat\Lambda}_1(0)}^{-1}$. Let $\psi=(U,V)$ with $U,V$ being given in the theorem.  We obtain
 $$
\tau_{1j}=\psi T_{1j}\psi^{-1}, \quad 1\leq j\leq p. 
 $$
 Let $f_j=\xi_j+\la_j\eta_j$ and $h_j=(\la_j\xi_j-\eta_j)^2$.  The invariant functions of $\{T_{11}, \ldots, T_{1p}\}$ are generated by
 $f_1,\ldots, f_p,   h_1,\ldots, h_p$. This shows that the invariant functions of $\{\tau_{11}, \ldots, \tau_{1p}\}$ are generated by 
 $f_1\circ\psi, \ldots, f_p\circ\psi, h_1\circ \psi, \ldots, h_p\circ\psi$.
 Set $g:=\ov{f\circ\psi\circ\rho}$.
 We can verify that $\phi=(f,g)$ is biholomorphic. Now
 $
 \phi\rho\phi^{-1}=\rho_0. 
 $
 Then $\tilde M$ is defined by  
 $$
 z_{p+j}=E_j(z',\ov z'), \quad 1\leq j\leq p,
 $$
 where $ E_j= h_j\circ\phi^{-1}$.  Then  $E_j\circ\phi$ and $z_j\circ\phi=f_j$
 are invariant by $\{\tau_{1k}\}$. This shows that \{$\phi\tau_{ij}\phi^{-1}\}$ has the same
 invariant functions as deck transformations of $\pi_1$
 of the complexification of $\tilde M$.   By \rl{twosetin}, 
 $\phi\tau_{1j}\phi^{-1}$ agrees with the unique set of generators for
  the deck transformations of $\pi_1$.
 Then $\hat M$ is a realization of $\{\tau_{11}, \ldots, \tau_{1p},\rho\}$. 
  \end{proof}

\setcounter{thm}{0}\setcounter{equation}{0}

\section{
Normal forms of completely integrable commuting biholomorphisms}\label{nfcb}

In this section, we shall consider a family of commuting germs of  holomorphic diffeomorphisms at a common fixed point, say $0\in \cc^n$. 
We shall give conditions that ensure that the family can be transformed simultaneously and holomorphically to a normal form. This means that there exists a germ of biholomorphism at the origin which conjugates each   germ
of biholomorphism in the family to a mapping that commutes with the linear part of every mapping of the family. We can achieve this under two conditions:
\bpp
\item
 {\it The family is ``formally completely integrable''}. This means that the  the normal form of the family 
 has the ``same resonances'' as the normal form of  the family of the linear parts. 
\item {\it The family of linear parts  is of ``Poincar\'e type''}. In general, individually, each linear part might not satisfy these conditions.
They are satisfied,  collectively, by the family.
\epp

For our convergence proof, both conditions will be essential.
To be more specific, let ${\mathbf D}_1:=\text{diag}(\mu_{11},\ldots,\mu_{1n}),\ldots, {\mathbf D}_\ell:=\text{diag}(\mu_{\ell1},\ldots,\mu_{\ell n})$ be diagonal invertible matrices of $\cc^n$.
Let us consider a family $F:=\{F_i\}_{i=1}^{\ell}$ of  germs of holomorphic diffeomorphisms  of $(\cc^n,0)$ of which the linear of $F_i(x)$ at the origin is 
$$D_i\colon x\to {\mathbf D}_ix.
$$
Let us set $D:=\{  D_i 
\}_{i=1,\ldots \ell}$. Thus
$$
F_i(x)={\mathbf D}_ix+f_i(x), \quad   f_i(0)=0,\quad Df_i(0)=0.
$$

The group of germs of (resp. formal) biholomorphisms tangent to identity acts on 
the family $F$ by  
$\Phi_*F:=\{\Phi^{-1}\circ F_i\circ\Phi\colon 1\leq i\leq\ell\}$.

Let us denote ${\mathcal O}_n$ (resp. $\widehat{\mathcal O}_n$) the ring of
germs of holomorphic functions at the origin (resp. ring of formal power series)
of $\cc^n$. Let $Q=(q_1,\ldots, q_n)\in \nn^n$ and
$x=(x_1,\ldots,x_n)\in \cc^n$, we shall write
$$
|Q|:=q_1+\cdots +q_n,\quad x^Q:=x_1^{q_1}\cdots x_n^{q_n}.
$$

Let us specialize to a family $\{F_i\}_{i=1,\ldots \ell}$ of {\bf commuting germs of holomorphic diffeomorphisms},
that is that $F_i\circ F_j=F_j\circ F_i$ for all $1\leq i,j\leq \ell$. Since it generates an abelian group,  such a
family  is   said to be {\bf abelian}. We emphasize that the family does not necessarily form a group and
$$
\ell<\infty.
$$

Let us recall a result by M. Chaperon (see theorem 4 in \cite{Ch86}, page 132):
\begin{prop} 
If the family of diffeomorphisms is abelian then there exists a formal diffeomorphism $  \Phi$, which is tangent to the identity,
 such that
$$
{\widehat F_i}({\mathbf D}_jx)= {\mathbf D}_j{\widehat F_i}(x),\quad 1\leq i,j\leq \ell
$$
where ${\widehat F_i} :=   \Phi_*F_i$, for $1\leq i\leq \ell$.
We call the family  $\{\widehat F_i\}$ a {formal normal form} of the family $F$ with respect to the family $D$ of linear maps. 
\end{prop}

As mentioned above, for convenience, we have restricted ourselves to changes of holomorphic coordinates that are tangent to the identity. 
Also $\Phi_*\{F_i\}_{i=1}^\ell=\{\tilde F_i\}_{i=1}^\ell$ means that 
 $$
 \Phi_*F_i=\tilde F_i, \quad 1\leq i\leq \ell.
 $$
These  restrictions will be removed by mild changes.  For instance, if $\Phi$ transforms a family $F$ into a family $\hat F$ that commutes with $LF$, the family of the linear part of the $F$, then $(L\Phi)^{-1}(LF_i)L\Phi=L\hat F_i$. Therefore, $\Phi(L\Phi)^{-1}$ is tangent to the identity
and  transforms $F$ into $ (L\Phi) \hat F (L\Phi)^{-1}$ which commutes with $LF$.

Let $\widehat{\mathcal O}_n^D$ be the ring of formal invariants of the family $D$, that is
$$
\widehat{\mathcal O}_n^D:=\{f\in \widehat{\mathcal O}_n\,|\;f({\mathbf D}_ix)=f(x),\; i=1,\ldots, \ell  \}.
$$
As defined in Definition~\ref{ccst},  ${\mathcal C}_2(D)$ is  the ``non-linear formal centralizer'' of $D$, that is
$$
{\mathcal C}_2(D)=\{H\in (\widehat{\mathfrak M}^2_n)^n\,|\;H({\mathbf D}_ix)={\mathbf D}_iH(x),\; i=1,\ldots, \ell  \}.
$$
Here $\widehat{\mathfrak M}_n$ denotes the maximal ideal of the ring $\widehat{\mathcal O}_n$ of formal power series, that is the set of formal power series vanishing at the origin of $\mathbf C^n$.
Let  $e_j=(0,\ldots,0,1,0,\ldots,0)$ denote the $j$th unit vector of $\cc^n$.  If $Q\in\nn^n$ with $|Q|>0$, then
    $x^Q\in \widehat{\mathcal O}_n^D$ if and only if
\eq{muiq-}
\nonumber
\mu_i^Q=1,\quad \forall\; 1\leq i\leq \ell.
\eeq
Here $\mu_i^Q:=\mu_{i1}^{q_1}\cdots\mu_{in}^{q_n}.$
If $|Q|>1$, then   $x^Qe_j\in {\mathcal C}_2(D)$ if and only if
\eq{muiq}\nonumber
\mu_i^Q=\mu_{ij},\quad \forall\; 1\leq i\leq \ell.
\eeq
%
%

It can be shown (as in proposition 5.3.2 of \cite{St00,GW05}) that $\widehat{ \mathfrak M}_n^D$
is a ring generated by a finite number of monomials $x^{R_1},\ldots, x^{R_p}$ ($R_i\in\nn^n$) and
that the non-linear centralizer ${\mathcal C}_2(D)$ of $D$ is a module over $\widehat{ \mathfrak M}_n^D$ of finite type.
 
\begin{defn}
A formal normal form $\{\hat F_i\}_{i=1,\ldots, \ell}$ is said to be {\bf completely integrable} if
\begin{enumerate}
\item each 
$\hat F_i$ has the form 
$$
x_j' = \hat\mu_{ij}(x)x_j,\quad j=1,\ldots,n
$$
where $ \hat\mu_{ij}$ are  invariant by $D$ (i.e. $\hat\mu_{ij}(x)\in \widehat{\mathcal O}_n^D$) and satisfy $\hat\mu_{ij}(0)=\mu_{ij}$;
\item for each  $(j,Q)\in \{1,\ldots, n\}\times\nn^n$ with $|Q|\geq 2$,
$$
\hat\mu_{i}(x)^Q\equiv \hat\mu_{ij}(x)\;\text{for all $i=1,\ldots, \ell$},\quad\text{if and only if} \quad \mu_{i}^Q=\mu_{ij}\;\text{for all $i=1,\ldots, \ell$}.
$$

\end{enumerate}
\end{defn}
\begin{defn}
A commutative family of germs of diffeomorphisms $F$ is said to be formally (resp. holomorphically) {\bf completely integrable} if it is formally (resp. holomorphically) conjugated to a  completely integrable normal form.
\end{defn}
 \begin{rem}\label{rem-intg}
For a completely integrable normal form, we have that for each $Q\in\nn^n$, $\mu_i^Q=1$ for all $1\leq i\leq \ell$, if and only if $\hat\mu_i(x)^Q=1$ for all $i=1,\ldots, \ell$. Indeed, if $\hat\mu_i(x)^Q=1$ for all $i=1,\ldots, \ell$, then evaluation at zero give the result. On the other hand, if $\mu_i^Q=1$ for all $1\leq i\leq \ell$, then $\mu_{ij}\mu_i^Q=\mu_{ij}$ for all $1\leq i\leq \ell$. Hence, according to the definition, $\hat \mu_{ij}(x)\hat\mu_i^Q(x)=\hat\mu_{ij}(x)$ for all $1\leq i\leq \ell$, which gives $\hat\mu_i^Q(x)=1$ for all $1\leq i\leq \ell$.

\end{rem}

We recall from Definition \ref{ccst} (iii) that a formal diffeomorphism $\Phi$, tangent to identity, is {\it normalized} (w.r.t.~$D$) if it does not have components along the centralizer of $D$, i.e. $\Phi_{j,Q}=0$ if $\mu_i^Q=\mu_{ij}$ for all $i$,  $Q\in \nn^n$ with $|Q|\geq 2$.   Let  $\cL C^{\mathsf c}(D)$ denote the set of the normalized mappings,
and let $\cL C_2^{\mathsf c}(D)$ denote the set of mappings $\Phi-I$ with $\Phi\in \cL C^{\mathsf c}(D)$. 

\begin{lemma}\label{pcn0}
Any formal diffeomorphism $\Phi$ of $(\cc^n,0)$, tangent to identity, can be written uniquely
as $\Phi=\Phi_1\circ\Phi_0^{-1}$  with $\Phi_1\in\cL C^{\mathsf c}(D)$ and $\Phi_0\in \cL C(D)$. Furthermore, $\Phi_0,\Phi_1$
are convergent when $\Phi$ is convergent. 
\end{lemma}
\begin{proof} This follows from \rl{fhg-}, where $\hat{\cL H}$ is replaced by $\cL C_2(D)$ and $\pi$
is defined by
$$
\pi\left(\sum f_{j,Q}x^Qe_j\right)=\sum_j\sum_{x^Qe_j\in\cL C_2(D)}f_{j,Q}x^Qe_j.
\qquad \qedhere$$
\end{proof}
\begin{lemma}\label{lem-nf-nf}
Let $\hat F:=\{\hat F_i\}$ be a formal normal form of the abelian family of diffeomorphisms $F:=\{F_i\}$.  Let $\tilde F:=\{\tilde F_i\}$ be another formal normal form of $F$. Then, there exists a  
formal diffeomorphism $\Phi$, tangent to identity at the origin, such that $ \Phi\in  {\mathcal C}(D)$ and $\Phi\circ\tilde F_i=\hat F_i\circ\Phi$. 
Furthermore, there is a unique $\Phi\in\cL C^c(D)$ that transforms the family $F$ into a normal form.
\end{lemma}
\begin{proof}
Since both $\hat F$ and $\tilde F$ are normal forms of $F$,  there exists a formal diffeomorphism $\Phi$, tangent to identity at the origin, such that $\tilde F_i\circ\Phi=\Phi\circ \hat F_i$. According to  \rl{pcn0}, we can decompose  $\Phi=\Phi_1\circ\Phi_0^{-1}$ where $\Phi_0\in \cL C(D)$ and $\Phi_1\in \cL C^{\mathsf c}(D)$. Hence, we have $\Phi_1^{-1}\circ\tilde F_i\circ\Phi_1=\Phi_0^{-1}\circ \hat F_i\circ\Phi_0$.  Let us set $G_i:=\Phi_0^{-1}\circ\hat F_i\circ\Phi_0$. Then $G_i$ is a formal diffeomorphism satisfying $G_i(x)-{\mathbf D}_ix\in {\mathcal C}(D)$. 
Let us show by induction on $N\geq 2$ that if $\Phi_1=I+\Phi_1^N+O(N+1)$ with $\Phi_1^N$ being
homogeneous of degree $N$, then $\Phi_1^N=0$. Indeed, a computation  shows that
$$
\{G_i\}_N=\{\hat F_i\}_N+ D_i\circ \Phi_1^N- \Phi_1^N\circ D_i.
$$
Express   $\Phi_1^N$ as sum of monomial mappings. The monomial mappings are not
in $\cL C(D)$,  while those of $F_i$ and $G_i$ are. We obtain $\Phi_1^N=0$.

To verify the last assertion,  assume that $\Psi_*F=\hat F$ and $\tilde\Psi_*F
=\tilde F$
are in the normal form.  Suppose that $\Psi,\tilde\Psi$ are normalized. Then $(\Psi^{-1}\tilde\Psi)_*\hat F
=\tilde\Psi_*(\Psi^{-1})_*\hat F$ is in the normal form.  Write $\Psi^{-1}\tilde\Psi=\psi_1\psi_0^{-1}$ with $\psi_1\in\cL C^{\mathsf c}(D)$
and $\psi_0\in\cL C (D)$. Then $(\psi_1)_*\hat F$ is in a normal 
form. From the above proof, we know that $\psi_1=I$. Now $\Psi=\tilde\Psi\psi_0$, which implies that $\Psi=\tilde\Psi$.
\end{proof}
%
\begin{lemma}\label{complete-lemma}
If a formal normal form of $F$ is completely integrable so are all other normal forms of $F$; in particular, the unique $\Phi$ in \rla{lem-nf-nf} transforms $F$ into a completely integrable normal form. 
\end{lemma}
\begin{proof}
By Lemma \ref{lem-nf-nf}, we transform a normal form $\{\hat F_i\}$ into another one $\{\tilde F_i\}$ by applying a transformation $\Phi$ that commutes with each $D_j$. Hence, we have $\tilde F_i:=\Phi^{-1}\circ \hat F_i\circ \Phi$, for all $i=1,\ldots, \ell$. Let us write $\Phi(x)=\sum_{Q\in\nn^n,\; 1\leq j\leq n}\phi_{j,Q}x^Q e_j$.
Then
$$
\Phi\circ \hat F_i(x)=\sum_{Q\in\nn^n}\phi_{j,Q}\mu_{i}(x)^Qx^Q e_j.
$$
Suppose that $\{\hat F_i\}$ is completely integrable and $\Phi$ commutes with each $D_j$. Then
$$
\Phi\circ \hat F_i(x)= \text{diag}(\mu_{i1}(x),\ldots,\mu_{in}(x))\cdot \Phi(x).
$$
The conjugacy equation leads to
$$
\text{diag}(\mu_{i1}(x),\ldots,\mu_{in}(x))\cdot \Phi(x)=\begin{pmatrix}\tilde {F}_{i1}(\Phi(x))\\\vdots\\ \tilde {F}_{in}(\Phi(x))\end{pmatrix}.
$$
As a consequence, we have
$$
\tilde F_i(x)=\text{diag}((\tilde \mu_{i1}(x),\ldots,\tilde \mu_{in}(x))\cdot x
$$
with
$$
(\tilde\mu_{ij}\circ \Phi(x))\cdot\Phi_  j (x)= \mu_{ij}(x)\cdot\Phi_j(x),\quad \text{  i.e.  }  \tilde \mu_{ij}=\mu_{ij}\circ \Phi^{-1}.
$$
Each function $\tilde\mu_{ij}$ is  an invariant function of $D$ since
$$
\tilde\mu_{ij}({\mathbf D}_k x)= \mu_{ij}\circ \Phi^{-1}({\mathbf D}_kx)= \mu_{ij}\circ D_k (\Phi^{-1}(x))= \mu_{ij}\circ \Phi^{-1}(x).
$$
The second and third conditions of the definition of the complete integrability is obviously satisfied by $\{\tilde{F}_i\}$ since $\tilde \mu_{ij}=\mu_{ij}\circ \Phi^{-1}$.
\end{proof}
\begin{lemma}
If a formal normal form of $F$ is linear so are all other normal forms of $F$.
\end{lemma}
\begin{proof}
According to Lemma \ref{lem-nf-nf}, we transform a    linear normal form $\{\hat F_i\}$ into another one $\{\tilde F_i\}$ by applying a transformation $\Phi$ that commutes with each $D_j$. Since $ \hat F_i(x)={\mathbf D}_ix$, we have $\tilde F_i=\Phi^{-1}(D_i\Phi(x))$, for all $i=1,\ldots, \ell$. Since $\Phi$ commutes with each map $x\mapsto {\mathbf D}_ix$, then
$$
\tilde F_i=\Phi^{-1}(D_i\Phi(x))=\Phi^{-1}(\Phi({\mathbf D}_ix))={\mathbf D}_ix.
\qquad\qedhere
$$
\end{proof}

\begin{defn}\label{small divisors}
We  say that 
 {\it the family $D$ is of Poincar\'e type}
 if there exist  constants $d>1$ and $c>0$ such that, for each $(j,Q)\in\{1,\ldots, n\} \times\nn^n$
 that satisfies
  $\mu_{m}^{Q}-\mu_{m j}\neq 0$ for some $m$,
  there exists $(i,Q')\in \{1,\ldots, n\}\times \nn^n$ such that
$\mu_{k}^{Q'}=\mu_k^Q$ for all $1\leq k\leq \ell$,
$\mu_{i}^{Q'}-\mu_{ij}\neq0$,
 and
\gan
\max\left(|\mu_{i}^{Q'}|,|\mu_{i}^{-Q'}|\right)>c^{-1}d^{|Q'|}, \quad
 \text{ $Q'-Q\in\nn^n\cup(-\nn^n$)}.
\end{gather*}
 \end{defn} 
  Such a condition has appeared in the definition of the good set in~\cite{BHV10}.
\begin{defn}
Let $f=\sum_{Q\in\nn^n}f_Qx^Q$ and $g=\sum_{Q\in\nn^n}g_Qx^Q$ be two formal power series. We   say that $g$  {\it majorizes} $f$, 
 written as $f\prec g$, if $g_Q\geq 0$ and $|f_Q|\leq g_Q$ for all $Q\in\nn^n$.  Set
$$
\bar f := \sum_{Q\in\nn^n}|f_Q|x^Q.
$$
\end{defn}
\begin{thm}\label{thm-conv-nf}
Let $F$ be an abelian family of germs of holomorphic diffeomorphisms at the origin of $\cc^n$. Assume that it is formally completely integrable and   that its linear part at the origin is of Poincar\'e type.
Then $F$ is holomorphically conjugated to a normal form $\hat F=\{\hat F_1,\ldots, \hat F_\ell\}$
 so that 
 each
 $\hat F_i$ is defined by
$$
x_j' = \mu_{ij}(x)x_j,\quad j=1,\ldots,n
$$
where  $\mu_{ij}(x)$ are   germs of holomorphic functions  invariant under $D$ 
and $\mu_{ij}(0)=\mu_{ij}$.
 In fact, the unique normalized mapping $\Phi$ in \rla{lem-nf-nf}
is convergent. 
\end{thm}
The last assertion follows  from \rl{pcn0}
and \rl{complete-lemma}.
Such a result for commuting germs of vector fields is known \cite{St00} under a Brjuno-type of small divisor  conditions.
Such an integrability result for a single germ of two-dimensional hyperbolic real analytic area-preserving mapping was proved by Moser~\cite{moser-hyperbolic}.
 For a single germ of reversible biholomorphism of very special type, this result was due to Moser-Webster \cite{MW83}; 
indeed, as shown by  Moser-Webster \cite{MW83}[lemma 3.2], a  germ of {(hyperbolic reversible) mapping}
 of the form $\phi=\tau_1\tau_2$ where the $\tau_1,\tau_2$ are germs holomorphic involutions, is formally completely integrable under some condition on the linear parts at the origin of $\tau_1,\tau_2$.
Our proof is inspired from these proofs.   However, 
 in Moser-Webster situation, there is only two eigenvalues $\mu$ and $\mu^{-1}$ and the remaining
 eigenvalues are $1$ with multiplicity. The Poincar\'e   type condition in the above theorem, that is $|\mu|\neq 1$, is necessary   to obtain the convergence as demonstrated by Moser-Webster.
We shall use our result in the next section in order to normalize a special kind of CR-singularities.

\begin{proof}
Let us conjugate, simultaneously, each $F_i= {\mathbf D}_ix+f_i$ to $\hat F_i:=\hat{\mathbf D_i}(x)x$ by the action of $\Phi(x)=x+\phi(x)$ where $\phi(0)=0$ and $\phi'(0)=0$.  Here, $\hat{\mathbf D}_i(x)$ denotes the matrix $\text{diag}(\hat\mu_{i1}(x),\ldots,\hat\mu_{in}(x))$ and each $\hat\mu_{ij}(x)$ is a formal power series invariant under $D$, i.e. $\hat\mu_{ij}(x)\in\widehat{\mathcal O}_n^D$. We can assume that $\Phi$ does not have a non-zero component along the centralizer of $D$; indeed,  by \rl{complete-lemma}, we can assume that $\Phi$ is normalized 
 w.r.t $D$.
 Then, for each $i=1,\ldots, \ell$, we have
$$ 
F_i\circ\Phi(x)= {\mathbf D}_ix+f_i(\Phi)(x)+{\mathbf D}_i\phi(x), \quad \Phi\circ \hat F_i(x)= \hat{\mathbf {\mathbf D}_i}(x)x+ \phi(\hat F_i)(x).
$$
Equation $F_i\circ\Phi=\Phi\circ \hat F_i$  reads
\begin{equation}\label{conjugacy}
\left(\phi(\hat {\mathbf D}_i(x)x)-{\mathbf D}_i\phi(x)\right) +\left(\hat {\mathbf D}_i(x)-{\mathbf D}_i\right)x = f_i(\Phi)(x)\quad i=1,\ldots, \ell.
\end{equation}
Our convergence proof is  based  on two conditions: the existence of a formal $\phi\in\cL C^{\mathsf c}(D)$ that satisfies
the above equation, and the Poincar\'e type condition on the linear part $D$. We already know that
$\phi$ is unique. We shall project equation \re{conjugacy} along the ``non-resonant'' space (i.e. the space ${\mathcal C}^{\mathsf c}(D)$ of normalized mappings w.r.t. ${\mathbf D}$). The mapping $\phi$ also solves this last equation and we shall majorize it using that projected equation.

Let us first decompose these equations along the ``resonant'' and ``non-resonant'' parts,   i.e.
  $\cL C_2(D)$ and $\cL C^{\mathsf c}_2(D)$.  
 Since $\phi=\sum_{Q\in \nn^n, |Q|\geq 2} \phi_{j,Q}x^Qe_j$ is normalized then $\phi_{j,Q}=0$ for some $Q\in\nn^n$, $|Q|\geq 2$ and $1\leq j\leq n$, if we have $\mu_{m}^Q=\mu_{mj}$ for all $m$. We recall that, since each ${\mathbf D}_i$ is a diagonal matrix, then a map belongs to the centralizer of $D$ if and only if each monomial   map of its Taylor expansion at the origin belongs to this centralizer as well.
Since the $\hat \mu_{ij}$ is a  formal invariant function then
$$
\phi(\hat {\mathbf D}_i(x)x)=\sum_{Q\in \nn^n, |Q|\geq 2} \phi_{j,Q}\hat\mu_i^Q(x)x^Qe_j=:\sum_{Q'\in \nn^n, |Q'|\geq 2} \psi_{j,Q'}x^Qe_j.
$$
The latter 
contains only non-resonant terms, that is that if   $\mu_{i}^{Q'}=\mu_{ij}$ for all $i$, then $\psi_{j,Q'}=0$.
  Indeed,  $\hat\mu_i^Q(x)$ contains monomials of the form $x^P$ with $\mu_i^P=1$ for all $1\leq i\leq \ell$.   Hence, $\psi_{j,Q'}$ is a linear combination of $\phi_{j,Q}$ such that $Q'=Q+P$ with $\mu_i^P=1$ for all $i$. Therefore, if $\mu_{i}^{Q'}=\mu_{ij}$ for all $i$, then for all these $Q$'s, we have $\mu_{i}^{Q}=\mu_{i}^{Q'}=\mu_{ij}$ for all $i$ so that  
  $\phi_{j,Q}=0$; that is $\psi_{j,Q'}=0$.

Hence, the projection on  the resonant mappings in $\cL C_2(D)$ 
 leads to
\beq\label{res}
\left(\hat {\mathbf D}_i(x)-{\mathbf D}_i\right)x = \{f_i(\Phi)(x)\}_{res},\quad i=1,\ldots, \ell.
\eeq
Here  for any formal mapping $g(x)=O(|x|^2)$ on $\cc^n$,  we define the projection on $\cL C_2(D)$
by
\eq{gres}
  (g(x))_{res}=\sum_j\sum_{\forall i,\mu_i^Q= \mu_{ij}} g_{j,Q}x^Qe_j.
\eeq
  The projection $g$ onto $\cL C_2^{\mathsf c}(D)$ is defined as
$ 
g(x)-(g(x))_{res},
$ 
i.e. it is the projection of $g$ on the non-resonant mappings.

Let us consider the projection on the non-resonant mappings. We first need to decompose power series according to a non-homogeneous equivalence relation on
their coefficients.
Let us define the equivalence relation on $ \{1,\ldots,n\}\times\nn^n$ by
 $$
 (j,Q)\sim (\tilde{j},\tilde Q), \  \text{if $\mu_{ij}-\hat\mu_i^Q(x)=
 \mu_{i\widetilde{j}}-\hat\mu_i^{\tilde Q}(x)$ for all $1\leq i\leq \ell$}.
 $$
Here the identities hold as formal power series.
Let $\Delta$ be the set of the equivalent classes   on the non-resonant   multiindex set 
$$
 \left\{(j,Q)\in\{1,\dots, n\}\times
\nn^n\colon (\mu_1^Q-\mu_{1j},\dots, \mu_{\ell}^Q-\mu_{\ell j})\neq 0,  |Q|>1\right\}.
$$

If  
 $\mu_k^Q-\mu_{kj}\neq 0$ for some $k$,  clearly 
 $\hat\mu_k^Q-\mu_{kj}$ is not identically zero.
We can decompose any  formal power series map $f$ along these equivalent classes
and the resonant part  of the mapping. Let $\delta \in \Delta$ and $f=\sum_{Q=\in\nn^n, 1\leq j\leq n}  f_{j,Q}x^Qe_j$
  with $f=O(2)$. We can  write
\eq{fdelx}
 f_{\delta}(x):=\sum_{  (j,Q)\in \delta}f_{j,Q}x^Qe_j,   \quad \sum_{\del\in\Del}
 f_{\del}(x)\prec \ov f(x).
\eeq
We   denote by  $\widehat { \mathfrak M}_{n,\delta}^n$ the vector space of such maps.
To a given equivalent class $\delta$, we now associate a representative $(j_{\delta}, Q_{\delta})$, and
 later we shall identify an equation among $n$ equations in \re{conjugacy}
for estimation.

  Since $\phi$ contains no resonant  mappings, 
  then
\eq{phdi}
\phi=\sum_{\delta\in\Delta}\phi_\delta.
\eeq
Let us decompose the projection onto   non-resonant mappings in $\cL C_2^{\mathsf c}(D)$
of equation
 $(\ref{conjugacy})$ along each equivalent class $\delta$ as follows. 
Using the definition of the equivalence class $\Delta$, we obtain
\begin{equation}\label{conjugacy-delta}
\left[\hat\mu_i^{Q_{\delta}}(x)-\mu_{ij_{\delta}}\right]\phi_{\delta}(x) = \left\{f_i(\Phi)\right\}_{\delta}(x),\quad \forall i=1,\ldots, \ell
\end{equation}
where $\{f\}_{\delta}$ denotes the projection of $f$ on $\widehat { \mathfrak M}_{n,\delta}^n$, defined by \re{fdelx}.

For each   $(j_\del,Q_\del) 
\in\Delta$, we know that   $\mu_k^{Q_\del}-\mu_{kj_\del}\neq 0$  
for some $k$.  
By the Poincar\'e type condition, there exist $i$
and   $Q_{\delta}'\in\nn^n$ such that
\eq{mulq}\mu_{i}^{Q_{\delta}'}-\mu_{ij_{\delta}}\neq 0; \quad
\mu_{m}^{Q_{\delta}'}=\mu_{m}^{Q_{\delta}},\quad \forall  1\leq m\leq   \ell;\quad Q_{\delta}'-Q_\del\in\nn^n\cup(-\nn^n)
\eeq
and, furthermore,
one of the following holds:
\ga\label{miqp}
|\mu_{i}^{Q'_{\delta}}|\leq cd^{-|Q'_{\delta}|}, \\
\label{miqp+}|\mu_{i}^{-Q'_{\delta}}|\leq cd^{-|Q'_{\delta}|}.
\end{gather}
  Here, $d>1$ does not depend on $Q_{\delta}$.
So, let us use the $i$th equation of $(\ref{conjugacy-delta})$ to estimate $\phi_{\delta}$. We have, for that $i$,
\eq{pdhm-}
\phi_{\delta} = \left[\hat\mu_{i}^{Q_{\delta}}-\mu_{ij_{\delta}}\right]^{-1}\left\{f_{i}(\Phi)\right\}_{\delta}.
\eeq
Therefore,  we have established the uniqueness of $\phi$ under \re{phdi} and \re{pdhm-}, and under the
condition that $\phi$ satisfies the equation when \re{conjugacy} is projected onto $\cL C^{\mathsf c}(D)$. 
The existence of $\phi$
is ensured by assumption.   We now consider the convergence of $\phi$. 
By \re{mulq} and 
  Remark~\ref{rem-intg},  we obtain $\hat\mu_{i}^{Q_{\delta}'-Q_{\delta}}=1$.
This allows us to rewrite \re{pdhm-} as
\eq{pdhm}
\phi_{\delta} = \left[\hat\mu_{i}^{Q_{\delta}'}-\mu_{ij_{\delta}}\right]^{-1}\left\{f_{i}(\Phi)\right\}_{\delta}.
\eeq
We majorize this power series.

  Recall that $\hat\mu_{ij}(0)=\mu_{ij}$. Let us set
$$
  M_{ij}(x):=\mu_{ij}^{-1} \hat \mu_{ij}(x). 
$$
We have $M_{ij}(0)=1$ and we decompose
$$
 M_{ij}(x)=\sum_{Q\in\nn^n} M_{ij,Q}x^Q.
$$
Let us set $\mu^*:=\max_{ij}\{|\mu_{ij}|,|\mu_{ij}^{-1}|\}$, and
$$
m_{i}=\sum_{Q\in\nn^n} \max_{1\leq j\leq n}|M_{ij,Q}|x^Q,\quad m=\sum_{Q\in\nn^n} \max_{1\leq i\leq \ell,\;1\leq j\leq n}|M_{ij,Q}|x^Q.
$$
Note that $m(0)=1$. Then $M_{ij}\prec m$ and
$$
M_{ij}^{-1}=\frac{1}{1+(M_i-1)}\prec\f{1}{1-(m-1)} =\f{1}{2-m}.
$$
  Here and in what follows, if $f(x)$ is a formal power series with $f(0)=0$, then for any number 
$a\neq0$, $\frac{1}{a-f(x)}$ stands for the  formal power series in $x$ for
$$
\frac{1}{a}\left \{1+\sum_{n=1}^\infty(a^{-1}f(x))^n\right\}.
$$

To simplify notation in  \re{pdhm}, let us write $Q$ 
for $Q_{\delta}'$ and $j$ for $j_{\delta}$.
Fix $d_1$ with $1<d_1<d$.
We  consider the first case that $\mu^*cd^{-|Q|}>d_1^{-|Q|}$. Since $d>d_1$,  we have only finitely many such  $Q's$   (recall
that each $Q$ has the form $Q_\del'$).
  The function  $M_i\mapsto \mu_{ij}-\mu_i^QM_i^Q$ is holomorphic in $M_i\in\cc^p$ at $M_i= 
  (1,\ldots, 1)$ and does not vanish at this point. 
  Hence, the function
$$
(\mu_{ij}-\hat\mu_i^Q)^{-1}=(\mu_{ij}-\mu_i^QM_i^Q)^{-1}
$$
is also holomorphic at $M_i=(1,\ldots, 1)$. For all  $Q's$  in the first case, 
we
have
\aln
  (\mu_{ij}-\hat\mu_i^Q)^{-1}&\prec
\frac{C}{1-C(\ov M_{i1}-1+\cdots+\ov M_{in}-1)}\prec\frac{C}{1-nC(m-1)}.
\end{align*}
We now consider   the second case that  $\mu^*c d^{-|Q|}\leq d_1^{-|Q|}$.  In   case \re{miqp}, we obtain
\begin{eqnarray*}
(\hat\mu_i^{Q}-\mu_{ij})^{-1} & = & -\mu_{ij}^{-1}(1 - \mu_{ij}^{-1}\mu_i^QM_i^Q)^{-1}\\
&\prec &\mu^*\left[1- \mu^*cd^{-|Q|}m^{|Q|}\right]^{-1}
\\
&\prec &\mu^*\left[1-d_1^{-|Q|}m^{|Q|}\right]^{-1}\\
&\prec &\mu^*\left[1-d_1^{-1}m \right]^{-1}.
\end{eqnarray*}
In case \re{miqp+}, we have
\begin{eqnarray*}
(\hat\mu_i^{Q}-\mu_{ij})^{-1} & = & - \mu_{i}^{-Q}M_i^{-Q}\left[1 - \mu_{ij}\hat\mu_i^{-Q}M_i^{-Q}\right]^{-1}\\
&\prec &cd^{-|Q|}(2-m)^{-|Q|}\left[1-\mu^*cd^{-|Q|}(2-m)^{-|Q|}\right]^{-1}\\
&\prec &(\mu^*)^{-1}d_1^{-|Q|}(2-m)^{-|Q|}\left[1-d_1^{-|Q|}(2-m)^{-|Q|}\right]^{-1}\\
&\prec &(\mu^*)^{-1} \left[1-d_1^{-1}(2-m)^{-1}\right]^{-1}.
\end{eqnarray*}
  We have obtained the estimates for the second case.
Therefore, we have shown that for any $Q=Q_\del'$ and $1\leq j\leq \ell$,
\eq{hmiq}
(\hat\mu_i^{Q}-\mu_{ij})^{-1}\prec S(m-1).
\eeq
Here $S(t)$ is a convergent power series in $t$ that is independent of   all  $Q's$ of the form $Q_\del'$.

Let us set
$$
f^*:=\sum_{Q\in \nn^n}\max_{1\leq i\leq \ell,\;1\leq j\leq n}|f_{ij,Q}|x^Q e_j.
$$
  By the definition of the equivalence relation on multiindices, we have
\eq{phdi+}
\sum_{\delta\in\Delta}f_\delta^*\prec f^*.
\eeq
According to \re{pdhm} and \re{hmiq}, we have
$$
\phi_{\delta} \prec S(m-1)\left\{f^*(\bar \Phi)\right\}_{\delta}.
$$
  Now \re{fdelx} and \re{phdi+} imply
\beq\label{maj-nonres}
\phi \prec  S(m-1)f^*(\bar \Phi).
\eeq

Let us  project $(\ref{res})$ onto   the $k$th components of $\cL C_2(D)$ as follows.
For a power series   map $g$, we define
$$
g_{res,k}(x)=\sum_{\mu^Q=\mu_k}\quad g_{k,Q}x^Q.
$$
  By the definition of $g_{res}$ in \re{gres}, $g_{res}=(g_{res,1},\ldots, g_{res,n})$. 
 We have
$$
\mu_{ik}\left(M_{i,k}(x)-1\right)x_k =\left(\hat \mu_{ik}(x)-\mu_{ik}\right)x_k = \{f_{ik}(\Phi)\}_{res,k}(x).
$$
Therefore, for all $1\leq k\leq n$,
\beq\label{maj-res}
(m-1)x_k\prec \frac{1}{\min_{i,j}|\mu_{i,j}|}f^*(\bar \Phi).
\eeq
Let us set $\mu_*:=\frac{1}{\min_{i,j}|\mu_{ij}|}$. We
 set $x_1= t, \ldots, x_n= t$ 
 in $\ov\Phi(x)$ and $m(x)$.
  Let $\phi( t)$,   $\ov\Phi( t)$,  and $ m( t)$
still denote $\phi( t,\ldots,  t)$, $\ov\Phi( t,\ldots, t)$, and $m( t,\ldots, t)$, respectively. Let
$$
 t W( t) := \phi( t) + (m( t)-1) t.
$$
We have $W(0)=0$, $\phi( t)\prec  t W( t)$, and $(m( t)-1)\prec W( t)$.
 From estimates $(\ref{maj-nonres})$ and $(\ref{maj-res})$, we obtain
\beq \label{maj-W}
 t W( t)\prec \mu_*f^*(\bar \Phi( t))+S(m( t)-1)f^*(\bar \Phi( t)).
\eeq
Since  ${f_{ij}}(x)=O(|x|^2)$, 
 there exists a constant $c_1$ such that
$$
f^*(x)\prec \frac{c_1(\sum_j x_j)^2}{1-c_1(\sum_j x_j)}.
$$
Hence, estimate $(\ref{maj-W})$ reads
\begin{eqnarray} \label{rel-dom}
 t W( t)&\prec &\left(\mu_*+S(m( t)-1)\right)\frac{c_1(n( t+\phi))^2}{1-c_1n( t+\phi)} 
 \\
&\prec & \left(\mu_*+S( W( t))\right)\frac{c_1 t^2(n(1+W( t)))^2}{1-c_1n t(1+W( t))}.\nonumber
\end{eqnarray}
Let us consider the equation in the unknown $U$ with $U(0)=0$~:
\beq\label{equ-maj}
U( t)(1-c_1n t(1+U( t)))= \left(\mu_*+S( W( t))\right)c_1 t(n(1+U( t)))^2.
\eeq
According to the implicit function theorem, there exists a unique germ of holomorphic function $ U( t)$, solution to $(\ref{equ-maj})$ with $U(0)=0$.
According to inequality $(\ref{rel-dom})$, the function $W$ is dominated by $U$~: $W( t)\prec U( t)$. This can be seen by induction on the degree of the Taylor polynomials at the origin. Therefore, $W$ converges at the the origin. The theorem is proved.
\end{proof}

\setcounter{thm}{0}\setcounter{equation}{0}

\section{Real manifolds with an abelian CR-singularity}
\label{abeliancr}

Let us consider a real analytic manifold $M$ with a CR-singularity at the origin, as in section~\ref{secinv}. We assume that its complexification ${\mathcal M}$ has the maximum number of deck transformations with respect to each projection $\pi_1$ and $\pi_2$. The deck transformations are then generated by germs of holomorphic involutions of $(\cc^{2p},0)$, which are denoted by
\eq{2taus}\nonumber
\{\tau_{11},\ldots,\tau_{1p}\},\quad \{\tau_{21},\ldots,\tau_{2p}\}.
\eeq
  Recall that both families are abelian, that is that
\eq{tijk}\nonumber
 \tau_{ij}\circ \tau_{ik}= \tau_{ik}\circ \tau_{ij}.
 \eeq
They are intertwined by the antiholomorphic involution $\rho$:
$$ 
\tau_{2j}=\rho\circ\tau_{1j}\circ\rho.
$$ 

Let us consider the following germs of holomorphic diffeomorphisms~:
\ga\label{sigma_i}
\sigma_i:=\tau_{1i}\circ\tau_{2i},\quad 1\leq i\leq e_*+h_*,\\
\label{sigma_i+}
\sigma_s:=\tau_{1s}\circ\tau_{2 (s_*+s)},\quad \sigma_{s+s_*}=\tau_{1 (s+s_*)}\circ\tau_{2s},\quad e_*+h_*<s\leq p-s_*.
\end{gather}
Notice that    the above  property holds for quadrics of  the complex case by
\rp{sigs}. 
The family $\{\sigma_i\}$  is reversible with respect to $\rho$. More precisely, we have the following relations
$$
\sigma_i^{-1}=\rho\sigma_i\rho, \quad 1\leq i\leq e_*+h_*;\quad
\sigma_{s+s_*}^{-1}=\rho\sigma_s\rho, \quad e_*+h_*<s\leq p-s_*.
$$

\begin{defn}
We   say that the manifold $M$ has an {\bf abelian CR-singularity} at the origin if its complexification ${\mathcal M}$ has the maximum number of deck transformation and if the family $\{\sigma_1,\dots,\sigma_p\}$ of germs of biholomorphisms at the origin of $\cc^{2p}$ is abelian, i.e. $$\sigma_i\sigma_j=\sigma_j\sigma_i.$$
\end{defn}
\begin{defn}
A
{\it product quadric} is a submanifold in $\cc^{2p}$ defined by
\begin{alignat*}{3}
z_{p+e} &= (z_e+2\gaa_e\ov z_e)^2,\quad &&1\leq e\leq e_*\\
z_{p+h}&= (z_{h}+2\gaa_{h}\ov z_h)^2,\quad&& e_*+1\leq h\leq e_*+h_*\\
z_{p+s}&= (z_s+2\gamma_s\ov z_{s+s_*})^2,\quad && 
\\
z_{p+s+s_*}&=( z_{s+s_*}+2   (1-\ov\gamma_{s})\ov z_{s}) ^2,\quad&& e_*+h_*<s\leq p-s_*
\end{alignat*}
with
$
 0<\gaa_e<1/2, 
 \gaa_h>1/2,$ and $
 \gaa_s\in   (1/2,\infty)\times i(0,\infty).
$
\end{defn}

In what follows, we assume that $M$ has an {\it abelian CR-singularity} at the origin and that $M$ is a {\it higher order perturbation of a product quadric}.

The aim of this section is to show that such an analytic perturbation with an abelian CR-singularity and no hyperbolic component is holomorphically conjugate to a normal form. We shall give two proofs of this result. The first one rests on Moser-Webster result \cite{MW83}[theorem 4.1] applied successively to each $\sigma_i$. The other one is based on the fact that the family $\{\sigma_i\}$ is formally completely integrable and their linear part is of Poincar\'e type. We then apply \rt{thm-conv-nf}.

\subsection{Normal forms  for real submanifolds with  an abelian CR singularity
}
\begin{thm}\label{abelinv}
Let $M$ be a germ of real analytic submanifold in $\cc^n$ at an
 abelian CR-singularity at the origin. Suppose that $M$  is a higher order perturbation of a product quadric of which 
 $\gaa_1, \ldots, \gaa_p$ satisfy \rea{0ge1}.  Assume that the associated $\sigma$ of $M$
 has distinct eigenvalues. 
Suppose that $M$ does not have a 
 hyperbolic component $($i.e. $e_*\geq 0, s_*\geq0, h_*=0)$.
Then there exists a germ of biholomorphism $\psi$ that commutes with $\rho$ and such that, for $1\leq i\leq p$ and
$k=1,2$
\ga\label{convnfsi}
\psi^{-1}\circ\sigma_{i}\circ\psi:\begin{cases}\xi'_i = M_i(\xi\eta)\xi_i\\ \eta'_i= M_i^{-1}(\xi\eta)\eta_i\\ \xi'_j= \xi_j\\ \eta'_j= \eta_j, \quad j\neq i,\end{cases}\quad
 \psi^{-1}\circ\tau_{ki}\circ\psi:\begin{cases}\xi'_i  = \Lambda_{ki} (\xi\eta)\eta_i\\ \eta'_i = \Lambda_{ki}^{-1}(\xi\eta)\xi_i\\ \xi'_j = \xi_j\\
 \eta'_j = \eta_j,\quad j\neq i.\end{cases}
\end{gather}
 Moreover, we have
\begin{alignat}{4}
\label{lam1e}
\Lambda_{1e}&=\overline{\Lambda_{1e}\circ{\rho_z}},\quad &&1\leq e\leq e_*\\
\label{lam1s}\Lambda_{1s}&=\overline{\Lambda_{1(s+s_*)}^{-1}\circ {\rho_z}},\quad&& e_*<s\leq p-s_*\\
\label{lam2j}\Lambda_{2j}&=\Lambda_{1j}^{-1}, \quad&&
1\leq j\leq p.
\end{alignat}
\end{thm}
\begin{proof}
  We will present two  convergence proofs: one is based on a convergent
  theorem of Moser and Webster and another is based on  \rt{thm-conv-nf}.
 We first use some formal results obtained by Moser and Webster~\cite{MW83} and  some results in section~\ref{nfcb}. 

Since $M$ is a higher order perturbation of a product quadric, there are linear coordinates such that, for $1\leq i\leq p$ and
$k=1,2$, $\tau_{k,i}$ and $\sigma_i$ are higher order perturbations of 
\gan
S_{i}:\begin{cases}\xi'_i = \mu_i\xi_i\\ \eta'_i= \mu_i^{-1}\eta_i\\ \xi'_j= \xi_j\\ \eta'_j= \eta_j, \quad j\neq i,\end{cases}\quad
 T_{ki}:\begin{cases}\xi'_i  = \lambda_{ki}\eta_i\\ \eta'_i = \lambda_{ki}^{-1}\xi_i\\ \xi'_j = \xi_j\\
 \eta'_j = \eta_j,\quad j\neq i.\end{cases}
\end{gather*}
For elliptic coordinates, this was computed in \cite{MW83} and recalled in \re{tau1e}. For complex coordinates, this is computed in \re{linears},\re{taus}. 
Recall that 
$$
\sigma_m=\tau_{1m}\circ\tau_{2m}, \quad 1\leq m\leq p.
$$
Since $|\mu_1|\neq 1$, then by theorem 4.1 of Moser-Webster (\cite{MW83}), there is a  unique convergent transformation $\psi_1$
 normalized w.r.t. $S_1$
 such that  $\sigma_1^*:=\psi_1^{-1}\circ\sigma_{1}\circ\psi_1$ and $ \tau_{i1}^*:=\psi_1^{-1}\circ\tau_{i1}
 \circ\psi_1$
 are given by   \ga\label{sigma1norm}
\sigma_1^*\colon 
\begin{cases}x'_1 = M_1(\xi,\eta)\xi_1\\ \eta'_1= M_1^{-1}(\xi,\eta)\eta_1\\ \xi'_j= \xi_j\\ \eta'_j= \eta_j, \quad j\neq 1,\end{cases}\qquad
 \tau_{k1}^*\colon 
 \begin{cases}\xi'_1  = \Lambda_{k1} (\xi,\eta)\eta_1\\ \eta'_1 = \Lambda_{k1}^{-1}(\xi,\eta)\xi_1\\ \xi'_j = \xi_j\\
 \eta'_j = \eta_j,\quad j\neq 1.\end{cases}
\end{gather}
Here $k=1,2$.  
It is a  simple fact (e.g. see  \rl{complete-lemma}, $D=\{ S_1\})$ that there is a unique 
$\phi_1\in \cL C^c(S_1)$ such that $\phi^{-1}\sigma_1\phi$ is in the centralizer of $S_1$. 
Therefore,  $\phi_1=\psi_1$ is also convergent. 

Furthermore, we have $M_1(\xi,\eta)=\Lambda_{11}(\xi,\eta)\Lambda_{21}^{-1}(\xi,\eta)$;
and $\Lambda_{11}, \Lambda_{21}, M_1$ are invariant by $ S_1$. In the new coordinates, let
us denote $\tau_{im},\sigma_m$ by the same symbols for $m>1$. However, $\sigma_1=\sigma_1^*$
and $\tau_{k1}=\tau_{k1}^*$. Since each $\sigma_m$ commutes with $\sigma_1$, then
$\sigma_m$ is in the centralizer of $S_1$. Indeed, according to \cite{MW83}[Lemma 3.1](or \rl{pcn0} with $D=\{ S_1\}$), we can decompose $\sigma_m=\sigma_m^1\sigma_m^0$ where $\sigma_m^1$ is normalized w.r.t $S_1$ and $\sigma_m^0$ is in the centralizer of $ S_1$.  Write $\sigma_1\sigma_m=\sigma_m\sigma_1$ as
 $$(\sigma_m^1)^{-1}\sigma_{1}\sigma_m^1=\sigma_m^0\sigma_{1}(\sigma_m^0)^{-1}.$$
Since $\sigma_m^0\sigma_{1}(\sigma_m^0)^{-1}$ belongs to ${\mathcal C}(S_1)$, so does $(\sigma_m^1)^{-1}\sigma_{1}\sigma_m^1$.
Then applying the uniqueness of $\psi_1$ stated earlier to $\sigma_m^1$, we conclude that $\sigma_m^1=I$
and $\sigma_m=\sigma_m^0$ is in the centralizer of $S_1$. 

Let us verify that $\sigma_m^0$ or in general  each (formal) transformation $\var$
  in $\cL C(S_1)$ preserves the form of $\sigma_{1}^*$ and $\tau_{i1}^*$. 
Indeed, $\varphi^{-1}$ commutes with $S_1$ too. Thus $\varphi^{-1}\sigma_1^*\varphi$ commutes with $S_1$ and its linear part is $S_1$.  The linear part of $\varphi_1$ must preserve the eigenspaces of $S_1$ and hence it is given by
$$
\xi_1\to a\xi_1,\quad \eta_1\to b\eta_1, 
\quad (\xi_*,\eta_*) \to \phi(\xi_*,\eta_*)
$$
for $\xi_*=(\xi_2,\dots, \xi_n)$ and $\eta_*=(\eta_2,\dots, \eta_n)$. By a simple computation, the linear part of $\varphi^{-1}\tau_{k1}^*\varphi$ still has  the form \re{sigma1norm}.   According to \cite{MW83}[lemma 3.2], there a unique normalized mapping $\Psi$ that normalizes $\varphi^{-1}\sigma_1^*\varphi$  and  the $\varphi^{-1}\tau_{k1}^*\varphi$'s. According to the uniqueness property of \rl{lem-nf-nf}, $\Psi=Id$. Therefore, $\var$ preserves the forms of $\tau_{i1}^*$ and $\sigma_1^*$. 
 
Let $\psi_2$ be the unique biholomorphic map  normalized w.r.t. $S_2$  such that $\psi_2^{-1}\sigma_2\psi_2=\sigma_2^*$ and $\psi_2^{-1}\tau_{k2}\psi_2=\tau_{k2}^*$ are in the  normal form~:  
\ga\label{sigma2norm}
\sigma_{2}^*:
\begin{cases}\xi'_2 = M_2(\xi,\eta)\xi_2\\ \eta'_2= M_2^{-1}(\xi,\eta)\eta_2\\ \xi'_j= \xi_j\\ \eta'_j= \eta_j, \quad j\neq 2,\end{cases}\quad
\tau_{k2}^*:\begin{cases}\xi'_2  = \Lambda_{k2} (\xi,\eta)\eta_2\\ \eta'_2 = \Lambda_{k2}^{-1}(\xi,\eta)\xi_2\\ \xi'_j = \xi_j\\
 \eta'_j = \eta_j,\quad j\neq 2.\end{cases}
\end{gather}
Here $k=1,2$, and  $M_2$ and $\Lambda_{k2}$ are invariant by $S_2$. 
Since $\sigma_2$ commutes with $S_1$, we have
$$
(S_1^{-1}\psi_2S_1)^{-1}\circ\sigma_2\circ(S_1^{-1}\psi_2S_1)=S_1^{-1}\sigma_2^*S_1.
$$
Note that  $S_1^{-1}\sigma_{2}^*S_1$ (resp. $S_1^{-1}\tau_{k2}^*S_1$) has the form \re{sigma2norm}  
 in which $M_{2}$ (resp. $\Lambda_{k2}$) is replaced by $M_{2}\circ S_1$ (resp. $\Lambda_{k2}\circ S_1$). In other words $S_1^{-1}\sigma_{2}^*S_1$ and $S_1^{-1}\tau_{k2}^*S_1$ are still of the form  \re{sigma2norm}. Since $S_1$ is diagonal, then
$S_1^{-1}\psi_2S_1$ remains normalized w.r.t. $S_2$. Applying the above uniqueness on $\psi_2$ for 
$\sigma_{2}$, we conclude that $\psi_2=S_1^{-1}\psi_2S_1$. This shows that $\psi_2$
preserves the forms of $\tau_{k1}^*$ and $\sigma_1^*$. By the same argument as above, we have $\sigma_m^*\in {\mathcal C}(S_1,S_2)$. 
 
In summary, we have found holomorphic coordinates so that $\tau_{ij}=\tau_{im}^*$ and $\sigma_m=\sigma_m^*$ for $m=1,2$.   As mentioned 
previously, we know that $\sigma_1^*, \sigma_2^*, \sigma_3, \ldots, \sigma_m$
commute with $S_1$ and $S_2$. In particular, $M_1, M_2$ are invariant by $S_1, S_2$. 
Repeating  this procedure,  we find  a holomorphic map $\phi$   so that all $\phi^{-1}\sigma_j\phi=\sigma_j^*$
and $\phi^{-1}\tau_{kj}\phi=\tau_{kj}^*$ are in the normal forms. Furthermore, $M_i$ and $\Lambda_{k,i}$ are invariant by $\{S_1, \ldots, S_p\}$.

By \rl{pcn0}, we decompose $\phi=\phi_1\phi_0^{-1}$ where $\phi_1$ is normalized w.r.t. $\{S_1, \ldots, S_p\}$
and $\phi_0$ is in the centralizer of $\{S_1, \ldots, S_p\}$.  Then 
$\phi_1^{-1}\sigma_j\phi_1=\sigma_j^*$
and $\phi^{-1}_1\tau_{ij}\phi_1=\tau_{ij}^*$ are in the normal forms,  since $\phi_0$ commutes with $S_j$. 
We want to show that $\phi_1$ commutes with $\rho$.

Note that  $\sigma_e^{-1}=\rho\sigma_e\rho$ and $\sigma_{s+s_*}^{-1}=\rho\sigma_s\rho$.
Thus $(\rho\phi_1\rho)^{-1}\sigma_j(\rho\phi_1\rho)=\tilde\sigma_j^*$ where $\tilde\sigma_e^*:= \rho(\sigma_e^*)^{-1}\rho$ and $\tilde\sigma_s^*:= \rho(\sigma_{s+s*}^*)^{-1}\rho$. According to \re{eqrh}, we see that $\rho\phi_1\rho$
is still normalized w.r.t. $\{S_1,\ldots, S_p\}$.
By \rla{lem-nf-nf},  we  know that there is a unique normalized formal  mapping $\phi_1$ such that $\phi_1^{-1}\sigma_j\phi_1$ are in the centralizer of $\{S_1,\ldots,
S_p\}$. Since $\tilde\sigma_j^*$ belongs to the centralizer of $\{S_1,\ldots,S_p\}$, then we have $\rho\phi_1\rho=\phi_1$. 

 Now, $\tau_{2j}^*=\rho\tau_{1j}^*\rho$ follows from $\tau_{2j}=\rho\tau_{1j}\rho$.
This shows that
\begin{eqnarray} 
\nonumber
  \Lambda_{2e}&=&\overline{\Lambda_{1e}^{-1}\circ\rho},\quad 1\leq e\leq e_*,\\
\Lambda_{2s}&=&\overline{\Lambda_{1 (s+s_*)} \circ \rho}, 
\nonumber
\\ 
\nonumber
\Lambda_{2(s+s_*)}&=&\overline{\Lambda_{1 s} \circ \rho},\quad e_*+h_*<s\leq p-s_*. 
\nonumber
\end{eqnarray}
Let $\phi_2$ be defined by $$
\xi_j'=(\Lambda_{1j}^{1/2}M_j^{1/4})(\xi\eta)\xi_j, \quad
\eta_j'=(\Lambda_{1j}^{-1/2}M_j^{-1/4})(\xi\eta)\eta_j, \quad
1\leq j\leq p.$$ For  a suitable choice of the roots, we have $\phi_2\rho=\rho\phi_2$. Furthermore, $\phi_2$ preserves all invariant functions of $\{S_1,\ldots,S_p\}$. Hence, each $\phi_2^{-1} \circ\phi_1^{-1}\circ\tau_{ki}\circ\phi_1\circ\phi_2$ has the form $\tau_{kj}^*$ stated in 
 \rt {abelinv}.

 \medskip
 
 We now present another proof by using the more general \rt{thm-conv-nf}.
 
Note that the above proof is valid at the formal level without using the convergence result of Moser and Webster. 
More specifically, 
 if $\tau_{ij}$ are given by formal power series with 
$\sigma_1, \ldots, \sigma_p$ commuting pairwise, there exists a formal map $\psi$ that is tangent to the identity and commutes
with $\rho$ such that \re{convnfsi} holds. Since each $\mu_j$ is not a root of unity,  then
\re{convnfsi} implies that the   conjugate family $\{\sigma_m^*\}$ is a completely integrable normal form. 

%
 
Let $\sigma_i$ be defined as above. Let $S_i$ be its linear part at the origin of $\cc^{n}$. The eigenvalues $\{\mu_{ij}\}_{1\leq j\leq n}$ of $S_i$ are either $\mu_i$, $\mu_i^{-1}$ or $1$. More precisely, if $Q\in\nn^n$, $|Q|\geq 2$ then
\eq{muiij}
\mu_m^Q-\mu_{mj}= \mu_m^{q_m-q_{m+p}} - \begin{cases}\mu_m\quad\text{if }j=m \\\mu_m^{-1}\quad\text{if }j=m+p \\1\quad\text{otherwise. }\end{cases}
\eeq
We need to verify the condition that the family of linear part  $\{ S_1, \ldots,  S_p\}$ is 
of  the Poincar\'e type. So we can apply \rt{thm-conv-nf}.
 
Suppose that $(j,Q)\in\{1,\ldots, 2p\}\times\nn^{2p}$ satisfies
$$
\mu_l^Q-\mu_{l j}\neq0
$$
for some $1\leq l\leq 2p$. Set $d=\{\min_i\max(|\mu_i|,|\mu_{i}^{-1}|)\}^{1/{(2p)}}$.
 We define
$$
Q'=Q-\sum_{i=1}^p\min(q_i,q_{i+p})(e_i+e_{i+p}):=(q_1',\ldots, q_{2p}').
$$
Then $\mu_i^Q=\mu_i^{Q'}$ for all $i$. Take $i=l$ if $|Q'|\leq 2p$. In this case, we easily get
\eq{miqp9}
\mu_i^{Q'}-\mu_{ij}\neq0, \quad |\mu_i^{Q'}|>c^{-1}d^{|Q'|}
\eeq
by choosing a sufficiently large $c$. Assume that $|Q'|>2p$.
Take $i$ such that
$$
q_{i}+q_{i+p}=\max_k (q_{k}+q_{k+p}).
$$
Then $q_i+q_{i+p}\geq |Q'|/p>1$. By \re{muiij}, we get the first inequality in \re{miqp9}.
We note that $(q'_i,q'_{i+p})=(q_i,0)$ or $(0,q_{i+p})$. Thus
$$
\max(|\mu_i^{Q'}|,|\mu_i^{-Q'}|)=(\max(|\mu_i|,|\mu_i|^{-1}))^{q_i+q_{i+p}}\geq d^{|Q'|}.
$$
This shows that $\{D\sigma_1(0),\ldots, D\sigma_p(0)\}$ is of the Poincar\'e type.

We now apply Theorem~\ref{thm-conv-nf} as follows. 
We decompose $\psi=\psi_1\psi_0^{-1}$ such that $\psi_1\in\cL C^{\mathsf c}( S_1, \ldots,
 S_p)$ and $\psi_0\in\cL C( S_1, \ldots,  S_p)$. Then each $\sigma_i^*=\psi_1^{-1}\sigma_i\psi_1$
still has the form in \re{convnfsi}; in particular, $\{\sigma_1^*, \ldots, \sigma_p^*\}$ is a completely 
 integrable formal normal form. By \rt{thm-conv-nf}, $\psi_1$ is convergent.  Now, $\psi_1^{-1}\tau_{kj}\psi_1
 =\psi_0^{-1}(\psi^{-1}\tau_{kj}\psi)\psi_0$ are still of the form \re{convnfsi}; however \re{lam1e}-\re{lam2j}
might not hold. 
As in the first proof, we can verify
 that $\psi_1\rho=\rho\psi_1$. Applying another change of coordinates that commutes with $\rho$ and
 each $ S_j$ as before, we achieve \re{convnfsi}-\re{lam2j}.  The proof of the theorem is complete.
\end{proof}

\begin{rem}\label{presform} 
When $M$ is non-resonant and   $\log \hat M$ is tangent to the identity, 
 we    apply \rt{ideal5} to obtain a further
holomorphic change of coordinates so that $(M_1,\ldots, M_p)$ are uniquely determined by the real analytic submanifold. 
Then by  \rp{ideal6+},   $\{\hat\tau_{i1},\dots,\tau_{ip}\}$, $i=1,2$, are formally equivalent to $\{\hat\tau_{i1},\dots,
\hat\tau_{ip}\}$, $i=1,2$,  defined by \re{7convnfsi}, if and only if the formal map $\Psi$
in \rp{ideal6+} is the identity.
In other words, $M$ has an abelian CR singularity at the origin if and only if 
$\Phi-I$ vanishes.
\end{rem}

As a corollary of \rt{abelinv}, we have the following normal form for real submanifolds. In order to study
the holomorphic flatness and hull of holomorphy,
 we choose a realization similar to the case
of Moser-Webster for $p=1$. 
\begin{thm}\label{abelm}
Let $M$ be a germ of  real analytic submanifold at an abelian CR singularity.  Assume that $M$ is a higher order perturbation of a product quadric of which 
 $\gaa_1, \ldots, \gaa_p$ satisfy \rea{0ge1}.
Suppose that $M$  has no hyperbolic component of complex tangent at the origin.   Suppose that the associated $\sigma$ of $M$ has distinct
 eigenvalues $\mu_1, \ldots, \mu_p$,
$\mu_1^{-1}, \ldots, \mu_p^{-1}$.
Then $M$ is holomorphically equivalent to
\begin{gather}\label{Hmpp}
\widehat M\colon
z_{p+j}=\Lambda_{1j}(\zeta)\zeta_j,\quad 1\leq j\leq
p,
\end{gather}
where $\zeta=(\zeta_1,\ldots, \zeta_p)$ are the convergent solutions to
\begin{align}\label{ze+L}
\zeta_e&=\frac{1+\Lambda_{1e}^{2}(\zeta)}{(1-\Lambda_{1e}^2(\zeta))^{2}}|z_e|^2-\frac{\Lambda_{1e}(\zeta)}{(1-\Lambda_{1e}^2(\zeta))^{2}}(z_e^2+\ov z_e^2),\\ 
\zeta_s&=\frac{\Lambda_{1s}(\zeta)+\Lambda_{1s}^3(\zeta)}{(1-\Lambda_{1,s}^2(\zeta))^2}z_s\ov z_{s+s_*}-\frac{\Lambda_{1s}(\zeta)}{(1-\Lambda_{1s}^{2}(\zeta))^2}
(z_s^2+\Lambda_{1s}^2(\zeta)\ov z_{s+s_*}^2),\\
\label{zsss9}\zeta_{s+s_*}&= \frac{\Lambda_{1(s+s_*)}(\zeta)+\Lambda_{1(s+s_*)}^3(\zeta)}{(1-\Lambda_{1(s+s_*)}^2(\zeta))^2}\ov z_s z_{s+s_*}\\
&\quad -\frac{\Lambda_{1(s+s_*)}(\zeta)}{(1-\Lambda_{1(s+s_*)}^{2}(\zeta))^2}
(z_{s+s_*}^2+\Lambda_{1(s+s_*)}^2(\zeta)\ov z_{s}^2).
\nonumber
\end{align}  
Here $\Lambda_{1j}(\zeta)=\la_j+O(\zeta)$ $(1\leq j\leq p)$ satisfy \rea{lam1e}-\rea{lam1s}.
In particular, $\widehat M$ is contained in $z_{p+e}=\ov z_{p+e}$ and $z_{p+s}=\ov z_{p+s+s_*}$. 
\end{thm} 
By \rl{disga2},  that $\sigma$ has distinct eigenvalues is equivalent  to $\gaa_1, \ldots, \gaa_p$ being distinct. 
\begin{proof}   We use a realization which is different from \re{realM2}. 
We assume
that $M$ already has the normal form as in \rt{abelinv}. Thus for $j=1,\ldots, p$, we have
\eq{ta1j}
\tau_{1j}\colon\xi_j'
=\Lambda_{1j}(\xi\eta)\eta_j,\quad \eta_j'=\Lambda_{1j}^{-1}(\xi\eta)\xi_j, \quad (\xi_k',\eta_k')=(\xi_k,\eta_k), \quad k\neq j.
\eeq
 Let us define
$$
f_j(\xi,\eta)=\xi_j+\xi_j\circ\tau_{1j}, \quad
g_j=\ov{f_j\circ\rho}, \quad 1\leq j\leq p. 
$$
The latter implies that the biholomorphic mapping $\varphi (\xi,\eta)= (f(\xi,\eta),
g(\xi,\eta))$ transforms $\rho$ into the standard
complex conjugation $(z',w')\to(\ov w',\ov z')$. 
Define $$
F_{j}(\xi,\eta)=\xi_j\circ\tau_{1j}(\xi,\eta) \xi_j, \quad
1\leq j\leq p.
$$
 Using the expressions of $\tau_{1j}$ given by \re{ta1j}, we verify that  $f_j$ and $F_{j}$ 
are invariant by $\tau_{1k}$.  
Note that the linear part of $f_j(\xi,\eta)$ is $\xi_j+\la_j\eta_j$
for $1\leq j\leq p$, and the
quadratic part of $F_{j}(\xi,\eta)$ is $\la_j\xi_j^2$. 
By \rl{twosetin},  $f_1,\dots, f_p$ and $F_{1},\dots, F_{p}$
 generate all invariant functions 
of $\{\tau_{11},\ldots,\tau_{1p}\}$. 

Next using $\ov{\Lambda_{1 e}\circ{\rho_z}}=\Lambda_{1e}$ and $\ov{\Lambda_{1s}\circ{\rho_z}}=\Lambda_{1(s+s_*
)}^{-1}$,
we rewrite $z_j=f_j(\xi,\eta), w_j=g_j(\xi,\eta)$ as
\begin{alignat*}{5}
 &\xi_e&&=\frac{z_e-\Lambda_{1e}(\xi\eta)w_e}{1-\Lambda_{1e}^2},
\quad &\eta_e&=\frac{w_e-\Lambda_{1e}(\xi\eta)z_e}{1-\Lambda_{1e}^2},\\
& \xi_s&&=\frac{z_s-\Lambda_{1s}^2(\xi\eta)w_{s+s_*}}{1-\Lambda_{1s}^2(\xi\eta)}, \quad
&\eta_s&=\frac{\Lambda_{1s}(\xi\eta)(w_{s+s_*}-z_s)}{1-\Lambda_{1s}^2(\xi\eta)},\\
 &\xi_{s+s_*}&&=\frac{z_{s+s_*}-\Lambda_{1(s+s_*)}^2(\xi\eta)w_{s}}{1-\Lambda_{1(s+s_*)}^2(\xi\eta)}, \quad
&\eta_{s+s_*}&=\frac{\Lambda_{1(s+s_*)}(\xi\eta)(w_{s}-z_{s+s_*}) }{1-\Lambda_{1(s+s_*)}^2(\xi\eta)}.
\end{alignat*}
Using the above formulae and  $w_j=\ov z_j$, we compute $\zeta_j=\xi_j\eta_j$ to 
obtain  \re{ze+L}-\re{zsss9}. 

Note that $F_j(\xi,\eta)=\zeta_j\Lambda_{1j}(\zeta)$. This shows 
that $z_{p+j}=F_j\circ\varphi^{-1}(z',\ov z')$ 
have the form \re{Hmpp}. Again, we use the
formula of $\tau_{1,k}$ to verify that $z=(z',z'')$
are invariant by all $ \var\tau_{1k}\var^{-1}$.   On the other hand, $z=(z',z'')$ generate invariant functions
of   the deck transformations
of $\pi_1$ for the complexification of $\hat M$ given by \re{Hmpp}.  This shows that $\{\var\tau_{11}\var^{-1}, \ldots, \var\tau_{1p}\var^{-1}\}$
and the deck transformations of $\pi_1$, of which each family consists of commuting involutions, have the same
invariant functions.  By \rl{twosetin}, we know that the two families must be identical. 
This shows
that \re{Hmpp} is a realization for $\{\tau_{11},
\ldots, \tau_{1p},\rho\}$.  

To verify the last assertion of the theorem, we 
set $z_{p+e}^*=\ov z_{p+e}$, $z_{p+s}^*=\ov z_{p+s+s_*}$,
and $z^*_{p+s+s_*}=\ov z_{p+s}$. Set $\zeta_e^*=
\ov \zeta_e$, $\zeta_s^*=\ov \zeta_{s+s_*}$,
and $\zeta^*_{s+s_*}=\ov \zeta_s$.  We take complex conjugate
on  identities \re{Hmpp}-\re{zsss9}. By \rea{lam1e}-\rea{lam1s}, we have
$$
  \Lambda_{1e}=\ov{\Lambda_{1e}\circ{\rho_z}}, \quad \Lambda_{1s}=\ov{\Lambda_{ 1(s+s_*)}^{-1}\circ{\rho_z}}.
$$
We verify that $z^*,\zeta^*$ still satisfy \re{Hmpp}-\re{zsss9}, if $z_{j+p},\zeta_{j}$ are replaced by
$z_{p+j}^*,\zeta_j^*$, respectively, and $z_j$ are unchanged for $1\leq j\leq p$. By the
uniqueness of solutions $\zeta$ to \re{ze+L}-\re{zsss9}, we conclude that $\zeta^*_j=\zeta_j$.
Therefore, $z_{p+j}=z_{p+j}^*$. 
The proof is complete. 
\end{proof}
 
\subsection{Hull of holomorphy of  real submanifolds with an abelian CR singularity}
Let $X$ be a subset of $\cc^n$. We  
define the hull of holomorphy of $X$, denoted by $\cL H(X)$,
to be the intersection of domains of holomorphy in $\cc^n$ that contain $X$.

We assume that $M$ is real analytic and has a non-resonant complex
tangent at the origin of elliptic type only.
By \rt{abelm}, we may assume that $M$ is given by
\begin{gather*}
M\colon
z_{p+j}=\Lambda_{1 j}(\zeta)\zeta_j,\quad 1\leq j\leq
p,
\end{gather*}
where $\zeta_j=\zeta_j(z')$ $(j=1,\ldots, p)$ are the convergent real-valued solutions to
\begin{gather}\label{zj1L}
\zeta_j=\frac{1+\Lambda_{1 j}^{2}(\zeta)}{(1-\Lambda_{1 j}^2)^{2}(\zeta)}|z_j|^2-\frac{\Lambda_{1 j}(\zeta)}{(1-\Lambda_{1 j}^2(\zeta))^{2}}(z_j^2+\ov z_j^2), \quad 1\leq j\leq p.
\end{gather}
For  $\zeta\in\rr^p$ with small $|\zeta|$, we know that  $\Lambda_{1 j}(\zeta)>1$.

In a neighborhood of the origin in $\rr^p$, let us define the following germ of real analytic  diffeomorphism:   
$$
R\colon \zeta\to 
\left(\Lambda_{1 1}(\zeta)\zeta_1,\dots,\Lambda_{1 p}(\zeta)\zeta_p \right).
$$
 If $\e$ is small enough, for each $x''\in [0,\e]^{p}$, we can define $\zeta=R^{-1}(x'')$.
  Note that $R$ sends $\zeta_j=0$ into $x_{p+j}=0$ for each $j$. We can write 
 $$
 R^{-1}(x'')=(x_{p+1}S_1(x''), \ldots, x_{2p}S_p(x''))$$
  with $S_j(0)>0.$
 Then
$M\cap\{z''=x''\}$ is given by \re{zj1L}. 
For $x''\in[0,\e]^p$ let $D_j(x'')$ be the compact   set in  the $z_j$ plane whose boundary is defined 
by the $j$th equation in \re{zj1L} where $\zeta=R^{-1}(x'')$.  When $x_{p+j}>0$,   the boundary 
of $D_j(x'')$ is an ellipse with  
\eq{djxpp}
D_j(x'')\subset \Del_{C_1\sqrt{x_{p+j}}}. 
\eeq
Here and in what follows constants  will depend only on $\la_1,\ldots, \la_p$. 
 Thus 
$$
D(x''):=D_1(x'')\times\cdots\times D_p(x'')\times\{x''\}\subset\cc^p\times\rr^p
$$ is a product
of ellipses and its dimension equals the number of positive numbers among $x_{p+1},\ldots, x_{2p}$. We will call 
$D(x'')$ an {\it analytic polydisc} and  $\pd^*D(x''):=\pd D_1(x'')\times
\cdots\times\pd D_p(x'')\times\{x''\}$ its distinguished boundary.  Note that $\pd^*D(x'')$ is contained in $M$.
In fact, $M$ is foliated by $\pd^*D(x'')$ as $x''$ vary in $[0,\e]^p$ and $\e$ is sufficiently small. We will specify the value
of $\e$ later.  We will use this foliation  
and Hartogs' figures in   analytic polydiscs  to find the local hull of holomorphy of $M$  at the origin.

As $x''$ vary in $[0,\e]^p$, let
$
M^\epsilon
$
be the union of $\pd^*D(x'')$, and  $\cL H^\e$  the union of  $D(x'')$. Both $\cL H^\e$ and $M^\e$ are 
  compact
subsets in $\cc^{2p}$.   For any open ball $B^{\e_*}$ in $\cc^{2p}$ centered at the origin with radius $\e_*$,
$$
B^{\e_*}+M^\e:=\{a+b\colon a\in B^{\e_*}, b\in M^\e\}
$$
is contained in a given neighborhood of $M^\e$, if $\e_*$ is sufficiently small.  Analogously,  $B^{\e_*}+\cL H^\e$
is a   connected open neighborhood of $\cL H^\e$.  
Let us first verify that  a function that
is holomorphic in a connected neighborhood of $M^\e$ in $\cc^n$  extends holomorphically to a neighborhood of $\cL H^\e$. 
 Assume that $f$ is holomorphic in a neighborhood $\cL U$ of $\pd_*^\e D:=\cup_{x''\in[0,\e]^p}\pd^*D(x'')$. 
 We first note that  $\cL H^\e$ is defined by
 \begin{align}\label{xpj}
A_j(x'')|z_j|^2  -B_j(x'')(z_j^2+\ov z_j^2)&\leq x_{p+j}, \quad 1\leq j\leq p;\\
y''=0, \quad x''&\in[0,\e]^p\label{y''x''}
\end{align} 
with
\begin{gather}\nonumber
A_j(x'')=\frac{1+\Lambda_{1 j}^2(R^{-1}(x''))}{S_j(x'')(1-\Lambda_{1 j}^2(R^{-1}(x''))},
\quad
B_j(x'')=\frac{\Lambda_{1 j}(R^{-1}(x''))}{S_j(x'')(1-\Lambda_{1 j}^2(R^{-1}(x''))}.
\end{gather}
Let $\del$ be a small positive number.  For $x''\in[-\del,\e]^p$, let $D_j^\del(x'')\subset\cc$ be defined by
 \begin{align}\nonumber
A_j(x'')|z_j|^2  -B_j(x'')(z_j^2+\ov z_j^2)&\leq x_{p+j}+\del.
\end{align} 
 Fix $\del>0$
sufficiently small.  Let
$
P_{\e,\del}
$
(resp.~$\pd^*P_{\e,\del}$)
be the set of $z=(z',z'')$ satisfying the following: 
$y''\in[-\del,\del]^p$,  $x''\in[-\del,\e]^p$,  and $z_j\in D_j^{\del}(x'')$ (resp.~$z_j\in\pd D_j^{\del}(x'')$) for $1\leq j\leq p$.
Let $\cL U_{\e,\del}$ (resp. $\cL U_{\e,\del_1}$) be a small neighborhood of $P_{\e,\del}$ (resp. $P_{\e,\del_1}$). 
Assume that $0<\del_1<\del$ and $\del_1$ is sufficiently small.  We may also assume that $\cL U_{\e,\del_1}$
is contained in $\cL U_{\e,\del}$ and  $\pd^*P_{\e,\del}\subset \cL U.$
  Thus, for $(z',z'')\in \cL U_{\e,\del_1}$, we can define $$
F(z',z'')=\int_{\zeta_1\in \pd D_1^{\del}(x'')}
\cdots\int_{\zeta_p\in \pd D^{\del}_p(x'')}\f{f(\zeta',z'')\, d\zeta_1\cdots d\zeta_p}{(\zeta_1-z_1)\cdots(\zeta_p-z_p)}.
$$
When $z$ is sufficiently small, $F(z)=f(z)$ as $f$ is holomorphic near the origin.  Fix $z_0\in\cL U_{\e,\del_1}$. We want to show that
  $F$ is holomorphic at $z_0$.  So $F$ is a desired extension of $f$.
By  continuity, when $z=(z_1,\dots, z_{2p})$ tends  to $z_0$, $x''$ tends to $x_0''$
and  $\pd D_j^\del(x'')$ tends to $\pd D_j^{\del}(x_0'')$,
while $z_j\in D_j^{\del}(x_0'')$ when $z$ is sufficiently close to $z_0$.
  By Cauchy theorem, for $z$ sufficiently close to $z_0$ we change contour integrals 
  successively  to get
\aln
F(z',z'')=
\int_{\zeta_1\in \pd D_1^{\del}(x_0'')}\int_{\zeta_2\in \pd D_2^{\del}(x_0'')}
\cdots\int_{\zeta_p\in \pd D^{\del}_p(x_0'')}\f{f(\zeta',z'')\, d\zeta_1\cdots d\zeta_p}{(\zeta_1-z_1)\cdots(\zeta_p-z_p)}.\end{align*}
The set of integration is fixed. The integrand is holomorphic in $z$. Hence   $F$ is holomorphic at $z=z_0$.

Next we want to show that $\cL H^\e$ is the hull of holomorphy of $M^\e$ in $B^{2p}_{\e_0}$ for suitable
$\e,\e_0$ that can be arbitrarily small. 

Let us first show that 
 $\cL H^\e$ is the intersection  of domains of holomorphy in $\cc^n$. 
  Recall that  $ \cL H^\e$ is defined by \re{xpj}-\re{y''x''}.
Next, we define   for $\delta':=(\delta_1,\ldots,\delta_p)$ with $\del_j>0$ 
\begin{align*}
\rho_j^{\delta'}&=A_j(x'')|z_j|^2-B_j(x'')(z_j^2+\ov z_j^2)-x_{p+j}+(  \delta_1^{-1}+\cdots+ \delta_p^{-1})
\sum_{i=1}^py_{p+i}^2\\
&\quad +\sum_{i\neq j} \delta_i^{-1}
\left\{A_i(x'')|z_i|^2-B_i(x'')(z_i^2+\ov z_i^2)-x_{p+i}\right\}.
\end{align*}
When $p=1$, the last summation is $0$. 
 The complex Hessian of $\rho_j^{\delta'}$ is
\begin{align*}
&\sum_{\all,\beta=1}^{2p}\frac{\pd^2\rho^{\delta'}_j}{\pd z_\alpha\ov z_\beta}t_\alpha\ov t_\beta=A_j(x'')|t_j|^2
+\frac{ \delta_1^{-1}+\cdots+ \delta_p^{-1}}{2}\sum_i |t_{p+i}|^2+\sum_{i\neq j}\frac{1}{ \delta_i}A _i(x'')|t_i|^2\\
&\qquad +\RE\sum_k a_{jk}(x''; z_j)t_j \ov t_{p+k} +
\RE\sum_{k,\ell} b_{j,k\ell}(x''; z_j)t_{p+k} \ov t_{p+\ell} \\
&\qquad+\RE\sum_{i\neq j} \sum_k\frac{1}{ \delta_i} c_{j,ik}(x''; z_i)t_i \ov t_{p+k}
+\RE\sum_{i\neq j} \sum_{k,\ell}\frac{1}{ \delta_i} d_{j,k\ell}(x''; z_i)t_{p+k} \ov t_{p+\ell}.\end{align*}
Here $a_{jk}(x'';0)=b_{j,kl}(x'';0)= c_{j,ik}(x'';0)=d_{j,kl}(x'';0)=0$, 
and  $i,j,k,\ell$ are in $\{1,\ldots, p\}$.  From the Cauchy-Schwarz inequality,
it follows  that for $|z|<\e_0$ with $\e_0>0$ sufficiently
small and  $0<\delta_j<1$, 
\begin{align*}
\sum_{\all,\beta=1}^{2p}\frac{\pd^2\rho^{\delta'}_j}{\pd z_\alpha\ov z_\beta}t_\alpha\ov t_\beta&\geq\f{1}{2}A_j(x'')|t_j|^2
+\frac{ \delta_1^{-1}+\cdots+ \delta_p^{-1}}{4}\sum_j|t_{p+j}|^2+ \frac{1}{2}\sum_{i\neq j} \delta_i^{-1}A_i(x'')|t_i|^2.
\end{align*} 
Therefore, each $\rho_j^{\delta'}$ is  strictly plurisubharmonic on $|z|<\e_0$ for all $0< \delta_i<1$. 
 Hence for $ \del=(\del_0,\ldots, \del_p)=(\del_0,\del')\in(0,1)^{p+1}$, 
$$
\rho^\delta(z)=\max_j\{\rho_j^{\delta'},|y''|^2- \delta_0^2, x_{p+j}^2-\e^2\}
$$
is plurisubharmonic on the ball $B^{2p}_{\e_0}$.   By \re{djxpp},   
$ D(x'')$ is contained in $B^{2p}_{C_2\e^{1/2}}$ for $x''\in[0,\e]^p$.   We now fix $\e<(\e_0/{C_2})^2$ to ensure
\eq{dxbe}
D(x'')\subset B_{\e_0}^{2p}, \quad \forall x''\in[0,\e]^p. 
\eeq
This shows that
$$
 \cL H_\del^\e:=\{z\in  B_{\e_0}^{2p}\;|\;\colon \rho_\del^\e(z)<0\}
$$
is a domain of holomorphy.   

Let us verify that
\eq{}\nonumber
\cL H^{\e}=\bigcap_{ \delta_0>0,\ldots,\del_p>0}\cL H_\del^\e.
\eeq
Fix $z\in \cL H^\e$. From \re{dxbe} we get $z\in B_{\e_0}^{2p}$. 
We have $y''=0$.  Hence \re{xpj} hold and $x_{p+j}^2\leq\e^2$. Clearly, $\rho_j^ \delta(z)<0$ for
each $j$ and $ \delta\in(0,1)^p$. This shows that $z\in \cL H^\e$ is in the intersection.  For the other inclusion, 
let us assume that $z$ is in the intersection. Then $y''=0$ and $x''\in[0,\e]^p$. So \re{y''x''} holds.
 With $\rho_j^\del(z)<0$, we let $ \delta_i$ tend to $0$ for $i\neq j$.
We conclude
\begin{gather}\nonumber
A_i(x'')|z_i|^2-B_i(x'')(z_i^2+\ov z_i^2)\leq x_{p+i}
\end{gather}
for all $i\neq j$, and hence for all $i$ as $p>1$.  When $p=1$ the above inequality can be obtained directly from
$\rho_1^ \delta$.  We have verified \re{xpj}. This shows that $z\in\cL H^\e$.
 
In view of \re{xpj}-\re{y''x''}, the boundary of $\cL H^\e$ is the union $\cup_{j=1}^p\cL H_j^\e$ with  $\cL H_j^\e$
being defined by
\begin{alignat*}{2}
A_j(x'')|z_j|^2 -B_j(x'')(z_j^2+\ov z_j^2)&= x_{p+j},\\
A_i(x'')|z_i|^2 -B_i(x'')(z_i^2+\ov z_i^2)&\leq x_{p+i}, \quad 1\leq i\leq p,\  i
\neq j;\\
y''=0, \qquad x_{p+j}&\leq\e \quad 1\leq j\leq p.
\end{alignat*}

Therefore, we have proved the following theorem.
\begin{thm}\label{hullj}
 Let $M$ be a germ of real analytic submanifold at an abelian CR singularity. Assume that the complex tangent of $M$
is purely elliptic and has distinct eigenvalues
 at the origin. There is a base of neighborhoods $\{U_j\}$ of the origin in $\cc^n$ which satisfies the following: For each $U_j$,  the local hull of holomorphy $H(M\cap U_j)$ of $M\cap U_j$ is foliated by embedding complex submanifolds with boundaries. Furthermore, near the origin $H(M\cap U_j)$
is the transversal intersection of $p$ real analytic submanifolds of dimension $3p$ with boundary.
 The boundary of $H(M\cap U_j)$ contains $M\cap U_j$; and two sets are the same if and only if $p=1$. 
\end{thm}
 
 \begin{rem} The proof shows that the hull of $H(M\cap U_j)$ is foliated by analytic polydiscs, where an analytic polydisc is a biholomorphic embedding
of closed unit polydisc in some $\cc^k$ with $1\leq k\leq p$. 
\end{rem}

\setcounter{thm}{0}\setcounter{equation}{0}

\section{Rigidity of product quadrics}\label{rigidquad}

The aim of this section is to prove the following rigidity theorem: Let us consider 
a higher order analytic perturbation of a product quadric. If this manifold is formally equivalent to the product quadric, then under a small divisors  condition, it is also holomorphically equivalent to it.

The proof goes   as follows~: Since the manifold is formally equivalent to the quadric,   the associated sets of involutions $\{\tau_{1 i}\}$ and $\{\tau_{2 i}\}$ are simultaneously linearizable by a formal biholomorphism that commutes with $\rho$. In particular,  $\sigma_1,\ldots, \sigma_p$,  as defined by \re{sigma_i} and \re{sigma_i+}, are formally linearizable and  they   commute pairwise. These are germs of biholomorphisms with a diagonal linear part. According to \cite{stolo-bsmf}[theorem 2.1], this abelian family can be holomorphically linearized under a collective Brjuno type condition \re{bruno-cond}. Furthermore, the transformation commutes with $\rho$.
Then, we linearize simultaneously and holomorphically both $\tau_1:=\tau_{1 1}\circ\cdots\circ \tau_{1 p}$ and $\tau_2:=\tau_{2 1}\circ\cdots\circ \tau_{2 p}$ by a transformation that commutes with both $\rho$ and $\cL S$, the family of linear parts of the $\sigma_1,\ldots,\sigma_p$.
Finally, we linearize simultaneously and holomorphically both families $\{\tau_{1 i}\}$ and $\{\tau_{2 i}\}$ by a transformation that commutes with $\rho$, $\cL S$, $T_1$ and $T_2$. 

These last two steps will be obtained through a majorant method and the application of a holomorphic implicit function theorem. This is done in 
Proposition~\ref{linearS}. They first require a complete description of the various centralizers and their   associated normalized  mappings, i.e. suitable complements. This is a goal of Proposition~\ref{STiR}.

Throughout this section, we do not assume that  $\mu_1,\ldots, \mu_p$ are non resonant in the
sense that $\mu^Q\neq1$ if $Q\in\zz^p$ and $Q\neq0$.  In fact, we will apply our results to $M$
which might be resonant. 
However, we will retain the assumption that $\sigma$ has distinct
eigenvalues when we apply the results to the manifolds.

\

We recall from \re{sigma_i} and \re{sigma_i+}, the definition of germs of holomorphic diffeomorphisms~:
\begin{eqnarray}\label{sigmai1}
\sigma_i&:=&\tau_{1i}\circ\tau_{2i},\quad 1\leq i\leq e_*+h_*;\\
\sigma_s&:=&\tau_{1s}\circ\tau_{2 (s_*+s)},\\
\sigma_{s+s_*}&:=&\tau_{1(s+s_*)}\circ\tau_{2s},\quad e_*+h_*<s\leq p-s_*.
\label{sigmai3}
\end{eqnarray}
They satisfy 
$$
\sigma_i^{-1}=\rho\sigma_i\rho, \quad 1\leq i\leq e_*+h_*;\quad
\sigma_{s+s_*}^{-1}=\rho\sigma_s\rho, \quad e_*+h_*<s\leq p-s_*.
$$
Recall the  linear maps
\begin{align}
&S\colon\xi_j'=\mu_j\xi_j,\quad \eta_j'=\mu_j^{-1}\eta_j;\nonumber\\
\label{rSxi-}
&S_j\colon\xi_j'=\mu_j\xi_j,\quad \eta_j'=\mu_j^{-1}\eta_j,\quad \xi_k'=\xi_k,\quad \eta_k'=\eta_k, 
\quad k\neq j;\\
\label{rTij-} 
&T_{ij}\colon\xi_j'=\la_{i j}\eta_j,\quad\eta_j'=\la_{i j}^{-1}\xi_j,\quad\xi_k'=\xi_k,\quad
\eta_k'=\eta_k, \quad k\neq j; 
\\
\label{rRho-}
& \rho\colon \left\{\begin{array}{ll}
(\xi_e',\eta_e',\xi_h',\eta_h')=
(\ov\eta_e,\ov\xi_e,\ov\xi_h,\ov\eta_h),\vspace{.75ex}
\\
(\xi_{s}', \xi_{s+s_*}',\eta_{s}',\eta_{s+s_*}')=(\ov\xi_{s+s_*},
\ov\xi_{s},\ov\eta_{s+s_*}, \ov\eta_{s}).
\end{array}\right.
\end{align}

We need to introduce   notation for the indices to describe various centralizers 
regarding $T_{1j}, S_j$ and $\rho$.
We first introduce index sets for the centralizer of $\cL S, T_1,\rho$. We recall that $\cL S$ and
$\cL T_i$ denote the families $\{S_1,\ldots, S_p\}$ and   $\{T_{i 1}, \ldots, T_{ip}\}$,
respectively. Also, $T_i=T_{i 1}\circ\cdots \circ T_{i p}$ .

Let  $(P,Q)\in\nn^p\times\nn^p$ and $1\leq j\leq p$. By definition, $\xi^P\eta^Qe_j$ belongs to the centralizer of $\cL S$ if and only if it commutes with each $S_i$. In other words,  $\xi^P\eta^Qe_j\in\cL C(\cL S)$ if and only if
\eq{mukpk}
 \mu_k^{p_k-q_k}=1,\quad \forall k\neq j; \quad\mu_j^{p_j-q_j}=\mu_j.
\eeq
Note that the same condition holds for  $\xi^Q\eta^Pe_{p+j}$ to belong to $\cL C(\cL S)$. This leads us to define the set of multiindices
$$
\cL R_j:=\{(P,Q)\in\nn^{2p}\colon\mu_j^{p_j-q_j}=\mu_j,\   \mu_i^{p_i-q_i}=1, \forall i\neq j\}, \quad 1\leq j\leq p.
$$
We observe that if $(P,Q)\in\cL R_j$, then
\gan
p_j=q_j+1, \quad j\neq h; \quad p_i=q_i, \quad \forall i\neq j,h;\\
 \la_h^{p_h-q_h}=\pm1, \quad h\neq j.
\end{gather*}

For convenience, we define  for $P=(p_e,p_h,p_s,p_{s+s_*})$ and $Q=(q_e,q_h,q_s,q_{s+s_*})$
\al
\rho( P Q) &:=
( q_e ,p_h, p_{s+s_*}, p_{s}, p_e, q_h, q_{s+s_*}, q_{s}),
\nonumber
\\
\label{rhoapq}\rho_a(PQ) &:= (q_e ,p_h, p_{s+s_*}, p_{s}),\\
\label{rhobpq}\rho_b(PQ) &:=(p_e, q_h, q_{s+s_*}, q_{s}),\\
\ov f_{\rho( P Q)} &:=(\ov{f\circ\rho})_{ P Q}=:\ov {f_{\rho( P Q)}}.
\nonumber
\end{align}
Hence, we have $\xi^P\eta^Q\circ \rho = \ov\xi
^{\rho_a(PQ)}\ov\eta^{\rho_b(PQ)}$ as well as $\rho( P Q)=(\rho_a( P Q),\rho_b( P Q))$. 

According to \re{mukpk} and  equation $(\ref{rRho-})$ of $\rho$,  
the restriction of $\rho$ to  $\cL R_h$ is an involution, which will be denoted by $\rho_h$. Moreover, $\rho$ is  a bijection  $\rho_s$
from $\cL R_s$ onto $\cL R_{s+s_*}$.   We define an involution on $\cL R_e$ by
$$ 
\rho_e(PQ):=(\rho_b(PQ),\rho_a(PQ)).
$$ 
Note that  $\rho_e$ is not a restriction of $\rho$, and $\rho_s$ is not an involution either.


Next, we introduce sets of indices to be used to compute the centralizers on $\cL T_1,\cL T_2,\rho$. Set
\gan
\cL N_j:=\cL R_j\cap\{(P,Q) \colon p_i\geq q_i, \quad \forall i\neq j\}, \quad 1\leq j\leq p.
\end{gather*}
Note that when $p=1$, $\cL N_j=\cL R_j$ for $j=e$ or $h$.
Let us set
\gan
A_{jk}( P, Q):=
\max\{ p_k, q_k\}, \quad k\neq j,\quad
A_{jj}( P, Q)=
 p_j;\\
B_{jk}( P, Q):=
\min\{ p_k, q_k\}, \quad k\neq j,\quad
B_{jj}( P, Q)=
 q_j.
\end{gather*}
We define a mapping
$$
(A_j,B_j)\colon \cL R_j\to\cL N_j
$$
with
$$
A_j:=(A_{j1},\ldots, A_{jp}),\quad B_j:=(B_{j1},\ldots, B_{jp}).
$$
We notice that, for $(P,Q)\in \cL N_j$, $A_j\circ \rho_j(P,Q)=(p_e,p_h,p_{s+s^*}, p_s)$ and $B_j\circ \rho_j(P,Q)=(q_e,q_h,q_{s+s^*}, q_s)$. In other words, on  $\cL N_j$
 for $j=e$ or $h$, $A_j\circ\rho_j$ just interchanges the $s$th   and  the $(s+s_*)$th coordinates for each $s$, so does $B_j\circ\rho_j$.

Finally, 
 for $(P,Q)\in\cL R_j$ we define
\al \label{nujab}
\nu_{ P Q}& :=\begin{cases}\prod_{h'}\la_{h'}^{ q_{h'}- p_{h'}}, & j\neq h, \\
  \la_h^{p_h-q_h-1}\prod_{h'\neq h }\la_{h'}^{ q_{h'}- p_{h'}},& j=h;
\end{cases}
\\
\nu_{ P Q}^+ &:=\begin{cases}
\prod_{h'| q_{h'}> p_{h'}}\la_{h'}^{ q_{h'}- p_{h'}},   & j\neq h;\\
\label{nujab+}\prod_{h'\neq h, q_{h'}> p_{h'}}\la_{h'}^{ q_{h'}- p_{h'}}, & j=h.
\end{cases}
\end{align}
Here $e_*<h',h\leq e_*+h_*$. 
Note that $\nu_{PQ}^+$ is only defined for $(P,Q)\in\cL R_j$. 
For convenience, we however define
$$
\nu_{QP}:=\nu_{PQ}, \quad (P,Q)\in\cL R_j. 
$$
 If $p=1$ we set $\nu_{P Q}^{ +}=1$.
 \begin{lemma}\label{nunu+} Let $(P,Q)\in\cL R_j$. Then
\begin{alignat}{3} \label{nupm1}
 &\nu_{PQ}=\pm1,   \quad  (P,Q)\in\cL R_j, \ j\neq h;\qquad && \nu_{PQ}^+=\pm1;\\
 \label{nupm1+}
 &\nu_{\rho_e(PQ)}=\nu_{PQ}, \quad (P,Q)\in\cL R_e; \quad &&\nu_{\rho(PQ)}=\nu_{PQ}, \quad
 \forall (P,Q);\\
 \label{nupm1++}  
 & \nu_{\rho_e(PQ)}^+=\nu_{PQ}\nu_{PQ}^+, \quad (P,Q)\in\cL R_e; \quad &&\nu_{\rho(PQ)}^+=  \nu_{PQ}^+, \quad
 (P,Q)\in\cL R_h\cup\cL R_{s+s_*}.
 \end{alignat}
\end{lemma}
\begin{proof}
 From the definition
of $\cL R_j$, we have $(\la_i^{p_{i}-q_{i}})^2=\mu_i^{p_{i}-q_{i}}=1$ for $i=h'$
in \re{nujab}-\re{nujab+}.
We also have $\mu_h^{p_{h}-q_{h}-1}=1$ for terms in \re{nujab}-\re{nujab+}. Thus
\eq{Lhph}
\nonumber
\la_{h'}^{p_{h'}-q_{h'}}=\pm1,\quad \la_h^{p_h-q_h-1}=\pm1.
\eeq
    Thus  we obtain \re{nupm1}; the rest identities 
follow from the definition of $\rho_e$, $\rho$, and the above identities.    \end{proof}

\begin{lemma}\label{comput-rhoT1}
For all multiindices  $(P,Q)\in  \cL R_e\cup\cL R_h$, we have
\ga\label{lambda}
\ov{\la^{\rho_a( P, Q)-\rho_b( P, Q)}}=
\la^{ Q- P},\quad \ov{\mu^{ \rho_b-\rho_a}}=\mu^{P-Q},\\
\label{T1rho}
\xi^P\eta^Q\circ \rho\circ T_1= \la^{ Q- P}\ov\xi^{\rho_b(PQ)}\ov\eta^{\rho_a(PQ)},\\
\label{rhoS-1}
\xi^P\eta^Q\circ \rho\circ S^{-1}= \mu^{ P-Q}\ov\xi^{\rho_a(PQ)}\ov\eta^{\rho_b(PQ)}.
\end{gather}
\end{lemma}
\begin{proof}
The first identity in \re{lambda} follows from \re{rhoapq}-\re{rhobpq} and the fact that $\la_e$ and $\mu_e$ are reals, $\la_h^{-1}=\ov\la_h$,  $p_s=q_s$, and  $p_{s+s_*}=
q_{s+s_*}$. 
This gives us the first identity in \rea{lambda}, and the second identity follows from the first.
A direct computation shows that 
$$
\xi^P\eta^Q\circ \rho\circ T_1= \ov\la^{ \rho_a- \rho_b}\ov\xi^{\rho_b(PQ)}\ov\eta^{\rho_a(PQ)},\quad \xi^P\eta^Q\circ \rho\circ S^{-1}= \ov\mu^{ \rho_b-\rho_a}\ov\xi^{\rho_a(PQ)}\ov\eta^{\rho_b(PQ)}.
$$
The result follows from \re{lambda}.
\end{proof}
 
It is tedious to find  necessary and sufficient conditions to describe the centralizer of   $\cL T_1, \cL T_2, \rho$, as the mappings
in the families are non diagonal. There are different ways to described these conditions too.
 To keep computation relatively simple, we 
do not aim a minimum set of conditions. Of course, when we use the centralizers we will verify all the
sufficient conditions.  
\begin{prop}\label{STiR} Let $\cL S=\{S_1,\ldots, S_p\}$,  $\cL T_i=\{T_{i1},\ldots, T_{ip}\}$
and $\rho$ be given by \rea{rSxi-}-\rea{rRho-}.
Let $\var=I+(U,V)$ be a formal biholomorphic map
that is tangent to the identity.
\bppp
\item $\var\in{\cL C}(\cL S)$ if and only if
\ga\label{muimuj}
U_{j,PQ}=0=V_{j,QP}, \quad\forall (P,Q)\not\in\cL R_j.
\end{gather}
Also, $\var\in{\cL C}(\cL S,\rho)$ if and only if additionally
\gan
\label{uhus} U_{h,PQ}= \ov{U}_{h,\rho(PQ)}, \  (P,Q)\in\cL R_h;
\quad U_{s+s_*,PQ}=\ov{U}_{s,\rho(PQ)},\ (P,Q)\in\cL R_{s+s_*};\\
 V_{e,QP}=\ov U_{e,\rho(PQ)},\quad (P,Q)\in\cL R_e;\\
V_{h,QP}=\ov V_{h,\rho(QP)},\  (P,Q)\in\cL R_h; \quad V_{s+s_*,QP}=\ov{V}_{s,\rho(QP)}, \ (P,Q)\in\cL R_{s+s_*}.
\label{uhqp}
\end{gather*}
\item $\var\in{\cL C}(\cL S, T_1)$ if and only if \rea{muimuj} holds and
\eq{vjlju}
V_{j,QP}=\la_j^{-1}\la^{P-Q}U_{j,PQ}, \quad \forall (P,Q)\in\cL R_j.
\eeq
Also, $\var\in{\cL C}(\cL S,T_1,\rho)$ if and only if in addition to \rea{muimuj} 
 and \rea{vjlju}
\begin{alignat}{4}\label{uenu}
 U_{e,PQ}&= \nu_{PQ}\ov{U}_{e, \rho_e(PQ)},\quad && (P,Q)\in\cL R_e;\\
 U_{h,PQ}&= \ov{U}_{h,\rho(PQ)},\quad  &&(P,Q)\in\cL R_h;
 \label{uhnu}\\
\label{usnu} U_{s,PQ}&= \ov{U}_{s+s^*,\rho(PQ)}, \quad && (P,Q)\in\cL R_{s}.
\end{alignat}
\item
$\var\in{\cL C}(\cL T_1,\cL T_2)$ if and only if  in addition to \rea{muimuj} and \rea{vjlju}
\begin{eqnarray}
\label{ujrjnj}
U_{j, P Q}&=&\nu_{P Q}^+U_{j,(A_j,B_j)( P, Q)},\quad (P,Q)\in\cL R_j\setminus\cL N_j. 
\end{eqnarray}
Also, $\var\in{\cL C}(\cL T_1,\cL T_2,\rho)$ if and only if
additionally 
\begin{alignat}{4}
\label{sqnupq}
U_{j, P Q}&=  
 \nu_{PQ}^+ 
\ov U_{j ,(A_j,B_j)\circ \rho_j( PQ)},\quad &&(P,Q)\in  \cL N_j, \quad j=e,h;\\  
 U_{s+s_*,  P Q}&=\nu_{PQ}^+\ov U_{s,(A_s,B_s)\circ  \rho( P Q)}, \quad  &&(P,Q)\in  \nn_{s+s_*}.
\label{sqnupq2}
\end{alignat}
\eppp
\end{prop}
 We remark that condition \re{ujrjnj}   holds trivially when $(P,Q)\in \cL N_j$, in which case it becomes
 $U_{j,PQ}=U_{j,PQ}$.  
\begin{proof}
 To simplify  notation, 
we abbreviate 
$$
\rho_a=\rho_a(PQ), \quad\rho_b=\rho_b(PQ), \quad A_j=A_j(P,Q), \quad B_j=B_j(P,Q).
$$

Recall that $\la_e=\ov\la_e, \la_h=\ov\la_h^{-1}$ and $\la_{s+s_*}=\ov\la_s^{-1}$. By definition,
\ga\label{rSet-}\nonumber
S_e=T_{1e}T_{2e}, \quad S_h=T_{1h}T_{2h}, \quad S_{s}=T_{1s}T_{2(s+s_*)}, \quad S_{s+s_*}=T_{1(s+s_*)}T_{2s}.
\end{gather}
In the proof, we will   use the fact
  that $S_j$ is reversible by both involutions in the composition for $S_j$. In particular,
\eq{tijsj}
T_{1j}S_jT_{1j}=S_j^{-1},\quad \forall j.
\eeq
 However, we have $T_{2(s+s_*)}S_sT_{2(s+s_*)}=S_s^{-1}$
and $T_{2s}S_{s+s_*}T_{2s}=S_{s+s_*}^{-1}$.
For simplicity,  we will  derive identities by using \re{tijsj} and
\eq{se-1}
 S_e^{-1}=\rho S_e\rho, \quad S_h^{-1}=\rho S_h\rho, \quad S_{s+s_*}^{-1}=\rho S_s\rho.
\eeq
 Finally, we need one more identity.
Recall that
\begin{alignat*}{4}
&T_{1e}T_{2j}=T_{2j}T_{1e}, \quad j\neq e;\quad
&&T_{1h}T_{2j}=T_{2j}T_{1h}, \quad j\neq h;\\
&T_{1s}T_{2j}=T_{2j}T_{1s}, \quad j\neq s+s_*;\quad
&&T_{1(s+s_*)}T_{2j}=T_{2j}T_{1(s+s_*)}, \quad j\neq s.
\end{alignat*}
Therefore, for any $j$ we have the identity
\eq{sjt1}
T_1S_j T_1=S_j^{-1}.
\eeq
In what follows, we will derive all identities by using \re{tijsj}, \re{se-1} and \re{sjt1}, as well
as $S_iS_j=S_jS_i$,  $T_{1i}T_{1j}=T_{1j}T_{1i}$ and $T_2=\rho T_1\rho$.

(i)
 The centralizer of $\cL S$ is easy to describe.
Namely, $\var\in\cL C(\cL S)$  if and only if
\gan
\label{ujs=-}
U_j\circ S_j=\mu_jU_j, \quad U_j\circ S_k=U_j, \quad k\neq j,\\
\quad V_j\circ S_j=\mu_j^{-1}V_j,\quad V_j\circ S_k=V_j, \quad k\neq j.\label{vjsj}
\end{gather*}
 For $\var\rho=\rho\var$, we need
\ga\label{uhou}
  U_h=\ov {U_h\circ\rho},
\quad U_{s+s_*}=\ov{U_s\circ\rho},\\
V_{e}=\ov{U_e\circ\rho},
 \quad V_h=\ov {V_h\circ\rho}, \quad V_{s+s_*}=\ov{V_s\circ\rho}.
\label{ueve}
\end{gather}

(ii) Suppose that $\var\in\cL C(\cL S,T_1)$. Then, it also belongs to $\cL C(S,T_1)$. Hence, it satisfies
\begin{gather}\label{ujs=}
V_j=\lambda_{j}^{-1}U_{j}\circ T_{1}.
\end{gather}
This implies $(\ref{vjlju})$.

Assume furthermore that $\var\in \cL C(\cL S, T_1, \rho)$. Eliminating $V_e$ in \re{ujs=} with \re{ueve}, we obtain
 \eq{uele}
 \nonumber
 U_e=\la_e\ov{U_e\circ\rho\circ T_1}.
 \eeq
According to \re{T1rho}, we obtain
$$
U_{e,\rho_b\rho_a}=\la_e\ov\la^{Q-P}\ov U_{e,PQ}.
$$
If $(P,Q)\in\cL R_e$ and since $\mu_e, \mu_s, \mu_{s+s^*}$ are of norm greater than $1$, then we have $p_{s+s_*}=q_{s+s^*}$, $p_{s}=q_{s}$ and $p_{e}=q_{e}+1$.  By
$\la_e\ov\la^{Q-P}=\ov \nu_{PQ}=\nu_{PQ}^{-1}$ we get 
 \re{uenu}.

Using \re{ujs=}, we eliminate $V_j$ from \re{ueve} and \re{uhou} to  obtain
 \gan
 \la_h^{-1}U_h\circ T_1=\ov{\la_h^{-1}U_h\circ T_1\circ\rho},\quad
 \la_{s+s_*}^{-1}U_{s+s_*}\circ T_1=\ov{\la_{s}^{-1}U_{s}\circ T_1\circ\rho}.
 \end{gather*}
Since $T_1\rho T_1= \rho T_2 T_1=\rho S^{-1}$, the previous equalities read
\gan
 U_h=\la_h^2\ov{U_h\circ \rho\circ S^{-1}},\quad
 U_{s+s_*}=\la_{s+s_*}^2\ov{U_{s}\circ\rho\circ S^{-1}}.
 \end{gather*}
We recall that $\la_{s+s_*}=\ov\la_{s}^{-1}$. According to \re{rhoS-1}, we obtain
$$
U_{h,\rho(PQ)}=\la_h^2 \ov\mu^{P-Q}\ov U_{h,PQ},\quad U_{s+s_*,\rho(PQ)}=\la_{s+s_*}^2 \ov\mu^{P-Q}\ov U_{s,PQ}.
$$
If $(P,Q)\in\cL R_h$, then $\ov\mu^{P-Q}=\ov\mu_h=\la_h^{-2}$. If $(P,Q)\in\cL R_s$, then $\ov\mu^{P-Q}=\ov\mu_s=\la_{s+s_*}^{-2}$. The result then follows.

(iii)  Let $\var\in\cL C(\cL T_1,\cL T_2)$. Then, in particular, we have
\aln 
U_{j}=U_j(T_{1k}), \quad k\neq j; \quad 
V_j=\la_j^{-1}U_j\circ T_1.
\end{align*}
Let $(P,Q)\in \cL R_j\setminus \cL N_j$. For each $k$ such that $q_k>p_k$, we compose $U_j$ by $T_{1 k}$. We emphasize that when $(P,Q)\in\cL R_j$,
 such a $k$ is a hyperbolic index. Using the previous identity, we obtain
\eq{uL}
U_{j,PQ}=L_{j,PQ}U_{j,A_jB_j}
\eeq
with 
$$
L_{j, P Q}:=\prod_{k\neq j,  p_k< q_k}\la_{k}^{ q_k- p_k}.
$$
By the definition of $\nu_{PQ}^+$, we conclude 
\eq{Ljpqn}
L_{j,PQ}=\nu_{PQ}^+, \quad (P,Q)\in\cL R_j. 
\eeq
If $(P,Q)\in  \cL N_j$, then $(A_j,B_j)=(P,Q)$ and we   have $L_{j,PQ}=\nu_{PQ}^+=1$, so that the   relation \re{uL} just becomes the identity $U_{j,PQ}=U_{j,PQ}$.
 

Assume now that $\var\in\cL C(\cL T_1,\cL T_2,\rho)$. In addition to the previous conditions, we have \re{uhou} and \re{ueve}.
Hence, \re{uenu}-\re{usnu} and \re{uL} lead to: 
\begin{alignat*}{4}
\nu_{PQ}\ov U_{e,\rho_e(PQ)}&=U_{e,PQ}=L_{e,PQ}U_{e,A_eB_e},\quad && (P,Q)\in \cL R_e;\\
\ov U_{h,\rho_h(PQ)}&=U_{h,PQ}=L_{h,PQ}U_{h,A_hB_h},\quad &&(P,Q)\in \cL R_h;\\
\ov U_{s+s_*,\rho(PQ)}&=U_{s,PQ}=L_{s,PQ}U_{s,A_sB_s},\quad &&(P,Q)\in \cL R_s. 
\end{alignat*}
Since $\rho_e, \rho_h$ are involutions on  $\cL R_e$ and $\cL R_h$, respectively, 
and since $\rho$ is a bijection from $\cL R_s$ onto $\cL R_{s+s_*}$, we obtain 
\begin{alignat*}{4}
\nu_{\rho_e(PQ)}\ov U_{e,PQ}&=L_{e,\rho_e(PQ)}U_{e,(A_e,B_e)\circ\rho_e(PQ)},\quad && (P,Q)\in \cL R_e;\\
\ov U_{h,PQ}&=L_{h,\rho_h(AB)}U_{h,(A_h,B_h)\circ\rho_h(PQ)},\quad && (P,Q)\in \cL R_h(PQ);\\
\ov U_{s+s_*,PQ}&=L_{s,\rho(AB)}U_{s,(A_s,B_s)\circ\rho(PQ)},\quad && (P,Q)\in \cL R_{s+s_*}.
\end{alignat*}
By \re{Ljpqn}, we copy  the values $L_{j,\rho( P Q)}=\nu_{\rho(PQ)}^+$ from  \re{nupm1++}. We have
\begin{alignat*}{4}
\nu^+_{\rho_j(PQ)}&=\nu^+_{PQ}, \quad &&\text{if $j \neq e
$, and $ (P,Q)\in\cL R_j$;}\\
 \nu^+_{\rho_e(PQ)}&=\nu_{PQ}\nu_{PQ}^+,
 \quad && \text{if $(P,Q)\in\cL R_e$};\\
 \nu_{\rho_e(PQ)}&=\nu_{PQ}, \quad &&\text{if $(P,Q)\in\cL R_e$}.
\end{alignat*}
Finally, we obtain
\begin{eqnarray*}
U_{j,PQ}&=& \nu_{PQ}^+ 
\ov U_{j,(A_j,B_j)\circ\rho_j(PQ)},\quad (P,Q)\in \cL R_j, \quad j=e,h; \\
U_{s+s_*,PQ}&=&\nu_{PQ}^+\ov U_{s,(A_s,B_s)\circ\rho(PQ)},\quad (P,Q)\in \cL R_{s+s_*}.
\end{eqnarray*}

Therefore, we have derived necessary conditions for the centralizers. 
Let us verify that the conditions   are also sufficient. 
Of course, the verification for (i) is straightforward. 
Furthermore,  that $\var=I+(U,V)$ commutes with $S_1, \ldots, S_p$
is equivalent to $U_{j,PQ}=V_{j,QP}=0$ for all $(P,Q)\in \cL R_j$, which
is also trivial in cases (ii) and (iii).

For (ii),  \re{muimuj} and \re{vjlju} imply that $\var$ commutes with $T_1$. We  verify that $\var$ commutes 
with $\rho$. Write $\rho\var\rho=(\tilde U,\tilde V)$.  Applying \re{vjlju} and \re{uenu} each twice, we get for $(P,Q)
\in \cL R_e$
$$
\tilde U_{e,PQ}=\ov V_{e,\rho(PQ)}=\la_e^{-1}\la^{\rho_b-\rho_a}\ov U_{e,\rho_e(PQ)}=\la_e^{-1}\la^{\rho_b-\rho_a}
\nu_{PQ}U_{e,PQ}.
$$
We get $\tilde U_{e,PQ}=U_{e,PQ}$. The  identities for  hyperbolic and complex  components of $\rho\var\rho=\var$ are easy to verify.

For (iii),  let us  verify that \re{ujrjnj}, 
 \re{muimuj}, and \re{vjlju} are sufficient conditions for $\var\in\cL C(\cL
T_1,\cL T_2)$. By \re{vjlju}, we get $\var T_1=T_1\var$.  
 Also,  for $\var\in\cL C(\cL T_1)$ it remains to show that for $(P,Q)\in\cL R_j$
 \eq{uit1j}
 (U_j\circ T_{1k})_{PQ}=U_{j,PQ}, \quad k\neq j; \quad (U_{j}\circ T_{1j})_{QP}=\la_jV_{j,QP}.
 \eeq
   We introduce $(P_j,Q_j)$ via $\xi^P\eta^Q\circ T_{1j}=\la_j^{p_j-q_j}\xi^{P_j}\eta^{Q_j}$
 and also denote $(P_j,Q_j)$ by $(P,Q)_j$. 
 
We first remark that \re{ujrjnj} 
  also holds for $(P,Q)\in\cL N_j$.
Therefore, we will use \re{ujrjnj} 
for all $(P,Q)\in\cL R_j$.

  For $k\neq j, h$, we  have $(P_k,Q_k)=(P,Q)$. Thus in this case we immediately 
 get the first identity in \re{uit1j}. 
  Using \re{ujrjnj} twice, we obtain for $j\neq h$
 \begin{align*}
 (U_j\circ T_{1h})_{PQ}&=\la_h^{p_h-q_h}U_{j,(PQ)_h}=\la_h^{p_h-q_h}\nu_{(PQ)_h}^+U_{j,(A_j,B_j)(P,Q)}\\
 &=\la^{p_h-q_h}\nu_{(PQ)_h}^+ \ov \nu_{PQ}^+U_{j,PQ}= U_{j,PQ}.
 \end{align*}
 Combining with the identities which we have proved, we get    $(U_j\circ T_{1j})_{QP}=(U_j\circ T_1)_{QP}=(\la_jV_j)_{QP}$ for $j\neq h$.
 This gives us all the identities in \re{uit1j} for $(P,Q)\in\cL R_j$. These identities are trivial when $(P,Q)$ is not in $\cL R_j$. Therefore, we have shown that these conditions are sufficient for $\var\in\cL C(\cL T_1,\cL T_2)$. 
 
 Finally, we need to verify that \re{muimuj}, \re{vjlju}, and \re{ujrjnj}-\re{sqnupq2} imply that $\var$ and $\rho$ commute.
 
 To shorten operations applied to multiindices, let us introduce the follow notation. 
 For $(P,Q)\in \cL \cL N_j$, define
 $$
 \iota_j\colon (P,Q)\mapsto (A_j,B_j)\circ\rho_j(P,Q), \quad j=e,h; \quad
  \iota_s\colon (P,Q)\mapsto (A_s,B_s)\circ\rho(P,Q). $$
Then $\iota_j$ 
 is an involution on $\cL N_j$ when $j=e,h$, and it is a bijection from $\cL N_{s+s_*}$ onto $\cL N_{s}$ when $j=s$. Furthermore, the inverse of $\iota_s$ is given by 
 $$
\iota_{s+s_*}\colon (P,Q)\mapsto  (A_{s+s_*},B_{s+s_*})
 \circ \rho(P,Q).
  $$
 
Fix $(P,Q)\in\cL R_e$. By \re{ujrjnj} and \re{sqnupq}, we have 
 \al\label{uepqn}
 U_{e,PQ}=\nu_{PQ}^+U_{e,(A_e,B_e)(P,Q)}=\nu_{PQ}^+  \nu^+_{(A_e,B_e) (P,Q)}  \ov U_{e,\iota_e\circ (A_e,B_e) (P,Q)}.
 \end{align}
  We know that $p_j=q_j$ when $j\neq h$ or $j$ does not equal the $e$
($j$ can represent other elliptic components). We know that $p_{e}=q_e+1$ for the   $e$.
By treating case by case for $p_h\geq q_h$ or $p_h<q_h$, i.e. $2^{h_*}$ cases in total, we verify that
$$
\iota_e\circ (A_e,B_e) (P,Q)=\iota_e(P,Q),
 \quad \nu_{PQ}^+  \nu^+_{(A_e,B_e) (P,Q)}=\la_e^{-1}\la^{\rho_b-\rho_a}\nu^+_{\rho_b\rho_a}.
$$
This allows us to apply \re{vjlju} and 
\re{ujrjnj} to rewrite the right-hand side of \re{uepqn} as $\ov V_{e,\rho(PQ)}$. We repeat a simpler  procedure for $U_{h,PQ}$
with $(P,Q)\in\cL R_h$: 
We apply $(A_h,B_h)$ to the multi-index $(P,Q)$ and use \re{sqnupq} once. We then check the multiindex and the coefficient to conclude
that the result is $\ov U_{h,\rho(PQ)}$. (Here we do not need apply \re{vjlju}.) For $U_{s+s_*, PQ}$ with $(P,Q)\in\cL R_{s+s_*}$, we apply
$(A_{s+s_*},B_{s+s_*})$ to $(P,Q)$ and use \re{sqnupq2} once. The result is $\ov U_{s,\rho(PQ)}$.  
With $U_{h,PQ}=\ov U_{h,\rho(PQ)}$ and $U_{s+s_*,PQ}=\ov U_{s,\rho(PQ)}$, we apply
 \re{vjlju} to obtain $V_{h,PQ}=\ov V_{h,\rho(PQ)}$ and $V_{s+s_*,PQ}=\ov V_{s,\rho(PQ)}$. This shows that $\var$ commutes with $\rho$.
 The proof is complete.
\end{proof}

We have described the conditions on centralizers. We now determine  complements of these conditions
to define normalized mappings. 
\begin{defn} Let $\var=I+(U,V)$ be a formal mapping tangent to the identity.
\bppp
\item We say that $\var$ is {\it normalized} with respect to $S_1,\ldots, S_p$ if
\eq{ujpq=0}\nonumber
U_{j,PQ}=0=V_{j,QP}, \quad \text{if $(P,Q)\in\cL R_j, \quad \forall j$}.
\eeq
Furthermore, $\rho\var\rho$ is  normalized w.r.t. $S_1,\ldots, S_p$ if and only if $\var$ is.
\item We say that $\var$ is {\it normalized} with respect to
$\{\cL S, T_1,\rho\}$ if
\begin{alignat}{4}
\label{uhus+1} U_{h,PQ}&=- \ov{U}_{h,\rho(PQ)}, \  &&\forall (P,Q)\in\cL R_h;
\\ 
\label{ussp}
 U_{s+s_*,PQ}&=-\ov{U}_{s,\rho(PQ)},\  &&\forall(P,Q)\in\cL R_{s+s_*};\\
\label{uhus+3} U_{e,PQ}&=-\nu_{PQ}\ov{U}_{e,\rho_e(PQ)},\quad  &&\forall(P,Q)\in\cL R_e;\\
\label{vjlju++}V_{j,QP}&=-\la_j^{-1}\la^{P-Q}U_{j,PQ}, \qquad && \forall (P,Q)\in\cL R_j.
\end{alignat}
\item We say that $\var$ is {\it normalized} w.r.t. $\{\cL T_1,\cL T_2,\rho\}$ if
\begin{alignat}{4}
\label{ujrjnjC}
U_{j, P Q} &=-\nu_{P Q}^+U_{j,(A_j,B_j)( P, Q)},\quad &&\forall(P,Q)\in\cL R_j\setminus\cL N_j,\\
\label{UePQ-}
U_{j, P Q} &= -  
\nu_{PQ}^+ 
 \ov U_{j,(A_j,B_j)\circ\rho_j( P, Q)},\quad &&\forall (P,Q)\in \cL N_j, \  j=e,h;\\
\label{UssP} U_{s+s_*,PQ}&=-\nu_{PQ}^+\ov U_{s,(A_s,B_s)\circ\rho( P, Q)}, \quad &&\forall (P,Q)\in  \cL N_{s+s_*}.
\end{alignat}
\eppp 
\end{defn}

\begin{lemma}\label{FHG-}
 Let $F$ be a formal map which is tangent to the identity.
There exists a unique formal decomposition $F=HG^{-1}$ with $G\in{\cL C}(\cL S,T_1,\rho)$
$($resp. ${\cL C}(\cL T_1,\cL T_2, \rho))$
and $H\in {\cL C}^\mathsf{c}(\cL S,T_1,\rho)$ $($resp. ${\cL C}^\mathsf{c}(\cL T_1,\cL T_2, \rho)))$.
If $F$ is convergent, then $G$ and $H$ are also convergent.
\end{lemma}
\begin{proof} We will apply \rl{fhg-} as follows. Let
$\hat H$   be the set of mappings 
in $ {\cL C}_2^{\mathsf{c}}(\cL S, \cL T_{1}, \rho)$.
Note that $\hat H$ is a $\rr$-linear subspace of   $({\widehat { \mathfrak M}}_n^2)^n$. 
We will define a $\rr$-linear projection $\pi$ from  $({\widehat { \mathfrak M}}_n^2)^n$ 
onto $\hat H$ such that $\pi$ preserves the degree of $F$ if $F$ is homogeneous. We will
 show that $\hat G=(\operatorname{I}-\pi)\hat H$  agrees with ${\cL C}_2^\mathsf{c}(\cL S, \cL T_{1}, \rho)$. 
We will derive estimates on $\pi$ stated in \rl{fhg-}, from which we conclude the convergence of $H, G$. 

The same argument will be applied  to  the second case of $\cL C(\cL T_1, 
\rho)$ and $\cL C^{\mathsf c}(\cL T_1,
\rho)$. 

For the first case, let us define a projection $\pi \colon ({\widehat { \mathfrak M}}_n^2)^n 
\to \hat H$.
We decompose
$$
(U,V)=(U'+U'',V'+V''),\quad
\pi (U,V)=(U',V').
$$
 We first define
\eq{notinrj}\nonumber
U'_{j,PQ}=U_{j,PQ}, \quad V_{j,PQ}'=V_{j,PQ}, \quad U_{j,PQ}''=0,
\quad V_{j,PQ}''=0, \eeq
for  $ (P,Q)\not\in\cL R_j.$
Suppose that $(P,Q)\in \cL R_e$. We have
\begin{align*}
U_{e,PQ}&=U_{e,PQ}'+U_{e,PQ}'', \\
U_{e,\rho_e(PQ)}&=U_{e,\rho_e(PQ)}'+U_{e,\rho_e(PQ)}''.
\end{align*}
According to \re{uhus+3} and  \re{uenu}, we  need to seek solutions that satisfy
\eq{uepq'}
U_{e,PQ}'+\nu_{PQ}\ov U_{e,\rho_e(PQ)}'=0, \quad
U_{e,PQ}''-\nu_{PQ}\ov U_{e,\rho_e(PQ)}''=0.
\eeq
Hence,  for $(P,Q)\in \cL R_e$ we  choose
\begin{gather}\label{uepq'+}
\nonumber
 U_{e,PQ}'=\f{1}{2}(U_{e,PQ}-\nu_{PQ}\ov U_{e,\rho_e(PQ)}), \   U_{e,PQ}''=\f{1}{2}(U_{e,PQ}+\nu_{PQ}\ov U_{e,\rho_e(PQ)}).
\end{gather}
 We verify directly that the solutions satisfy \re{uepq'} as follows: 
\aln
U_{e,PQ}'+\nu_{PQ}\ov U_{e,\rho_e(PQ)}'&=
\f{1}{2}(U_{e,PQ}-\nu_{PQ}\ov U_{e,\rho_e(PQ)})\\ &\quad +
\f{1}{2}(\nu_{PQ}\ov U_{e,\rho_e(PQ)}-\nu_{PQ}\nu_{\rho_e(PQ)}\ov U_{e,PQ})=0.
\end{align*}
 Here we have used the fact that $\rho_e$ is an involution 
and $\nu_{\rho_e(PQ)}\nu_{PQ}=1$ from \re{nupm1+}.

Analogously, for $(P,Q)\in \cL R_h$,   we achieve   \re{uhus+1} and \re{uhnu} by taking
\ga\label{uhpq'}\nonumber
 U_{h,PQ}'=\f{1}{2}(U_{h,PQ}- \ov U_{h,\rho_h(PQ)}),
\quad  U_{h,PQ}''=\f{1}{2}(U_{h,PQ}+ \ov U_{h,\rho_h(PQ)}).
\end{gather}
For $(P,Q)\in\cL R_{s+s_*}$,  we achieve \re{usnu} and \re{ussp} by taking
\ga\label{usspq'}\nonumber
 U_{s+s_*,PQ}'=\f{1}{2}(U_{s+s_*,PQ}-\ov U_{s,\rho(PQ)}),
\quad  U_{s+s_*,PQ}''=\f{1}{2}(U_{s+s_*,PQ}+ \ov U_{s,\rho(PQ)}).
\end{gather}

We have determined coefficients for $U_{j,PQ}', U_{j,PQ}''$ with $(P,Q)\in\cL R_j$.
Let us set for $(P,Q)\in\cL R_j$,
\begin{eqnarray}
V_{j,QP}' &=& -\la_j^{-1}\la^{P-Q}U_{j,PQ}'\label{v'},\\
V_{j,QP}'' &=& \la_j^{-1}\la^{P-Q}U_{j,PQ}''\label{v''}.
\end{eqnarray}
This fulfills the conditions on $V_j'$ and $V_j''$ easily. Note that the last identity means that $(U'',V'')$
commutes with $T_1$. We have obtained the required formal decomposition. 

To prove the convergence,  we start with 
\begin{equation}\label{nupq}
\la_j^{-1}\la^{P-Q} = \nu_{PQ}=\pm1
\end{equation}
 for $(P,Q)\in\cL R_j$.
So $\pi$ is indeed an $\rr$-linear projection which preserves  degrees. Since  
$|\nu_{PQ}|=1$, 
we have that 
$$
|U_{PQ}'|\leq \max_{(P',Q')}|U_{P'Q'}|.
$$ 
Here $(P',Q')$ runs over all permutations of $(P,Q)$ in $2p$ coordinates. 
The same holds for $V'$. Hence, with the notation of \rl{fhg-}, we have
$$
\{\pi (U,V)\}_{sym}\prec (U, V)_{sym}.
$$
The existence and uniqueness as well as the convergence also follow
 from \rl{fhg-}.

We now consider the second case of $\cL C(\cL T_1,\cL T_2,\rho)$ by  minor changes.
Let us define a projection $\pi\colon(({\widehat { \mathfrak M}}_n^2)^n\to \hat H$. Here $\hat H$
is the space associated with the mappings satisfying the normalized conditions \re{ujrjnjC}-\re{UssP}.
Let  $\hat G=(\operatorname{I}-\pi)\hat H.$ 
We decompose as above
$$
(U,V)=(U'+U'',V'+V''),\quad
\pi (U,V)=(U',V').
$$
We  choose~:
\begin{alignat}{4}
\label{ujrjnjCs}
U_{j, P Q}'' &=\f{1}{2}(U_{j, P Q}'+\nu_{P Q}^+U_{j,(A_j,B_j)( P, Q)}),\  && (P,Q)\in\cL R_j\setminus\cL N_j, 
\\
\label{consist-1}
U_{j, P Q}' &=\f{1}{2}(U_{j, P Q}'-\nu_{P Q}^+U_{j,(A_j,B_j)( P, Q)}),\  && (P,Q)\in\cL R_j\setminus\cL N_j, 
\\
U_{j,PQ}''&=\f{1}{2}(U_{j,PQ}+ \nu_{PQ}^+ 
\ov U_{j,(A_j,B_j)\circ\rho_j(PQ)}), \quad && (P,Q)\in\cL N_j,\  j=e,h,\\
\label{consist1} U_{j,PQ}'&=\f{1}{2}(U_{j,PQ}-  \nu_{PQ}^+ 
\ov U_{j,(A_j,B_j)\circ\rho_j(PQ)}),\quad  && (P,Q)\in\cL N_j,\  j=e,h,\\ 
U_{s+s_*,PQ}'' &= \f{1}{2}(U_{s+s_*,PQ}+\nu_{PQ}^+\ov U_{s,(A_s,B_s)\circ\rho(PQ)}),\quad && (P,Q)\in\cL N_{s+s_*},\\
\label{consist3}
U_{s+s_*,PQ}'&=\f{1}{2}(U_{s+s_*,PQ}-\nu_{PQ}^+\ov U_{s,(A_s,B_s)\circ\rho(PQ)}),\quad && (P,Q)\in\cL N_{s+s_*}.
\end{alignat}
We set $U''_{j,PQ}=0=V_{j,QP}''$ for $(P,Q)\not\in\cL R_j$. 
Let us   verify that $\pi (U,V)=(U',V')$ is in $\hat H$.  Recall that
$$
\iota_e\colon (P,Q)\to (A_e,B_e)\circ\rho_e(PQ)=(A_e,B_e)(\rho_b(P,Q),\rho_a(P,Q)), \quad (P,Q)\in \cL N_e.
$$
To verify  \re{UePQ-} for  $j=e$,   via  \re{consist1} 
we compute
\aln
U'_{e,PQ}+  \nu_{PQ}^+ 
\ov U'_{e,(A_e,B_e)\circ\rho_e(PQ)}&=
\f{1}{2}(U_{e,PQ}-  \nu_{PQ}^+ 
\ov U_{e,\iota_e(PQ)})
\\  &\quad +
\f{ \nu_{PQ}^+}{2}(U_{e,\iota_e(PQ)}- \nu_{\iota_e(PQ)}^+ 
\ov U_{e,PQ})=0.
\end{align*}
Here we have used the fact that $\iota_e\colon\cL N_e\to\cL N_e$ is an involution and 
\eq{forrem2}\nonumber
\nu_{PQ}\nu_{\iota_e(PQ)}=1, \quad \nu_{PQ}^+ 
{\nu_{\iota_e(PQ)}^+}=1, \quad (P,Q)\in\cL \cL N_e.
\eeq
Recall that
$$
\iota_j(P,Q)=(A_j,B_j)(\rho_a(PQ),\rho_b(PQ)), \quad j=h,s.
$$
We also know that  $\iota_h$ is an involution on $\cL N_h$ and $\iota_s$ is a bijection from $\cL N_{s+s_*}$ onto $\cL N_{s}$.
Analogously, we verify \re{UePQ-}  for $U'_h$ and  \re{UssP} via \re{consist1} and \re{consist3}.
Note that $(P,Q)\to (A_j,B_j)(P,Q)$ is a projection on $\cL R_j$. Analogously, we verify \re{ujrjnjC} via \re{consist-1}.
This shows that $\pi (U,V)$ is in $\hat H$. We can also verify that $(U'',V'')=(\operatorname{I}-\pi
)(U,V)$ satisfies the conditions
on the centralizer,  i.e. it  is in $\hat G$. 
  
As before, we have 
$$
|U_{j,PQ}'|, |U_{j,PQ}''|\leq \max_i\max_{(P',Q')\text{permutation of } (P,Q)}|U_{i,P'Q'}|. 
$$
Equations \re{v'}, \re{v''} lead to the same inequality for $V',V''$. Hence, again the result follows from \rl{fhg-}.
\end{proof}

\begin{prop}\label{linearS} Assume that the family of involutions $\{\cL T_1, \cL T_2, \rho\}$ is 
formally linearizable. Assume further that $\sigma_1,\ldots, \sigma_p$
 defined by \rea{sigmai1}-\rea{sigmai3}, are linear.   
 \bppp
\item There is
 a  biholomorphic mapping in the centralizer of $\{\cL S,\rho\}$
which linearizes $\tau_1$ and $\tau_2$.
\item
Assume further that $\tau_1=T_1$ and $\tau_2=T_2$.
Then $\{\tau_{11},\ldots, \tau_{1p},\rho\}$ is holomorphically linearizable.
\eppp
\end{prop}
\begin{proof} (i) Suppose that $\Psi$ is a formal mapping satisfying
\eq{conj-tj}\nonumber
\Psi^{-1}\tau_{1j}\Psi=T_{1i_j}, \quad \Psi\rho=\rho\Psi.
\eeq
Then $T_{1j}=(L\Psi)\circ T_{1i_j}\circ (L\Psi)^{-1}$, and $L\Psi$ commutes with $\rho$. 
 Replacing $\Psi$ by $\Psi\circ L\Psi^{-1}$, we may assume that $\Psi$ is tangent to the identity and $i_j=j$.
We decompose
 $
\Psi=\Psi_1\Psi_0^{-1},
$
where $\Psi_1$ is normalized w.r.t. $\cL S,T_1,\rho$ and $\Psi_0$ is in the centralizer of $\cL S,T_1,\rho$.
Since $\Psi,\Psi_0$ commute with $S_j$ and $\rho$, then $\Psi_1$ commutes with $S_j,\rho$ too.
We now let $\Psi$ denote $\Psi_1$.

To be more specific, let us write
$$
\tau_{1}\colon \begin{cases}
    \xi'_i= \lambda_{i}\eta_{i}+ f_{i}(\xi,\eta)
    \quad i=1,\ldots, p,\\ \eta'_i=\lambda_{i}^{-1}\xi_{i}+ g_{i}
    (\xi,\eta)\quad i=1,\ldots, p,\\
\end{cases}
$$
and
$$
\Psi\colon\begin{cases}
\xi'_i= \xi_{i}+U_{i}(\xi,\eta)\quad i=1,\ldots,p,\\
\eta'_i=\eta_{i}+V_{i}(\xi,\eta)\quad i=1,\ldots, p.\\
\end{cases}
$$
Let us write that $\Psi$ conjugates $\tau_{1}$ to
$$
T_1\colon
    \xi'_i= \lambda_{i}\eta_{i},\quad
      \eta'_i=\lambda_{i}^{-1}\xi_{i}, \quad i=1,\ldots,p.
$$
 We have $\Psi\circ T_{1}= \tau_{1} \circ
\Psi$; that is
\ga
\lambda_{i}V_i-U_i\circ T_{1}=   - f_{i}\circ \Psi(\xi,\eta)\quad i=1,\ldots, p,\\
\label{lai-1}\lambda_{i}^{-1}U_i-V_i\circ T_{1}= -g_{i}\circ \Psi(\xi,\eta)\quad i=1,\ldots, p.
\end{gather}
Since  $\Psi$ commutes with each $S_j$, then $U_{j,PQ}=V_{j,QP}=0$ for $(P,Q)\not\in\cL R_j$.
Let us find an equation involving only the unknown $U_e, U_h,  V_h, U_{s}, V_s$. By the reality conditions,
they determine $U,V$ completely.

Since the normalized mapping $\Psi$ commutes with $\rho$, we  have
\ga
U_{h,PQ}= \ov{U}_{h,\rho(PQ)}, \  (P,Q)\in\cL R_h,
\quad U_{(s+s_*),PQ}=\ov{U}_{s,\rho(PQ)},\ (P,Q)\in\cL R_{s+s_*},\nonumber\\
\label{VeQP} 
V_{e,QP}=\ov U_{e,\rho_e(QP)},\quad (P,Q)\in\cL R_e,\\
V_{h,QP}=\ov V_{h,\rho(PQ)},\  (P,Q)\in\cL R_h, \quad V_{s+s_*,QP}=\ov{V}_{s,\rho(QP)}, \ (P,Q)\in\cL R_{s+s_*}.\nonumber
\end{gather}
Let us combine the above identities with the (first two) normalizing conditions 
\begin{alignat}{4}
U_{h,PQ}&=- \ov{U}_{h,\rho(PQ)}, \  &&(P,Q)\in\cL R_h,\nonumber\\
U_{s+s_*,PQ}&=-\ov{U}_{s,\rho(PQ)},\  &&(P,Q)\in\cL R_{s+s_*},\nonumber\\
\label{UePQ}
U_{e,PQ}&=-\nu_{PQ}\ov{U}_{e,\rho_e(PQ)},\quad && (P,Q)\in\cL R_e,\nonumber\\
V_{j,QP}&=-\la_j^{-1}\la^{P-Q}U_{j,PQ}, \quad && (P,Q)\in\cL R_j.\nonumber
\end{alignat}
Recall that $\Psi$ belongs to the centralizer of $\cL S$ so that $U_{j,PQ}=V_{j,QP}=0$ for $(PQ)\not\in\cL R_j$ 
and  all $j$.
We then immediately see that $U_h,U_s,U_{s+s_*},V_h,V_s,V_{s+s_*}$ are $0$.

We now use the two last conditions to determine $U_e,V_e$ and majorize them.  By \re{lai-1}, \re{VeQP} and \re{nupq}, we obtain
$$
U_{e,PQ}-\nu_{PQ}\ov{U}_{e,\rho_e(QP)}=-\la_e\{g_e\circ\Psi\}_{PQ}.
$$ 
Using \re{uhus+3}, we obtain that, for $(P,Q)\in\cL R_e$,
$$
U_{e,PQ}=-\f{1}{2}\la_e\{g_e\circ\Psi\}_{PQ},
$$
as well as
$$
V_{e,QP}=\f{1}{2}\nu_{PQ}\la_e\{g_e\circ\Psi\}_{PQ}.
$$
Therefore, we have
$$
|V_{e,QP}|,|U_{e,PQ}|\leq C \left|\{g_e\circ\Psi\}_{PQ}\right|.
$$
 In view of \re{FiGp},  
we then have
$$
\psi_{sym}\prec C g_{sym}\circ\Psi_{sym}= g_{sym}\circ(I_{sym}+\psi_{sym}).
$$
Therefore, $\psi_{sym}$ is convergent at the origin and so is $\Psi$.

(ii)
Assume now that $\sigma=S, \tau_1=T_1, \tau_2=T_2$ are linear. Suppose that $\Psi$ linearizes the $\{\tau_{ij}\}$ and
commutes with $\rho$. We decompose $\Psi=\Psi_1\Psi_0^{-1}$ with $\Psi_1$ being normalized w.r.t. $\cL S,T_1,T_2,\rho$ and 
with $\Psi_0$ being in the centralizer of $\cL S,T_1,T_2,\rho$. 
 From (i), we know that $\Psi$ is diagonal and $\Psi^{-1}\tau_{ij}\Psi=T_{ij}$.  We have
$$
\Psi_1^{-1}\tau_{i j}\Psi_1=\Psi_0^{-1}T_{i j}\Psi_0=T_{i j}.
$$
Hence, $\Psi_1$ linearizes the $\tau_{i j}$ and is normalized w.r.t $\cL S, T_1, T_2,\rho$. Since $\Psi,\Psi_1$ commute with $\cL S$ and $\rho$, so does $\Psi_1$. Let us denote $\Psi=\Psi_1$ and let us write $\Phi=I+(U,V)$.

We recall
$$
T_{1 j}\colon \begin{cases}
    \xi'_j= \lambda_{j}\eta_{j} \\
    \eta'_j=\lambda_{j}^{-1}\xi_{i} \\
    \xi'_k=\xi_{k},\quad k\neq j\\
    \eta'_k=\eta_{k},\quad k\neq j,\\
\end{cases}
\quad
\tau_{1 j}\colon \begin{cases}
    \xi'_j= \lambda_{j}\eta_{j}+ f_{j j}(\xi,\eta)\\
    \eta'_j=\lambda_{j}^{-1}\xi_{i}+ g_{j j}(\xi,\eta)\\
    \xi'_k=\xi_{k}+ f_{jk}(\xi,\eta),\quad k\neq j\\
    \eta'_k=\eta_{k}+ g_{jk}(\xi,\eta),\quad k\neq j.\\
\end{cases}
$$
Since we have $\Psi\circ T_{1j}= \tau_{1j} \circ \Psi$, we obtain the following relations
\begin{equation}\label{lin-invol}
\begin{cases}
\lambda_{j}V_j-U_j\circ T_{1 j}= - f_{j j}\circ \Psi\\
\lambda_{j}^{-1}U_j-V_j\circ T_{1 j}=-g_{j j}\circ \Psi\\
U_k-U_k\circ T_{1 j}=- f_{jk}\circ \Psi,\quad k\neq j\\
V_k-V_k\circ T_{1 j}=- g_{jk}\circ \Psi,\quad k\neq j.
\end{cases}
\end{equation}
According to \re{vjlju}, the left-hand side of the two first equations are zero. We shall use the two last ones to obtain estimates. 
According to the normalizing conditions, we find as above, that $U_h,U_s,U_{s+s_*},V_h,V_s,V_{s+s_*}$ are $0$. Thus we only have to show that $U_e, V_e$ are convergent.

In the second last  identity  in \re{lin-invol} with $k=e$, let us compose on the right by all $T_{1 j}$ with $j\neq e$. We have
  for  $j'\neq e$ and 
$j\neq e$,
$$
U_e\circ T_{1 j}-U_e\circ T_{1 j}\circ T_{1 j'}=- f_{j,e}\circ \Psi\circ T_{1 j'} =f_{j,e}\circ \tau_{1 j'}\circ \Psi.
$$
Repeating this for all $T_{1j}$ except for $j=e$ and taking summation, we get
$$
U_{e}-U_e\circ T_{1e}^{-1}\circ T_1 = -\left\{\sum_{i=1}^p f_{j,e}\circ\tau_{11}\cdots
\circ\widehat{\tau_{1e}}\circ\cdots\circ\tau_{1i} \right\}\circ \Psi.
$$
Here $\widehat{\tau_{1e}}$ means that $\tau_{1e}$ is not included in composition if $i\geq e$.
Thus 
$$
U_{e}\circ T_{1e}-U_e \circ T_1 = -\left\{\sum_{i=1}^p f_{j,e}\circ\tau_{11}\cdots
\circ\widehat{\tau_{1e}}\circ\cdots\circ\tau_{1i} \right\}\circ \tau_{1e}\circ \Psi.
$$
Combining with the first identity in \re{lin-invol}  and eliminating $U_{e}\circ T_{1e}$, we obtain 
$$
\la_e V_{e}-U_e\circ T_{1} = \tilde f_e\circ \Psi
$$
for a convergent power series $\tilde f_e$.  
The normalizing condition \re{vjlju++} says that $\la_1^{Q-P}U_{e,PQ}=-\la_eV_{e,QP}$ for $(P,Q)\in R_e$. We obtain
$$
V_{e,QP}=\frac{1}{2\la_e}\{\tilde f_e\circ\Psi\}_{QP}
\prec \frac{1}{2\la_e}\{\ov{\tilde f_e}\circ\ov {\Psi}\}_{QP} , \quad (P,Q)\in \cL R_e.  $$
If $(P,Q)$ is not in $\cL R_e$, the above still holds as $V_{e,QP}=0$.

Indeed this is the key point, if $U_{j,PQ}\neq 0$ then $(P,Q)\in \cL R_j$ so that 
$\mu_j^{p_j-q_j}=\mu_j$ and $\mu_\ell^{p_\ell-q_\ell}=1$, $\ell\neq j$. As we have observed and since $j\neq h$ ($j$ is actually $e$),  this implies that $p_j=q_j+1$ and $p_\ell=q_\ell$, $\ell\neq j,h$. Since the hyperbolic $\lambda_h$  are of modulus one, we have either $|\la^{P-Q}|=\la_e$. Thus
$$
|U_{e,PQ}|=|V_{e,QP}|\leq\frac{1}{2\la_e}\{\ov{\tilde f_e}\circ\ov {\Psi}\}_{QP}
\leq  \frac{1}{2\la_e}\{\ov{\tilde f_e}\circ\ov \Psi_{sym}\}_{PQ}.
$$
We obtain
$$
\psi_{sym}\prec  C\left( \ov {\tilde f}_{sym}\circ (I_{sym} +\psi_{sym}\right).
$$
Therefore,  $U_e,V_e$
 are convergent at the origin since they are majorized by a solution of an analytic implicit function theorem.
\end{proof}

\begin{rem}The results obtained so far in this section does not require that $\sigma$ has 
distinct eigenvalues. To apply the results to the real manifolds, we impose it again as in previous sections.
\end{rem}
 
\begin{thm}\label{rigidQ} Let $M$ be a germ of analytic submanifold that is a third order perturbation of a product quadric $Q$ in $\cc^{2p}$.
  Suppose that $M$, i.e. its $\sigma$,
has $n$ distinct eigenvalues.  Suppose that $M$ is formally equivalent to
the   product quadric $Q$. Suppose that each hyperbolic component has an eigenvalue $\mu_h$ which
is either a root of unity or satisfies the Brjuno condition \re{bruno-cond}. Then $M$ is holomorphically equivalent
to the product quadric.
\end{thm}
\begin{proof}
We first apply Theorem \ref{theo-invariant} with $\mathcal I=0$ (\cite{stolo-bsmf}) that linearize simultaneously and holomorphically the $\sigma_1,\ldots, \sigma_p$.
 Note that the small divisor condition in this special case
is equivalent that each $\mu_h$ is either a root of unity or a Brjuno number. 
 Then, we apply successively the two assertions of \rp{linearS}. Hence, in good holomorphic coordinates, $\{\tau_{11},\ldots,\tau_{1p},\rho\}$ are linear. Then, by \rp{inmae}, the manifold is holomorphically equivalent to the quadric.
\end{proof}

We present two   convergence proofs for \rt{abelinv}: one is based on normalization for
each member of the family $\{\sigma_1, \ldots, \sigma_p\}$, and another is based on simultaneous 
normalization for the whole family.    
 Besides the simultaneous linearization in a more general frame work~\cite{stolo-bsmf}  used above, 
  the first approach by linearizing the 
family $\{\sigma_1,\ldots, \sigma_p\}$ one by one is still valid. Here it is crucial that 
the linear maps of $\{\sigma_1, \ldots, \sigma_p\}$ have a very simple structure. 
Indeed, let $\phi_1$ be
a holomorphic mapping that linearizes $\sigma_1$; the existence of such a convergent
$\phi_1$ is ensured \cite{Ru02}.  With the transformation by $\phi_1$, we
may assume that $\sigma_1$ is the linear $\hat S_1$.
Let $\phi_2$ be the unique holomorphic mapping that is normalized w.r.t. $\hat S_2$ and 
linearizes $\sigma_2$. Since $\hat S_1$ and $\sigma_2$ commute, we verify that $\hat S_1
\phi_2\hat S_1^{-1}$ is normalized w.r.t. $\hat S_2$ and  linearizes $\sigma_2$. Then $\phi_2$
commutes with $\hat S_1$ and linearizes $\sigma_2$. Inductively, we find a biholomorphic mapping
that linearizes all $\sigma_1, \ldots, \sigma_p$.   The remaining argument is as in the   proof of 
the theorem.

\setcounter{thm}{0}\setcounter{equation}{0}
\section{
Existence of attached  complex manifolds}
\label{secideal}

We are interested in 
complex submanifolds $K$  in $\cc^{2p}$ that    intersect the real
submanifold $M$ at the origin. Recall that $M$ has real dimension $2p$. Generically, the origin is 
an isolated intersection point if $\dim K=p$.  Let us consider the situation when
the intersection has dimension $p$. Without further restrictions, there are many such
complex submanifolds; for instance, we can take a $p$-dimensional
 totally real and real analytic
submanifold $K_1$ of $M$. We then let $K$ be the complexification of $K_1$.  
 To ensure the uniqueness or finiteness of  the complex submanifolds $K$,
we therefore introduce
the following.
\begin{defn} Let $M$ be a formal real submanifold of dimension $2p$ in $\cc^n$.
We say that a formal complex submanifold $K$ is  {\it attached} to $M$ if $K\cap M$
 contains at least two germs of  totally real and formal submanifolds $K_1, K_2$  of dimension $p$
that intersect transversally at the origin. Such a pair $\{K_1,K_2\}$ are called a pair of {\it asymptotic} formal submanifolds of $M$.  \end{defn}

Before we present the details, let us describe the main steps to derive the results. 
We first  derive the results at the formal level.
We then apply the results of \cite{Po86} and \cite{stolo-bsmf}. The proof of the co-existence of convergent and divergent attached
submanifolds will rely on a theorem of
P\"oschel  on stable invariant submanifolds and Siegel's small divisor technique used in
the proof of  the divergent
normal form in  section~\ref{div-sect}. However, the argument for the divergent part will be simpler. 

We now describe the formal results. 
When $p=1$ and $M$ has a non-resonant hyperbolic complex tangent,
it admits a unique attached formal holomorphic curve~\cite{Kl85}.  When $p>1$,  new situations arise. 
First, we show that there are obstructions to
attach formal submanifolds. 
However, the  formal obstructions disappear when  
 $M$ admits the maximum number
of deck transformations and   $M$ is non-resonant. 
 These two conditions allow us to express $M$
in an 
equivalent form \re{masym}.  This equivalent form for $M$, which has not been used so far, will play an essential
role in our proof 
 for $p>1$.

 We will   consider  a real submanifold 
$M$ which is   a higher order perturbation of a non-resonant  product quadrics. 
By adapting the proof of Klingenberg~\cite{Kl85} to  the manifold $M$  \re{masym}, 
we will show the existence of a unique attached  formal submanifold for a prescribed non-resonance condition. As in~\cite{Kl85},
we  also show that  the complexification of $K$ in $\cL M$ is a pair of invariant formal submanifolds $\cL K_1,\cL K_2$ of $\sigma$.
Furthermore,  $K$ is convergent if and only if $\cL K_1$ is convergent. 
 


Let us first recall the values of the Bishop invariants. The types of the invariants play an important role
 for the existence 
and the convergence of  attached formal complex submanifolds. 
 From \re{gsss}, and \re{gens}, 
we recall that
\ga\label{gens11}
\gamma_e=\frac{1}{\la_e+\la_e^{-1}}, \quad
\gamma_h=\frac{ 1}{\la_h+\ov\la_h},\quad
\gamma_s=\frac{1}{1+\la_s^{ -2}},\\
\label{gehs11}
 0<\gaa_e<1/2, \quad
 \gaa_h>1/2, \quad
\gaa_s\in (1/2,\infty)+i(0,\infty), \quad
\gaa_{s+s_*}=1-\ov\gaa_s.
\end{gather}
As in \rl{t1t2sigrho}, we normalize
\ga\label{Larange}
\la_e>1,\quad |\la_h|=1,\quad |\la_s|>1,\quad\la_{s+s_*}=\ov\la_{s}^{-1};\\
\arg\la_{h}\in(0,\pi/2), \quad\arg\la_s\in(0,\pi/2).
\label{Larange+}
\end{gather}
Recall that $\mu_j=\la_j^2$. By \re{gens11},  we have 
\eq{gsgssmu}
\gaa_j^2=\frac{\mu_j}{(1+\mu_j)^2}, \quad j=e,h; \qquad
\gaa_s\ov\gaa_{s+s_*}= \frac{\mu_s}{(1+\mu_s)^2}.
\eeq
We first verify the following.
\le{disga2} Let $\gaa_j,\la_j$ be given by \rea{gens11}-\rea{Larange+}. Let $\mu_j=\la_j^2$.
Assume that $\mu_1$,   $\mu_1^{-1}, \ldots$,  $ \mu_p^{-1}$ are distinct. Then 
  $$\gaa_{e}^2, \quad \gaa_{h}^2, \quad \ov\gaa_s\gaa_{s+s_*}, \quad \gaa_{s}\ov\gaa_{s+s_*}$$
are distinct $p$ numbers.  The latter is equivalent to $\gaa_1, \ldots, \gaa_p$ being distinct. 
\ele
\begin{proof} Note that $x^{-1}+x$ and $x^{-1}$
   decrease strictly  on $(0, 1)$.  So $\gaa_e^2,\gaa_h^2$ are distinct. 
 We also have
 $$
 \gaa_s\ov\gaa_{s+s_*}=\gaa_s-\gaa_s^2. 
 $$
If $a,b$ are complex numbers, then 
 $a-a^2=b-b^2$ if and only if $a=b$ or $a+b=1$. Since $\gaa_s$ is not real,
 then $\gaa_s\ov\gaa_{s+s_*}$ are different from $\gaa_e^2$ and $\gaa_h^2$. 
 For any distinct complex numbers $a_{1},   a_2$ in $(0,\infty)+i(1/2,\infty)$.
 We have $1-a_2\neq 1-a_1, a_1, a_2$. The lemma
 is proved.
 \end{proof}

Let us 
first investigate the numbers  of pairs of formal asymptotic submanifolds and attached formal submanifolds. 
\begin{lemma}\label{asynum}
Let $M$ be a formal submanifold that is  a third order perturbation of  a product quadric $Q$ in $\cc^{2p}$.  
Assume that the associated $S$
of   $Q$ has distinct eigenvalues $$\mu_1,\dots, \mu_p, \quad\mu_1^{-1}, \ldots, \mu_p^{-1}.$$
  \bppp
\item If $M$ admits an attached formal submanifold,  its CR singularity has no elliptic component. 
\item If $Q$ has no elliptic components, then
$Q$ has at least $2^{h_*+s_*-1}$ pairs of asymptotic totally real and real analytic submanifolds
and all of  them are contained in a single  attached complex submanifold.
\item  There is no formal  submanifold   attached to
$$
M\colon z_3=(z_1+2\gaa_1\ov z_1)^2+(z_2+2\gaa_2\ov z_2)^3,
\quad z_4=(z_2+2\gaa_2\ov z_2)^2.
$$
Here  $M$ has a hyperbolic complex tangent at the origin. 
\item Assume that $M$ has no elliptic component and it admits the maximum number of formal deck transformations. 
Let 
\eq{nunue-}
\nu=\mu_\e=(\mu_1^{\e_1},\ldots, \mu_p^{\e_p}), \quad \e_j=\pm1, \quad \e_{s+s_*}=\e_s. 
\eeq
Suppose that  
\eq{nunue+}
\nu^Q\neq\nu_j^{-1}, \quad \forall Q\in\nn^p, \quad |Q|>0, \quad 1\leq j\leq p.
\eeq
Then $M$ admits  a unique pair of asymptotic formal submanifolds $K_1,K_2$ such that  each  $K_i$
is defined by $z'=\rho_i(z')$ for   a formal anti-holomorphic involution $\rho_i$ and 
the linear part of  $\ov\rho_2^{-1}\ov\rho_1$  has eigenvalues   $\nu_1,\ldots, \nu_p$.  
In particular, if \rea{nunue+} holds for each $\nu$ of the form \rea{nunue-} then 
  $M$ admits exactly $2^{h_*+s_*-1}$ pairs
of asymptotic formal submanifolds.  \eppp
 \end{lemma}
\begin{proof} (i) Let $M$ be defined by
$$
z_{p+j}=Q_j(z',\ov z')+H_j(z',\ov z'), \quad 1\leq j\leq p
$$
where $H_j(z',\ov{z'})=O(|z'|^3)$ and each $Q_j$ is quadratic. 
Let $\{K_1,K_2\}$ be a pair of asymptotic formal submanifolds of $M$. 
We know that $K_1,K_2$ are tangent to $M$ at the origin. Let $K_i'$ be the projection of $K_i$
onto the $z'$-subspace. Since $T_0M$ is a $p$-dimensional
complex subspace, then $K_1',K_2' $ are still totally real. 
Let $K_1'$ be defined by
\gan
K_1'\colon \ov z'={\mathbf A}z'+R(z'), \quad \ov {\mathbf A}{\mathbf A}={\mathbf I},  \quad R(z')=O(2)
\end{gather*}
such that $\rho_1(z'):=\ov{\mathbf A}\ov z'+\ov R(\ov z')$ defines anti-holomorphic  formal involutions.  Let $K_2$ be the (formal)
 fixed-point
set of  the anti-holomorphic
involution 
$\rho_2(z')=\ov{\tilde{\mathbf A}}\ov z'+\ov{\tilde R(z')}$ with $\tilde R(z')=O(2)$.  Then $K_1,K_2$ intersect transversally at the origin
if and only if 
$$
\det(\tilde{\mathbf A}-\mathbf A)\neq0. 
$$
Let us define  holomorphic mappings
$$
\ov\rho_i(z'):=\ov{\rho_i(z')}, \quad i=1,2. 
$$
Then $K$ is given by
$$
z_{p+j}''=Q_j(z',\ov\rho_{i}(z'))+H_j(z',\ov\rho_i(z')), \quad i=1,2, \quad j=1, \ldots, p.
$$
The  two equations agree, if and only if
\eq{Qzpb}
Q_j(z', \ov\rho_{1}(z'))+H_j(z',\ov\rho_1(z'))=Q_j(z',\ov\rho_{2}(z'))+H_j(z',\ov\rho_2(z')), \quad 1\leq j\leq p. 
\eeq
Recall that
\aln
Q_j(z',\ov z)&=(z_j+2\gaa_j\ov z_j)^2,\quad j=e,h;\\
Q_s(z',\ov z')&=(z_{s+s_*}+2\gaa_{s+s_*}\ov z_s)^2,\\
Q_{s+s_*}(z',\ov z')&=
(z_{s}+2\gaa_{s}\ov z_{s+s_*})^2.
\end{align*}
Let us first find necessary conditions on the linear parts of $\rho_i$ for \re{Qzpb} to  be solvable.
Let $w'={\mathbf A}z'$ and $\tilde w'=\tilde{\mathbf A}z'$. 
Comparing the quadratic terms in \re{Qzpb}
for $i=1,2$, we see that
\gan
(z_j+2\gaa_jw_j)^2=(z_j+2\gaa_j\tilde
w_j)^2,\\
( z_{s+s_*}+2\gaa_{s+s_*}w_s)^2=( z_{s+s_*}+2\gaa_{s+s_*}\tilde
w_s)^2,\\
(z_{s}+2\gaa_{s}w_{s+s_*})^2=( z_{s}+2\gaa_{s}\tilde
w_{s+s_*})^2.
\end{gather*}
Here $\gaa_{s+s_*}=1-\ov\gaa_s$, by \re{gehs11}.
For each $j$, $w_j\neq\tilde w_j$. Otherwise, the fixed points of $\rho_1$ and $\rho_2$
do not intersect transversally.   Therefore, the above $3$ identities can be written as
\gan
z_j+2\gaa_jw_j=-(z_j+2\gaa_j\tilde
w_j),\\
z_{s+s_*}+2\gaa_{s+s_*}w_s=-( z_{s+s_*}+2\gaa_{s+s_*}\tilde
w_s),\\
z_{s}+2\gaa_{s}w_{s+s_*}=-( z_{s}+2\gaa_{s}\tilde
w_{s+s_*}).
\end{gather*}
In the matrix form, we get $\tilde {\mathbf{A}}=-{\boldsymbol{\gamma}}^{-1} -\mathbf{A}$ with
$$
\boldsymbol{\gamma}:=\begin{pmatrix}
  \boldsymbol{\gaa}_{e_*}    &\mathbf{0} &\mathbf{0} &\mathbf{0}\\
    \mathbf{0}&\boldsymbol\gaa_{h_*}&\mathbf{0}& \mathbf{0}  \\
   \mathbf{0}&\mathbf{0}&\mathbf{0}&\boldsymbol{\gaa}_{s_*} \\
   \mathbf{0}&\mathbf{0}&{\boldsymbol{\tilde \gaa}}_{s_*}&\mathbf{0}
\end{pmatrix}.
$$
Here $\tilde{\boldsymbol{\gaa}}_{s_*}=\mathbf I_{s_*}-\ov {\boldsymbol{\gaa}}_{s_*}$.  
Let us express in block matrices
$$
{\mathbf A}=\begin{pmatrix}
{\mathbf A}_{e_*e_*}  &{\mathbf A}_{e_*h_*} &{\mathbf A}_{e_*s_*} &{\mathbf A}_{e_*(2s_*)} \\
{\mathbf A}_{h_*e_*}  &{\mathbf A}_{h_*h_*} &{\mathbf A}_{h_*s_*} &{\mathbf A}_{h_*(2s_*)} \\
{\mathbf A}_{s_*e_*}  &{\mathbf A}_{s_*h_*} &{\mathbf A}_{s_*s_*} &{\mathbf A}_{s_*(2s_*)} \\
{\mathbf A}_{(2s_*)e_*}  &{\mathbf A}_{(2s_*)h_*} &{\mathbf A}_{(2s_*)s_*} &{\mathbf A}_{(2s_*)(2s_*)}
\end{pmatrix}
$$
where the diagonal block matrices are of sizes $e_*\times e_*,h_*\times h_*,s_*\times s_*$, and $s_*\times s_*$, respectively.
When  $\mathbf{A}\ov {\mathbf{A}}=\mathbf I$, for $\tilde {\mathbf{A}}\ov{\tilde {\mathbf{A}}}=\mathbf I$ we need
$
{\boldsymbol{\gamma}}^{-1}+\mathbf{A}+\boldsymbol{\gamma} ^{-1}\ov {\mathbf{A}}\ov{\boldsymbol{\gamma}}=0.
$
Recall that $\gaa_1^2,\ldots, \gaa_{e_*+h_*}^2$ 
 are real and distinct. It is easy to see that $\mathbf{A}_{e_*h_*}=0$, $\mathbf{A}_{h_*e_*}=0$, and $\mathbf{A}_{e_*e_*}, \mathbf{A}_{h_*h_*}$ are diagonal. Also,
\eq{aeses}
\mathbf{A}_{e_*e_*}+\ov {\mathbf{A}}_{e_*e_*}=-\boldsymbol{\gamma}_{e_*}^{-1}, \quad \mathbf{A}_{h_*h_*}+\ov {\mathbf{A}}_{h_*h_*}=-\boldsymbol{\gamma}_{h_*}^{-1}.
\eeq
In block matrices, we obtain 
\ga 
{\boldsymbol{\gamma}}_j^{-1} \ov {\mathbf{A}}_{j(2s_*)}\ov{\tilde{\boldsymbol{\gamma}}}_{s_*}=- {\mathbf{A}}_{js_*},   \qquad
{\tilde{\boldsymbol{\gamma}}}_{s_*} ^{-1}\ov{\mathbf{A}}_{ (2s_*)j}\ov{\boldsymbol{\gamma}}_{j}=-{\mathbf{A}}_{s_*j}; \\ 
{\boldsymbol{\gamma}}_j^{-1}\ov {\mathbf{A}}_{js_*}\ov{\boldsymbol{\gamma}}_{s_*}=-{\mathbf{A}}_{j(2s_*)},  
 \qquad{\boldsymbol{\gamma}}_{s_*}^{-1} \ov {\mathbf{A}}_{s_*j}{\boldsymbol{\gamma}}_{j}=- {\mathbf{A}}_{(2s_*)j}; \\ 
{\tilde{\boldsymbol{\gamma}}}_{s_*}^{-1}\ov {\mathbf{A}}_{(2s_*)(2s_*)}\ov{
\tilde{\boldsymbol{\gamma}}}_{s_*}=- {\mathbf{A}}_{s_*s_*},  
 \qquad{\tilde{\boldsymbol{\gamma}}}_{s_*}^{-1}\ov {\mathbf{A}}_{(2s_*)s_*}\ov{\boldsymbol{\gamma}}_{s_*}=-{\mathbf{A}}_{s_*(2s_*)}-{\tilde{\boldsymbol{\gamma}}}_{s_*}^{-1},
\label{tgss}
\\
{\boldsymbol{\gamma}}_{s_*}^{-1}\ov {\mathbf{A}}_{s_*(2s_*)}\ov{\tilde{\boldsymbol{\gamma}}}_{s_*}=- {\mathbf{A}}_{(2s_*)s_*}-{\boldsymbol{\gamma}}_{s_*}^{-1}, 
 \qquad  {\boldsymbol{\gamma}}_{s_*}^{-1}\ov {\mathbf{A}}_{s_*s_*}\ov{\boldsymbol{\gamma}}_{s_*}=-{\mathbf{A}}_{(2s_*)(2s_*)}.
\end{gather} 
In the first $4$ equations,  we have $j=e_*,h_*$. 
Note that the last two equations are of the form  
 \re{tgss}.

By \rl{disga2},  we know that $\gaa_e^2,\gaa_h^2$, and $\gaa_{s}\ov\gaa_{s+s_*}$ are distinct.
 Thus,  ${\mathbf{A}}_{js_*}={\mathbf{A}}_{j(2s)_*}=\mathbf 0$
and ${\mathbf{A}}_{s_*j}={\mathbf{A}}_{(2s_*)j}=\mathbf 0$ for $j=e_*,h_*$.  Since $\gaa_{s}\ov{{\gaa}}_{s+s_*}$ is
different from all $\gaa_{s+s_*}\ov\gaa_{s}$, then ${\mathbf{A}}_{s_*s_*}={\mathbf{A}}_{(2s_*)(2s_*)}=\mathbf 0$ while
${\mathbf{A}}_{s_*(2s_*)}$, ${\mathbf{A}}_{(2s_*)s_*}$ are diagonal.  
Now ${\mathbf{A}}\ov {\mathbf{A}}=\mathbf I$ implies that
\eq{aehs}
{\mathbf{A}}_{e_*e_*}\ov {\mathbf{A}}_{e_*e_*}=\mathbf I, \quad
{\mathbf{A}}_{h_*h_*}\ov {\mathbf{A}}_{h_*h_*}=\mathbf I, \quad
{\mathbf{A}}_{s_*(2s_*)}\ov {\mathbf{A}}_{(2s_*)s_*}=\mathbf I.
\eeq
 Combining the first identities  in \re{aeses} and \re{aehs}, we know that  the  diagonal $e$th element $a_e$ of 
 ${\mathbf{A}}_{e_*e_*}$ must satisfy
 $$
 2a_e=\gaa_e^{-1}, \quad a_e^2=1.
 $$
 Since   $0<\gaa_e<1/2$,  there is no such solution $a_e$ if $e_*>0$.
  We have verified (i).

Note that $\gaa_h^{-1}=\la_h+\ov\la_h$ with $|\gaa_h|=1$.   For the hyperbolic components, by the second identities in 
  \re{aeses} and \re{aehs}, one set of solutions is given by
  \ga\label{ahsh}\nonumber
\mathbf A_{h_*h_*}=-{\boldsymbol\la}_{h_*}
\quad
\tilde {\mathbf A}_{h_*h_*}=-{\boldsymbol{\la}}_{h_*}^{-1}.
\end{gather}
For the complex components, we  use  ${\mathbf{A}}_{s_*(2s_*)}\ov {\mathbf{A}}_{(2s_*)s_*}=\mathbf I$
and multiply both sides of the second identity in \re{tgss} by $\mathbf A_{s_*(2s_*)}$. The $(s-s_*)$th diagonal element $a_s$
must satisfy
$$
a_s(a_s+\tilde\gaa_s^{-1})+\tilde \gaa_s^{-1}\ov\gaa_s=0.
$$
By the last identity in \re{gsgssmu}, we get 
$$
a_s^2+(1-\ov\mu_s)a_s+\ov\mu_s=0.
$$
Hence, $a_s=-1$ or $a_s=-\ov\mu_s$.  We get one set of solutions 
\begin{alignat}{4}
\nonumber
 {\mathbf A}_{(2s_*)s_*}&=-\mathbf I, \quad& {\mathbf A}_{s_*(2s_*)}&=-\mathbf I,\\
\tilde {\mathbf A}_{(2s_*)s_*}&=-{\boldsymbol \mu}_{s_*}^{-1}, \quad
&\tilde {\mathbf A}_{s_*(2s_*)}&=-\ov{\boldsymbol{\mu}}_{s_*}.
\nonumber
\end{alignat}
There are exactly $2^{h_*+s_*-1}$ solutions for $\mathbf A,\tilde {\mathbf A}$  since we can only determine the
pairs 
$$\{\mathbf A_{h_*h_*}, \tilde {\mathbf A}_{h_*h_*}\}, \quad
\{\mathbf A_{s_*(2s_*)},\mathbf A_{s_*(2s_*)}\}.
$$
Note that 
\begin{gather}\label{ata-1}
\mathbf {\tilde A}^{-1}\mathbf A=\diag\nu,\\
 \label{nunue}
 \nu=\mu_\epsilon=(\mu_1^{\e_1},\ldots,\mu_{p}^{\e_{p}}), \quad \e_j^2=1, \quad \nu_{s+s_*}=\ov \nu_s^{-1},
 \end{gather}
 where there are $2^{h_*+s_*-1}$ distinct combinations.  
Thus,  we get exactly $2^{h_*+s_*-1}$ pairs $\{K_{\e}^1,K_{\e}^2\}$
of asymptotic linear submanifolds indexed by  $\e=(\e_1, \ldots, e_{h_s+s_*})$ with $\e_j^2=1$
 for the product quadric.   We may restrict to $\e_1=1$.
 The attached formal submanifolds associated  to these linear asymptotic submanifolds are  unique and it  is given by
\aln
z_{p+h}&=(1-{4}\gaa_h^{2})z_h^2,\\
z_{p+s}&=(1-2\gaa_{s+s_*})^2z_{s+s_*}^2,\\
z_{p+s+s_*}&=(1-2\gaa_s)^2z_{s}^2.
\end{align*}
This finishes the proof of  (ii).

(iii).
Let us continue the computation for the perturbations. 
We have determined linear parts of antiholomorphic involutions $\rho_i$.
We expand components of $R(z')$ as
$$
R_{j}(z')=\sum_{k=2}^{\infty}R_{j; k}(z'), \quad 1\leq j\leq p.
$$
Here $R_{j;k}$ are homogeneous terms of degree $k$.  We expand $\tilde R_j$ analogously. 
Suppose that terms of order up to $k-1$ in $R_j,\tilde R_j$
have been determined. For the hyperbolic components, we need to solve the equations
\ga\label{zhrh}
4\sqrt{1-4\gaa_h^{2}}z_h(R_{h;k}(z')+{\tilde R_{h;k}}(z'))=\cdots,
\end{gather}
where the right-hand side has been determined.  Indeed, let us compute the $(k+1)$-jet of \re{Qzpb}. We obtain
$$
(1-2\gaa_j\lambda_j)^2z_j^2 + 2(1-2\gaa_j\lambda_j)z_jR_{j;k}= (1-2\gaa_j\lambda_j^{-1})^2z_j^2 + 2(1-2\gaa_j\lambda_j^{-1})z_j\tilde R_{j;k}+ {\mathcal R}
$$
where ${\mathcal R}$ is polynomial that depends on $\tilde R_{j;l},R_{j;l}$, $l<k$. Since $(1-2\gaa_j\lambda_j)=-(1-2\gaa_j\lambda_j^{-1})$, we obtain \re{zhrh}. 
  
	When $p>1$, the system of equations \re{zhrh} cannot be solved
even formally, unless the right-hand side is divisible by $z_h$.  When $p=1$, 
the equation \re{zhrh} is clearly solvable. In fact, under the non-resonant condition on $\mu_1$, the formal  anti-holomorphic  involutions $\{\rho_1,\rho_2\}$
can be uniquely determined.

 Let us keep the above notation and compute for the example stated in (iii).
We need to solve
\aln
(z_1+2\gaa_1\tilde  w_1)^2+(z_2+2\gaa_2\tilde  w_2)^3& =
(z_1+2\gaa_1 w_1)^2+(z_2+2\gaa_2 w_2)^3,\\
(z_2+2\gaa_2\tilde  w_2)^2&=(z_2+2\gaa_2 w_2)^2.
\end{align*}
Again $\tilde w_2- w_2$ cannot be identically zero. Thus
$
\tilde w_2=- w_2-\gaa_2^{-1} z_2.
$
Then we need to solve
\aln
(z_1+2\gaa_1\tilde  w_1)^2 =
(z_1+2\gaa_1 w_1)^2+2(z_2+2\gaa_2 w_2)^3.
\end{align*}
By (ii), we know that $ w_1=\la_1z_1+R_1(z')$ and $ w_2=\la_2 z_2+R_2(z')$ with $R_i(z')=O(2)$.
Also $\tilde  w_1=\ov\la_1z_1+\tilde R_1(z')$ and $\tilde w_2=\ov\la_2 z_2+\tilde R_2(z')$.
Comparing the cubic terms implies that $z_1$ must divide
$
2(1+2{\gaa_2}\la_2)^3z_2^3,
$
which is a contradiction. 

(iv)   
For a general $M$, following Klingenberg \ci{Kl85} we reformulate the problem by considering the
following equations
\aln 
h(z')&=q(z',\ov \rho_i(z'))+H(z',\ov \rho_i(z')),\quad i=1,2,\\
h^*(\ov\rho_i(z'))&=\ov q( \ov\rho_i(z'), z')+\ov {H}(\ov\rho_i(z'),z'), \quad i=1,2.
\end{align*}
Here $h,h^*, \ov\rho_i$ are unknowns.  Initially, we only require that
 $\ov\rho_1,\ov\rho_2$ be arbitrary biholomorphic maps, except 
 their linear parts  
match with $z'\to Az'$ and $z'\to\tilde Az'$.
This will ensure that the solutions $\ov\rho_i$ are unique and they are involutions.

 As demonstrated in (iii), in general there is no formal submanifold attached to $M$.
 We now assume that $M$ admits the maximum number of deck transformation.
By \rl{sd1} and \rp{inmae} we know that in suitable holomorphic coordinates, $M$ is given by  
\aln
z_{p+j}&=\Bigl(\sum_hb_{jh}(z_h+2\gaa_h\ov z_h)+
\sum_sb_{js}(z_s+2\gaa_s\ov z_{s+s_*})\\
&\quad +\sum_sb_{j(s+s_*)}(
z_{s+s_*}+2\gaa_{s+s_*}\ov  z_s)+E_j(z,\ov z)\Bigr)^2, \quad 1\leq j\leq p.
\end{align*}
Here $(b_{jk})$  is invertible and $E_j(z,\ov z)=O(2)$.  
This special form, which has not played significant roles until now,  will allow us
removing the obstruction to formal solutions $\rho_i$. 

For the proof of our result,  we will restrict  $(b_{jk})$ to be the identity matrix. 
Let $M$ be defined by 
\aln
z_{p+h}&=(z_h+2\gaa_h\ov z_h+E_j(z',\ov z'))^2, \\
\quad z_{p+s}&=(z_s+2\gaa_s\ov z_{s+s_*}+  E_{p+s}(z',\ov z'))^2,\\
z_{p+s+s_*}&=(
z_{s+s_*}+2\gaa_{s+s_*}\ov  z_s+  E_{p+s+s_*}(z',\ov z'))^2.
\end{align*}
 We fix linear parts of $\rho_i$ such that 
 $$
  \rho_1(z')=\ov A\ov z'+\ov R(\ov z'), \quad  \rho_2(z')=\ov {\tilde A}\ov z'+\ov {\tilde R}(\ov z').
   $$
For $i=1,2$ we then need to solve $w,\tilde w$ from
 \al\label{zh2g}
 z_h+2\gaa_h\ov \rho_{ih}+ E_h(z',\ov\rho_i)&=(-1)^if_h,\\ 
 z_s+2\gaa_s\ov \rho_{is+s_*}+E_{p+s}(z',\ov\rho_i)&=(-1)^if_{s},\\ 
\label{zsss}
 z_{s+s_*}+2\gaa_{s+s_*}\ov\rho_{is}+E_{p+s+s_*}(z',\ov\rho_i)&=(-1)^if_{s+s_*},\\ 
 \label{2ghz}
  2\gaa_hz_h+\ov \rho_{ih}+ \ov E_h(\ov\rho_i,z')&=(-1)^i f^*_h(\ov\rho_i),\\ 
  2\ov\gaa_sz_{s+s_*}+\ov \rho_{is}+\ov E_{p+s}(\ov\rho_i,z')&=(-1)^if_s^*(\ov\rho_i),\\ 
  2\ov\gaa_{s+s_*}z_s+\ov\rho_{is+s_*}+\ov E_{p+s+s_*}(\ov\rho_i,z')&=
  (-1)^if_{s+s_*}^*(\ov\rho_i). \label{zh2g6}
  \end{align}
 Suppose that we have already determined terms of $R_{j},\tilde R_j,
 f_j,f_{j}^*$ of order $<k$. 
 We have 
 $$
 \ov\rho_1(z')=\mathbf Az'+R(z'), \quad \ov\rho_1^{-1}(z')=\mathbf A^{-1}z'-\mathbf A^{-1}R'(\mathbf A^{-1}z'),
 $$
 where  the terms in $R'-R$ of order $k$ depend only on terms of $R$ of order $<k$. 
 Recall that $\mathbf{\tilde A}^{-1}{\mathbf A}=\diag\nu$ is given by \re{ata-1}.  For terms of
 order $k$, by eliminating $f_j,f_j^*$, we therefore  need to solve  
 \ga\label{rjq}
 R_{jQ}+\tilde R_{jQ}=\cdots
 \end{gather}
 where the dots denote terms which have been determined.   
 We compose from
 right in the last 3 identities for $i=1$ (resp. $i=2$)  by $\ov\rho_1^{-1}$ (resp $\ov\rho_2^{-1}$).  From the new identities, we obtain
 $$
 \mathbf A^{-1}R (\mathbf A^{-1}z') +
  \tilde {\mathbf A}^{-1}\tilde R (\tilde{\mathbf A}^{-1}z') =\cdots.
  $$
Recall that the linear part of $\tilde{\mathbf A}^{-1}{\mathbf A}$ is $\diag\nu$ with $\nu:=\nu_{\e}$. 
 Thus we need to solve \re{rjq} and 
$$
\nu_j^{-1}R_{j,Q}+\nu^Q {\tilde R}_{j,Q}=\cdots.$$  
  The equation admits a unique solution as   \eq{nuqn0}
  \nu^Q\neq\nu_j^{-1}, \quad Q\in\nn^p, \quad |Q|>1,\quad 1\leq j\leq p.
 \eeq
This shows that  $R_{j,Q}, \tilde R_{j,Q}$ are uniquely determined. 
 
 To verify that $\rho_i$ are involutions,  we compose by $\ov \rho_i^{-1}$ from
 right  
 in \re{zh2g}-\re{zsss}, and we apply complex conjugate to the coefficients of the new identities.
 This results in   \re{2ghz}-\re{zh2g6} in which $(\ov\rho_i, f_j^*)$ are replaced
 by $(\ov{(\ov\rho_i)^{-1}}, \ov f_i)$. We can also start with  \re{2ghz}-\re{zh2g6}
  and apply the same procedure to get \re{zh2g}-\re{zsss}, in which $(\ov\rho_i, f_i)$ 
 are replaced by $(\ov{(\ov\rho_i)^{-1}}, \ov f^*_i)$.  By the uniqueness of the solutions, we conclude
 that  $\ov {(\ov\rho_i)^{-1}}=\ov \rho_i$ as both sides have the same linear part.
 We now have
 $
 \ov{(\ov\rho_i)^{-1}(\ov z')}=\ov{\rho_i(z')}$. Hence $ \ov z'=\ov\rho_i(\rho_i(z'))=\ov{\rho_i^2(z')}.
  $
   This shows that each $\rho_i$ is an involution.
 \end{proof}
 We remark that given  complex numbers
 $$
 \mu_1,\ldots, \mu_{h_*}, \quad \mu_{h_*+1},\ldots, \mu_{h_*+s_*}, \quad
 \mu_{h_*+s_*+s}=\ov\mu_{h_*+s}^{-1}
 $$
 with $|\mu_h|=1$. Let $\nu=\mu_\e$ be given by \re{nunue-}.  The set of $\nu$ that violate \re{nunue+} is contained in the union of the sets defined by
 $
 \nu^Q=\nu_j^{-1}.
 $
 Here $Q\in\nn^p$, $|Q|>1$ and $1\leq j\leq p$. For each $Q,j$, the above equations define an algebraic set of codimension at least $1$ in the space
 $
 (S^1)^{h_*}\times\cc^{s_*}.
 $
   
We now can prove the following theorem.
\begin{thm}\label{invep}
  Let $M$ be a higher order perturbation of a product quadric. 
  Assume that in $(\xi,\eta)$ coordinates,
   its associated  $\sigma$ has a linear part given by   the diagonal matrix
   with diagonal entries $\mu_1,\ldots, \mu_p$,  $\mu_1^{-1}, \ldots, \mu_p^{-1}$.
Let $\nu=\nu_e$ be of the form \rea{nunue} and satisfy \rea{nuqn0}.
  Then $M$ admits 
a unique pair of asymptotic submanifold  $\{K^{\e}_1, K^{\e}_2\}$ such that the complexification  
 of $K^{\e}_1$ in $\cL M$
is an invariant formal submanifold $\cL H_\e$
of $\sigma$ that is tangent to 
\eq{clhe}
\cL H_\e=
\Bigl(\bigcap_{e_j=1, 1\leq j\leq p}\{\xi_j=0\}\Bigr)   \cap\Bigl(\bigcap_{e_i=-1, 1\leq i\leq p}\{ \eta_i=0\}\Bigr).
\eeq 
Furthermore, the complexification of $K_2^\e$ equals $\tau_1\cL H_\e$. \end{thm}
\begin{proof}
We will follow Klingenberg's approach for $p=1$,  by using the deck transformations. 
Here we
assume that $M$ admits the maximum number of deck transformations.  Suppose that
$K$ is an attached formal complex submanifold which intersects with $M$ at two
totally real formal submanifolds $K_1,K_2$.
We first embed
$K_{1}\cup K_{2}$ into $\cL M$ as $M$ is embedded into $\cL M$. Let $\cL K_{i}$ be
the complexification of $K_{i}$ in $\cL M$. Since $\rho$ fixes $K_{i}$ pointwise,
then $\rho\cL K_{i}=\cL K_{i}$.

 We want to show that $\tau_1(\cL K_{1})=\cL K_{2}$;
thus $\cL K_{i}$ is invariant under $\sigma$.  Recall that $\cL K_{i}$ is defined by
\eq{briz}
\ov\rho_i(z')=w'.
\eeq
On $\cL K_{1}$,  by \re{zh2g} and \re{zsss} we have  
$\tilde L(z',w')+E(z',w')=-f(z')$.  The latter defines a complex
submanifold of dimension $p$. Thus it must be  $\cL K_{1}$.
On $\cL M$,
$$
(\tilde L_j(z',w')+E_j(z',w'))^2=z_{p+j}
$$
are invariant by $\tau_{1}$. Thus each $\tilde L_j(z',w')+E_j(z',w')$ is either
invariant or skew-invariant by $\tau_1$. Computing the linear part, we conclude
that they are all skew-invariant by $\tau_1$. Hence $\tau_1(\cL K_1)$ is defined by
$$
\tilde L(z',w')+E(z',w')=f(z'),
$$
which is the defining equations for $\cL K_2$.   

Finally, if $\cL K_1$ is convergent, then \re{briz} implies that $\ov\rho_1$ is convergent. Hence $K_1$, the fixed point set of $\rho_1$,
 is convergent.\end{proof}

We now study the convergence of attached formal submanifolds. 
Let us first recall a  theorem of P\"oschel~\cite{Po86}. Let $\nu$ and $\e$ be as in \re{nunue}.
Define
$$
\omega_{\nu}(k)=\min_{1<|P|\leq 2^k, P\in\nn^p}\min_{1\leq i\leq p}
\left\{|\nu^P-\nu_i|,
|\nu^P-\nu_i^{-1}|\right\}.
$$
Suppose that 
\eq{omnu}
-\sum\frac{\log\omega_{\nu}(2^k)}{2^k}<\infty.
\eeq
Then the unique invariant formal submanifold of $\sigma$ that is tangent  to the $\cL H_\e$  defined by \re{clhe}
is convergent. 

We now obtain a consequence of \rt{invep} and P\"oschel's theorem. 
\begin{thm}\label{pocor}
 Let $M$ be a higher order perturbation of a product quadric. Suppose that $M$ admits the maximum number of deck
transformations. Assume that the CR singularity of $M$ has no elliptic components. 
 Let $\nu=\mu_\e$ be given by \rea{nunue}.   Assume that $\nu=(\mu_1^{\epsilon_1}, \ldots, \mu_{p}^{\epsilon_{p}})$ 
satisfy  \rea{omnu}. Then $M$ admits an attached complex submanifold. 
\end{thm}

Since the eigenvalues of $\sigma$ are special, we verify that the condition \re{omnu} can be satisfied.

Let us first prove \rp{oneconv}, by  considering the case when the complex tangent has pure complex type. 
Then     condition \re{omnu} always holds  if $\nu_1,\ldots, \mu_p$ satisfy  the weaker non-resonance condition
\re{nuqn0}.
  Indeed,  in this case, the eigenvalues of $\sigma$ are
$$
\mu_s, \quad \mu_{s_*+s}=\ov\mu_s^{-1},\quad \mu_{p+s}=\mu_s^{-1}, \quad \mu_{p+s_*+s}=\ov\mu_s.
$$
 Recall that $1\leq s\leq s_*$ and $p=2s_*$.  We may assume that $|\mu_s|>1$.  We take
 \eq{nusmu}
 \nu_s=\mu_s, \quad \nu_{s+s_*}=\ov\mu_s. 
 \eeq
  Assume that
\eq{nu-n}
\nonumber
\nu^Q-\nu_j\neq0, \quad Q\in\nn^p, \quad |Q|>1, \quad 1\leq j\leq p.
\eeq
Under the condition \re{nusmu},  we can find a positive integer $r$ such that
\eq{nu1r}
\nonumber
\min\left\{|\nu_1|^r, \ldots, |\nu_p|^r\right\}>\max\left\{|\nu_1|,\ldots, |\nu_p|\right\}. 
\eeq
It is easy to see that  $|\mu^P-\nu_j|\geq c$ for some positive constant and all $Q\in\nn^p$ with $|Q|>1$. Hence \re{omnu} holds.  
We have proved \rp{oneconv}.

We now consider the general case by showing that the set of  $\{\mu_h,\mu_{s+s_*},\ov\mu_{s+s_*}^{-1}\}$
that satisfy \re{omnu} for some choice of $\nu$ has the full measure. Without loss of generality, we may assume that
$ 
|\mu_{h_*+s}|>1.
$ 
Thus we  list the eigenvalues of $\sigma$ as
\begin{gather*}
 \boldsymbol{\mu}_{h_*}, \quad \boldsymbol{\mu}_{s_*}, \quad \tilde {\boldsymbol{\mu}}_{s_*},\quad
  \ov{\boldsymbol{\mu}_{h_*}},\quad \ov{\tilde{\boldsymbol{\mu}}_{s_*}}, \quad \ov{\boldsymbol{\mu}_{s_*}}
  \end{gather*}
  with
 $$
 \boldsymbol{\mu}_{h_*}=(\mu_1,\ldots,\mu_{h_*}), \quad
\boldsymbol{\mu}_{s_*}=(\mu_{h_*+1},\ldots,\mu_{h_*+s_*}), 
\quad\tilde {\boldsymbol{\mu}}_{s_*}=(\ov\mu_{h_*+1}^{-1},\ldots,\ov\mu_{h_*+s_*}^{-1}).
 $$
We take $(\nu_1,\ldots,\nu_p)=( \boldsymbol{\mu}_{h_*},  \boldsymbol{\mu}_{s_*}, \ov{\boldsymbol{\mu}_{s_*}})$.
We first note that there are only finitely many $R_1,\ldots, R_d\in\nn^{2s_*}$ such that
$$
|(\nu_{h_*+1}, \ldots, \nu_p)^{R_i}|<2C_*, \quad C_*:=\max\{|\nu_1|, \ldots,|\nu_p|\}.
$$
 Denote by $\{b_1,\ldots, b_{4s_*d}\}$ the set of numbers:
 $$
 \nu_j(\nu_{h_*+1}, \ldots, \nu_p)^{-R_i}, \quad  \nu_j^{-1}(\nu_{h_*+1}, \ldots, \nu_p)^{-R_i}
 $$
with $ h_*< j\leq p$ and $1\leq i\leq d$. Let $\cL S_m$ be the set of $ \boldsymbol{\mu}_{h_*}\in(S^1)^{h_*}$ satisfying the Siegel  
condition
\eq{mijm}
\min_{i,j}\left\{\left | \boldsymbol{\mu}_{h_*}^P-b_j \right|, | \boldsymbol{\mu}_{h_*}^P-\mu_h|,| \boldsymbol{\mu}_{h_*}^P-\mu_h^{-1}|\right\}\geq \frac{C}{(1+|P|)^{m}}, 
\eeq
for  $P\in\nn^{h_*}$ and $|P|>2.$
  One can verify that, $\cup_{m=2}^\infty \cL S_m$ has the full measure on $(S^1)^{h_*}$ for a fixed set of $\{b_j\}$.

We take any $\mu_{h_*+1}, \ldots, \mu_p$ such that 
$$
(|\mu_{h_*+1}|, \ldots, |\mu_p|)^Q\neq |\mu_j|, \quad \forall Q\in\nn^{p-_*}, \quad |Q|>1.
$$
We then take $(\mu_1,\ldots, \mu_{h_*})\in(S^1)^{h_*}$ satisfying \re{mijm}.  We have
\eq{nuPn}
|\nu^P-\nu_j|\geq \frac{C}{(1+|P|)^{m}}, \quad P\in\nn^{h_*}, |P|>2.
\eeq
To verify it, we write $P=(P',P'')$ with $P'\in\nn^{h_*}$. If $P''\neq R_i$ for $1\leq i\leq d$, we have $|\nu^P-\nu_j|\geq C_*$, which satisfies \re{nuPn}. Suppose that $P''=R_i$. Then
for $j>h_*$ we have
$$
|\nu^P-\nu_j|\geq |(\nu_{h_*+1},\ldots,\nu_p)^{R_i}|\min_i\{|(\mu_1,\ldots,\mu_{h_*})^{P'}-b_i|\}\geq \frac{C'}{(1+|P|)^m}.
$$
Suppose that $1\leq j\leq h_*$. If $P''\neq0$, we have 
$$
|\nu^P-\nu_j|\geq\min( |\mu_{h_*+1}|, \ldots, |\mu_p|)-1.
$$
 This gives us \re{nuPn}. Suppose now that $P''=0$. Then we get \re{nuPn}
immediately.  We have verified \re{nuPn} for all cases.

 We have proved that the non-resonant product quadric has a unique attached complex submanifold. 
  Let us first show 
 that the unique 
complex submanifolds attached 
to the product quadric may split into two  attached submanifolds after a  perturbation. In fact,  a stronger result hods;
it could split into a
divergent attached submanifold and a convergent one simultaneously.

\begin{prop}\label{coexist}
There is a non-resonant $4$-dimensional real analytic submanifolds $M$ that has pure complex type and 
 admits a convergent attached submanifold and a divergent one too.  \end{prop}
\begin{proof}
By  \rp{oneconv},  it suffice to show the existence of a divergent attached submanifold. 
The proof is an application of small divisors, as shown in previous divergent result. However, the proof is much simple. We will
be brief.  

Consider 
$$
M\colon z_{3}=(z_1+2\gaa_1\ov z_{2}+ a(z_1z_2))^2,\\
\quad z_{4}=(
z_{2}+2\gaa_{2}\ov  z_1)^2.
$$
Here $a$ is holomorphic in $z_1z_2$ and $a(0)=a'(0)=0$.  By \re{zh2g} for $i=1,2$, we eliminate $f_1$ to obtain
\begin{align*}
2\gaa_1R_{1}+2\gaa_1\tilde R_{1}+a(z_1z_2)&=\cdots,\\
R_{1}\circ\ov\rho_1\circ\ov\rho_2+R_2+\ov a((z_1z_2)\circ\ov\rho_1(z_1,z_2))&=\cdots.
\end{align*}
Here the right-hand sides depend on coefficients of lower orders. Thus for $Q=(k,k)$, we have
$$
R_{1,kk}=\frac{a_{k}-(\mu_1\ov\mu_1^{-1})^k\ov a_{k}+e_{kk}}{(\mu_s\ov\mu_s^{-1})^k-1}. 
$$
Here $e_{kk}$ depends only on coefficients of $a_j$ with $j<k$. We will choose $a_k$ as follows. If $|e_{kk}|>1$, 
we choose $a_k=0$. If $|e_{kk}|\leq 1$, we choose an $a_k$ such that $|a_k|=1$ and 
$|a_{k}-(\mu_1\ov\mu_1^{-1})^k\ov a_{k}|=2$.  In both cases, we obtain
$$
|R_{1,kk}|\geq 
\frac{1}{|(\mu_1\ov\mu_1^{-1})^k-1|}.
$$
We can find $\mu_1$ such that $0<|(\mu_1\ov\mu_1^{-1})^k-1|\leq \frac{1}{k!}$ for a sequence of integer $k=k_j\to
\infty$. Furthermore, $\mu_1\ov\mu_1^{-1}$ is not a root of unity and $|\mu_1|\neq 1$.
 This shows that $R_1$ is divergent. 
\end{proof}

\begin{rem}
It is plausible that  there are $2^{h_*+s_*-1}$ attached formal complex submanifolds
to a generic $M$ that is a higher order perturbation of  non-resonant product
quadric and has the maximum number of deck transformations.
\end{rem}

 To study the existence of convergence of all attached formal manifolds, we   use the following theorem
in~\cite{stolo-bsmf} to conclude simultaneous convergence of all attached formal submanifolds. In fact the conclusion
is much more stronger. Here we recall the technique
of  linearization of $\sigma$ on the resonant ideal, i.e. the ideal generated by $\xi_1\eta_1,\ldots, \xi_p\eta_p$. 

For the convenience of the reader, we state the result only  for the family $F=\{F_1,\ldots, F_l\}$, where $F$ is  
$\{\sigma_1, \ldots, \sigma_p\}$, or a single mapping $\sigma$. 
 Recall that the linear part $D=\{D_i\colon 1\leq i\leq l\}$ of $F$ is  $\{S_1,\ldots, S_p\}$ or $S$.  The matrix of $D_i$
is diagonal, which is denoted by $\diag\mu_i$ for $\mu_i=(\mu_{i,1}, \ldots, \mu_{i,n})$.  
 Let $\cL I$ be a monomial ideal on $\cc^n$. Define 
$$
\omega_k(D,{\cL  I})=\inf\left\{\max_{1\leq i\leq l}|\mu_{i}^Q-\mu_{i,j}|\neq 0\colon
|\;2\leq |Q|\leq 2^k, 1\leq j\leq n,Q\in \nn^n, x^Q\not\in {\cL  I}\right\}
$$
where $\mu_i^Q:=\mu_{i,1}^{q_1}\cdots\mu_{i,n}^{q_n}$.
Let $\{\omega_k(D)\}_{k\geq 1}$ be the sequence of positive numbers defined by
$$
\omega_k(D)=\inf\left\{\max_{1\leq i\leq l}|\mu_{i}^Q-\mu_{i,j}|\neq 0\colon 
|\;2\leq |Q|\leq 2^k, 1\leq j\leq n,Q\in \nn^n \right\}.
$$
According to~\cite{stolo-bsmf}, we say that the family $D$ is {\it diophantine} (reps. on $\cL I$), if 
\eq{bruno-cond}
-\sum\frac{\log w_k(D)}{2^k}<\infty, \quad (resp. 
-\sum\frac{\log w_k(D,\cL I)}{2^k}<\infty). 
\eeq
When $D$ is reduced to one element and ${\mathcal I}=\{0\}$, this condition is Brjuno condition \cite{BR71,Ru02}. Let $\cL C^D$ denote the centralizer of the family $D$. 
 
 We now state the following theorem proved  in~\cite{stolo-bsmf}. 
\begin{thm}\label{theo-invariant}
Let ${\cL I}$ be a monomial ideal on $\cc^n$. Let $F$ be the above family of holomorphic mappings.  
Assume that the family $D$ is diophantine  on
${\cL  I}$.  Suppose that there is a formal mapping $\Phi$ satisfies the following: 
\bppp
\item
$\Phi$  is tangent to the identity and has a zero  projection on ${\cL  C}_D\cup \hat{\cL I}^n$, i.e.   
$\Phi=(\Phi_1,\ldots, \Phi_n)$ satisfy that $\Phi_{j,Q}x^Qe_j=0$ if $x^Q\in\cL I$ or $x^Qe_j\in\cL C_D$. 
\item $\Phi^{-1}F_i\Phi=D_i$ modulo $\hat{\cL I}^n$ for all $i$. 
\epp
Then $\Phi$ is convergent. 
\end{thm}

We apply the above theorem to $\Phi$ in $\cL C^c(S, ResI)$ and $\sigma$ which arises from a real analytic submanifold which
is a higher order perturbation of a non-resonant product quadric. Note that $\cL C_S$ is contained in the resonant ideal and
 the condition on the projection (ii) of the above
theorem is satisfied by the unique normalized map that linearizes $\sigma$ on $\hat{\cL I}^n$.
 
 As a corollary of the above theorem,
 we have the following result.
 \begin{cor} \label{stcor}
 Let $ResI$ be the resonant ideal of $S$. Assume that $\sigma$ satisfies the diophantine on $\cL I$. Then 
$\sigma$ is holomorphically linearizable on $\cL I$. In particular, if $\{\mu_1,\ldots, \mu_p\}$ is  non-resonant in $\zz^p$, i.e. 
$\mu^Q\neq1$ for all $Q\in\zz^p$ with $|Q|>0$,  then in suitable holomorphic coordinates, the
$\sigma$ is linear  and diagonal on the $(\xi_{i_1},\ldots, \xi_{i_s}, \eta_{i_{s+1}}, \ldots,\eta_{i_p})$-subspace
for any partition $\{i_{1}, \ldots, i_p\}=\{1,\ldots, p\}$. 
\end{cor}

  As a consequence of \nrc{stcor} and \rt{invep}, we obtain immediately \rt{allconv}, which we restate
  here in a stronger form.  
\begin{thm}
Let $M$ be a third order perturbation of a product quadric.  
Suppose that $M$  admits the maximum number of deck transformations and 
is non resonant. Suppose that $M$ has no elliptic component and that the eigenvalues of $\sigma$ satisfy diophantine 
 condition \rea{bruno-cond},  then
 all attached formal submanifolds are convergent. 
 Moreover, and the restrictions of $\sigma$ on these
invariant submanifolds are simultaneously linearizable by a single 
change of holomorphic coordinates of  the ambient space.
\end{thm}

As mentioned earlier, the eigenvalues of $\sigma$ are special. 
Let us  verify that the set of $\mu$ that satisfy the
diophantine  condition \re{bruno-cond} has the full measure. 
Recall that the resonant ideal  is 
generated by $\xi_1\eta_1,\ldots, \xi_p\eta_p$. 
Suppose that
$\xi^P\eta^Q$ is not in the ideal. Then $p_jq_j=0$ and $|p_i-q_j|=p_j+q_j$. We need to consider non-zero small divisors of the form
$$
 \boldsymbol{\mu}_{h_*}^P\boldsymbol{\mu}_{s_*}^Q \ov{\boldsymbol{\mu}}_{s_*}^R-\mu_j,\quad 1\leq j\leq p$$
for $(P,Q,R)\in\zz^p$.
Let $m$ be a positive number such that 
$$
\left||\boldsymbol{\mu}_{s_*}^Q \ov{\boldsymbol{\mu}}_{s_*}^R|-|\mu_j|\right|\geq \frac{1}{2}\min_j\{|\mu_j|,|\mu_j|^{-1}\}, \quad|Q+R|\geq m.
$$
Let us define
$$
\mu_s=r_s\nu_s, \quad r_s=|\mu_s|, \quad \mu_h=\nu_h, \quad r_h=1.
$$
Then we can write
$$
\mu^{P-Q}-\mu_j=\mu_j^{-1}\left(r_j^{-1}\nu_j^{-1}\prod_sr_s^{p_s-p_s'-q_s+q_s'}\cdot \prod\nu_h^{p_h-q_h}\prod\nu_s^{p_s+p'_s-q_s-q'_s}-1\right).
$$
Here $P=(p_h,p_s,p_s')$ and $Q=(q_h,q_s,q_s')$. We set 
$$
\prod_sr_s^{p_s-p_s'-q_s+q_s'}=r^{R'},\quad
\nu_j^{-1}\prod\nu_h^{p_h-q_h}\prod\nu_s^{p_s+p'_s-q_s-q'_s}=
\nu_j^{-1}\nu^R.
$$ 
Note that $|R'|\leq |P|+|Q|$ and $|R|\leq |P|+|Q|$. 
In view of  $$|\rho e^{i\theta}-1|^2=(r-1)^2+r\sin^2(\theta/2)\geq C\max\{|r-1|^2,|e^{i\theta}-1|\}$$ we obtain
$$
|\mu^{P-Q}-\mu_j|\geq Cr_j^{-1}\max\left\{|r_j^{-1}r^{R'}-1|, |\nu_j^{-1}\nu^{R}-1|\right\}.
$$
Now one can see that the set of $\mu=\{\mu_h,\mu_{h_*+s}, \ov\mu_{h_*+s}^{-1}\}$ that satisfies the diophantine 
condition \re{bruno-cond}
has the full measure.

Finally, we indicate a consequence of  $\sigma$ being linear on the zero set of the resonant
 ideal.  In this case the solutions $\{\ov\rho_1,\ov\rho_2\}$
to \re{zh2g}-\re{zh2g6} are linear and there are  $2^{h_*+s_*-1}$  pairs $\{\ov\rho_{j1},\ov\rho_{j2}\}$ of solutions.  Now \re{zh2g}-\re{zh2g6} imply that
$$
E(z',\ov\rho_1(z'))=-E(z',\ov\rho_2(z')), \quad \ov E(\ov\rho_1(z'),z')=-\ov E(\ov\rho_2(z'),z'). 
$$
The complex submanifold associated to $\{\rho_{i1},\rho_{i2}\}$ then has the form
$$
K_j\colon z_{p+i}=(L_i(z',\ov\rho_{j1}(z'))+E_i(z',\ov\rho_{j1}(z')))^2, \quad 1\leq i\leq p.
$$
Of course, there are additional hidden symmetries in $E$ for $\sigma$ to preserve the resonant ideal. On the other hand, $E$  can be quite general
as shown by the algebraic example (Example~\ref{texpl}).

\bibliographystyle{alpha}
\bibliography{gong-stolo1}

\def\cprime{$'$}
\begin{thebibliography}{BMR02}

\bibitem[AG09]{AG09}
P.~Ahern and X.~Gong.
\newblock Real analytic manifolds in {$\Bbb C^n$} with parabolic complex
  tangents along a submanifold of codimension one.
\newblock {\em Ann. Fac. Sci. Toulouse Math. (6)}, 18(1):1--64, 2009.

\bibitem[Art68]{artin68}
M.~Artin.
\newblock On the solutions of analytic equations.
\newblock {\em Invent. Math.}, 5:277--291, 1968.

\bibitem[BER97]{BER97}
M.~S. Baouendi, P.~Ebenfelt, and L.~Preiss Rothschild.
\newblock Parametrization of local biholomorphisms of real analytic
  hypersurfaces.
\newblock {\em Asian J. Math.}, 1(1):1--16, 1997.

\bibitem[BER00]{BER00}
M.~S. Baouendi, P.~Ebenfelt, and L.~Preiss Rothschild.
\newblock Convergence and finite determination of formal {CR} mappings.
\newblock {\em J. Amer. Math. Soc.}, 13(4):697--723 (electronic), 2000.

\bibitem[BG83]{bedford-gaveau}
E.~Bedford and B.~Gaveau.
\newblock Envelopes of holomorphy of certain {$2$}-spheres in {${\bf C}^{2}$}.
\newblock {\em Amer. J. Math.}, 105(4):975--1009, 1983.

\bibitem[BHV10]{BHV10}
P.~Bonckaert, I.~Hoveijn, and F.~Verstringe.
\newblock Local analytic reduction of families of diffeomorphisms.
\newblock {\em J. Math. Anal. Appl.}, 367(1):317--328, 2010.

\bibitem[Bis65]{Bi65}
E.~Bishop.
\newblock Differentiable manifolds in complex {E}uclidean space.
\newblock {\em Duke Math. J.}, 32:1--21, 1965.

\bibitem[BK91]{bedford-kling}
E.~Bedford and W.~Klingenberg.
\newblock On the envelope of holomorphy of a {$2$}-sphere in {${\bf C}^2$}.
\newblock {\em J. Amer. Math. Soc.}, 4(3):623--646, 1991.

\bibitem[BMR02]{BMR02}
M.~S. Baouendi, N.~Mir, and L.~Preiss Rothschild.
\newblock Reflection ideals and mappings between generic submanifolds in
  complex space.
\newblock {\em J. Geom. Anal.}, 12(4):543--580, 2002.

\bibitem[Brj71]{BR71}
A.~D. Brjuno.
\newblock Analytic form of differential equations. {I, II}.
\newblock {\em Trudy Moskov. Mat. Ob\v s\v c.}, 25:119--262 (1971); ibid. 26
  (1972), 199--239, 1971.

\bibitem[Bur13]{Bu13}
V.~Burcea.
\newblock A normal form for a real 2-codimensional submanifold in
  {$\mathbb{C}^{N+1}$} near a {CR} singularity.
\newblock {\em Adv. Math.}, 243:262--295, 2013.

\bibitem[Car32]{Ca32}
{\'E}.~Cartan.
\newblock Sur la g\'eom\'etrie pseudo-conforme des hypersurfaces de l'espace de
  deux variables complexes {II}.
\newblock {\em Ann. Scuola Norm. Sup. Pisa Cl. Sci. (2)}, 1(4):333--354, 1932.

\bibitem[Car33]{Ca33}
{\'E}.~Cartan.
\newblock Sur la g\'eom\'etrie pseudo-conforme des hypersurfaces de l'espace de
  deux variables complexes.
\newblock {\em Ann. Mat. Pura Appl.}, 11(1):17--90, 1933.

\bibitem[Cha86]{Ch86}
M.~Chaperon.
\newblock G\'eom\'etrie diff\'erentielle et singularit\'es de syst\`emes
  dynamiques.
\newblock {\em Ast\'erisque}, (138-139):1--440, 1986.

\bibitem[Chi89]{Ch89}
E.~M. Chirka.
\newblock {\em Complex analytic sets}, volume~46 of {\em Mathematics and its
  Applications (Soviet Series)}.
\newblock Kluwer Academic Publishers Group, 1989.
\newblock Translated from the Russian by R. A. M. Hoksbergen.

\bibitem[CM74]{chern-moser}
S.~S. Chern and J.~K. Moser.
\newblock Real hypersurfaces in complex manifolds.
\newblock {\em Acta Math.}, 133:219--271, 1974.

\bibitem[Cof06]{Co06}
A.~Coffman.
\newblock Analytic stability of the {CR} cross-cap.
\newblock {\em Pacific J.\ Math.}, 226(2):221--258, 2006.

\bibitem[Gon94]{Go94}
X.~Gong.
\newblock On the convergence of normalizations of real analytic surfaces near
  hyperbolic complex tangents.
\newblock {\em Comment. Math. Helv.}, 69(4):549--574, 1994.

\bibitem[Gon96]{Go96}
X.~Gong.
\newblock Divergence of the normalization for real {L}agrangian surfaces near
  complex tangents.
\newblock {\em Pacific J. Math.}, 176(2):311--324, 1996.

\bibitem[Gon12]{Go12}
X.~Gong.
\newblock Existence of divergent {B}irkhoff normal forms of {H}amiltonian
  functions.
\newblock {\em Illinois J. Math.}, 56(1):85--94, 2012.

\bibitem[Gun90]{Gunning2}
R.C. Gunning.
\newblock {\em Introduction to holomorphic functions of several variables.
  {V}ol. {II}}.
\newblock The Wadsworth \& Brooks/Cole Mathematics Series. Wadsworth \&
  Brooks/Cole Advanced Books \& Software, Monterey, CA, 1990.
\newblock Local theory.

\bibitem[GW05]{GW05}
T.~Gramchev and S.~Walcher.
\newblock Normal forms of maps: formal and algebraic aspects.
\newblock {\em Acta Appl. Math.}, 87(1-3):123--146, 2005.

\bibitem[HK95]{HK95}
X.~Huang and S.G. Krantz.
\newblock On a problem of {M}oser.
\newblock {\em Duke Math. J.}, 78(1):213--228, 1995.

\bibitem[Hua98]{H98}
X.~Huang.
\newblock On an $n$-manifold in {${\bf C}^n$} near an elliptic complex tangent.
\newblock {\em J. Amer. Math. Soc.}, 11(3):669--692, 1998.

\bibitem[HY09a]{HY09}
X.~Huang and W.~Yin.
\newblock A {B}ishop surface with a vanishing {B}ishop invariant.
\newblock {\em Invent. Math.}, 176(3):461--520, 2009.

\bibitem[HY09b]{HY09b}
X.~Huang and W.~Yin.
\newblock A codimension two {CR} singular submanifold that is formally
  equivalent to a symmetric quadric.
\newblock {\em Int. Math. Res. Not. IMRN}, (15):2789--2828, 2009.

\bibitem[HY12]{HY12}
X.~Huang and W.~Yin.
\newblock Flattening of a {CR} singular point in a codimension two real
  submanifold in {$\mathbb C^n+1$}, 2012.
\newblock arXiv:1210.5146 [math.CV].

\bibitem[JL13]{JL13}
R.~Juhlin and B.~Lamel.
\newblock On maps between nonminimal hypersurfaces.
\newblock {\em Math. Z.}, 273(1-2):515--537, 2013.

\bibitem[Kli85]{Kl85}
W.~Jr. Klingenberg.
\newblock Asymptotic curves on real analytic surfaces in {${\bf C}^2$}.
\newblock {\em Math. Ann.}, 273(1):149--162, 1985.

\bibitem[KS13]{KS13}
I.~Kossovskiy and R.~Shafikov.
\newblock Divergent {CR}-equivalences and meromorphic differential equations,
  2013.
\newblock arXiv:1309.6799 [math.CV].

\bibitem[KW82]{KW82}
C.E. Kenig and S.M. Webster.
\newblock The local hull of holomorphy of a surface in the space of two complex
  variables.
\newblock {\em Invent. Math.}, 67(1):1--21, 1982.

\bibitem[KW84]{KW84}
C.E. Kenig and S.M. Webster.
\newblock On the hull of holomorphy of an $n$-manifold in {${\bf C}^n$}.
\newblock {\em Ann. Scuola Norm. Sup. Pisa Cl. Sci. (4)}, 11(2):261--280, 1984.

\bibitem[Mos56]{moser-hyperbolic}
J.~Moser.
\newblock The analytic invariants of an area-preserving mapping near a
  hyperbolic fixed point.
\newblock {\em Comm. Pure Appl. Math.}, 9:673--692, 1956.

\bibitem[Mos73]{moser-princeton}
J.~Moser.
\newblock {\em Stable and random motions in dynamical systems}.
\newblock Princeton University Press, Princeton, N. J.; University of Tokyo
  Press, Tokyo, 1973.
\newblock With special emphasis on celestial mechanics, Hermann Weyl Lectures,
  the Institute for Advanced Study, Princeton, N. J, Annals of Mathematics
  Studies, No. 77.

\bibitem[Mos85]{moser-zero}
J.~Moser.
\newblock Analytic surfaces in {${\bf C}\sp 2$} and their local hull of
  holomorphy.
\newblock {\em Ann. Acad. Sci. Fenn. Ser. A I Math.}, 10:397--410, 1985.

\bibitem[MW83]{MW83}
J.~Moser and S.M. Webster.
\newblock Normal forms for real surfaces in {${\bf C}^{2}$} near complex
  tangents and hyperbolic surface transformations.
\newblock {\em Acta Math.}, 150(3-4):255--296, 1983.

\bibitem[OZ11]{OZ11}
A.G. O'Farrell and D.~Zaitsev.
\newblock Formally-reversible maps of {$\mathbb C^2$}, 2011.
\newblock arXiv:1111.6984 [math.CV].

\bibitem[PM03]{PM03}
R.~P{\'e}rez-Marco.
\newblock Convergence or generic divergence of the {B}irkhoff normal form.
\newblock {\em Ann. of Math. (2)}, 157(2):557--574, 2003.

\bibitem[P{\"o}s86]{Po86}
J.~P{\"o}schel.
\newblock On invariant manifolds of complex analytic mappings near fixed
  points.
\newblock {\em Exposition. Math.}, 4(2):97--109, 1986.

\bibitem[R{\"u}s02]{Ru02}
H.~R{\"u}ssmann.
\newblock Stability of elliptic fixed points of analytic area-preserving
  mappings under the bruno condition.
\newblock {\em Ergodic Theory Dynam. Systems}, 22(5):1551--1573, 2002.

\bibitem[Sie41]{siegel-integral}
C.~L. Siegel.
\newblock On the integrals of canonical systems.
\newblock {\em Ann. of Math. (2)}, 42:806--822, 1941.

\bibitem[Sie54]{siegel-divergent}
C.L. Siegel.
\newblock \"{U}ber die {E}xistenz einer {N}ormalform analytischer
  {H}amiltonscher {D}ifferentialgleichungen in der {N}\"ahe einer
  {G}leichgewichtsl\"osung.
\newblock {\em Math. Ann.}, 128:144--170, 1954.

\bibitem[SM71]{siegel-moser-book}
C.~L. Siegel and J.~K. Moser.
\newblock {\em Lectures on celestial mechanics}.
\newblock Springer-Verlag, New York-Heidelberg, 1971.
\newblock Translation by Charles I. Kalme, Die Grundlehren der mathematischen
  Wissenschaften, Band 187.

\bibitem[Sto00]{St00}
L.~Stolovitch.
\newblock Singular complete integrability.
\newblock {\em Inst. Hautes \'Etudes Sci. Publ. Math.}, (91):133--210 (2001),
  2000.

\bibitem[Sto05]{stolo-annals}
L.~Stolovitch.
\newblock Normalisation holomorphe d'alg\`ebres de type {C}artan de champs de
  vecteurs holomorphes singuliers.
\newblock {\em Ann. of Math. (2)}, 161(2):589--612, 2005.

\bibitem[Sto07]{St07}
L.~Stolovitch.
\newblock Family of intersecting totally real manifolds of {$(\mathbb C^n,0)$}
  and {CR}-singularities, 2007.
\newblock preprint, p.1--30.

\bibitem[Sto13]{stolo-bsmf}
L.~Stolovitch.
\newblock Family of intersecting totally real manifolds of {$(\mathbb C^n,0)$}
  and germs of holomorphic diffeomorphisms, 2013.
\newblock p. 1-15, to appear in Bull. Soc. math. France.

\bibitem[Tan62]{Ta62}
N.~Tanaka.
\newblock On the pseudo-conformal geometry of hypersurfaces of the space of
  {$n$}\ complex variables.
\newblock {\em J. Math. Soc. Japan}, 14:397--429, 1962.

\bibitem[Web92]{We92}
S.~M. Webster.
\newblock Holomorphic symplectic normalization of a real function.
\newblock {\em Ann. Scuola Norm. Sup. Pisa Cl. Sci. (4)}, 19(1):69--86, 1992.

\end{thebibliography}
\end{document}